\documentclass[11pt]{article}
\usepackage{xcolor}
\usepackage[colorlinks=true, linkcolor=blue, citecolor=red]{hyperref}
\usepackage{amsmath}
\usepackage{amssymb}
\usepackage{mathrsfs}
\usepackage{bbm}
\usepackage{amsfonts}
\usepackage{amsthm}
\usepackage{bm}
\usepackage[numbers,sort&compress]{natbib} 
\usepackage{enumerate}

\numberwithin{equation}{section}
\theoremstyle{plain}
\newtheorem{theorem}{Theorem}[section]

\newtheorem{lemma}{Lemma}[section]
\newtheorem{prop}{Proposition}[section]
\newtheorem{remark}{Remark}[section]
\newtheorem{definition}{Definition}[section]

\newcommand{\R}{\mathbb{R}}

\newcommand{\E}{\mathbb{E}}

\def \cald {{  {\mathcal{D}}  }}

\def\bt{\begin{theorem}}
\def\et{\end{theorem}}
\def\bl{\begin{lemma}}
\def\el{\end{lemma}}
\def\br{\begin{remark}}
\def\er{\end{remark}}
\def\l^2{\ell^2}

\def\5{$}

\def\be{\begin{equation}}
\def\ee{\end{equation}}

\allowdisplaybreaks
\everymath{\displaystyle}

\usepackage{lineno} 

\begin{document}

\title{
 {                    \large\bf { 
   Existence and vanishing noise limit of measure attractors for McKean-Vlasov $p$-Laplacian lattice systems with delay driven by L\'{e}vy noise
}}}
 \author{Zhang Chen $^{a,}$\thanks{E-mail address: zchen@sdu.edu.cn} , \ \
        Xiaoxiao Sun $^{a,}$\thanks{E-mail address: xsun@mail.sdu.edu.cn} , \ \
        Bixiang Wang $^{b,}$\thanks{E-mail address: bwang@nmt.edu}       
\\
\normalsize{$^a$ School of Mathematics, Shandong University, Jinan 250100,  China} \\
\normalsize{$^b$ Department of Mathematics, New Mexico Institute of Mining and Technology,} \\
\normalsize{Socorro,  NM~87801, USA}\\
}
\date{}
\maketitle

\begin{minipage}{14cm}
{\bf Abstract.}
 This paper is concerned with the existence and
 the  limiting behavior of
 measure attractors
 of distribution laws of 
 the  
 solution segment process  for
 the  McKean-Vlasov stochastic $p$-Laplace lattice system
 with time delay
 driven by    L\'{e}vy noise.
 The nonlinear drift and  diffusion terms are
 allowed to have
  superlinear growth.
  Due to time delay,
  the Skorohod metric space
  is employed to
  describe the trajectories
  of the solutions
  with jumps. 
 We first prove the
 existence and uniqueness
 of  c\`{a}dl\`{a}g solutions for the lattice system,
 and then define a   non-autonomous cocycle
 acting on the Borel
 probability measures
 in the Skorohod space.
 This cocycle is continuous
 in bounded subsets
 of the space of probability
 measures only when time is
 sufficiently large.
 We then prove the existence
 of pullback absorbing sets
 and the asymptotic compactness
 of the cocycle as well as
 the existence and uniqueness
 of pullback 
 measure attractors.
 We finally investigate
 the limiting behavior
 of 
 measure
  attractors
 of the lattice system
 as the noise intensity approaches zero,
 and establish the 
     optimal convergence rate of singleton measure attractors in
     the  Wasserstein distance 
      of order $\theta$.
 
\vspace{2mm}


\vspace{2mm}

\end{minipage}

\tableofcontents

\section{Introduction}
 In this paper, we 
 study the  McKean-Vlasov stochastic $p$-Laplace delay lattice system defined on the integer set $\mathbb{Z}$ with superlinear drift and diffusion terms driven by L\'{e}vy noise:
  \begin{equation} \label{lmvlds-1}
  \left\{
  \begin{aligned}
   & du_i(t)
   + (Au)_i (t)
            = - \lambda u_i(t) dt
          - f_i(u_i(t), \mathcal{L}_{u(t)}) dt
          + g_i(t) dt  
          + F_i(u_i(t - \rho),\mathcal{L}_{u(t - \rho)}) dt \\
   & \ \ \
       + \sqrt{\varepsilon} \sum_{k \in \mathbb{N} }
               \left( h_{k,i}(t) + \sigma_{k,i}(u_i(t), \mathcal{L}_{u(t)})
               \right) dW_k(t) \\
   & \ \ \
       + \sqrt{\varepsilon} \int_{y \in \mathbb{Y}}
             \left( \widetilde{h}_{i}(t,y) + \widetilde{\sigma}_{i}(u_i(t-),\mathcal{L}_{u(t)},y)
             \right) \widetilde{N}(dt,dy),  \ \ \  t > \tau,  \\
   & u_i(s) = \varphi_i(s-\tau),\  s \in [\tau-\rho, \tau],
  \end{aligned}
  \right.
\end{equation}
 where
 $i \in \mathbb{Z}$,
 $\lambda$ is a positive constant,  
 $\rho \in (0,1]$ is a time delay parameter,
 $u_i(t-) = \lim_{s \uparrow t} u_i(s)$,
 $\mathcal{L}_{u(t)} $ is the probability distribution of $u(t) = (u_i(t))_{i \in \mathbb{Z}}$;
 $g(t) = (g_i(t))_{i \in \mathbb{Z}}$, $h(t) = (h_{k,i}(t))_{k \in \mathbb{N}, i \in \mathbb{Z}} $
 and $ \widetilde{h}(t,y) = (\tilde{h}_{i}(t,y))_{ i \in \mathbb{Z}}$ are  given    sequences
 in $\ell^2$; 
 for every $k \in \mathbb{N}$ and $i \in \mathbb{Z}$,
 $f_i$, $\sigma_{k,i}$
 and  $\widetilde{\sigma}_{i}$ are superlinear functions
 and   
 $F_i$ is a Lipschitz   function;
 $\varepsilon \in (0,1]$ is  a noise intensity parameter,
 $\{W_k\}_{k \in \mathbb{N}}$ is a sequence of independent real-valued two-sided Wiener processes on a complete filtered probability space
 $(\Omega, \mathcal{F}, \{\mathcal{F}_t\}_{t \in \mathbb{R}}, \mathbb{P})$
 satisfying the usual conditions;
 and
 $\widetilde{N} (dt, dy)
 = {N} (dt, dy) -\nu (dy) dt
 $ 
 is an independent Poisson martingale measure with Poisson random measure $N$ 
 and $\sigma$-finite characteristic measure $\nu$ on the measure space
 $(\mathbb{Y}, \mathcal{Y}, \nu)$.
 The operator $A$ in \eqref{lmvlds-1}
 is the discrete
 $p$-Laplace operator
 with $p>2$ which is
 given by:
 for every 
 $ u =(u_i)_{i
 	\in \mathbb{Z}}
 \in \ell^2$,
   \be\label{mar02a1}
   (Au)_i
   =  |u_{i}  - u_{i-1}|^{p-2} \left( u_{i}  - u_{i-1}\right)
- |u_{i+1} - u_{i}|^{p-2} \left( u_{i+1}  - u_i \right)
  ,
\quad \forall \  i
\in \mathbb{Z}.
\ee

Indeed, the 
operator 
$A$ in \eqref{mar02a1} 
corresponds to the
discretization of the 
continuous $p$-Laplace
operator on 
$\mathbb{R}$ 
which is given by
 $$
   Au
   =    - \frac{\partial}{\partial x}                               \left(  \left|\frac{\partial u(x)}{\partial x} \right|^{p-2}
    \frac{\partial u(x)}{\partial x}
         \right) ,
         \quad \forall\  x\in \mathbb{R} .
$$
 The $p$-Laplace equation   describes  the type of diffusion depending on the gradient and is closely related to the heterogeneous medium or turbulence \cite{dt1994SIMA, ev1988ARMA}.
 For 
 the stochastic $p$-Laplace equations, 
 the well-posedness and invariant measures  have been investigated 
 in    \cite{sz2023SAA, bwang2026DCDSB, wliu2009JMAA, gt2016SIMA, wct2024JDE}
 for white noise
 and
 in    
  \cite{blz2014NA, neelima2022PA, km2025JEE} for
  L\'{e}vy noise.

  When modeling the solidification of alloys or neural networks with chains of coupled oscillators, the underlying spatial medium is inherently discrete rather than continuous.
  So it is significant to investigate lattice systems in 
 this case   instead of spatially continuous equations.  
  Moreover, lattice models also have wide applications in image processing \cite{cr1993IEEE}, pattern recognition \cite{cm1995IEEE}, climate science \cite{MT2016CPAM}, and nonlinear optics \cite{cls2003Nature}. 
  Recently, random attractors and invariant measures for stochastic $p$-Laplacian lattice systems
  driven by Brownian motion
   have been studied in \cite{WW2020SPA, WG2025Non} without  delay and 
   in \cite{ZL2025AMO} with  delay.
 The time    delay reflects the dependence of 
 physical systems on the past states due to   signal propagation, processing speeds and reaction time of the systems.   
  For  lattice systems
   and partial differential equations with delay effect,   the 
  asymptotic behavior of solutions
  such as 
  the  global 
  attractors and invariant measures  
   has been 
  studied in  \cite{CMV2014DCDS, LYC2023SINUM, XCV2024QTDS, CW2025SPDEAC} and the references therein.
  
 The stochastic lattice system \eqref{lmvlds-1} depends not only on the state of the solution, but also on the probability distribution of the solution. Such  equations are called the McKean-Vlasov stochastic equations or the mean-field equations \cite{blpr2017AP, cdll2019}, which  are used to describe the limit of interacting particle systems \cite{lwz2021CMP}.
 The
 McKean-Vlasov stochastic equations
 without delay
  have been
  extensively  
  studied in 
   \cite{ad1995SPA, W2018SPA, wr2025WSP, HDS2021AP, RZ2021Ber, LMW2021AIHP, W2023Ber,HL2021AMO, LSZZ2023PA}.
   Recently, 
   the solutions and the asymptotic
   behavior of 
  the McKean-Vlasov   delay equations driven by Brownian motion
   have been
   investigated in
   \cite{HRW2019DCDS, bsy2022EJP, SSL2024arXiv} and the references therein.
   
   As far as the authors are aware,
   it seems that there is no result in the literature
   on the dynamics
   of the  McKean-Vlasov    equations
   with time delay
    driven by 
    L\'{e}vy noise.
    The purpose of the present  paper
    is to fill
     this gap and study the
    existence, uniqueness and zero-noise
    limit of measure attractors
    of the delay
    $p$-Laplace lattice system
    \eqref{lmvlds-1} with  L\'{e}vy noise.

      The 
      concept of measure attractor 
      was introduced 
      in \cite{S1991MN}
      for the 
        autonomous stochastic Navier-Stokes 
        equations in bounded domains
         (see also   \cite{CC1998EJP}).  
        This concept was extended to
        the non-autonomous case and studied
        in
        \cite{LW2024JDE} for
        the  non-autonomous stochastic reaction-diffusion equation on bounded thin domain  and  in 
   \cite{CWZ2025JNS}     
   for the  superlinear stochastic Schr\"{o}dinger 
   delay lattice system. 
 In these works, 
 the  stochastic equation
 does not  depend on the distributions of solutions.    
 Recently, 
 the  measure attractors
 have been 
 investigated 
 in  \cite{SSW2026JDE}
  for the  McKean-Vlasov 
  stochastic   reaction-diffusion equation
  in unbounded domain driven by Brownian motion,
  and in 
  \cite{BCS2026JDE}  for
  the  McKean-Vlasov stochastic lattice system
  driven by L\'{e}vy noise.

 In contrast to \cite{SSW2026JDE, BCS2026JDE},  
 a striking feature of \eqref{lmvlds-1} is that 
 the
 nonlinearity  depends on the delay term $u(t-\rho)$ and its probability distribution,
 and  hence the arguments of
  \cite{SSW2026JDE, BCS2026JDE}
  for the McKean-Vlasov  equations without delay
  cannot be applied to   
  the McKean-Vlasov  delay system \eqref{lmvlds-1}.
  On the other hand,
  since system \eqref{lmvlds-1}
  is driven by 
  L\'{e}vy noise,
  the trajectory of the solution
  is discontinuous, and thus
  the arguments of 
  \cite{SSL2024arXiv, CWZ2025JNS}
 for  the   delay 
 equations driven by
 Brownian motion 
 do not work for   \eqref{lmvlds-1}.

 Indeed, the 
   discontinuity of solutions 
   of \eqref{lmvlds-1}
   is a major obstacle for studying
   the existence of
   measure attractors
   of the non-autonomous
   dynamical system
   associated with the
   segment process of 
   solutions as explained below.

   The trajectory of the solution
   of \eqref{lmvlds-1}
   is 
    c\`adl\`ag; that is,
    it is right-continuous with left-hand limits,
    which implies that
    the segment process of solutions
    lives  in the 
    Skorohod space
     $ 
    ( D([-\rho,0], \ell^2), \ d^0)$
    of   c\`adl\`ag functions from 
    $[-\rho,0]$ to $ \ell^2$
    rather than
    $  ( C([-\rho,0], \ell^2), \| \cdot \|_{\infty})$
    of continuous functions
     from 
    $[-\rho,0]$ to $ \ell^2$,
    where
    $d^0$ is the 
     Skorohod metric 
     in  $ 
  D([-\rho,0], \ell^2)$
  and  $  \| \cdot \|_{\infty} $
  is the uniform norm
  in  $ C([-\rho,0], \ell^2)$.
  Note that the 
     Skorohod metric 
     $d^0$ is quite different from 
     $\|\cdot \|_{\infty}$.
     For example,
     if a sequence converges
     in 
   $  ( C([-\rho,0], \ell^2), \| \cdot \|_{\infty})$,
   then it also converges
   in  $  ( C([s, 0], \ell^2), \| \cdot \|_{\infty})$
   for any $-\rho <s <0$.
   However, this property does not hold
   true 
   in the space
    $ 
    ( D([-\rho,0], \ell^2), \ d^0)$;
    that is,
   the convergence of a  sequence
    in
     $ 
    ( D([-\rho,0], \ell^2), \ d^0)$
    does  not necessarily  
  imply the convergence
  of the sequence
  in  $ 
    ( D([s,0], \ell^2), \ d^0)$
    for $ -\rho <s <0$.
      In other words,
    the convergence in
      $ 
    ( D([-\rho,0], \ell^2), \ d^0)$
    does not imply the 
    convergence
    in  subintervals
     of $[-\rho, 0]$
    in terms of the Skorohod metric
    $d^0$,
    which
    introduces a fundamental
    problem for the continuity
    of distributions of
    the segment process
    of solutions 
    with respect to
    the distributions of initial data
    in
     $ 
    ( D([-\rho,0], \ell^2), \ d^0)$.
   Indeed,
   it is impossible
   for the dynamical system associated with \eqref{lmvlds-1}
   defined in the space of probability measures
   in $( D([-\rho,0], \ell^2), \ d^0)$
   to be continuous for all time,
   and hence we are unable to  
   {directly
    apply \cite[Proposition 2.2]{SSW2026JDE} and \cite[Proposition 3.6]{bwang2012JDE} to obtain the existence of measure attractors} for \eqref{lmvlds-1}
   due to the fact that such invariance depends heavily on the continuity of the dynamical system for all time.

    Despite this fact,
    we  observe that the dynamical system for \eqref{lmvlds-1} is still continuous 
    when time is sufficiently large
    (see Theorem \ref{consys}).
    We carefully employ 
     this kind of continuity
     and the cocycle property
     of the dynamical system to successfully
  prove
  that  the measure attractor of \eqref{lmvlds-1} is 
  {invariant for all time
     (see Theorem \ref{pis} and Remark \ref{difference}).
     }

    Another difficulty of the paper
    is caused by the superlinear growth
    of the coefficient of  L\'{e}vy noise.
    In order to establish the existence of
    measure attractors, we need to
    construct a closed pullback
    absorbing set for system
     \eqref{lmvlds-1},  
     for which we must derive
     the higher-order uniform moment
     estimates of solutions.
     Both   coefficients of the white noise and
     the 
     L\'{e}vy noise
     of \eqref{lmvlds-1}
      are superlinear, but the
      uniform estimates on  the 
      integral against
     L\'{e}vy noise are much more
     involved than the integral against 
    white noise.
    A detailed analysis on the nonlinear coefficient
    is required to obtain the uniform estimates
    on the  L\'{e}vy  integral
    (see Lemma \ref{pue}).
    
    The superlinear  L\'{e}vy noise
    also introduces a major difficulty
    for proving the pullback
    asymptotic compactness
    of solutions
    in the  Skorohod space
    $( D([-\rho,0], \ell^2), \ d^0)$,
    for which we need to
    establish the 
    equi-continuity of solutions
    in probability.
    For stochastic equations
    driven by Brownian motion,
    such continuity can be 
    obtained by 
    the   higher-order
    uniform  moment  estimates of solutions,
    see  e.g., \cite{CWZ2025JNS, wwm2019SCL, ll2022DCDSB, csw2025PRSEA}.
    However,
    this method does not work
    for stochastic equations driven by superlinear
     L\'{e}vy noise.
     In order to prove the equi-continuity of
     solutions in probability and hence the 
     pullback asymptotic compactness
     of  \eqref{lmvlds-1} with superlinear
    L\'{e}vy noise,
    we will
    first
    employ  Jakubowski's theorem
     \cite{J1986AIHP}
     to reduce the infinite-dimensional
     problem to a one-dimensional problem
     and then apply Aldou's theorem
     \cite{,A1978AOP}
     to obtain the desired results
 (see Lemma \ref{pac}).
 We mention  that
 the argument of this paper 
 is 
          different from
          that in
           \cite{RRG2006SPA, bc2016Stoch} where the tightness of solutions was 
           established  for delay   equations with linearly
           growing   L\'{e}vy noise. 
 Based on 
 the existence of
 closed pullback
 absorbing sets
 and asymptotic
 compactness of solutions,
 we then conclude
 the existence and uniqueness
 of measure attractors
 for system  \eqref{lmvlds-1}
 in the space of probability measures
 in    $ 
    ( D([-\rho,0], \ell^2), \ d^0)$
     (see Theorem \ref{pis}).

 We 
 remark that
 the method of this paper
  is  applicable to 
  a  wide class of
   stochastic delay equations driven by 
   either Brownian motion or L\'{e}vy noise
   (or both).
   In particular,  
if  the stochastic equation
 \eqref{lmvlds-1} doesn't depend on the distribution of solutions, 
 then 
 by virtue of Lemma \ref{pac}, 
 one can obtain the existence of invariant measures   
  by   Krylov-Bogolyubov's argument
   as in \cite{RRG2006SPA, bc2016Stoch}.
  
The last goal of this paper is
to investigate the limiting behavior
of measure attractors
of  \eqref{lmvlds-1} 
as the noise intensity
approaches zero.
We will also establish the optimal
convergence rate
in terms of the noise intensity
  when the measure attractor
is a singleton.
 
We mention that
the   upper semi-continuity  of   measure attractors with respect to time delay  
and noise intensity
has been explored in
 \cite{CWZ2025JNS} and   \cite{mll2024AMO},
 respectively,
 for the distribution independent
   stochastic equations driven by Brownian motion.
   In this paper, we deal with system
   \eqref{lmvlds-1}    depending 
   on distributions of solutions
   driven by 
   L\'{e}vy noise.
   
   We will first
   establish the convergence rate of 
   solutions of
   \eqref{lmvlds-1} 
   in finite time interval as
   $\varepsilon
   \to 0$ in Lemma~\ref{solcon},
   and then  
   prove the convergence of pullback measure attractors  
 in Theorem \ref{uppercont}.  
 Under further dissipative conditions,
 we show the measure attractor is  
 a singleton
  in Theorem \ref{sin-attractor},
  which is actually 
  an evolution family of 
  pullback mixing
  probability measures.
  Finally, in Theorem \ref{conv-rate},
  by the convergence of solutions in $L^2(\Omega, D_\rho)$,
  we  establish the optimal convergence rate of 
  singleton
   measure attractors  in
   terms of the  Wasserstein distance of order $\theta$ with respect to $\varepsilon$,
 which generalizes the convergence rate of invariant measures \cite{CW2024AML} to 
 the case of evolution family of probability measures.
 The reader is referred to \cite{DR2007, LL2025CMP, WCT2023PAMS} for more details
 on the existence and  limiting    behavior of   evolution systems of probability measures
 for stochastic equations.

This paper is organized as follows. 
 In Section~\ref{2},
  we introduce  preliminaries and 
 discuss the  assumptions
 on the nonlinear terms
 of    \eqref{lmvlds-1}.
  Section~\ref{3} is devoted to the well-posedness of \eqref{lmvlds-1} as well
  as the uniform esitmates of solutions. 
In Section~\ref{4},
 we 
prove the existence of pullback absorbing sets,
  pullback asymptotic compactness 
 and the existence of pullback
 measure attractors.
 In Section~\ref{5},
  we discuss
  the zero-noise limit of  measure attractors 
  of \eqref{lmvlds-1},
  and establish the optimal convergence rate in
  terms of the Wasserstein distance of order $\theta$ 
   with respect to noise intensity.  

 For convenience, throughout this paper, we denote by
 $ \ell^r =  \{ u = (u_i)_{i \in \mathbb{Z} } : \sum\limits_{i \in \mathbb{Z}} |u_i|^r < \infty  \}$
  with   norm $\|u\|_r =  ( \sum\limits_{i \in \mathbb{Z}} |u_i|^r  )^{ \frac{1}{r}
 }$ for  $r>0$; and 
  $\ell^\infty =  \{ u = (u_i)_{i \in \mathbb{Z} } : \sup_{i \in \mathbb{Z}} |u_i| <  \infty  \}$
  with norm
   $\|\cdot\|_{\ell^\infty}$.
  If $r = 2$, we
  will write  the norm and the inner product of $\ell^2$ 
  as  $\|\cdot\|$ and $\langle \cdot, \cdot \rangle$, respectively.

\section{Preliminaries}\label{2}

 In this section, we first recall
 the  compensated Poisson random measures and  the the non-autonomous
 dynamical systems
 defined on probability measure spaces.
 We then introduce the assumptions on 
 the nonlinear terms   in \eqref{lmvlds-1},
 which will be frequently used in the sequel.
 
Let  $(\Omega, \mathcal{F}, \{\mathcal{F}_t\}_{t \in \mathbb{R}}, \mathbb{P})$
be a
  complete filtered probability space
  satisfying the usual conditions, and 
 $\{W_k\}_{k \in \mathbb{N}}$
 be  a sequence of independent real-valued two-sided Wiener processes
 on  $(\Omega, \mathcal{F}, \{\mathcal{F}_t\}_{t \in \mathbb{R}}, \mathbb{P})$.
 Suppose $(\mathbb{Y}, \mathcal{Y}, \nu)$ is a $\sigma$-finite measure space,
 and $N$ is a Poisson random measure on 
 $(\mathbb{R}^+ \times \mathbb{Y}, \mathcal{B} (\mathbb{R}^+) \times \mathcal{Y},  
 dt \times \nu)$,
  which is also referred to as a Poisson random measure on
  $( \mathbb{Y},  \mathcal{Y}, \nu)$ with characteristic measure $\nu$.
   The random measure
    $\widetilde{N} (dt, dy)
      = {N} (dt, dy) -\nu (dy) dt$
  is called the compensated Poisson random measure of $N$.
  We assume that
  $\widetilde{N} $
  and the sequence $\{W_k\}_{k\in \mathbb{N}}$ are independent, 
  and $\widetilde{N} $ is a martingale measure on 
   $(\Omega, \mathcal{F}, \{\mathcal{F}_t\}_{t \in \mathbb{R}}, \mathbb{P})$. 
 We refer the reader to \cite{APPL2009} for more details on Poisson random measures and 
 martingale-valued measures.

 Denote by $   D([a,b], \ell^2) $
 the space of $ \ell^2 $-valued right-continuous functions with left-hand limits (c\`{a}dl\`{a}g)
with norm
$\|u\|_{\infty} = \sup_{t \in [a,b]} \|u(t)\| $ for $ u \in    D([a,b], \ell^2) $.
 Then  $ (  D([a,b], \ell^2), \|\cdot\|_{\infty}) $ is a  Banach space  but  it is not separable.

 In the sequel, 
 we will also use the Skorohod metric $ d^0 $ in $D([a,b], \ell^2)$ which is given by:
 for all $ u,v \in D([a,b], \ell^2)$,   
 \begin{align*}
       d^0(u,v) := \inf_{\lambda \in \Lambda}
                               \{ \|\lambda\|^0 \vee \| u - v \circ \lambda \|_{ \infty } \},
 \end{align*}
 where $ \Lambda $  is the set of all strictly increasing and continuous mappings of $[a,b]$ onto itself
 and $ \|\lambda\|^0 := \sup_{s,t \in [a,b], s \neq t}
                                  \left| \ln \frac{\lambda(t)-\lambda(s)}{t-s} \right| < +\infty $.
 Then the metric space $( D([a,b], \ell^2), d^0)$ is complete and separable \cite{B2023BOOK},
 and $d^0(u,0) = \|u\|_{\infty }$
 for $u\in  D([a,b], \ell^2) $.
 
 Throughout this paper, 
 when we  consider  
 $  D([a,b], \ell^2) $ as a measurable space,
 we always mean  
 $  D([a,b], \ell^2) $ is equipped with its  
  Borel $\sigma$-algebra
 $\mathscr{B}( D([a,b], \ell^2))$   under the Skorohod metric $ d^0 $.

 Next, we recall the theory of cocycles on the space of probability measures.
 Denote by $ D_{\rho} = D([-\rho,0], \ell^2) $ if $ \rho > 0 $, and $ D_{\rho}= \ell^2 $ if $ \rho = 0 $.
 Let
  $ \mathcal{P}(D_{\rho})$ be 
  the space of probability measures on $\left( D_{\rho}, \mathscr{B}(D_{\rho}) \right)$   
 where $\mathscr{B}(D_{\rho})$ is the Borel $\sigma$-algebra of $ D_{\rho} $ under the Skorohod metric $ d^0 $.
 Denote by $L_b(D_{\rho})$ the set of all bounded Lipschitz functions on $D_{\rho}$.
 Let
\begin{align*}
   \|\psi\|_{Lip} = \sup_{x \neq y}
                       \frac{| \psi(x) - \psi(y)|}{d^0(x,y)}, \ \ \
   \text{and}\
   \|\psi\|_{L_b} = \sup_{ x \in D_{\rho}} |\psi(x)| + \|\psi\|_{Lip}.
\end{align*}

 Define a metric on $ \mathcal{P}(D_{\rho}) $ corresponding to the Fortet-Mourier norm by
 $$ d_{\mathcal{P}(D_{\rho})} (\mu_1, \mu_2)
    = \sup_{\psi \in L_b(D_{\rho}), \|\psi\|_{L_b} \le 1}
        \left| ( \psi, \mu_1 ) - ( \psi, \mu_2 ) \right|, \ \  \text{for}\  \mu_1, \mu_2 \in \mathcal{P}(D_{\rho}).$$
 Then $\left( \mathcal{P}(D_{\rho}), d_{ \mathcal{P}( D_{\rho}) } \right)$ is a Polish space \cite{AJ2023},
 and $\mu_k \to \mu$ in $\left( \mathcal{P}(D_{\rho}), d_{ \mathcal{P}( D_{\rho}) } \right)$ if and only if $\mu_k \to \mu$ weakly.

 For $ m \ge 1 $, let
 $$ \mathcal{P}_m( D_{\rho} ) = \left\{ \mu \in \mathcal{P}(D_{\rho}): \int_{D_{\rho}} \|\xi\|_{ \infty }^m \mu(d\xi) < \infty \right\},
 $$
 where $\|\xi\|_{ \infty } = \sup_{t \in [-\rho,0]} \| \xi(t) \|$. Then
   the space $\left( \mathcal{P}_m( D_{\rho} ), d_{ \mathcal{P}( D_{\rho}) } \right)$ is separable but not complete.
 For any $\mathcal{U} \subseteq \mathcal{P}_m(D_{\rho})$, denote by
 $ \|\mathcal{U}\|_{\mathcal{P}_m(D_{\rho})}
   = \sup_{\mu \in \mathcal{U}}
                \left( \int_{D_{\rho}}
                             \| \xi \|^m_{\infty
                             } \mu(d \xi)
                \right)^{\frac{1}{m} }
 $.
 If $ \|\mathcal{U}\|_{\mathcal{P}_m(D_{\rho})}
 <\infty$, then $ \|\mathcal{U}\|_{\mathcal{P}_m(D_{\rho})}
 $
  is said to be a bounded subset of
  $\mathcal{P}_m(D_{\rho})$.

 Now we introduce non-autonomous cocycles  defined on the metric space $\left( \mathcal{P}_m(D_{\rho}), d_{\mathcal{P}(D_{\rho})} \right)$.

\begin{definition}\label{def_ds}
  A family $\Phi =
                           \left\{\Phi(t, \tau): t \in \mathbb{R}^+, \tau \in \mathbb{R}
                           \right\}$
  of mappings from $\left( \mathcal{P}_m(D_{\rho}), d_{\mathcal{P}(D_{\rho})} \right)$
  to itself 
  is called a non-autonomous cocycle on $\left( \mathcal{P}_m(D_{\rho}), d_{\mathcal{P}(D_{\rho})} \right)$
  if $ \Phi $ satisfies 
  the following conditions:
   for all $\tau \in \mathbb{R}$ and $t,s \in \mathbb{R}^+$:

  (i) $\Phi(0, \tau) = I$,
      where $ I $ is the identity operator on $\left( \mathcal{P}_m(D_{\rho}), d_{\mathcal{P}(D_{\rho})} \right)$;

  (ii) $\Phi(t+s, \tau) = \Phi(t,s+\tau) \circ \Phi(s,\tau)$.

  In addition, if for every bounded subset $\mathscr{B}$ of $ \mathcal{P}_m (D_{\rho}) $,
  $ \Phi $ is continuous from  $\left( \mathscr{B} , d_{\mathcal{P}(D_{\rho})} \right)$
  to $\left( \mathcal{P}_m(D_{\rho}),  d_{\mathcal{P}(D_{\rho})} \right)$,
  then   $\Phi$ is 
  said to be continuous on bounded subsets of $ \mathcal{P}_m(D_{\rho}) $.
\end{definition}

Recall that
 for every $ m \ge 1 $,  
 $$ \mathcal{P}_m(\ell^2) = \left\{ \mu \in \mathcal{P}(\ell^2): \mu(\|\cdot\|^m) := \int_{\ell^2} \|\xi\|^m \mu(d\xi) < \infty \right\}.$$
 The Wasserstein distance 
 in   $\mathcal{P}_m(\ell^2)$ is given by:
 for all  $\mu_1, \mu_2 \in \mathcal{P}_m(\ell^2)$,  
 $$ \mathbb{W}_m(\mu_1,\mu_2)
    = \inf_{\pi \in \mathcal{C}(\mu_1,\mu_2)}
            \left( \int_{\ell^2 \times \ell^2}
                       \|\xi_1 - \xi_2\|^m \pi(d\xi_1, d\xi_2)
            \right)^{\frac{1}{m}},
 $$
 where $\mathcal{C}(\mu_1,\mu_2)$ is the set of all couplings of $\mu_1$ and $\mu_2$.
 Then $\left( \mathcal{P}_m(\ell^2), \mathbb{W}_m \right)$ is a Polish space.

 In what follows, we
 introduce the  assumptions on
 the  drift and diffusion terms in \eqref{lmvlds-1}.

\noindent
{\bf(H1)}
 $g =(g_i)_{i \in \mathbb{Z}},
  h = (h_{k,i})_{k \in \mathbb{N}, i \in \mathbb{Z}}$
 and $\widetilde{h} = (\widetilde{h}_{i})_{ i \in \mathbb{Z}}$
 are $\ell^2$-valued functions  such that
 $$ \int_{\tau}^{\tau + T}
       \left(\| g(t) \|^2
              + \| h(t) \|^2
              + \| \widetilde{h}(t) \|^2_{L^2(\mathbb{Y}, \nu; \ell^2)}
       \right) dt
    < \infty, \ \ \  \forall\  \tau \in \mathbb{R},~ T > 0,
 $$
  where
$
  L^2(\mathbb{Y}, \nu; \ell^2)$
  is the space of all 
  square-integrable
  functions from $(\mathbb{Y},
  \mathcal{Y}, \nu)$
  to $\ell^2$.

 \noindent
{\bf(H2)}\  For every
$i\in \mathbb{Z}$,
$f_i: \mathbb{R} \times  \mathcal{P}_2(\ell^2)
\to \mathbb{R}$ is  a continuous map,
 $f_i(0,\delta_0) = 0$,
  and for any $ s, s_1, s_2 \in \mathbb{R} $ and $\mu, \mu_1, \mu_2 \in  \mathcal{P}_2(\ell^2)$,
 \begin{align}\label{f1}
    f_i(s,\mu) s
    \ge \lambda_1 |s|^p - \phi_{1,i} \left( 1 + |s|^2 + \mu(\|\cdot\|^2) \right),
 \end{align}
 \begin{align}\label{f2}
    | f_i(s_1,\mu_1) - f_i(s_2, \mu_2) |
    \le \phi_{2,i} \left( 1 + |s_1|^{p-2} + |s_2|^{p-2} \right) |s_1 - s_2 |
         + \phi_{3,i} \mathbb{W}_2(\mu_1,\mu_2),
 \end{align}
 \begin{align}\label{f3}
    \left( f_i(s_1,\mu_1) - f_i(s_2,\mu_2) \right) (s_1 - s_2)
    \ge & \lambda_2 \left( |s_1|^{p-2} s_1 - |s_2|^{p-2} s_2 \right) ( s_1 - s_2 ) \nonumber\\
         & - \phi_{4,i} \left( |s_1 - s_2|^2 + \mathbb{W}_2^2(\mu_1,\mu_2) \right),
\end{align}
 where
      $\delta_0 $ is the Dirac measure at $ 0 $,
      $\lambda_1 > 0, \lambda_2 > 0, p > 2$,
      $\phi_1 = \{\phi_{1,i}\}_{i \in \mathbb{Z}} \in \ell^1$,
      $\phi_2 = \{ \phi_{2,i} \}_{ i \in \mathbb{Z} } \in \ell^{\infty}$,
      $\phi_3 = \{ \phi_{3,i} \}_{ i \in \mathbb{Z} } \in \ell^2$,
 and
      $\phi_4 = \{\phi_{4,i}\}_{i \in \mathbb{Z}} \in \ell^1$.

\noindent
 {\bf(H3)}\ 
  For every
$i\in \mathbb{Z}$,
$F_i: \mathbb{R} \times  \mathcal{P}_2(\ell^2)
\to \mathbb{R}$ is  globally Lipschitz continuous;
that is,
    there exists a constant $L_{F,i} > 0$
 such that for all $ s_1, s_2 \in \mathbb{R} $ and $\mu_1, \mu_2 \in \mathcal{P}_2(\ell^2)$,
 \begin{align}\label{F1}
    | F_i(s_1, \mu_1) - F_i(s_2,\mu_2) |
    \le L_{F,i} \left( |s_1 - s_2| + \mathbb{W}_2(\mu_1,\mu_2) \right),
 \end{align}
 where $L_F = (L_{F,i})_{i \in \mathbb{Z}} \in \ell^2$.
 In addition, we assume  
 $(F_i(0, \delta_0))_{i \in \mathbb{Z}} \in \ell^2$.

\noindent
 {\bf(H4)}\
  For every
$i\in \mathbb{Z}$
and $k\in \mathbb{N}$,
$ \sigma_{k,i}: \mathbb{R} \times  \mathcal{P}_2(\ell^2)
\to \mathbb{R}$ is  
continuous;
and
$  \widetilde{\sigma}_{i}: \mathbb{R} \times  \mathcal{P}_2(\ell^2) \times \mathbb{Y}
\to \mathbb{R}$ is 
$(\mathcal{B}
(\mathbb{R})
\times \mathcal{B}
(\mathcal{P}_2(\ell^2))
\times \mathcal{Y}, \mathcal{B}
(\mathbb{R}) )$-measurable
such that 
 for all $s_1, s_2 \in \mathbb{R}$, $\mu_1, \mu_2 \in \mathcal{P}_2(\ell^2)$ and $y \in \mathbb{Y}$,
\begin{align} \label{sigma1}
  |\sigma_{k,i}(s_1, \mu_1) - \sigma_{k,i}(s_2, \mu_2) |
  \le L_{\sigma,k,i}
                 \left[ \left( |s_1|^{\frac{q}{2} - 1} + |s_2|^{\frac{q}{2} - 1} \right)
                        |s_1 - s_2|
                        + \mathbb{W}_2(\mu_1,\mu_2)
                 \right],
\end{align}
and 
\begin{align}\label{sigma2} |\widetilde{\sigma}_{i}(s_1, \mu_1, y) - \widetilde{\sigma}_{i}(s_2, \mu_2, y)| 
	\le L_{\tilde{\sigma},i}(y)
	 \left[ \left( |s_1|^{\frac{q}{2} - 1} + |s_2|^{\frac{q}{2} - 1} \right)
	|s_1 - s_2|
	+ \mathbb{W}_2(\mu_1,\mu_2)
	\right],
\end{align}
where $q\in [2,p)$ and
$$
\|L\|^2 
:= \sum_{k\in\mathbb{N}}\sum_{i\in\mathbb{Z}} L_{
\sigma, k,i}^2 
+ \sum_{i\in\mathbb{Z}}  \int_{y\in
\mathbb{Y}
} 
\bigl(L_{\tilde{\sigma},i}(y)\bigr)^2 \,\nu(dy) 
  <\infty,
$$
and
$$
\|\hat{\sigma}\|^2
:=  \sum_{k\in\mathbb{Z}} \sum_{i\in\mathbb{Z}}
  |\sigma_{k,i} (0,\delta_0)|^2
+  \sum_{i\in\mathbb{Z}}  \int_{ y\in \mathbb{Y}} 
|\widetilde{\sigma}_{i}(0,\delta_0,y)|^2 \,\nu (dy)
  <\infty.
 $$

Notice that
{\bf(H2)}  implies that
  for all $ s \in \mathbb{R} $ and $ \mu \in \mathcal{P}_2(\ell^2) $,
 \begin{align}\label{f4}
    | f_i(s,\mu) |
    \le \phi_{2,i} \left( 1 + |s|^{p-2} \right) |s|
          + \phi_{3,i} \mathbb{W}_2(\mu,\delta_0).
 \end{align}
 In the sequel, we will write 
 $h_k(t) = (h_{k,i}(t))_{i \in \mathbb{Z} }$
 for $k \in \mathbb{N}$
 and $t\in
 \mathbb{R}$
 for convenience.
 
 \begin{remark}
 In this paper, for simplicity,  
    we assume that
    the growth rate of 
   the nonlinear drift   $f_i$
  is 
    same   as that of 
    the  discrete $p$-Laplace
    operator  in \eqref{lmvlds-1}.
    Actually,  the growth rate
    of $f_i$ may be greater than
    the power   of the $p$-Laplace operator. 
    For example, we may 
    take  $f_i(s,\mu) = \lambda |s|^{p-2}s + f_i^*(s, \mu)$,
    where 
    $\lambda>0$ and $f_i^*(s, \mu)$ satisfies  {\bf(H2)} with $p$
    replaced by $p^*>p$. 
 \end{remark}

 To reformulate \eqref{lmvlds-1} as an abstract equation in $\ell^2$,
  we  define the discrete $p$-Laplace operator $A: \ell^2 \to \ell^2 $ by
 $$ (Au)_i = |u_{i}  - u_{i-1}|^{p-2} \left( u_{i}  - u_{i-1} \right)
             - |u_{i+1} - u_{i}|^{p-2} \left( u_{i+1}  - u_i \right),
 $$
  and the  linear  bounded operators $B$ and $B^*$ from $\ell^2$ to $\ell^2 $ by
 $$(Bu)_i = u_{i+1} - u_i \ \ \text{and}\ \  (B^*u)_i = u_{i-1} - u_{i}$$
 for $u=(u_i)_{i \in \mathbb{Z} } \in \ell^2$.
 Then
 $ Au= B^*
                        \left( \left( |(Bu)_i|^{p-2} (Bu)_i \right)_{i \in \mathbb{Z}}
                        \right)
 $
 and
 $(Bu, v) = (u, B^*v)$
 for all $u,v \in \ell^2$

 Denote by
 $f(u,\mu) = (f_i(u_i, \mu ))_{i \in \mathbb{Z}}$,
 $F(u,\mu) = (F_i(u_i, \mu))_{i \in \mathbb{Z}}$,
 $\sigma_k(u, \mu) = (\sigma_{k,i}(u_i, \mu))_{i \in \mathbb{Z}}$
 and
 $\widetilde{\sigma} (u, \mu, y) = (\widetilde{\sigma}_{i}(u_i, \mu, y))_{i \in \mathbb{Z}}$
 for $u = (u_i)_{i \in \mathbb{Z}} \in \ell^2$,
       $\mu \in \mathcal{P}_2(\ell^2)$
 and $k \in \mathbb{N}$.

With  above notation, system \eqref{lmvlds-1} can be rewritten as
\begin{equation}\label{lmvlds-2}
  \left\{
    \begin{aligned}
       & d u (t)
         + Au(t) dt
         + \lambda u(t) dt
         + f(u(t),\mathcal{L}_{u(t)}) dt \\
       & = g(t) dt
           + F( u(t - \rho), \mathcal{L}_{u(t - \rho)} ) dt
           + \sqrt{\varepsilon}\sum_{k \in \mathbb{N} }
              \left( h_k(t) + \sigma_k(u(t), \mathcal{L}_{u(t)}) \right) dW_k(t)  \\
       & \ \ \
           + \sqrt{\varepsilon}
            \int_{y\in \mathbb{Y}}
                \left( \widetilde{h}(t,y) + \widetilde{\sigma}
                (u(t-), \mathcal{L}_{u(t)}, y)
                \right) \widetilde{N}(dt,dy), \ \ \ t > \tau, \\
       & u(s) = \varphi(s-\tau) := ( \varphi_i(s-\tau) )_{i \in \mathbb{Z}} ,\  s \in  [\tau-\rho, \tau].
    \end{aligned}
  \right.
\end{equation}

 By {\bf(H3)}, we obtain that for all $ u_1, u_2, u \in \ell^2$, and $ \mu_1, \mu_2, \mu \in \mathcal{P}_2(\ell^2) $
 \begin{align}\label{F2}
    \| F(u_1,\mu_1) - F(u_2,\mu_2) \|^2
    \le 2 \|L_{F}\|_{\ell^\infty}^2 \|u_1 - u_2\|^2
         + 2 \|L_{F}\|^2\mathbb{W}_2^2(\mu_1,\mu_2),
 \end{align}
 and
 \begin{align}\label{F3}
    \|F(u,\mu)\|^2
    \le 2 \|F(0,\delta_0)\|^2
            + 4 \|L_{F}\|_{\ell^\infty}^2 \|u\|^2
            + 4 \|L_{F}\|^2 \mathbb{W}_2^2(\mu,\delta_0).
 \end{align}

 We will frequently use the following elementary inequalities to deal with 
 the nonlinearity $f$:
 for any $ s_1, s_2 \in \mathbb{R} $ and $ p \ge 2 $,
 \begin{align*}
   \left( |s_1|^{ p-2 } s_1 - |s_2|^{p-2} s_2 \right) (s_1 - s_2)
    \ge 2^{1-p} |s_1 - s_2|^p,
 \end{align*}
 and
 \begin{align*}
   (|s_1|^{p-2} s_1 - |s_2|^{p-2} s_2)(s_1 - s_2)
   \ge \frac{1}{2} (|s_1|^{p-2} + |s_2|^{p-2}) (s_1 - s_2)^2.
 \end{align*}

  From {\bf(H4)}, it follows that there exist 
   $\alpha_{k,i}>0$
   and
   $\beta_i>0$
     with
    $\| \alpha \|^2 := \sum_{k \in \mathbb{N}} \sum_{i \in \mathbb{Z}} \alpha_{k,i}^2 <  \infty$
    and $\| \beta \|^2 :=   \sum_{i \in \mathbb{Z}} \beta_{i}^2 <  \infty$
    such that for all $s \in \mathbb{R}$ and $\mu \in \mathcal{P}_2(\ell^2)$,
 \begin{align}\label{siglinear-1}
    |\sigma_{k,i}(s, \mu)|^2
    \le   \alpha_{k,i}^2
         + 4L_{\sigma, k,i}^2
           \left( |s|^{q}
                  + \mathbb{W}^2_2(\mu,\delta_0)
           \right),
 \end{align}
 and
 \begin{align}\label{siglinear-2}
     \int_{y\in \mathbb{Y} }
              |\widetilde{\sigma}_i(s, \mu, y)| ^2 \nu (dy)
    \le \beta_{i}^2
         + 4 \int_{y\in \mathbb{Y} }
            \left( L_{ \widetilde{\sigma}, i}(y)
            \right)^2 \nu(dy)
            \left( |s|^{q}
                     + \mathbb{W}^2_2(\mu, \delta_0)
            \right).
 \end{align}
  Then for  every  $\tilde{\theta} > 0$
  and $\lambda_0>0$, 
  by Young's inequality we have
 \begin{align}\label{sigma3}
       & \tilde{\theta}\sum_{k \in \mathbb{N}}  \|\sigma_k(u,\mu)\|^2
           +  \tilde{\theta} \int_{y \in \mathbb{Y} } \|\widetilde{\sigma} (u,\mu,y)\|^2 \nu (dy)    \nonumber\\
 \le & \tilde{\theta} \sum_{k \in \mathbb{N}} \sum_{i \in \mathbb{Z}}
             \left( \alpha^2_{k,i} 
                     + 4L_{\sigma, k, i}^2 |u_i|^{q} 
             \right)
          + \tilde{\theta} 
                 \sum_{i \in \mathbb{Z}}
                       \left( \beta^2_{i} 
                               + 4 |u_i|^{q}
                                      \int_{y \in \mathbb{Y} }
                                          \left( L_{\widetilde{\sigma}, {i}}(y)
                                          \right)^2 \nu(dy)
                       \right)
           + 4 \tilde{\theta} \|L\|^2 \mathbb{W}^2_2(\mu,\delta_0)  \nonumber\\
 \le & \tilde{\theta} ( \| \alpha \|^2 + \|\beta\|^2)
          + 4 \tilde{\theta} \|L\|^2 \mathbb{W}_2^2(\mu,\delta_0)
          + \frac{\lambda_0}{4} \| u \|^p_p
          + C_1  \nonumber\\
   = & \frac{\lambda_0}{4} \| u \|^p_p
          + 4 \tilde{\theta} \|L\|^2 \mathbb{W}^2(\mu,\delta_0)
          +C_2,
 \end{align}
 where
 $C_1 = C_1(\tilde{\theta}, \lambda_0, p, q, \|L\|^2) > 0$ and
 $C_2= C_1 + \tilde{\theta}( \| \alpha \|^2 + \|\beta\|^2)$.
 
 Similarly, 
 by {\bf(H4)} we  also  obtain that for all $u_1, u_2 \in \ell^2 $, $\mu_1, \mu_2  \in \mathcal{P}_2(\ell^2) $ and $\tilde{\theta} > 0$,
 \begin{align} \label{sigma4}
      & \tilde{\theta}\sum_{k \in \mathbb{N}}
               \|\sigma_k(u_1,\mu_1) - \sigma_k(u_2,\mu_2)\|^2
        +
   \tilde{\theta}    
                  \int_{y\in \mathbb{Y}
                  }
                          \|\widetilde{\sigma}(u_1,\mu_1,y) - \widetilde{\sigma}(u_2, \mu_2, y)\|^2\nu (dy)  \nonumber \\
 \le & \tilde{\theta} \sum_{k \in \mathbb{N} } \sum_{i \in \mathbb{Z}}
         L^2_{\sigma, k,i}
              \left[
                     \left( | u_{1,i} |^{\frac{q}{2} - 1}  + | u_{2,i} |^{\frac{q}{2} - 1} \right)
                     | u_{1,i} - u_{2,i} |
                     + \mathbb{W}_2(\mu_1, \mu_2)
              \right]^2  \nonumber\\
      & + \tilde{\theta}   \sum_{i \in \mathbb{Z}}
             \int_{y \in \mathbb{Y}}
                 \left( L_{\widetilde{\sigma}, i} (y) \right )^2 \nu(dy)
                       \left[
                                \left( | u_{1,i} |^{\frac{q}{2} - 1}  + | u_{2,i} |^{\frac{q}{2} - 1} \right)
                                | u_{1,i} - u_{2,i} |
                                + \mathbb{W}_2(\mu_1, \mu_2)
                       \right]^2  \nonumber\\
 \le & \tilde{\theta} \sum_{k \in \mathbb{N} } \sum_{i \in \mathbb{Z}}
         4 L^2_{\sigma, k,i}
                \left[
                       \left( | u_{1,i} |^{q-2} + | u_{2,i} |^{q-2} \right)
                       | u_{1,i} - u_{2,i} |^2
                \right]
           + \tilde{\theta} 
                  \sum_{k \in \mathbb{N} } 
                        \sum_{i \in \mathbb{Z}}  2 L^2_{\sigma, k,i} \mathbb{W}_2^2(\mu_1, \mu_2)  \nonumber\\
      & + \tilde{\theta}   \sum_{i \in \mathbb{Z}}
             4 \int_{\mathbb{Y}}
                    \left( L _{\widetilde{\sigma}, i}(y) \right )^2 \nu(dy)
                         \left[
                                  \left( | u_{1,i} |^{q-2} + | u_{2,i} |^{q-2} \right)
                                  | u_{1,i} - u_{2,i} |^2
                         \right]  \nonumber\\
      & + 2 \tilde{\theta}    
                   \mathbb{W}_2^2(\mu_1, \mu_2) 
                          \sum_{i \in \mathbb{Z}}   
                              \int_{y \in \mathbb{Y}}
                                  \left ( L_{\widetilde{\sigma}, i} (y) \right )^2 \nu(dy)  \nonumber\\
 \le & \frac{\lambda_2}{4}
              \sum_{ i \in \mathbb{Z}} 
                    \left( |u_{1,i}|^{p-2} + |u_{2,i}|^{p-2} \right) |u_{1,i} - u_{2,i}|^2
         + 2 \tilde{\theta} \| L \|^2 
                    \mathbb{W}_2^2(\mu_1, \mu_2)
         + C_3 \|u_1 - u_2 \|^2,
 \end{align}
 where
   $C_3 = C_3(\tilde{\theta}, p, q, \|L\|^2) > 0$
 is a constant.

\section{Well-posedness and uniform esitmates of solutions} \label{3}

 In this section, 
 we establish  the existence and uniqueness of solutions to \eqref{lmvlds-2} 
 and also derive the  uniform estimates of solutions.
 
 In the sequel, 
 for a process  $u(t)$,
 $\tau -\rho \le t \le \tau +T$,
 we will use $u_t$ to denote the segment of $u$ on $[t-\rho, t]$ which is given by:
 $$ u_t (s) = u(t+s), \quad -\rho \le s \le 0.
 $$
 Let
 $L^2( \Omega, \mathcal{F}_\tau;  D_{\rho} )$ be the set of all $D_{\rho}$-valued 
 $\mathcal{F}_\tau$-measurable random variables  $\varphi$ such that
 $\mathbb{E} \left[ \|\varphi\|^2_\infty \right] < \infty$.

\begin{definition}\label{def-sol}
  Given  $T>0$, $\rho > 0$ and 
  a  $D_{\rho}$-valued 
  $\mathcal{F}_\tau$-measurable 
   random variable $\varphi$,
   an $\ell^2$-valued
    c\`{a}dl\`{a}g process $u(t)$,
    $\tau -\rho \le t \le \tau +T$,
   is called a solution to \eqref{lmvlds-2} on
    $[\tau-\rho, \tau+T]$ if
    $u_\tau = \varphi$, 
   $u (t)$ is  $\mathcal{F}_t$-adapted
   for $t\ge \tau$, and
   for all $t \in [\tau, \tau+T]$, 
    $ \mathbb{P} $-almost surely,
   \begin{align*}
       u(t) = & \varphi(0)
                + \int_{\tau}^{t}
                     \left( - Au(s) - \lambda u(s)
                            - f(u(s), \mathcal{L}_{u(s)})
                            + g(s)
                            + F( u(s - \rho), \mathcal{L}_{u(s - \rho)} )
                     \right)ds \\
              & + \sqrt{\varepsilon}\sum_{k \in \mathbb{N} }
                     \int_{\tau}^{t}
                        \left( h_k(s) + \sigma_k(u(s),\mathcal{L}_{u(s)}) \right) dW_k(s) \\
              & + \sqrt{\varepsilon} 
                     \int_{\tau}^{t}\int_{
                     y\in \mathbb{Y} }
                        \left( \widetilde{h}(s,y) + \widetilde{\sigma}(u(s-),\mathcal{L}_{u(s)},y)
                        \right) \widetilde{N}(ds,dy)
                        \ \ \text{in }  \ell^2 .
   \end{align*}
 \end{definition}

\begin{theorem}\label{wp-sol}
   Suppose that {\bf(H1)}-{\bf(H4)} hold
   and $\varepsilon \in (0,1]$.
   Then for every $\varphi \in L^2( \Omega, \mathcal{F}_\tau;  D_{\rho} )$, the stochastic 
   equation
   \eqref{lmvlds-2} has a unique solution $u$ in the sense of Definition \ref{def-sol},
   which satisfies,
  for any $T > 0$,
   \begin{align}\label{secondmoment}
          & \mathbb{E} \left[ \sup_{t \in [\tau,\tau+T]} \|u(t)\|^2 \right]
            + \mathbb{E} \left[ \int_{\tau}^{\tau+T} \|u(s)\|^p_p ds \right]  \nonumber\\
     \le & C \mathbb{E} \left[ \sup_{s \in [-\rho,0]} \|\varphi(s)\|^2 \right] e^{C T}  
         + C \left( T
                       + \int_{\tau}^{\tau + T}
                               \left( \|g(s)\|^2 + \|h(s)\|^2 +  \| \widetilde{h}(s) \|^2_{L^2(\mathbb{Y}, \nu; \ell^2)}
                               \right)ds
                \right) e^{C T},
   \end{align}
   where   $C>0$
   is constant    independent of $\tau$, $T$, $\varphi$
   and $\varepsilon$.
   
   Moreover, 
   if  
   $\theta \in \left(2, \tfrac{2(p-2)}{q-2} \right)$,
   \begin{equation}\label{thetaassum1}
   	     \int_{\tau}^{\tau+T}
   	         \bigl( \|g(s)\|^\theta 
   	                   + \|h(s)\|^\theta
   	                   + \|\widetilde{h}(s)\|_{L^\theta(\mathbb{Y}, \nu; \ell^2)}^\theta
   	                   + \|\widetilde{h}(s)\|_{L^2(\mathbb{Y}, \nu; \ell^2)}^\theta
          	 \bigr) ds
   	      < \infty,
   \end{equation}
   and
   \begin{equation}\label{thetaassum2}
         \int_{y \in \mathbb{Y} }
   	         \biggl[ 
   	                    \Bigl( \sum_{i \in \mathbb{Z}} L_{\tilde{\sigma}, i}(y) 
   	                    \Bigr)^{\theta}
   	                   + 
   	                   \Bigl( \sum_{i \in \mathbb{Z}} |\widetilde{\sigma}_{i}(0, \delta_{0}, y) |
   	                      \Bigr)^{\theta}
   	         \biggr] \nu(dy)
   	      < \infty 
   \end{equation}
   hold, 
   then for every 
   $\varphi \in L^{\theta}(\Omega,\mathcal{F}_\tau;\mathcal{D}_{\rho})$, 
   we have 
   \[ 
      \mathbb{E}
             \left[ \sup_{t \in [\tau, \tau+T] } 
                        \|u(t)\|^{\theta}
             \right]
       + \mathbb{E}
                 \left[ \int_{\tau}^{\tau+T} 
                              \|u(s)\|^{\theta-2} \|u(s)\|_{p}^{p} 
                           ds
                 \right]  
   \le
   M_{\theta,T,\varphi} 
   \] 
   where $M_{\theta,T,\varphi}>0$ is a constant   depending 
   only on $\theta$, $T$ and $\varphi$, but not
   on $\varepsilon$.
\end{theorem}

\begin{proof}
Given  $ \varphi \in L^2(\Omega, \mathcal{F}_{\tau}; D_\rho)$,
we will apply the fixed point 
theorem 
 to prove the existence 
and uniqueness of solutions
in the space:
$$
D^\varphi([\tau-\rho, \tau+T], \mathcal{P}_2(\ell^2))
=
\left \{
\mu \in D([\tau-\rho, \tau+T], \mathcal{P}_2(\ell^2)):
\   {\mu}(t) = \mathcal{L}_{\varphi(t-\tau)},
\   t \in [\tau-\rho, \tau]  
\right \}.
$$
 Let  $\mu \in 
 D^\varphi([\tau-\rho, \tau+T], \mathcal{P}_2(\ell^2))$ 
 and  consider the distribution independent stochastic delay system:
 \begin{align}\label{distr-indep-lds}
      & d u_{\mu}(t)
        + Au_{\mu}(t) dt
        + \lambda u_{\mu}(t) dt
        + f^{\mu}(u_{\mu}(t)) dt \nonumber \nonumber \\
    = & g(t) dt
        + F( u(t-\rho), {\mu}(t-\rho) ) dt
        + \sqrt{\varepsilon}\sum_{k \in \mathbb{N} }
                \left( h_k(t) + \sigma_k^{\mu}( u_{\mu}(t))
                \right) dW_k(t) \nonumber \\
       & + \sqrt{\varepsilon}  
                   \int_{y \in \mathbb{Y} }
                        \left( \widetilde{h}(t) 
                                  + \widetilde{\sigma}^{\mu}(u_{\mu}(t-), y)
                        \right)
                   \widetilde{N}(dt,dy), \  t > \tau,
 \end{align}
 with initial data
 $u_{\mu}(s) = \varphi(s-\tau)$ for $s \in [\tau-\rho, \tau]$,
 where
  $f^{\mu}(u(t)) = f(u(t),\mu (t) )$,
  $\sigma_k^{\mu}( u(t)) = \sigma_k( u(t), \mu
  (t))$ and 
  $\widetilde{\sigma}^{\mu}(u(t-), y) = \widetilde{\sigma}(u(t-), \mu (t), y)$.

 Similar to the arguments in \cite[Theorem 3.1]{CYZ2023SAA}
 together with \cite{WW2020SPA},
 it follows from {\bf(H1)}-{\bf(H4)} that \eqref{distr-indep-lds} has a unique solution $ u_{\mu} \in L^2(\Omega, D([\tau-\rho,\tau + T],\ell^2))$
 which satisfies  the uniform estimates
 \begin{align*}
    \mathbb{E}
              \left[ \sup_{t \in [\tau-\rho,\tau + T]} \|u_{\mu}(t)\|^2
              \right]
    + \int_{\tau}^{\tau+T} \mathbb{E} \left[ \|u_\mu(t)\|^p_p \right] dt
   \le  M_{\tau, T,\mu} \left( 1 + \mathbb{E} \left[ \sup_{r \in [-\rho,0]} \|\varphi(r)\|^2 \right]  \right),
 \end{align*}
 where $ M_{\tau, T, \mu}>0$
 is a constant
   independent of $\varphi$
   and $\varepsilon$.
   Next, we prove the distribution
   of the solution
   belongs
   to 
   $ D^\varphi([\tau-\rho, \tau + T], \mathcal{P}_2(\ell^2))$.

 {\bf Step 1:}  Prove
  $\mathcal{L}_{u_{\mu}} \in D^\varphi([\tau-\rho, \tau + T], \mathcal{P}_2(\ell^2))$.

 By $u_{\mu} \in  L^2(\Omega, D([\tau
 -\rho, \tau + T], \ell^2))$ and the Lebesgue dominated convergence theorem,
 we can know that $u_{\mu} \in D([\tau-\rho, \tau + T], L^2(\Omega,\ell^2) )$.
 Noting that
 $$\mathbb{W}_2(\mathcal{L}_{u_{\mu}(t)},
                       \mathcal{L}_{u_{\mu}(s)} )
    \le
    \left( \mathbb{E} \left[ \| u_{\mu}(t) - u_{\mu}(s) \|^2 \right]
    \right)^{\frac{1}{2}}, ~~ \forall s, t \in [\tau-\rho, \tau + T],$$
 we find that for any $t_0 \in [\tau-\rho, \tau + T)$,
 \begin{align*}
    \lim_{t \to t_0 +} \mathbb{W}_2(\mathcal{L}_{u_{\mu}(t)},
                       \mathcal{L}_{u_{\mu}(t_0 )} )
    \le \lim_{t \to t_0 +}
    \left( \mathbb{E} \left[ \| u_{\mu}(t) - u_{\mu}(t_0 ) \|^2 \right]
    \right)^{\frac{1}{2}}
    =0,
 \end{align*}
that is,
$ \lim_{t \to t_0 +} \mathcal{L}_{u_{\mu}(t)} = \mathcal{L}_{u_{\mu}(t_0)} $ in $\left( \mathcal{P}_2(\ell^2), \mathbb{W}_2 \right)$.

 On the other hand, for any $ t_0, t_n \in (\tau
 -\rho, \tau + T] $
 such that $t_n < t_0$ and $t_n \to t_0$ as $n \to \infty$,
 it follows from $u_{\mu} \in D([\tau,\tau + T], L^2(\Omega,\ell^2) )$ that
 for any $ \varepsilon' > 0 $, there is $ N_0 \ge 1 $ such that
 $\left( \mathbb{E} \left[ \| u_{\mu}(t_n) - u_{\mu}(t_m) \|^2 \right]
  \right)^{\frac{1}{2}}
  \le \varepsilon' $
 for $ n, m \ge N_0$.
 Then we have
 \begin{align*}
    \mathbb{W}_2(\mathcal{L}_{u_{\mu}(t_n)},
                       \mathcal{L}_{u_{\mu}(t_m)} )
     \le \varepsilon', \ \text{for}\ n,m \ge N_0,
 \end{align*}
 which implies $ \lim_{n \to \infty} \mathcal{L}_{u_{\mu}(t_n)} $ exists in $\left( \mathcal{P}_2(\ell^2), \mathbb{W}_2 \right) $.
 By a   contradiction argument,
 we  find  that
    $\lim_{t \to t_0 -} \mathcal{L}_{u_{\mu}(t)}$ exists in $\left( \mathcal{P}_2(\ell^2), \mathbb{W}_2 \right)$, and thus
    we obtain 
    $\mathcal{L}_{u_{\mu} } \in D^\varphi
    ([\tau
    -\rho, \tau + T], \mathcal{P}_2(\ell^2))$.

 {\bf Step 2:} Prove existence and uniqueness of solutions of \eqref{lmvlds-2}.

By  {\bf Step 1},
 we   define a map $\Phi^{\varphi}: D^\varphi([\tau
 -\rho, \tau + T],\mathcal{P}_2(\ell^2)) \to D^\varphi([\tau-\rho, \tau + T],\mathcal{P}_2(\ell^2))$ by
 \begin{align*}
    (\Phi^{\varphi}\mu)(t) = \mathcal{L}_{u_{\mu}(t)},
                                                        ~ t \in [\tau-\rho, \tau + T],
                                                        ~ \mu \in D^\varphi( [\tau-\rho, \tau + T], \mathcal{P}_2(\ell^2) ),
 \end{align*}
 where $u_{\mu}$ is the solution of \eqref{distr-indep-lds} with initial value
 $\varphi$
 at  $\tau$.
 We  will
 find a fixed point 
 of  $ \Phi^{\varphi}$ 
 in the complete metric space
     $\left( D
     ^\varphi([\tau-\rho,\tau + T],
      \mathcal{P}_2(\ell^2)),
             d_{ D^\varphi([\tau-\rho,\tau + T],
              \mathcal{P}_2(\ell^2))}
      \right)$
 where the metric $d_{ D
  ^\varphi([\tau-\rho,\tau + T],
   \mathcal{P}_2(\ell^2) ) }$ is defined by
 \begin{align*}
    d_{D ^\varphi([\tau-\rho,\tau + T], \mathcal{P}_2(\ell^2))}(\mu_1, \mu_2)
    = \sup_{t \in [\tau-\rho,\tau + T]}
        e^{- \gamma t} \mathbb{W}_2(\mu_1(t), \mu_2(t)),
        ~~ \forall \mu_1, \mu_2 \in D
      ^\varphi([\tau-\rho,\tau + T], \mathcal{P}_2(\ell^2)),
 \end{align*}
 where  $\gamma > 0$
 is a number to be
  determined later.

 Let $ \mu_1, \mu_2 \in D^\varphi([\tau-\rho,\tau + T], \mathcal{P}_2(\ell^2) )$,
     $ u_{\mu_1} $ and $ u_{\mu_2} $ be the solutions of \eqref{distr-indep-lds}.
 By It\^{o}'s formula,
 we obtain that for all $t\ge \tau$,
 \begin{align}\label{wps-1}
    & \|u_{\mu_1}(t) - u_{\mu_2}(t) \|^2
       + 2 \int_{\tau}^{t}
       \langle  A u_{\mu_1}(s) - A u_{\mu_2}(s) ,
               u_{\mu_1}(s) - u_{\mu_2}(s)
       \rangle ds \nonumber \\
   & + 2 \int_{\tau}^{t}
       \langle f^{\mu_1}(u_{\mu_1}(s)) - f^{\mu_2}(u_{\mu_2}(s)),
               u_{\mu_1}(s) - u_{\mu_2}(s)
       \rangle ds \nonumber \\
  = & - 2 \lambda \int_{\tau}^{t}
        \|  u_{\mu_1}(s) - u_{\mu_2}(s) \|^2 ds
      + \varepsilon \int_{\tau}^{t}
              \sum_{k \in \mathbb{N}}
                                     \| \sigma^{\mu_1}_k(u_{\mu_1}(s)) - \sigma^{\mu_2}_k(u_{\mu_2}(s)) \|^2
           ds \nonumber\\
    & + 2 \int_{\tau}^{t}
             \langle F( u_{\mu_1}(s-\rho),  
             {\mu_1}(s-\rho) ) - F( u_{\mu_2}(s-\rho), 
             {\mu_2}(s-\rho) ),
                     u_{\mu_1}(s) - u_{\mu_2}(s)
             \rangle ds \nonumber\\
    & + 2\sqrt{\varepsilon} \int_{\tau}^{t}
             \sum_{k \in \mathbb{N}}
                 \langle \sigma^{\mu_1}_k(u_{\mu_1}(s)) - \sigma^{\mu_2}_k(u_{\mu_2}(s)),
                         u_{\mu_1}(s) - u_{\mu_2}(s)
                 \rangle
          dW_k(s)  \nonumber \\
    & + \varepsilon \int_{\tau}^{t} 
                 \int_{y\in \mathbb{Y}
                 }
                             \| \widetilde{\sigma}^{\mu_1}(u_{\mu_1}(s-),y)
                                - \widetilde{\sigma}^{\mu_2}(u_{\mu_2}(s-),y) \|^2
            {N}(ds, dy)  \nonumber \\
    & +2 \sqrt{\varepsilon} \int_{\tau}^{t} 
              \int_{y\in \mathbb{Y}
              }
                   \langle\widetilde{\sigma}^{\mu_1}(u_{\mu_1}(s-),y) 
                                         - \widetilde{\sigma}^{\mu_2}(u_{\mu_2}(s-),y) ,\     u_{\mu_1}(s-) - u_{\mu_2}(s-)    \rangle
           \widetilde{N}(ds, dy).
 \end{align}
  Note that
 \begin{align*}
       & \langle  A u_{\mu_1}(s) - A u_{\mu_2}(s) ,
               u_{\mu_1}(s) - u_{\mu_2}(s)
         \rangle  \nonumber \\
   = & \sum_{i \in \mathbb{Z} }
            \Big( |u_{\mu_1,i}(s) - u_{\mu_1,i-1}(s)|^{p-2} \left( u_{\mu_1,i}(s) - u_{\mu_1,i-1}(s) \right)
                        \nonumber \\
       & \ \ \ \ \ \
                  - |u_{\mu_2,i}(s) - u_{\mu_2,i-1}(s)|^{p-2} \left( u_{\mu_2,i} (s) - u_{\mu_2,i-1}(s) \right)
            \Big)       \nonumber \\
       & \ \ \ \ \
            \cdot
            \left( u_{\mu_1,i}(s) - u_{\mu_1,i-1}(s) - u_{\mu_2,i}(s) + u_{\mu_2,i-1}(s) \right)  \nonumber \\
  \ge & 2^{1-p} \| u_{\mu_1}(s) - u_{\mu_2}(s) \|_p^p.
 \end{align*}
Then by \eqref{wps-1}
(using a stopping time
if necessary),   we obtain
 \begin{align}\label{wps-2}
        & \mathbb{E}
             \left[ \|u_{\mu_1}(t) - u_{\mu_2}(t) \|^2 \right]
          + 2 \mathbb{E}
                 \left[ \int_{\tau}^{t}
                           \langle f^{\mu_1} (u_{\mu_1}(s)) - f^{\mu_2} (u_{\mu_2}(s)),
                                   u_{\mu_1}(s) - u_{\mu_2}(s)
                           \rangle ds
                 \right]  \nonumber \\
  \le & - 2 \lambda \int_{\tau}^{t}
             \mathbb{E}
                \left[ \| u_{\mu_1}(s) - u_{\mu_2}(s) \|^2 \right] ds  \nonumber \\
       & + 2 \mathbb{E}
                \left[ \int_{\tau}^{t}
                          \langle F( u_{\mu_1}(s-\rho),  {\mu_1}(s-\rho) ) - F( u_{\mu_2}(s-\rho),  {\mu_2}(s-\rho) ),
                                  u_{\mu_1}(s) - u_{\mu_2}(s)
                          \rangle ds
                \right]  \nonumber\\
       & + \varepsilon \mathbb{E}
              \left[ \int_{\tau}^{t}
                        \sum_{k \in \mathbb{N}}
                           \| \sigma^{\mu_1}_k(u_{\mu_1}(s)) - \sigma^{\mu_2}_k(u_{\mu_2}(s)) \|^2 ds
              \right]  \nonumber \\
       & + \varepsilon \mathbb{E}
              \left[
                     \int_{\tau}^{t}  
                        \int_{
                        y\in \mathbb{Y}
                        } \| \widetilde{\sigma}^{\mu_1}(u_{\mu_2}(s),y) - \widetilde{\sigma}^{\mu_2}(u_{\mu_2}(s),y) \|^2
                        \nu (dy)
                     ds
              \right].
 \end{align}

 For the second term on the left-hand side of \eqref{wps-2},
 by \eqref{f3} and Young's equality, we obtain
 \begin{align} \label{wps-3}
        & 2 \mathbb{E}
               \left[ \int_{\tau}^{t}
                         \langle f^{\mu_1} (u_{\mu_1}(s)) - f^{\mu_2} (u_{\mu_2}(s)),
                                 u_{\mu_1}(s) - u_{\mu_2}(s)
                        \rangle ds
               \right]  \nonumber \\
  \ge & 2 \mathbb{E}
              \left[ \int_{\tau}^{t} \sum_{i \in \mathbb{Z}}
                        \left( \frac{\lambda_2}{2}
                                \left( |u_{\mu_1,i}(s)|^{p-2} + |u_{\mu_2,i}(s)|^{p-2} \right)
                                | u_{\mu_1,i}(s) - u_{\mu_2,i}(s) |^2
                                - \phi_{4,i} | u_{\mu_1,i}(s) - u_{\mu_2,i}(s) |^2
                        \right) ds
              \right]  \nonumber \\
       & - 2 \int_{\tau}^{t} \sum_{i \in \mathbb{Z}}
                             |\phi_{4,i}| \mathbb{W}^2_2(\mu_1(s), \mu_2(s))
                          ds  \nonumber \\
  \ge & \lambda_2
           \mathbb{E}
              \left[ \int_{\tau}^{t} \sum_{i \in \mathbb{Z}}
                        \left( \left( |u_{\mu_1,i}(s)|^{p-2} + |u_{\mu_2,i}(s)|^{p-2} \right)
                                | u_{\mu_1,i}(s) - u_{\mu_2,i}(s) |^2
                        \right) ds
              \right]  \nonumber \\
       & - 2 \|\phi_4\|_{\ell^\infty} \int_{\tau}^{t}
               \mathbb{E}
                  \left[ \| u_{\mu_1}(s) - u_{\mu_2}(s) \|^2 ds
                  \right]
         - 2 \|\phi_4\|_1 \int_{\tau}^{t}
               \mathbb{W}^2_2(\mu_1(s), \mu_2(s)) ds.
 \end{align}

 For the second term on the right-hand side of \eqref{wps-2},
 by \eqref{F2}, we have
 \begin{align} \label{wps-4}
       & 2 \mathbb{E}
                \left[ \int_{\tau}^{t}
                          \langle F( u_{\mu_1}(s-\rho),  {\mu_1}(s-\rho) ) - F( u_{\mu_2}(s-\rho),  {\mu_2}(s-\rho) ),
                                  u_{\mu_1}(s) - u_{\mu_2}(s)
                          \rangle ds
                \right]  \nonumber\\
  \le & \mathbb{E}
                \left[ \int_{\tau}^{t}
                          \| F( u_{\mu_1}(s-\rho),  {\mu_1}(s-\rho) ) - F( u_{\mu_2}(s-\rho), 
                          {\mu_2}(s-\rho) ) \|^2 ds
                \right]  \nonumber \\
        & + \mathbb{E}
              \left[ \int_{\tau}^{t} \| u_{\mu_1}(s) - u_{\mu_2}(s) \|^2 ds
              \right]  \nonumber \\
    \le & \mathbb{E}
                \left[ \int_{\tau-\rho}^{t-\rho}
                          \| F( u_{\mu_1}(s),  {\mu_1}(s) ) - F( u_{\mu_2}(s),  {\mu_2}(s) ) \|^2 ds
                \right]  \nonumber \\
         & + \mathbb{E}
              \left[ \int_{\tau}^{t} \| u_{\mu_1}(s) - u_{\mu_2}(s) \|^2 ds
              \right]  \nonumber\\
  \le & \left( 2 \|L_{F}\|_{\ell^\infty}^2 +1 \right)
            \int_{\tau}^{t}
               \mathbb{E}
                  \left[ \| u_{\mu_1}(s) - u_{\mu_2}(s) \|^2 \right]
            ds
          + 2 \|L_{F}\|^2 \int_{\tau}^{t} \mathbb{W}_2^2(\mu_1(s),\mu_2(s) ) ds,
 \end{align}

 For the last two terms on the right-hand side of \eqref{wps-2},
 by \eqref{sigma4} with $\tilde{\theta}=1$, we obtain for all $\varepsilon \in (0,1]$,
\begin{align}\label{wps-5}
       & \varepsilon\mathbb{E}
            \left[ \int_{\tau}^{t} \sum_{k \in \mathbb{N}}
                      \| \sigma^{\mu_1}_k(u_{\mu_1}(s))
                         - \sigma^{\mu_2}_k(u_{\mu_2}(s)) \|^2 ds
            \right]  \nonumber \\
       & + \varepsilon \mathbb{E}
              \left[ \int_{\tau}^{t} 
                     \int_{y\in\mathbb{Y} }
                           \| \widetilde{\sigma}^{\mu_1}(u_{\mu_2}(s),y) - \widetilde{\sigma}^{\mu_2} (u_{\mu_2}(s),y) \|^2  \nu(dy)ds
              \right]  \nonumber \\
  \le & \frac{\lambda_2}{4}
         \mathbb{E}
            \left[ \int_{\tau}^{t} \sum_{i \in \mathbb{Z}}
                         \left( |u_{\mu_1,i}(s)|^{p-2} + |u_{\mu_2,i}(s) |^{p-2} \right)
                         |u_{\mu_1,i}(s) - u_{\mu_2,i}(s)|^2 ds
            \right]  \nonumber\\
       & + 2 \|L\|^2 \int_{\tau}^{t} \mathbb{W}^2_2(\mu_1(s),\mu_2(s) ) ds
         + C_1
             \int_{\tau}^{t}
                \mathbb{E} \left[ \|u_{\mu_1}(s) - u_{\mu_2}(s) \|^2 \right] ds,
\end{align}
where $C_1=C_1(p,q, L)>0$ is a constant
independent of $\varepsilon$.

 By \eqref{wps-2}-\eqref{wps-5}, we get
 for all $ t \in [\tau, \tau + T]$,
 \begin{align}\label{wps-6}
       & \mathbb{E}
            \left[ \|u_{\mu_1}(t) - u_{\mu_2}(t) \|^2 \right]
            + \frac{3\lambda_2}{4}
              \mathbb{E}
                 \left[ \int_{\tau}^{t} \sum_{i \in \mathbb{Z}}
                           \left( |u_{\mu_1,i}(s)|^{p-2} + |u_{\mu_2,i}(s) |^{p-2} \right)
                                  |u_{\mu_1,i}(s) - u_{\mu_2,i}(s)|^2 ds
                 \right]  \nonumber\\
  \le & \left( 2 \|\phi_4\|_{\ell^\infty}
                + 2 \|L_{F}\|_{\ell^\infty}^2
                + 1
                + C_1
         \right)
         \int_{\tau}^{t}
                  \mathbb{E} \left[ \| u_{\mu_1}(s) - u_{\mu_2}(s) \|^2 \right] ds
                  \nonumber \\
       & + \left( 2 \|\phi_4\|_1
                  + 2 \|L_{F}\|^2
                  + 2 \|L\|^2
           \right)
           \int_{\tau}^{t} \mathbb{W}_2^2(\mu_1(s),\mu_2(s) ) ds.
\end{align}
 Applying the Gronwall inequality to \eqref{wps-6}, we obtain  that for all $ t \in [\tau, \tau + T]$,
 \begin{align}\label{wps-7}
        & \mathbb{E}
             \left[ \|u_{\mu_1}(t) - u_{\mu_2}(t) \|^2 \right]  \nonumber\\
   \le & \left( 2 \|\phi_4\|_1
                  + 2 \|L_{F}\|^2
                  + 2 \|L\|^2
          \right)
          \int_{\tau}^{t} \mathbb{W}_2^2(\mu_1(s),\mu_2(s) ) ds \
           e^{ \left( 2 \|\phi_4\|_{\ell^\infty}
                      + 2 \|L_{F}\|_{\ell^\infty}^2
                      + 1
                      + C_1
                \right) (t-\tau) }  \nonumber \\
   \le & C_T \int_{\tau}^{t}
                 \mathbb{W}^2_2(\mu_1(s), \mu_2(s) ) ds,
 \end{align}
 where
   $ C_T = \left( 2 \|\phi_4\|_1
                  + 2 \|L_{F}\|^2
                  + 2 \|L\|^2
           \right)
            e^{ \left( 2 \|\phi_4\|_{\ell^\infty}
                      + 2 \|L_{F}\|_{\ell^\infty}^2
                      + 1
                      + C_1
                \right) T }$.
 By \eqref{wps-7}, we obtain that for $ t \in [\tau, \tau + T]$,
 \begin{align} \label{wps-8}
        e^{-2 \gamma t}
        \mathbb{E}
           \left[ \|u_{\mu_1}(t) - u_{\mu_2}(t) \|^2 \right]
 \le & C_T \int_{\tau}^{t} e^{-2 \gamma (t - s)} e^{-2 \gamma s}
                 \mathbb{W}^2_2(\mu_1(s), \mu_2(s)) ds  \nonumber \\
 \le & \frac{C_T}{2 \gamma}
        \sup_{s \in [\tau,\tau + T]} e^{-2 \gamma s} \mathbb{W}^2_2(\mu_1(s), \mu_2(s)).
 \end{align}
 By \eqref{wps-8} and the definition of $ \mathbb{W}_2(\cdot,\cdot)$,
 we have
 \begin{align}\label{wps-9}
    e^{-2 \gamma t}
    \mathbb{W}^2_2( \mathcal{L}_{u_{\mu_1}(t)}, \mathcal{L}_{u_{\mu_2}(t)} )
    \le \frac{C_T}{2 \gamma}
         \sup_{s \in [\tau, \tau + T]} e^{-2 \gamma s}
                 \mathbb{W}^2_2(\mu_1(s), \mu_2(s)),
                 \ 
         \text{ for all } t \in [\tau, \tau + T].
 \end{align}
 By \eqref{wps-9} and the definition of $ d_{D^\varphi([\tau-\rho,\tau + T], \mathcal{P}_2(\ell^2))}$,
 we obtain
 \begin{align}\label{wps-10}
    d_{D^\varphi([\tau-\rho,\tau + T], \mathcal{P}_2(\ell^2))}
       \left( \mathcal{L}_{u_{\mu_1}}, \mathcal{L}_{u_{\mu_2}} \right)
    \le \left( \frac{C_T}{2\gamma} \right)^{\frac{1}{2}}
         d_{D^\varphi([\tau-\rho,\tau + T], \mathcal{P}_2(\ell^2))}(\mu_1,\mu_2).
 \end{align}
 Choosing $\gamma > 0$ such that $\frac{C_T}{2\gamma} < \frac{1}{4}$,
 by \eqref{wps-10} and the definition of $\Phi^{\varphi}$,
 we get
 \begin{align*}
    d_{D^\varphi([\tau-\rho, \tau + T], \mathcal{P}_2(\ell^2))}
       \left( \Phi^{\varphi}\mu_1, \Phi^{\varphi}\mu_2 \right)
    \le \frac{1}{2} d_{D^\varphi([\tau-\rho, \tau + T], \mathcal{P}_2 (\ell^2))}(\mu_1,\mu_2)
 \end{align*}
 for all $\mu_1, \mu_2 \in D^\varphi([\tau-\rho,\tau + T], \mathcal{P}_2(\ell^2) )$,
 which shows that $\Phi^{\varphi}$ is a contraction mapping in
   $(D^\varphi([\tau-\rho,\tau + T], \mathcal{P}_2(\ell^2)), d_{D^\varphi([\tau-\rho,\tau + T], \mathcal{P}_2(\ell^2))})$.
 Therefore, $\Phi^{\varphi}$ has a unique fixed point
   $\bar{\mu} \in D^\varphi([\tau-\rho,\tau + T], \mathcal{P}_2(\ell^2)) $.
 Then $u_{\bar{\mu}}$ is 
 the unique  solution of \eqref{lmvlds-2} on $t \in [\tau,\tau + T]$ in the sense of Definition \ref{def-sol}.

 {\bf Step 3:} Derive uniform estimates of solutions.

For any initial value $\varphi \in L^2( \Omega, \mathcal{F}_\tau; D_{\rho} )$,  
applying It\^o's formula to \eqref{lmvlds-2}, we obtain
 \begin{align} \label{ue-1}
       & \| u(t)\|^2
         + 2 \int_{\tau}^{t} \|B u(s)\|_p^p ds
         + 2 \lambda \int_{\tau}^{t}\| u(s)\|^2 ds
         + 2 \int_{\tau}^{t}
               \langle f(u(s), \mathcal{L}_{u(s)}), u(s) \rangle ds  \nonumber \\
     = & \|\varphi(0)\|^2
         + 2 \int_{\tau}^{t} \langle g(s), u(s) \rangle ds
         + 2 \int_{\tau}^{t} \langle F(u(s-\rho),\mathcal{L}_{u(s-\rho)}),  u(s) \rangle ds  \nonumber \\
       & + 2 \sqrt{\varepsilon}\sum_{k \in \mathbb{N} } \int_{\tau}^{t}
                \langle \sigma_k(u(s), \mathcal{L}_{u(s)}) + h_k(s),
                        u(s)
                \rangle dW_k(s)
         +\varepsilon \sum_{k \in \mathbb{N} } \int_{\tau}^{t}
               \| \sigma_k(u(s), \mathcal{L}_{u(s)}) + h_k(s)\|^2 ds \nonumber \\
       & +  2 \sqrt{\varepsilon} 
              \int_{\tau}^{t} \int_{y \in \mathbb{Y} }
                   \langle
                              \widetilde{\sigma}(u(s-), \mathcal{L}_{u(s)}, y) 
                               + \widetilde{h}(s, y), 
                              u(s-) 
                   \rangle
              \widetilde{N}(ds, dy)  \nonumber \\
       & + \varepsilon  
              \int_{\tau}^{t} \int_{y \in \mathbb{Y} }
                  \| \widetilde{\sigma}(u(s-), \mathcal{L}_{u(s)}, y) 
                     + \widetilde{h}(s, y) \|^2
               N(ds, dy).
 \end{align}
 For the fourth term on the left-hand side of \eqref{ue-1},
 by \eqref{f1}, we obtain
 \begin{align}\label{ue-2}
        & 2 \int_{\tau}^{t}
             \langle f(u(s), \mathcal{L}_{u(s)}), u(s) \rangle ds  \nonumber \\
 \ge & 2 \lambda_1 \int_{\tau}^{t} \|u(s)\|^p_p ds
          - 2 \int_{\tau}^{t}
              \sum_{i \in \mathbb{Z}}
                    |\phi_{1,i}|
                     \left[ 1 +  |u_i(s)|^2  +  \mathcal{L}_{u(s)}(\|\cdot\|^2)
                     \right] ds  \nonumber \\
\ge & 2 \lambda_1 \int_{\tau}^{t} \|u(s)\|^p_p ds
          - 2 \|\phi_1\|_{\ell^\infty} \int_{\tau}^{t} \| u(s)\|^2 ds
          - 2 \int_{\tau}^{t} \|\phi_1\|_1
                    \left( 1 + \mathcal{L}_{u(s)}(\|\cdot\|^2) \right) ds.
 \end{align}
 By \eqref{ue-1} and \eqref{ue-2}, we 
 obtain for all $t\ge \tau$,
 \begin{align} \label{ue-3}
       & \mathbb{E} \left[ \sup_{r \in [\tau, t]} \|u(r)\|^2 \right]
         + 2 \mathbb{E} \left[ \int_{\tau}^{t} \|B u(s)\|_p^p ds \right]
         + 2 \lambda_1 \mathbb{E} \left[ \int_{\tau}^{t} \|u(s)\|^p_p ds \right]  \nonumber \\
  \le & 3 \mathbb{E} \left[ \|\varphi(0)\|^2 \right]
         + 6 \left( \|\phi_1\|_{\ell^\infty} + \|\phi_1\|_1 \right)
                 \mathbb{E} \left[ \int_{\tau}^{t} \| u(s)\|^2 ds \right]
         + 6 \|\phi_1\|_1 (t-\tau)  \nonumber \\
       & + \frac{3}{\lambda} \int_{\tau}^{t} \mathbb{E} \left[ \|g(s)\|^2 \right] ds
         + 6 \mathbb{E}
                \left[ \int_{\tau}^{t} \left| \langle F(u(s-\rho),\mathcal{L}_{u(s-\rho)}),  u(s) \rangle \right| ds
                \right] \nonumber \\
       & + 6 \mathbb{E}
                \left[ \sup_{r \in [\tau, t]}
                       \left| \int_{\tau}^{r} \sum_{k \in \mathbb{N} }
                                 \langle \sigma_k(u(s), \mathcal{L}_{u(s)}) + h_k(s),
                                         u(s)
                                 \rangle dW_k(s)
                       \right|
                \right]  \nonumber \\
       & + 3\mathbb{E}
              \left[ \sum_{k \in \mathbb{N} }
                        \int_{\tau}^{t}
                           \| \sigma_k(u(s), \mathcal{L}_{u(s)}) + h_k(s)\|^2 ds
              \right] \nonumber \\
       & + 6 \mathbb{E}
                \left[ \sup_{r \in [\tau, t]}
                       \left| 
                             \int_{\tau}^{r}
                              \int_{y \in \mathbb{Y} }
                \langle
              \widetilde{\sigma}(u(s-), \mathcal{L}_{u(s)},y) +  \widetilde{h}(s,y), 
                       u(s-) \rangle
               \widetilde{N}(ds, dy) 
                        \right|
                \right]  \nonumber \\
       & + 3 \mathbb{E}
                \left[  
                          \int_{\tau}^{t} \int_{y \in \mathbb{Y} }
                                \|\widetilde{\sigma}(u(s), \mathcal{L}_{u(s)},y) + \widetilde{h}(s,y)\|^2
                          \nu(dy) ds
                \right]. 
 \end{align}
              
 For the fifth term on the right-hand side of \eqref{ue-3},
 by \eqref{F1}, we obtain
 \begin{align}\label{ue-4}
        & 6 \mathbb{E}
                \left[ \int_{\tau}^{t} \left| \langle F(u(s-\rho),\mathcal{L}_{u(s-\rho)}),  u(s) \rangle \right| ds
                \right]  \nonumber \\
   \le & 6 \|F(0,\delta_0)\|^2 (t-\tau)
          + 12 \|L_{F}\|_{\ell^\infty}^2 \int_{\tau}^{t} \mathbb{E} \left[ \|u(s-\rho)\|^2 \right] ds  \nonumber \\
        & + 12 \|L_{F}\|^2 \int_{\tau}^{t} \mathbb{W}_2^2(\mathcal{L}_{u(s-\rho)}, \delta_0) ds
          + 3 \int_{\tau}^{t} \mathbb{E} \left[ \|u(s)\|^2 \right] ds  \nonumber \\
   \le & 6 \|F(0,\delta_0)\|^2 (t-\tau)
          + 3 \left( 4 \|L_{F}\|_{\ell^\infty}^2 + 4 \|L_{F}\|^2 + 1 \right)
            \int_{\tau}^{t} \mathbb{E} \left[ \|u(s)\|^2 \right] ds  \nonumber \\
        & + 12 \left( \|L_{F}\|_{\ell^\infty}^2 + \|L_{F}\|^2 \right)
            \int_{-\rho}^{0} \mathbb{E} \left[ \|\varphi(s)\|^2 \right] ds.
\end{align}
  Then by \eqref{ue-3} and \eqref{ue-4}, we have
  for $t\ge \tau$,
 \begin{align} \label{ue-5}
       & \mathbb{E} \left[ \sup_{r \in [\tau, t]} \|u(r)\|^2 \right]
         + 2 \mathbb{E} \left[ \int_{\tau}^{t} \|B u(s)\|_p^p ds \right]
         + 2 \lambda_1 \mathbb{E} \left[ \int_{\tau}^{t} \|u(s)\|^p_p ds \right]  \nonumber \\
  \le & 3 \mathbb{E} \left[ \|\varphi(0)\|^2 \right]
         + 3 \left( 2 \|\phi_1\|_{\ell^\infty} + 2 \|\phi_1\|_1 + 4 \|L_{F}\|_{\ell^\infty}^2 + 4 \|L_{F}\|^2 + 1
             \right)
           \mathbb{E}
              \left[ \int_{\tau}^{t} \| u(s)\|^2 ds
              \right]  \nonumber \\
       & + 6 \|\phi_1\|_1 (t-\tau)
         + \frac{3}{\lambda} \int_{\tau}^{t}  \|g(s)\|^2 ds
         + 6 \|F(0,\delta_0)\|^2 (t-\tau)  \nonumber \\
       & + 12 \left( \|L_{F}\|_{\ell^\infty}^2 + \|L_{F}\|^2 \right)
              \int_{-\rho}^{0} \mathbb{E} \left[ \|\varphi(s)\|^2 \right] ds \nonumber \\
       & + 6 \mathbb{E}
                \left[ \sup_{r \in [\tau, t]}
                       \left| \int_{\tau}^{r} \sum_{k \in \mathbb{N} }
                                 \langle \sigma_k(u(s), \mathcal{L}_{u(s)}) + h_k(s),
                                         u(s)
                                 \rangle dW_k(s)
                       \right|
                \right]  \nonumber \\
       & + 3 \mathbb{E}
                   \left[ \sum_{k \in \mathbb{N} } \int_{\tau}^{t}
                                \| \sigma_k(u(s), \mathcal{L}_{u(s)}) + h_k(s)\|^2 ds
                   \right] \nonumber \\
                   & + 6 \mathbb{E}
                \left[ \sup_{r \in [\tau, t]}
                       \left| 
                             \int_{\tau}^{r}
                              \int_{y \in \mathbb{Y}}
                \langle
              \widetilde{\sigma}(u(s-), \mathcal{L}_{u(s)},y) +  \widetilde{h}(s,y), 
                       u(s-) \rangle
               \widetilde{N}(ds, dy) 
                        \right|
                \right]  \nonumber \\
       & + 3 \mathbb{E}
                \left[  
                          \int_{\tau}^{t} \int_{y \in \mathbb{Y} }
                                \|\widetilde{\sigma}(u(s), \mathcal{L}_{u(s)},y) + \widetilde{h}(s,y)\|^2
                          \nu(dy) ds
                \right].
    \end{align}

 For the eighth and tenth terms on the right-hand side of \eqref{ue-5},
 by \eqref{sigma3} with $\tilde{\theta}=6$, we obtain
 \begin{align}\label{ue-6}
       & 3 \mathbb{E}
                   \left[ \sum_{k \in \mathbb{N} } \int_{\tau}^{t}
                             \| \sigma_k(u(s), \mathcal{L}_{u(s)}) + h_k(s)\|^2 ds
                   \right]  \nonumber \\
       & + 3 \mathbb{E}
                 \left[  
                           \int_{\tau}^{t} \int_{
                           y\in \mathbb{Y}
                           }
                                 \|\widetilde{\sigma}(u(s), \mathcal{L}_{u(s)},y) + \widetilde{h}(s,y)\|^2
                           \nu(dy) ds
                 \right]  \nonumber \\
 \le & 6 \mathbb{E}
               \left[ \sum_{k \in \mathbb{N} } \int_{\tau}^{t} \|\sigma_k(u(s), \mathcal{L}_{u(s)}) \|^2  ds
               \right]
          + 6 \int_{\tau}^{t}
                   \left( \|h(s)\|^2 + \|\widetilde{h}(s)\|^2_{L^2(\mathbb{Y}, \nu; \ell^2)}          
                    \right) ds  \nonumber \\
      & + 6 \mathbb{E}
                 \left[ 
                           \int_{\tau}^{t} \int_{
                           y\in \mathbb{Y}}
                                 \|\widetilde{\sigma}(u(s), \mathcal{L}_{u(s)},y) \|^2
                           \nu(dy) ds
                 \right]  \nonumber \\ 
 \le & \frac{\lambda_1}{4}
             \mathbb{E}
                \left[ \int_\tau^t \| u(s) \|^p_p ds \right]
          + 24 \|L\|^2 \int_\tau^t \mathbb{E} \left[ \|u(s)\|^2 \right] ds
          +  C_2  (t-\tau)  \nonumber \\
      & + 6 \int_{\tau}^{t}
                    \left( \|h(s)\|^2 
                             + \|\widetilde{h}(s)\|^2_{L^2(\mathbb{Y}, \nu; \ell^2)}   
                    \right)ds,
 \end{align}
 where $C_2 = C_2(p, q, \|L\|^2) > 0$ is a constant.

 For the seventh term on the right-hand side of \eqref{ue-5},
 by \eqref{sigma3}  and the Burkholder-Davis-Gundy (BDG) inequality, we have
 \begin{align}\label{ue-7}
       & 6 \mathbb{E}
                \left[ \sup_{r \in [\tau, t]}
                       \left| \int_{\tau}^{r} \sum_{k \in \mathbb{N} }
                                 \langle \sigma_k(u(s), \mathcal{L}_{u(s)}) + h_k(s),
                                         u(s)
                                 \rangle dW_k(s)
                       \right|
                \right]  \nonumber \\
 \le & 6 c_1 \mathbb{E}
                   \left[
                         \left( \sum_{k \in \mathbb{N} }
                                   \int_{\tau}^{t} \| u(s)\|^2
                                         \|\sigma_k(u(s), \mathcal{L}_{u(s)}) + h_k(s)\|^2 ds
                         \right)^{\frac{1}{2}}
                   \right] \nonumber\\
 \le & 6 c_1
          \mathbb{E}
             \left[ \sup_{s \in [\tau, t]} \| u(s)\|
                       \left( \int_{\tau}^{t}
                                 \sum_{k \in \mathbb{N} }
                                     \|\sigma_k(u(s), \mathcal{L}_{u(s)}) + h_k(s)\|^2
                               ds
                       \right)^{\frac{1}{2}}
             \right]  \nonumber\\
 \le & \frac{1}{4} \mathbb{E} \left[ \sup_{s \in [\tau, t]} \| u(s)\|^2 \right]
          + 36 c_1^2
            \mathbb{E}
               \left[ \int_{\tau}^{t}
                         \sum_{k \in \mathbb{N} } \|\sigma_k(u(s), \mathcal{L}_{u(s)}) + h_k(s)\|^2
                       ds
               \right]  \nonumber \\
 \le & \frac{1}{4} \mathbb{E} \left[ \sup_{s \in [\tau, t]} \| u(s)\|^2 \right]
          + \frac{\lambda_1}{4}
               \mathbb{E}
                  \left[ \int_{\tau}^{t} \|u(s)\|^p_p ds
                  \right]
          + 288 c_1^2 \|L\|^2
            \int_{\tau}^{t}
               \mathbb{E} \left[ \|u(s)\|^2 \right] ds  \nonumber \\
      & + C_3  (t-\tau)
          + 72 c_1^2
               \int_{\tau}^{t} \|h(s)\|^2 ds,
 \end{align}
 where $c_1>0$ is a constant
 from the BDG inequality, and
 $C_3=C_3(p,q, \|L\|^2, c_1)>0$ from
  \eqref{sigma3} with $\tilde{\theta}=72 c_1^2$. 

 For the ninth term on the right-hand side of \eqref{ue-5},
 it follows from the BDG inequality of discontinuous martingale in \cite[Theorem 3.50]{PZ2007} that
  \begin{align}\label{ue-8}
        & 
        6 \mathbb{E}
                \left[ \sup_{r \in [\tau, t]}
                       \left| 
                             \int_{\tau}^{r}
                              \int_{y \in \mathbb{Y}}
                \langle
              \widetilde{\sigma}(u(s-), \mathcal{L}_{u(s)},y) +  \widetilde{h}(s,y), 
                       u(s-) \rangle
               \widetilde{N}(ds, dy) 
                        \right|
                \right]   
           \nonumber \\
            \le & 6 c_2 
          \mathbb{E}
                \left[   \left (
                             \int_{\tau}^{t}
                              \int_{y \in \mathbb{Y}}
              \|
              \widetilde{\sigma}(u(s-), \mathcal{L}_{u(s)},y) +  \widetilde{h}(s,y)\|^2
                       \|u(s-) \|^2
              {N}(ds, dy)  
              \right )^{\frac 12}
                \right]   
             \nonumber \\
      \le & \frac{1}{4}
           \mathbb{E}
              \left[ \sup_{r \in [\tau, t]} \| u(s)\|^2 \right]
           + 36 c_2^2 \mathbb{E}
               \left[  
                         \int_{\tau}^{t} \int_{y \in \mathbb{Y} }
                               \| \widetilde{\sigma}(u(s), \mathcal{L}_{u(s)},y) + \widetilde{h}(s,y) \|^2
                         \nu (dy)ds
               \right],
  \end{align}
 where $c_2>0$ is a constant
 from the BDG inequality,
 which along with 
  \eqref{sigma3} with $\tilde{\theta}=72 c_2^2  $ implies that for $t\ge \tau$,
 \begin{align}\label{ue-9}
            & 
        6 \mathbb{E}
                \left[ \sup_{r \in [\tau, t]}
                       \left| 
                             \int_{\tau}^{r}
                              \int_{y \in \mathbb{Y}}
                \langle
              \widetilde{\sigma}(u(s-), \mathcal{L}_{u(s)},y) +  \widetilde{h}(s), 
                       u(s-) \rangle
               \widetilde{N}(ds, dy) 
                        \right|
                \right]   
           \nonumber \\
        \le & \frac{1}{4}
           \mathbb{E}
              \left[ \sup_{s \in [\tau, t]} \| u(s)\|^2 \right]
          + \frac{\lambda_1}{4}
               \mathbb{E}
                  \left[ \int_{\tau}^{t} \| u(s) \|^p_p ds
                  \right]
          +  288 c_2^2   \|L\|^2
              \int_{\tau}^{t} \mathbb{E} \left[ \| u(s) \|^2 \right] ds
                    \nonumber \\
        & + C_4  (t-\tau)
          + 72 c_2^2  
            \int_{\tau}^{t} 
            \|\widetilde{h}(s)\|^2
             _{L^2(\mathbb{Y}, \nu; \ell^2)} ds,
 \end{align}
 where $C_4 = C_4(p, q, \|L\|^2) > 0$ is a constant.

 From \eqref{ue-5}-\eqref{ue-7} and \eqref{ue-9},
 it follows that all $ t \in [\tau, \tau + T] $,
 \begin{align*}
       & \frac{1}{2} \mathbb{E} \left[ \sup_{r \in [\tau, t]} \|u(r)\|^2 \right]
         + 2 \mathbb{E} \left[ \int_{\tau}^{t} \|B u(s)\|_p^p ds \right]
         + \lambda_1 \mathbb{E} \left[ \int_{\tau}^{t} \|u(s)\|^p_p ds \right]  \nonumber \\
  \le & 3 \mathbb{E} \left[ \|\varphi(0)\|^2 \right]
         + 3 \left( 2 \|\phi_1\|_{\ell^\infty} + 2 \|\phi_1\|_1 + 4 \|L_{F}\|_{\ell^\infty}^2 + 4 \|L_{F}\|^2 + 1
             \right)
           \mathbb{E}
              \left[ \int_{\tau}^{t} \| u(s)\|^2 ds
              \right]  \nonumber \\
       & + 6 \|\phi_1\|_1 (t-\tau)
         + \frac{3}{\lambda} \int_{\tau}^{t}  \|g(s)\|^2 ds
         + 6 \|F(0,\delta_0)\|^2 (t-\tau)  \nonumber \\
       & + 12 \left( \|L_{F}\|_{\ell^\infty}^2 + \|L_{F}\|^2 \right)
              \int_{-\rho}^{0} \mathbb{E} \left[ \|\varphi(s)\|^2 \right] ds \nonumber \\
       & + 24 \left( 12 c_1^2 + 12 c_2^2 + 1 \right) \|L\|^2
            \int_{\tau}^{t}
               \mathbb{E}\left[ \|u(s)\|^2 \right] ds  \nonumber \\
       & + \left( C_2 +C_3 +C_4
           \right) (t-\tau)  \nonumber \\
       & + 6 \left( 12 c_1^2 + 1 \right)
               \int_{\tau}^{t} \|h(s)\|^2 ds
         + 6 \left( 12 c_2^2 + 1 \right)
            \int_{\tau}^{t} \|\widetilde{h}(s) \|^2
             _{L^2(\mathbb{Y}, \nu; \ell^2)}    ds
              ,
 \end{align*}
 which together with Gronwall's inequality, yields that for all $t \in [\tau, \tau+T]$,
 \begin{align*}
       & \mathbb{E} \left[ \sup_{r \in [\tau, t]} \|u(r)\|^2 \right]
         + 4 \mathbb{E} \left[ \int_{\tau}^{t} \|B u(s)\|_p^p ds \right]
         + 2 \lambda_1 \mathbb{E} \left[ \int_{\tau}^{t} \|u(s)\|^p_p ds \right]  \nonumber \\
  \le & \bar{C}_1  
         \left( \mathbb{E} \left[ \sup_{s \in [-\rho, 0]} \|\varphi(s)\|^2 \right]
                + T
                + \int_{\tau}^{\tau+T}
                        \left( \|g(s)\|^2 + \|h(s)\|^2 + \|\widetilde{h}(s)\|^2
         _{L^2(\mathbb{Y}, \nu; \ell^2)}                   
                        \right)
                  ds
         \right)
         e^{\bar{C}_2 (t-\tau) },
 \end{align*}
 where
 \begin{align*}
    \bar{C}_1 = & 6 + 24 \left( \|L_{F}\|_{\ell^\infty}^2 + \|L_{F}\|^2 \right)  
                            + 12 \|\phi_1\|_1 + 12 \|F(0,\delta_0)\|^2
                            + 2 \left( C_2 +C_3+C_4
                                  \right)  \\
                        & + \frac{6}{\lambda} 
                            + 12 \left( 12 c_1^2 + 1 \right)
                            + 12 \left( 12 c_2^2 + 1 \right),
 \end{align*}
 and
 \begin{align*}
    \bar{C}_2 = & 6 \left( 2 \|\phi_1\|_{\ell^\infty} 
                                         + 2 \|\phi_1\|_1 
                                         + 4 \|L_{F}\|_{\ell^\infty}^2 
                                         + 4 \|L_{F}\|^2 + 1
                                \right) 
                           + 48 \left( 12 c_1^2 + 12c_2^2 + 1 \right) \|L\|^2.
 \end{align*} 
 This completes the proof of \eqref{secondmoment}. 
 
 {\bf Step 4:} Derive uniform higher-order moment estimates  of solutions. 
 
Let 
 $u_\tau\in L^\theta(\Omega,\mathcal{F}_\tau;\ell^2)$. 
 For $m\in\mathbb N$, define a stopping time by
 \[ 
     \tau_m(\omega)
     = \inf\{t \in [\tau, +\infty) : \|u(t)\| > m \}\wedge(\tau+T),
 \]
 where $\inf \emptyset = +\infty$.  
 By Itô's formula, we obtain that for any $t \in [\tau, \tau+T]$,
\begin{align}\label{theta1}
	  & \mathbb{E} \left[ \| u(t \wedge \tau_m) \|^{\theta} \right]
	      + \theta
	             \mathbb{E}
	                    \left[ \int_{\tau}^{t \wedge \tau_m}
	                                      \| u(s) \|^{\theta-2} \, \| B u(s) \|_p^{p} ds
          	           \right]  \nonumber\\
	  & + \theta \lambda 
	             \mathbb{E}
	                   \left[ \int_{\tau}^{t \wedge \tau_m}
	                                     \| u(s) \|^{\theta} ds
	                   \right]  
	      + \theta
             	\mathbb{E}
             	       \left[ \int_{\tau}^{t \wedge \tau_m}
	                                     \| u(s) \|^{\theta-2}
	                                     \langle f(u(s), \mathcal{L}_{u(s)}), u(s) \rangle ds
	                  \right]  \nonumber\\
 \le & \mathbb{E}\left[ \| \varphi(0) \|^{\theta} \right] 
       + \theta
	           \mathbb{E}
	                 \left[ \int_{\tau}^{t \wedge \tau_m}
        	                           \| u(s) \|^{\theta-2}
	                                  \langle g(s), u(s) \rangle ds
	                \right]  \nonumber\\
	  & + \theta
	            \mathbb{E}
	                   \left[ \int_{\tau}^{t \wedge \tau_m}
	                                     \| u(s) \|^{\theta-2}
	                                     \langle F(u(s-\rho), \mathcal{L}_{u(s-\rho)}), u(s) \rangle ds
	                  \right]  \nonumber\\
	  & + \frac{\theta (\theta-1) }{2}
	             \mathbb{E}
	                   \left[ \sum_{k \in \mathbb{N} } 
	                                   \int_{\tau}^{t \wedge \tau_m}
	                                            \| u(s) \|^{\theta-2}
	                                            \| \sigma_k(u(s), \mathcal{L}_{u(s)}) + h_k(s) \|^{2} ds
	                   \right]  \nonumber\\
	  & + \mathbb{E}
	                \bigg[ \int_{\tau}^{t \wedge \tau_m}
	                             \int_{y \in \mathbb{Y} }
	                                 \Bigl( \| u(s) 
	                                             + \sqrt{\varepsilon}
	                                                       \widetilde{\sigma}(u(s), \mathcal{L}_{u(s)}, y)
	                                             + \sqrt{\varepsilon}\widetilde{h}(s, y) 
	                                           \|^{\theta}
      	                                       - \| u(s) \|^{\theta}  \nonumber\\
	  & \qquad\qquad\qquad\qquad
	                                           - \sqrt{ \varepsilon} \theta \| u(s) \|^{\theta-2}
	                                           \langle u(s),
	                                                      \widetilde{\sigma}(u(s), \mathcal{L}_{u(s)}, y)
	                                                      + \widetilde{h}(s,y) 
	                                          \rangle
	                                 \Bigr) \nu(dy) ds
	                \bigg].
\end{align}
For the last term on the left-hand side of  
\eqref{theta1}, 
by \eqref{f1} and {\bf Step 3},  we 
find that
 there exists $C_{1,T} \ge 1$ depending on $T$ such that 
\begin{align}\label{theta2}
      &	\theta \mathbb{E}
              \left[ \int_{\tau}^{t \wedge \tau_m}
	                            \| u(s) \|^{\theta-2} \,
	                            \big\langle f \big( u(s), \mathcal{L}_{u(s)} \big), u(s) \big\rangle ds
              \right]  \nonumber \\
\ge & \theta \mathbb{E}
               \left[ \int_{\tau}^{t \wedge \tau_m}
	                            \| u(s) \|^{\theta-2} \sum_{i \in \mathbb{Z}}
                            	\Bigl( \lambda_1 |u_i(s)|^{p}
	                                      - \phi_{1,i} 
	                                                       \bigl( 1 + |u_i(s)|^{2}
	                                                                   + \mathcal{L}_{u(s)}( \|\cdot\|^{2} ) 
	                                                       \bigr)
	                            \Bigr) ds
            	\right]  \nonumber\\
\ge &\theta \lambda_1
      	       \mathbb{E}
      	              \left[ \int_{\tau}^{t \wedge \tau_m}
	                                    \| u(s) \|^{\theta-2} \, \| u(s) \|_{p}^{p} ds
	                  \right]  
	      - \theta \| \phi_1 \|_{\ell^\infty} 
	              \mathbb{E}
	                     \left[ \int_{\tau}^{t \wedge \tau_m}
	                                      \| u(s) \|^{\theta} ds
	                     \right] \nonumber\\	                  
	   & - \theta
	             \mathbb{E}
	                    \left[ \int_{\tau}^{t \wedge \tau_m}
	                                      \| u(s) \|^{\theta-2} \, \| \phi_1 \|_{1}
	                                      \bigl( 1 + \mathbb{E} \left[ \| u(s) \|^{2} \right] \bigr) ds
	                    \right]  \nonumber\\
\ge & \theta \lambda_1
	            \mathbb{E}
	                   \left[ \int_{\tau}^{t \wedge \tau_m}
	                                     \| u(s) \|^{\theta-2} \, \| u(s) \|_{p}^{p} ds
	                   \right]
	     - \theta \| \phi_1 \|_{\ell^\infty} \,
	               \mathbb{E}
	                      \left[ \int_{\tau}^{t \wedge \tau_m}
	                                        \| u(s) \|^{\theta} ds
	                      \right]  \nonumber\\	                   
	  & - \theta \| \phi_1 \|_{1}
	              \left[ 1 + C_{1,T}
	                               \left( \mathbb{E} 
	                                               \left[ \sup_{s \in [-\rho,0]} \| \varphi(s) \|^{2} 
	                                               \right] + 1 
	                              \right)
	             \right]
	             \mathbb{E}
	                   \left[ \int_{\tau}^{t \wedge \tau_m}
	                                     \| u(s) \|^{\theta-2} ds
	                   \right]  \nonumber\\
\ge & \theta \lambda_1 \,
	            \mathbb{E}
	                   \left[ \int_{\tau}^{t \wedge \tau_m}
	                                     \| u(s) \|^{\theta-2} \, \| u(s) \|_{p}^{p} ds
	                   \right]  \nonumber\\
	  & - 2 \| \phi_1 \|_{1}
               \left[ 1 + C_{1,T}
	                                \bigg( \mathbb{E} 
	                                                \Big[ \sup_{s \in [-\rho,0]} \| \varphi(s) \|^{2} 
	                                                \Big] 
	                                          + 1 
	                               \bigg)
 	           \right] (t - \tau)  \nonumber\\
	  & - \left\{ \theta \| \phi_1 \|_{\ell^\infty}
	                   + (\theta - 2) \| \phi_1 \|_{1}
	                       \left[ 1 + C_{1,T}
       	                                        \bigg( \mathbb{E} 
       	                                                        \Big[ \sup_{s \in [-\rho,0]} \| \varphi(s) \|^{2} \Big] 
       	                                                  + 1 
       	                                        \bigg)
	                      \right]
	        \right\}
	        \mathbb{E}
	               \left[ \int_{\tau}^{t \wedge \tau_m}
	                                 \| u(s) \|^{\theta} ds
	              \right].
\end{align} 
 For the second term on the right-hand side of \eqref{theta1}, by Young's inequality, we have
 \begin{align} \label{thetag}
 	   \theta \mathbb{E}
 	                     \Biggl[ \int_{\tau}^{t\wedge\tau_m}
 	                                          \|u(s)\|^{\theta-2}
 	                                          \langle g(s), u(s) \rangle ds
 	                     \Biggr]
    \le (\theta-1)
             \mathbb{E}
                    \Biggl[ \int_{\tau}^{t\wedge\tau_m}
 	                                     \|u(s)\|^{\theta} ds
 	                \Biggr]
 	       + \int_{\tau}^{t} \|g(s)\|^{\theta} ds.
 \end{align}
 For the third term on the right-hand side of \eqref{theta1}, 
 by \eqref{F1}  and {\bf Step 3} we obtain 
 \begin{align}\label{thetaF}
 	& \theta \mathbb{E}\left[
 	\int_{\tau}^{t \wedge \tau_m}
 	\|u(s)\|^{\theta-2} \,
 	\big\langle F\big(u(s-\rho), \mathcal{L}_{u(s-\rho)}\big), u(s) \big\rangle \, ds
 	\right] \nonumber\\
 	& \quad \le
 	\theta \mathbb{E} \bigg[
 	\int_{\tau}^{t \wedge \tau_m}
 	\|u(s)\|^{\theta-2}
 	\Bigl(
 	 \|F(0,\delta_0)\|^2
 	+ 2 \|L_{F}\|_{\ell^\infty}^2 \, \|u(s-\rho)\|^2
 	\nonumber \\
 	& \qquad \qquad  \qquad \qquad \qquad \qquad 
 	   + 2 \|L_{F}\|^2 \, \mathbb{W}_2^2\big(\mathcal{L}_{u(s-\rho)}, \delta_0\big)
 	+ \tfrac{1}{2} \|u(s)\|^2
 	\Bigr) ds
 	\bigg] \nonumber\\
 	& \quad \le
 	\left(
 	(\theta-2) C_{T,\varphi}
 	+ 2 (\theta-2) \|L_F\|_{\ell^\infty}^2
 	+ \frac{\theta}{2}
 	\right)
 	\mathbb{E}\left[
 	\int_{\tau}^{t \wedge \tau_m}
 	\|u(s)\|^{\theta} \, ds
 	\right]
 	+  2C_{T,\varphi} (t-\tau) \nonumber\\
 	& \qquad
 	+ 4 \|L_F\|_{\ell^\infty}^2
 	\mathbb{E}\left[
 	\int_{\tau}^{t \wedge \tau_m}
 	\|u(s-\rho)\|^{\theta} \, ds
 	\right] \nonumber\\
 	& \quad \le
 	\left(
 	(\theta-2) C_{T,\varphi}
 	+ 2 (\theta-2) \|L_F\|_{\ell^\infty}^2
 	+ \frac{\theta}{2}
 	+ 4 \|L_F\|_{\ell^\infty}^2
 	\right)
 	\mathbb{E}\left[
 	\int_{\tau}^{t \wedge \tau_m}
 	\|u(s)\|^{\theta} \, ds
 	\right]
 	+  2C_{T,\varphi} (t-\tau) \nonumber\\
 	& \qquad
 	+ 4 \|L_F\|_{\ell^\infty}^2 \rho \,
 	\mathbb{E}\left[
 	\sup_{s \in [-\rho,0]} \|\varphi(s)\|^{\theta}
 	\right], 
 \end{align}
 where 
 $C_{T,\varphi}
    = \|F(0,\delta_0)\|^2 
        + 2 \|L_{F}\|^2
            C_{1,T} 
                         \bigg( \mathbb{E} 
                                                      \Big[ \sup_{s \in [-\rho,0]} \|\varphi(s)\|^2 
                                                      \Big]
                                    + 1 
                         \bigg). 
 $ \\
 For the fourth term on the right-hand side of \eqref{theta1},
 by Young's inequality
 and
 \eqref{sigma3} with $\tilde{\theta}$ replaced by $\theta-1$,
 we obtain
 \begin{align}\label{moment_gjg1}
 	   & \frac{\theta(\theta-1)}{2}
 	            \mathbb{E} 
 	                   \left[ \sum_{k \in \mathbb{N}}
                             	      \int_{\tau}^{t \wedge \tau_m}
 	                                       \|u(s)\|^{\theta-2}
 	                                       \|\sigma_k(u(s), \mathcal{L}_{u(s)}) + h_k(s)\|^2 ds
 	                   \right]  \nonumber\\
 \le & \theta(\theta-1)
             	\mathbb{E}
             	       \left[ \sum_{k \in \mathbb{N}}
 	                                  \int_{\tau}^{t \wedge \tau_m}
 	                                      \|u(s)\|^{\theta-2}
 	                                      \|\sigma_k(u(s), \mathcal{L}_{u(s)})\|^2 ds
 	                  \right]  \nonumber\\
 	  & + \theta(\theta-1)
 	                 \mathbb{E} 
 	                       \left[ \int_{\tau}^{t \wedge \tau_m}
 	                                    \|u(s)\|^{\theta-2}
 	                                    \|h(s)\|^2 ds
 	                       \right]  \nonumber\\[1ex]
 \le & \frac{\theta \lambda_1}{4}
 	            \mathbb{E}
 	                   \left[ \int_{\tau}^{t \wedge \tau_m}
 	                                \|u(s)\|^{\theta-2} \|u(s)\|_p^{p} ds
                       \right] 
 	         + 4 \theta(\theta-1) \|L\|^2 \,
         	         \mathbb{E} 
 	                       \left[ \int_{\tau}^{t \wedge \tau_m}
 	                                    \|u(s)\|^{\theta-2}
 	                                    \mathbb{E} \left[ \|u(s)\|^2 \right] ds
 	                       \right]  \nonumber\\
 	  & + \theta C_5
 	                 \mathbb{E} 
            	          \left[ \int_{\tau}^{t \wedge \tau_m}	\|u(s)\|^{\theta-2} ds
                         \right]  
 	         + (\theta-1)(\theta-2)
                 	\mathbb{E}
            	        \left[ \int_{\tau}^{t \wedge \tau_m}
                                 	\|u(s)\|^{\theta} ds
                        \right]
 	        + 2(\theta-1) \int_{\tau}^{t} \|h(s)\|^{\theta} ds  \nonumber \\[1ex]
 \le & \frac{\theta \lambda_1}{4} 
 	             \mathbb{E}
 	                   \left[ \int_{\tau}^{t \wedge \tau_m}
                                   	\|u(s)\|^{\theta-2} \|u(s)\|_p^{p} ds
 	                   \right]  \nonumber\\
 	  & + (\theta-2)
 	           \Bigg[ 4(\theta-1)\|L\|^2 C_{1,T}
 	                       \bigg( \mathbb{E} 
 	                                      \Big[ \sup_{s \in [-\rho,0]} \|\varphi(s)\|^2 
 	                                      \Big] + 1 
 	                       \bigg)  
                         + C_5
                         + \theta - 1
 	          \Bigg]
              \mathbb{E}
                     \left[ \int_{\tau}^{t \wedge \tau_m}
 	                                  \|u(s)\|^{\theta} ds
                  	\right]  \nonumber\\
 	  & + \left[ 8(\theta-1) \|L\|^2 C_{1,T}
                                \bigg( \mathbb{E} 
                                                \Big[ \sup_{s \in [-\rho,0]} \|\varphi(s)\|^2 \Big]
                                          + 1 
                                \bigg) 
 	                     + 2 C_5 
             	\right] (t-\tau)  
             	+ 2(\theta-1)	\int_{\tau}^{t}	\|h(s)\|^{\theta} ds,
 \end{align} 
 where $C_5 = C_2(p, q, \|L\|^2) > 0$ is a constant.
 
 For the last term on the right-hand side of \eqref{theta1},  applying Taylor's
  formula
  to the function
  $\left \| u(s)  
 	+\delta
 	\left (
 	\sqrt{\varepsilon}\widetilde{\sigma}(u(s), \mathcal{L}_{u(s)}, y)
 	+\sqrt{\varepsilon} \widetilde{h}(s,y)
 	\right )
 	\right \|^{\theta}
  $
  with respect to $\delta$
   we infer that
  there exists $\delta \in (0,1)$ such that
  \begin{align}\label{moment_gjhjan20}
 	 & \left \| u(s)  
 	                 + \sqrt{\varepsilon} \widetilde{\sigma}(u(s), \mathcal{L}_{u(s)}, y)
 	                 + \sqrt{\varepsilon} \widetilde{h}(s,y)
 	     \right \|^{\theta}
 	     - \|u(s)\|^{\theta}  \nonumber\\  
 	&  - \theta\sqrt{\varepsilon} \|u(s)\|^{\theta-2}
 	       \big\langle
                        	 u(s),\
 	                        \widetilde{\sigma}(u(s), \mathcal{L}_{u(s)}, y)
 	                        + \widetilde{h}(s, y)
 	       \big\rangle  \nonumber\\
 = &
 {\frac 12} \theta (\theta -2)
  \left \| u(s)
 	+ \delta \sqrt{\varepsilon}\widetilde{\sigma}(u(s), \mathcal{L}_{u(s)}, y)
 	+ \delta \sqrt{\varepsilon} \widetilde{h}(s,y)
 	\right \|^{\theta-4}
 	\nonumber\\
     & \  \cdot
  \left (
  	\big\langle
  	u(s)
 	+ \delta \sqrt{\varepsilon}\widetilde{\sigma}(u(s), \mathcal{L}_{u(s)}, y)
 	+ \delta \sqrt{\varepsilon} \widetilde{h}(s,y),
 	\
   \sqrt{\varepsilon}\widetilde{\sigma}(u(s), \mathcal{L}_{u(s)}, y)
 	+  \sqrt{\varepsilon} \widetilde{h}(s,y)
  	 \big\rangle
  	 \right )^2	\nonumber\\
  & + 
    {\frac 12} \theta
  \left \| u(s)
 	+ \delta \sqrt{\varepsilon}\widetilde{\sigma}(u(s), \mathcal{L}_{u(s)}, y)
 	+ \delta \sqrt{\varepsilon} \widetilde{h}(s,y)
 	\right \|^{\theta-2}
 \left \| 
 	  \sqrt{\varepsilon}\widetilde{\sigma}(u(s), \mathcal{L}_{u(s)}, y)
 	+   \sqrt{\varepsilon} \widetilde{h}(s,y)
 	\right \|^{2} \nonumber\\
 \le &  
    {\frac 12} \theta
    (\theta-1)
  \left \| u(s)
 	+ \delta \sqrt{\varepsilon}\widetilde{\sigma}(u(s), \mathcal{L}_{u(s)}, y)
 	+ \delta \sqrt{\varepsilon} \widetilde{h}(s,y)
 	\right \|^{\theta-2}
 \left \| 
 	  \sqrt{\varepsilon}\widetilde{\sigma}(u(s), \mathcal{L}_{u(s)}, y)
 	+   \sqrt{\varepsilon} \widetilde{h}(s,y)
 	\right \|^{2}\nonumber\\
 \le & 
      \theta(\theta-1) 2^{\theta -3}
  \left \| u(s)
   \right \|^{\theta-2}
 \left \| 
 	  \sqrt{\varepsilon}\widetilde{\sigma}(u(s), \mathcal{L}_{u(s)}, y)
 	+   \sqrt{\varepsilon} \widetilde{h}(s,y)
 	\right \|^{2}
 	\nonumber\\
 	& + 
      \theta(\theta-1) 2^{\theta -3}
  \left \| 
 	  \sqrt{\varepsilon}\widetilde{\sigma}(u(s), \mathcal{L}_{u(s)}, y)
 	+   \sqrt{\varepsilon} \widetilde{h}(s,y)
 	\right \|^{\theta}.
\end{align}
By \eqref{moment_gjhjan20} we have
for all $\varepsilon\in (0,1]$,
 \begin{align}\label{moment_gjh}
 	& \mathbb{E}\left[
 	\int_{\tau}^{t \wedge \tau_m}
  \int_{y\in \mathbb{Y}
  }
 	\Big(
 	\|u(s)
 	+ \sqrt{\varepsilon}\widetilde{\sigma}(u(s), \mathcal{L}_{u(s)}, y)
 	+\sqrt{\varepsilon} \widetilde{h}(s,y)\|^{\theta}
 	- \|u(s)\|^{\theta} \right. \nonumber\\
 	& \qquad\qquad\qquad\quad\
 	\left.
 	- \theta\sqrt{\varepsilon} \|u(s)\|^{\theta-2}
 	\big\langle
 	u(s),
 	\widetilde{\sigma}(u(s), \mathcal{L}_{u(s)}, y)
 	+ \widetilde{h}(s,y)
 	\big\rangle
 	\Big) \nu (dy) \, ds
 	\right] \nonumber\\
 	\le &
 	C_{\theta} 
 	\mathbb{E}\left[
 	\int_{\tau}^{t \wedge \tau_m}
 	\|u(s)\|^{\theta-2}
  \int_{y\in \mathbb{Y}
  }
 	\|\widetilde{\sigma}(u(s), \mathcal{L}_{u(s)}, y)
 	+ \widetilde{h}(s,y)\|^{2}
 	\,\nu(dy) \, ds
 	\right] \nonumber\\
 	& + C_{\theta} 
 	\mathbb{E}\left[
 	\int_{\tau}^{t \wedge \tau_m}
  \int_{y\in\mathbb{Y}
  }
 	\|\widetilde{\sigma}(u(s), \mathcal{L}_{u(s)}, y)
 	+ \widetilde{h}(s,y)\|^{\theta}
 	\, \nu(dy) \, ds
 	\right] \nonumber\\
 	=:& I_1 + I_2,
 \end{align}
 where
 $ C_{\theta} = \theta(\theta-1) 2^{\theta-3}
 $.
 
 Similar to the argument of \eqref{moment_gjg1},
 by \eqref{sigma3} with $\tilde{\theta}$ replaced by $2 \theta^{-1} C_{\theta}$ and Young's inequality,
 we obtain
 \begin{align} \label{moment_gjh_1}
 	 I_1
 	\le & \frac{\theta \lambda_1}{4} 
 	\mathbb{E}\left[
 	\int_{\tau}^{t \wedge \tau_m} 
 	\|u(s)\|^{\theta-2} \|u(s)\|^p_p ds 
 	\right]
   +
   C_6  \left( 
   1+   \mathbb{E} \Big (
    \sup_{s \in [-\rho,0]} \|\varphi(s)\|^2 
    \Big ) \right)   
 	    \mathbb{E} 
 	          \left (
 	           \int_{\tau}^{t \wedge \tau_m} \|u(s)\|^{\theta} ds 
 	          \right ) \nonumber\\
 	& + 
 	  C_6  \left( 
   1+   \mathbb{E} \Big (
    \sup_{s \in [-\rho,0]} \|\varphi(s)\|^2 
    \Big ) \right)   (t-\tau)
    +C_6   \int_{\tau}^{t} \|\widetilde{h}(s)
    \|^{\theta} _{L^2(
    \mathbb{Y},
    \nu; \ell^2)}ds . 
 \end{align}
 where $C_6=C_6(p,q, L, T, \theta)>0$
 is a constant.
 
  By \eqref{sigma2} and H\"older's inequality, we have
 \begin{align}\label{h_22_simple}
 	\|\widetilde{\sigma}(u(s), \mathcal{L}_{u(s)}, y)\|
 	& =
 	\Bigg(
 	\sum_{i \in \mathbb{Z}}
 	\big|\widetilde{\sigma}_{i}(u_i(s), \mathcal{L}_{u(s)}, y)\big|^2
 	\Bigg)^{\frac12}                                            \nonumber\\
 	& \le
 	\Bigg(
 	\sum_{i \in \mathbb{Z}}
 	\Big(
 	L_{\tilde{\sigma},i}(y)
 	\Big(|u_i(s)|^{\frac{q}{2}}
 	+ \sqrt{\mathcal{L}_{u(s)}(\|\cdot\|^2)}\Big)
 	+ |\widetilde{\sigma}_{i}(0, \delta_0, y)|
 	\Big)^2
 	\Bigg)^{\frac12}                                            \nonumber\\
 	& \le
 	\sqrt{2}
 	\Bigg(
 	\sum_{i \in \mathbb{Z}}
 	\big(L_{\tilde{\sigma},i}(y)\big)^2
 	\Big(|u_i(s)|^{\frac{q}{2}}
 	+ \sqrt{\mathcal{L}_{u(s)}(\|\cdot\|^2)}\Big)^2
 	\Bigg)^{\frac12}
 	 + \sqrt{2}
 	\Bigg(
 	\sum_{i \in \mathbb{Z}}
 	|\widetilde{\sigma}_{i}(0, \delta_0, y)|^2
 	\Bigg)^{\frac12}                                          \nonumber\\
 	& \le
 	\sqrt{2}
 	\sum_{i \in \mathbb{Z}}
 	L_{\tilde{\sigma},i}(y)
 	\Big(|u_i(s)|^{\frac{q}{2}}
 	+ \sqrt{\mathcal{L}_{u(s)}(\|\cdot\|^2)}\Big) 
 	+ \sqrt{2}
 	\sum_{i \in \mathbb{Z}}
 	|\widetilde{\sigma}_{i}(0, \delta_0, y)|              \nonumber\\
 	& \le
 	\sqrt{2}
 	\Bigg(
 	\sum_{i \in \mathbb{Z}}
 	\big(L_{\tilde{\sigma}, i}(y)\big)^{\frac{\theta}{\theta-1}}
 	\Bigg)^{\frac{\theta-1}{\theta}}
 	\Bigg(
 	\sum_{i \in \mathbb{Z}}
 	|u_i(s)|^{\frac{q\theta}{2}}
 	\Bigg)^{\frac1\theta}                                      \nonumber\\
 	& \quad
 	+ \sqrt{2}\,\sqrt{\mathcal{L}_{u(s)}(\|\cdot\|^2)}
 	\sum_{i \in \mathbb{Z}}
 	L_{\tilde{\sigma}, i}(y)
 	+ \sqrt{2}
 	\sum_{i \in \mathbb{Z}}
 	|\widetilde{\sigma}_{i}(0, \delta_0, y)| .
 \end{align}
 By \eqref{h_22_simple} and Young's inequality, we obtain
\begin{align}\label{moment_gjh11}
   I_2
 \le & 2^{\theta-1} C_{\theta} 
	           \mathbb{E} 
	                 \left[ \int_{\tau}^{t \wedge \tau_m}
	                              \int_{y \in \mathbb{Y} }
         	                          \big\| \widetilde{\sigma}(u(s), 
         	                                   \mathcal{L}_{u(s)}, y) 
         	                          \big\|^{\theta}
	                           \nu(dy) ds
    	            \right]  
    	   + 2^{\theta-1} C_{\theta} 
	               \int_{\tau}^{t}
                         \big\| \widetilde{h}(s) \big\|^{\theta}_{L^{\theta}(  \mathbb{Y}, \nu; \ell^2)} ds
                                      \nonumber\\[1ex]
 \le & 2^{\theta-1} C_{\theta}
	            \mathbb{E} 
	                  \Bigg[ \int_{\tau}^{t \wedge \tau_m}
                                   \int_{y \in \mathbb{Y} }
	                                   \bigg( \sqrt{2}
	                                                   \Big( \sum_{i \in \mathbb{Z}}
	                                                                       \big( L_{\tilde{\sigma}, i}(y) \big)^{\frac{\theta}{\theta-1}}
	                                                   \Big)^{\frac{\theta-1}{\theta}}
	                                                   \Big( \sum_{i \in \mathbb{Z}}
	                                                                  |u_i(s)|^{\frac{q \theta}{2}}
	                                                  \Big)^{\frac{1}{\theta}}  \nonumber\\
	  &                                           + \sqrt{2} \sum_{i \in \mathbb{Z}} 
	                                                         L_{\tilde{\sigma}, i}(y)
	                                                               \sqrt{\mathcal{L}_{u(s)}(\|\cdot\|^2) }
	                                                + \sqrt{2} \sum_{i \in \mathbb{Z}}
	                                                         \big| \widetilde{\sigma}_{i}(0, \delta_0, y) \big|
	                                \bigg)^{\theta}
	                        \nu(dy) ds
               	 \Bigg]  \nonumber\\
	  & + 2^{\theta-1} C_{\theta} 
           	       \int_{\tau}^{t}
                       \big\| \widetilde{h}(s) \big\|
                       ^{\theta}_{L^{ \theta}( \mathbb{Y}, \nu; \ell^2)} \, ds
	                                    \nonumber\\[1ex]
 \le & 6^{\theta-1} 2^{\frac{\theta}{2}} C_{\theta}
                \int_{y \in \mathbb{Y} }
	                \Big( \sum_{i \in \mathbb{Z} }
	                              \big( L_{\tilde{\sigma}, i}(y) \big)^{\frac{\theta}{\theta-1}}
	               \Big)^{\theta-1}
	          \nu(dy)
	          \mathbb{E} 
	                \left[ \int_{\tau}^{t \wedge \tau_m}
	                             \sum_{i \in \mathbb{Z}}
	                                   |u_i(s)|^{\frac{q \theta}{2}} ds
	                \right]  \nonumber\\
	 & + 6^{\theta-1} 2^{\frac{\theta}{2}} C_{\theta}
                  \int_{y \in \mathbb{Y} }
       	              \Big( \sum_{i \in \mathbb{Z} }
	                                 L_{\tilde{\sigma}, i}(y)
	                  \Big)^{\theta}
	               \nu(dy)
	         C_{1,T}^{\frac{\theta}{2}}
	         \bigg( \mathbb{E} 
	                         \Big[ \sup_{s \in [-\rho,0]} \|\varphi(s)\|^2
	                         \Big] + 1
	         \bigg)^{\frac{\theta}{2}}
	         (t-\tau)   \nonumber\\
	  & + 6^{\theta-1} 2^{\frac{\theta}{2}} C_{\theta}
                   \int_{y \in \mathbb{Y} }
                   	   \Big( \sum_{i \in \mathbb{Z}}
	                                  \big| \widetilde{\sigma}_{i}(0, \delta_0, y) \big|
	                   \Big)^{\theta}
	                \nu(dy) (t-\tau)  
	     + 2^{\theta-1} C_{\theta}
	               \int_{\tau}^{t} 
	                   \big\| \widetilde{h}(s) \big\|
	                   ^{\theta}_{L^{\theta}( \mathbb{Y}, \nu; \ell^2)} ds.
 \end{align}

 Let 
 $ a = \tfrac{\theta(q-2)}{2(p-2)}$ and $
     b = (1-a) \theta.
 $
Then $a \in (0,1), \tfrac{b}{1-a} > 2$ 
 and $a (\theta-2+p) + b = \tfrac{q \theta}{2}$.
 For the first term on the right-hand side of \eqref{moment_gjh11},
 by H\"older's and Young's inequalities, we have
 \begin{align}\label{I11}
 	   & 6^{\theta-1} 2^{\frac{\theta}{2}} C_{\theta}  
 	             \int_{y\in\mathbb{Y}
 	             } 
                     \Big( \sum_{i \in \mathbb{Z}}
 	                                  \big( L_{\tilde{\sigma},i}(y) \big)^{\frac{\theta}{\theta-1}}
 	                 \Big)^{\theta-1} \nu(dy) 
 	                 \mathbb{E} 
 	                         \left[ \int_{\tau}^{t \wedge \tau_m}
 	                                      \sum_{i \in \mathbb{Z}}
 	                                      |u_i(s)|^{\frac{q \theta}{2}} ds 
                    	     \right]  \nonumber \\
   = & 6^{\theta-1} 2^{\frac{\theta}{2}} C_{\theta}  
                	\int_{y \in \mathbb{Y} } 
                      	\Big( \sum_{i \in \mathbb{Z}}
 	                                     \big( L_{\tilde{\sigma},i}(y) 
 	                                     \big)^{\frac{\theta}{\theta-1}}
                    	\Big)^{\theta-1} \nu(dy) 
                        \mathbb{E} 
 	                            \left[ \int_{\tau}^{t \wedge \tau_m}
 	                                         \sum_{i \in \mathbb{Z}}
 	                                               \big(|u_i(s)|^{\theta-2+p}\big)^{a}
                                                	|u_i(s)|^{\,b} ds 
                            	\right]  \nonumber\\
  \le & 6^{\theta-1} 2^{\frac{\theta}{2}} C_{\theta}  
 	                    \int_{y \in \mathbb{Y} } 
 	                        \Big( \sum_{i \in \mathbb{Z}}
 	                                         \big( L_{\tilde{\sigma},i}(y) 
 	                                         \big)^{\frac{\theta}{\theta-1}}
 	                        \Big)^{\theta-1} \nu(dy)  \nonumber \\
 	   & \cdot \mathbb{E} 
 	                      \left[ \int_{\tau}^{t \wedge \tau_m}
 	                                   \Big( \sum_{i \in \mathbb{Z}} |u_i(s)|^{\theta-2+p} 
 	                                   \Big)^{a}
 	                                   \Big( \sum_{i \in \mathbb{Z}} 
 	                                                     |u_i(s)|^{\frac{b}{1-a}}
 	                                   \Big)^{1-a} ds 
 	                       \right]  \nonumber\\
 \le & 6^{\theta-1} 2^{\frac{\theta}{2}} C_{\theta} 
                    \int_{y \in \mathbb{Y} } 
 	                            \Big( \sum_{i \in \mathbb{Z}}
 	                                             \big( L_{\tilde{\sigma},i}(y) 
 	                                             \big)^{\frac{\theta}{\theta-1}}
 	                            \Big)^{\theta-1} \nu (dy)  
 	               \mathbb{E} 
 	                        \left[ \int_{\tau}^{t \wedge \tau_m}
 	                                     \big( \|u(s)\|^{\theta-2} \|u(s)\|_{p}^{p} 
 	                                     \big)^{a}
                                     	\|u(s)\|^{b} ds 
                        	\right]  \nonumber\\
 \le & \frac{\theta \lambda_1 }{4}
 	                     \mathbb{E} 
 	                           \left[ \int_{\tau}^{t \wedge \tau_m} 
                                          	\|u(s)\|^{\theta -2} \|u(s)\|_p^{p} ds
                               \right]  \nonumber\\
 	  & + (1-a) 
 	            \bigg( 6^{\theta-1} 2^{\frac{\theta}{2}} C_{\theta}  
                              \int_{y \in \mathbb{Y} } 
                       	          \Big( \sum_{i \in \mathbb{Z}}
 	                                                \big( L_{\tilde{\sigma},  i}(y) \big)^{\frac{\theta}{\theta-1}}
 	                              \Big)^{\theta-1} \nu(dy) 
 	            \bigg)^{\frac{1}{1-a}} 
 	            \left( \frac{4 a}{\lambda_1 \theta}
 	            \right)^{\frac{a}{1-a}}  \nonumber\\
 	    & \quad
 	           \cdot \mathbb{E} 
 	                           \left[ \int_{\tau}^{t \wedge \tau_m} \|u(s)\|^{\theta} ds
 	                          \right].
 \end{align}
 
 Using assumptions \eqref{thetaassum1}–\eqref{thetaassum2}, we have
\begin{align*}
      &  
                \int_{
                y\in \mathbb{Y}
                }
	                     \left( \sum_{i \in \mathbb Z} ( L_{\tilde\sigma,i}(y) )^{\frac{\theta}{\theta-1}}
	                    \right)^{\theta-1}
	             \nu(dy) 
	        < \infty,  \nonumber\\
      &	 
            	\int_{y\in\mathbb{Y}
            	} 
	                 \left( \sum_{i \in \mathbb{Z}}
                                	\big| \widetilde{\sigma}_{i}(0, \delta_0, y) \big|
                 	\right)^{\theta} \nu(dy)
	       < \infty, \nonumber\\
      & \int_{\tau}^{t}  
	             \big\| \widetilde{h}(s)\big\|
	             ^{\theta} _{	 
  L^2(  \mathbb{Y}, \nu; \ell^2)}ds
	      \ \vee\ 
	      \int_{\tau}^{t}  
                 \big\| h(s) \big\|^{\theta} ds
          \ \vee \ 
		  \int_{\tau}^{t} 
   	             \big\| g(s) \big\|^{\theta} ds
	       < \infty. 
\end{align*}
Then from \eqref{theta1}–\eqref{I11} it follows that there exists a constant 
$C_{\theta,T,\varphi}>0$, depending only on $\theta$, $T$ and $\varphi$, such that
for all $t\in[\tau,\tau+T]$ and $m\in\mathbb{N}$,
\begin{align}\label{moment_hb}
	  & \mathbb{E} \left[ \|u(t \wedge \tau_m)\|^{\theta} \right]
          + \frac{\theta \lambda_1}{4} 
       	           \mathbb{E} 
       	                  \bigg[ \int_{\tau}^{t \wedge \tau_m}
	                                      \|u(s)\|^{\theta-2} \|u(s)\|_p^p ds 
	                      \bigg]  \nonumber\\
\le & C_{\theta, T, \varphi}
	       \mathbb{E}
	               \bigg[ \int_{\tau}^{t \wedge \tau_m} 
	                               \|u(s)\|^{\theta} ds 
	               \bigg]
	        + C_{\theta,T,\varphi}.
\end{align}
Applying Gronwall's inequality to \eqref{moment_hb} we obtain, for all
$t\in[\tau,\tau+T]$ and $m\in\mathbb{N}$,
\begin{align}\label{moment_hb1}
	\mathbb{E}\big[ \|u(t \wedge \tau_m)\|^{\theta} \big]
	  + \frac{\theta \lambda_1}{4} 
       	           \mathbb{E} 
       	                  \bigg[ \int_{\tau}^{t \wedge \tau_m}
	                                      \|u(s)\|^{\theta-2} \|u(s)\|_p^p ds 
	                      \bigg]   
	\le C_{\theta,T,\varphi} e^{C_{\theta,T,\varphi} T}.
\end{align}
Letting $m \to +\infty$ in \eqref{moment_hb1} and using Fatou's lemma, we infer
that, for all $t \in [\tau,\tau+T]$,
\begin{align}\label{moment_hb2}
	\mathbb{E}\big[ \|u(t)\|^{\theta} \big]
	  + \frac{\theta \lambda_1}{4} 
       	           \mathbb{E} 
       	                  \bigg[ \int_{\tau}^{t}
	                                      \|u(s)\|^{\theta-2} \|u(s)\|_p^p ds 
	                      \bigg]   
	\le C_{\theta,T,\varphi} e^{C_{\theta,T,\varphi} T}.
\end{align}

By \eqref{lmvlds-2}  and It\^o's formula, we obtain that for all $t \ge \tau$,
\begin{align}\label{uniform-1}
	  & \mathbb{E} 
	           \left[ \sup_{r \in [\tau, t]} 
	                         \left( \| u(r) \|^{\theta} 
	                                  + \theta
	                                           \int_{\tau}^{r}
	                                           	    \| u(s) \|^{\theta-2}
	                                                     \langle f(u(s), \mathcal{L}_{u(s)}), u(s) \rangle ds 
	                         \right)
          	   \right]  \nonumber\\
 \le & \mathbb{E}\left[ \| u(\tau) \|^{\theta} \right] 
          + \theta
	           \mathbb{E}
	                 \left[ \int_{\tau}^{t}
        	                           \| u(s) \|^{\theta-2}
	                                  \big| \langle g(s), u(s) \rangle \big| ds
	                \right]  \nonumber\\
	  & + \theta
	            \mathbb{E}
	                   \left[ \int_{\tau}^{t}
	                                     \| u(s) \|^{\theta-2}
	                                     \big| \langle F(u(s-\rho), \mathcal{L}_{u(s-\rho)}), u(s) \rangle \big| ds
	                  \right]  \nonumber\\
	  & + \frac{\theta (\theta-1) }{2}
	             \mathbb{E}
	                   \left[ \sum_{k \in \mathbb{N} } 
	                                   \int_{\tau}^{t}
	                                            \| u(s) \|^{\theta-2}
	                                            \| \sigma_k(u(s), \mathcal{L}_{u(s)}) + h_k(s) \|^{2} ds
	                   \right]  \nonumber\\
	  & + \mathbb{E}
	                \bigg[ \int_{\tau}^{t}
	                             \int_{y \in \mathbb{Y} }
	                             	 \Big| 
 	                                   \| u(s-)
 	                                      + \sqrt{\varepsilon} 
 	                                         \widetilde{\sigma}\big( u(s-), \mathcal{L}_{u(t)}, y \big)
 	                                      + \sqrt{\varepsilon} \widetilde{h}(s,y)
 	                                   \|^{\theta}
 	                                   - \|u(s-)\|^{\theta}   \nonumber\\
 	& \qquad\qquad\qquad\quad
 	                                   - \theta \|u(s-)\|^{\theta-2}
 	                                          \sqrt{\varepsilon}
 	                                                \big\langle
 	                                                    \widetilde{\sigma}\big(u(s-), \mathcal{L}_{u(t)}, y\big)
 	                                                    + \widetilde{h}(s,y),
 	                                                    u(s-)
 	                                                \big\rangle 
 	                                  \Big| N(ds,dy)	                                 
	               \bigg] \nonumber\\
& + \mathbb{E} 
	         \left[ \sup_{r \in [\tau, t]}
                     	\Big| \theta 
 	                               \int_{\tau}^{r}
                             	 \|u(s)\|^{\theta-2}
 	                            \big\langle  \sum_{k \in \mathbb{N} }
                                   	\left( \sigma_k(u(s), \mathcal{L}_{u(s)}) + h_k(s) \right) dW_k(s),
 	                               u(s)
                            	\big\rangle
                     	\Big|
              \right] \nonumber\\
 	& +  \mathbb{E} 
 		         \left[ \sup_{r \in [\tau,t]}
 	                 \Big| \int_{\tau}^{r} 
 	                            \int_{y \in \mathbb{Y} }
 	                                \|u(s-)\|^{\theta-2}
                                    	\big\langle
 	                                             \widetilde{\sigma}\big(u(s-), \mathcal{L}_{u(t)}, y\big)
 	                                             + \widetilde{h}(s,y),
 	                                             u(s-)
             	                        \big\rangle
 	                                 \widetilde{N}(ds,dy)
 	                 \Big|
 	              \right].
\end{align}

For the last two terms in \eqref{uniform-1}, 
from the BDG inequality, it follows that there exists a constant $C>0$ such that
\begin{align} \label{uniform-2}
	  &  \mathbb{E} 
	  	         \left[ \sup_{r \in [\tau, t]}
	                       	\Big| \theta 
	   	                               \int_{\tau}^{r}
	                               	 \|u(s)\|^{\theta-2}
	   	                            \big\langle  \sum_{k \in \mathbb{N} }
	                                     	\left( \sigma_k(u(s), \mathcal{L}_{u(s)}) + h_k(s) \right) dW_k(s),
	   	                               u(s)
	                              	\big\rangle
	                       	\Big|
	                \right] \nonumber\\
 	& +  \mathbb{E} 
 		         \left[ \sup_{r \in [\tau,t]}
 	                 \Big| \int_{\tau}^{r} 
 	                            \int_{y \in \mathbb{Y} }
 	                                \|u(s-)\|^{\theta-2}
                                    	\big\langle
 	                                             \widetilde{\sigma}\big(u(s-), \mathcal{L}_{u(t)}, y\big)
 	                                             + \widetilde{h}(s,y),
 	                                             u(s-)
             	                        \big\rangle
 	                                 \widetilde{N}(ds,dy)
 	                 \Big|
 	              \right]  \nonumber\\	                
 \le & \theta C
        	\mathbb{E}
        	     \Biggl[
 	                  \biggl( \int_{\tau}^{t}
                                    	\|u(s)\|^{2\theta-2}
 	                                       \sum_{k \in \mathbb{N} }
 	                                           \bigl\| \sigma_k(u(s),\mathcal{L}_{u(s)}) + h_k(s)
 	                                           \bigr\|^{2}
 	                                ds
 	                 \biggr)^{\frac12}
 	             \Biggr]  \nonumber \\
	 & + \theta C
            	\mathbb{E}
            	     \Biggl[
                       	 \Biggl( \int_{\tau}^{t} 
                       	                \int_{\mathbb{Y} } 
                                       	   \|u(s-)\|^{2\theta-2}
                    	                   \bigl\| \widetilde{\sigma}(u(s-),\mathcal{L}_{u(s)}, y) 
	                                                 + \widetilde{h}(s,y)
	                                       \bigr\|^{2}
	                                 N(ds, dy)
	                     \Biggr)^{\frac{1}{2}}
	                 \Biggr]  \nonumber \\ 	 
   \le & \frac{1}{2} 
            \mathbb{E}
         	     \bigg[ \sup_{s \in [\tau, t] } \|u(s)\|^{\theta}
         	     \bigg]
           + \theta^2 C^2
           	    \mathbb{E}
 	                 \bigg[ \int_{\tau}^{t}
                                     \|u(s)\|^{\theta-2}
 	                                 \sum_{k \in \mathbb{N} }
 	                                       \bigl\| \sigma_k(u(s),\mathcal{L}_{u(s)}) + h_k(s)
 	                                       \bigr\|^{2}
 	                             ds
 	                 \bigg]  \nonumber \\
	   & + \theta^2 C^2
            	\mathbb{E}
            	      \bigg[ \int_{\tau}^{t} 
                       	                \int_{\mathbb{Y} } 
                                       	   \|u(s)\|^{\theta-2}
                    	                   \bigl\| \widetilde{\sigma}(u(s),\mathcal{L}_{u(s)}, y) 
	                                                 + \widetilde{h}(s,y)
	                                       \bigr\|^{2}
	                                 \nu(dy) ds
	                    \bigg]. 		                    		                             
\end{align}
Then together with the arguments in \eqref{theta2}-\eqref{moment_gjg1} and \eqref{moment_gjh}-\eqref{I11}, 
it follows from \eqref{moment_hb2}-\eqref{uniform-2} that there exists $C_{\theta, T, \varphi}^\prime> 0$ depending only on $\theta$, $T$ and $\varphi$ such that for all $t \in [\tau, \tau+T]$,
\begin{align*}
	  \mathbb{E} 
	        \left[ \sup_{r \in [\tau, t]} 
	                          \| u(r) \|^{\theta} 
	        \right]
	  + \mathbb{E} 
	  	        \left[ \int_{\tau}^{t}
	  	                                      \|u(s)\|^{\theta-2} \|u(s)\|_p^p ds 
	  	        \right]   
     \le C_{\theta, T, \varphi}^\prime,
\end{align*}
as desired. 
 \end{proof}

\section{Existence of  measure attractors
 in  $\mathcal{P}_{\theta}(D_{\rho})$} \label{4}
 This section is devoted to the existence 
 of measure attractors
of  \eqref{lmvlds-2} in 
$\mathcal{P}_{\theta}(D_{\rho})$   
 for $ \theta \in \left(2, \tfrac{2(p-2)}{q-2} \right) $. 
 To that  end,
 we need to
 establish  the continuity of 
 the  cocycle associated with \eqref{lmvlds-2} 
 on bounded subsets 
 of  $\mathcal{P}_{\theta}(D_{\rho})$. 
 Since the space $D_{\rho}$
 is equipped with the 
 Skorohod topology,
the continuity of the cocycle in this case
is  quite  different from
that in
 \cite{SSW2026JDE, SSL2024arXiv, BCS2026JDE} .

\subsection{The cocycle on $ \mathcal{P}_{\theta}(D_{\rho})$}

 Denote by $\mathcal{B}_b(D_{\rho})$ the set of all bounded Borel measurable functions on $D_{\rho}$.
 For any $ \psi \in \mathcal{B}_b(D_{\rho}) $ and $ r \le t $, define
   $(P_{r,t} \psi)(\varphi)
     = \mathbb{E} \left[ \psi(u_t(\cdot;r,\varphi) ) \right]
   $
 for all $\varphi \in D_{\rho}$,
 where the segment process $u_t(s; r,\varphi) = u(t+s; r,\varphi)$ for $s \in [-\rho, 0]$,
 and $u(t; r,\varphi)$ is the solution of \eqref{lmvlds-2} with initial data $\varphi$ at initial time $r$.

 By Theorem \ref{wp-sol} in Section 3 
 and the Yamada-Watanabe theorem with jumps \cite{BLP2015IJSA, FHK2025arXiv}, 
 we can define 
 an  operator 
     $P^{*}_{r,t}:  \mathcal{P}_{\theta}(D_{\rho}) \to \mathcal{P}_{\theta}(D_{\rho})$ 
 by $P^{*}_{r,t} \mu = \mathcal{L}_{u_{t}(\cdot; r, \varphi)}$
 for every $\mu \in  \mathcal{P}_{\theta}(D_{\rho})$
 where  
  $\varphi \in L^{\theta}(\Omega, \mathcal{F}_r; \mathcal{D}_{\rho})$
with
  $  \mathcal{L}_{\varphi}=\mu$.

 For every $t \in \mathbb{R}^+$ and $\tau \in \mathbb{R}$,
 define 
 a
  mapping $\Phi(t,\tau): \mathcal{P}_{\theta}(D_{\rho}) \to \mathcal{P}_{\theta}(D_{\rho})$ by
   $\Phi(t,\tau) \mu = P^{*}_{\tau, \tau + t} \mu$
 for every $\mu \in \mathcal{P}_{\theta}(D_{\rho})$.
 Note that
 $\Phi(t,\tau) \mu$ is the law of the segment process of solution to \eqref{lmvlds-2} at $ \tau + t $ with
 initial distribution $\mu$ at  initial time $\tau$.
 Then $\Phi(0,\tau) = I$, and by the uniqueness of solutions of \eqref{lmvlds-2},
 we  see  that for all $t,s \in \mathbb{R}^{+}$, $\tau \in \mathbb{R}$ and $\mu \in \mathcal{P}_{\theta}(D_{\rho})$,
   $\Phi(t + s,\tau) \mu = \Phi(t, s+\tau) \circ \Phi(s,\tau)\mu$.
 Hence $\Phi$ is a non-autonomous cocycle on $\left( \mathcal{P}_{\theta}(D_{\rho}), d_{\mathcal{P}(D_{\rho})} \right)$.

\begin{theorem}\label{consys}
 If {\bf(H1)}-{\bf(H4)} hold,
then for  every 
   $ \theta \in \left(2, \tfrac{2(p-2)}{q-2} \right) $,
    the cocycle $\{ \Phi(t,\tau) \}_{t \ge { \rho}, \tau \in \mathbb{R}}$ associated with \eqref{lmvlds-2}
   is continuous on bounded subsets of $\mathcal{P}_{\theta}(D_{\rho})$.
\end{theorem}

\begin{proof}
 By Definition \ref{def_ds},
 we need to 
  prove that if $ \mu_n \to \mu $ in $
  \left( \mathcal{P}_{\theta}(D_{\rho}), d_{\mathcal{P}(D_{\rho})} \right)$,
 where $\mu$ and $\mu_n $ are 
 in
  $  \mathcal{P_{\theta}(D_{\rho})}$
 such that
   $\int_{D_{\rho}}
       \|\xi\|_{
       \infty }^{\theta} \mu(d\xi) \vee \int_{D_{\rho}} \|\xi\|_{\infty }^{\theta} \mu_n(d \xi)
       \le R$  for some
       positive number  $R  $ independent of $n$,
 then  
 $\Phi(t,\tau)\mu_n \to \Phi(t,\tau)\mu 
 $ in 
$ \left( \mathcal{P}_{\theta}(D_{\rho}), d_{\mathcal{P}(D_{\rho})} \right)$
for every $ \tau \in \mathbb{R} $ and $ t \ge \rho$.

 Since $ \mu_n \to \mu $ in
 $\left(\mathcal{P}_{\theta}(D_{\rho}), d_{\mathcal{P}(D_{\rho})} \right)$,
 it follows that $ \mu_n \to \mu $ weakly.
 Then by the Skorokhod representation theorem,
 there exist a probability space
   $(\widetilde{\Omega}, \widetilde{\mathcal{F}}, \widetilde{\mathbb{P}})$
 and random variables $\widetilde{\varphi}$ and $\widetilde{\varphi}_n$ defined on
 $(\widetilde{\Omega}, \widetilde{\mathcal{F}}, \widetilde{\mathbb{P}})$ such that the distributions of
 $\widetilde{\varphi}$ and $\widetilde{\varphi}_n$ coincide with
 $\mu$ and $\mu_n$, respectively.
 Furthermore,
   $  \widetilde{\varphi}_n 
   \to  \widetilde{\varphi}$
   in $ ({D_\rho}, d^0)$ 
   $\widetilde{\mathbb{P}}$-almost surely.
 Note that $\widetilde{\varphi}, \widetilde{\varphi}_n, W_k$
 and $\widetilde{N}$ can be considered as random variables defined in the product space
 $(\Omega \times \widetilde{\Omega},
   \mathcal{F} \times \widetilde{\mathcal{F}},
   \mathbb{P} \times \widetilde{\mathbb{P}} )
 $.
 So we may consider the solutions of system \eqref{lmvlds-2} in the product space
 with initial data $\widetilde{\varphi}$ and $\widetilde{\varphi}_n$ instead of the solutions in $(\Omega, \mathcal{F}, \mathbb{P})$
 with initial data $\varphi$ and $\varphi_n$.
 For convenience, denote by
     $\hat{u}_{n}(t) = u(t;\tau, \widetilde{\varphi}_n)$
 and $\hat{u}(t) = u(t; \tau, \widetilde{\varphi})$.

{\bf Step 1:} Prove the  estimates:
 for all $t \in [\tau,\tau+T]$,
 \begin{align}\label{jan23a}
          \widehat{\mathbb{E}} \left[ \|\hat{u}_n(t) - \hat{u}(t)\|^2 \right]   
\le   \left( \widetilde{\mathbb{E}} 
                             \left[ \|\widetilde{\varphi}_n(0) - \widetilde{\varphi}(0)\|^2 \right]
                    + \widetilde{C} \widetilde{\mathbb{E}}
                             \left[ \int_{-\rho}^{0}
                                           \|\widetilde{\varphi}_n(s) - \widetilde{\varphi}(s)\|^2 ds
                             \right]
            \right) e^{\widetilde{C}(t-\tau)} ,
 \end{align}
  where   $\widetilde{C}
 =\widetilde{C} (p,q, L)>0$ is a constant, 
 $\widehat{\mathbb{E}}$ is the mathematical expectation on 
 the product probability space
 $( \Omega \times \widetilde{\Omega},
    \mathcal{F} \times \widetilde{\mathcal{F}},
    \mathbb{P} \times \widetilde{\mathbb{P}}
  )$,
 and $\widetilde{\mathbb{E}}$ is the mathematical expectation on probability space
 $(\widetilde{\Omega}, \widetilde{\mathcal{F}}, \widetilde{\mathbb{P}})$.
By It\^o's formula and \eqref{lmvlds-2}, we obtain that for all $t \in [\tau,\tau+T]$,
 \begin{align}\label{consys-1}
        & \widehat{\mathbb{E}}
             \left[\|\hat{u}_n(t) - \hat{u}(t)\|^2 \right]
           + 2 \widehat{\mathbb{E}}
               \left[ \int_{\tau}^{t}
                         \langle A \hat{u}_n(s) - A \hat{u}(s),
                                \hat{u}_n(s) - \hat{u}(s)
                         \rangle ds
               \right]  
          + 2 \lambda 
                     \widehat{\mathbb{E}} 
                                           \left[ \int_{\tau}^{t} \|\hat{u}_n(s) - \hat{u}(s)\|^2 ds \right]  \nonumber\\
      & + 2 \widehat{\mathbb{E}}
            \left[ \int_{\tau}^{t}
                      \langle f(\hat{u}_n(s), \mathcal{L}_{\hat{u}_n(s)}) - f(\hat{u}(s), \mathcal{L}_{\hat{u}(s)}),
                              \hat{u}_n(s) - \hat{u}(s)
                      \rangle ds
            \right]  \nonumber\\
 \le & \widehat{\mathbb{E}} \left[ \|\widetilde{\varphi}_n(0) - \widetilde{\varphi}(0)\|^2 \right] \nonumber\\
      & + 2 \widehat{\mathbb{E}}
            \left[ \int_{\tau}^{t}
                      \langle F(\hat{u}_n(s-\rho), \mathcal{L}_{\hat{u}_n(s-\rho)}) - F(\hat{u}(s-\rho), \mathcal{L}_{\hat{u}(s-\rho)}),
                              \hat{u}_n(s) - \hat{u}(s)
                      \rangle ds
            \right]  \nonumber\\
      & + \widehat{\mathbb{E}}
             \left[ \int_{\tau}^{t}
                       \sum_{k \in \mathbb{N} }
                             \| \sigma_k(\hat{u}_n(s),\mathcal{L}_{\hat{u}_n(s)}) - \sigma_k(\hat{u}(s),\mathcal{L}_{\hat{u}(s)})
                             \|^2 ds
             \right]  \nonumber\\
      & + \widehat{\mathbb{E}}
             \left[ \int_{\tau}^{t} 
                          \int_{
                          y\in\mathbb{Y}
                          }
                             \| \widetilde{\sigma}(\hat{u}_n(s),\mathcal{L}_{\hat{u}_n(s)},y)
                                - \widetilde{\sigma}(\hat{u}(s),\mathcal{L}_{\hat{u}(s)},y)
                             \|^2
                          \nu(dy) ds
             \right].
 \end{align}

 For the fourth term on the left-hand side of \eqref{consys-1},
 by \eqref{f3}, we have
 \begin{align}\label{consys-2}
        & 2 \widehat{\mathbb{E}}
            \left[ \int_{\tau}^{t}
                      \langle f(\hat{u}_n(s), \mathcal{L}_{\hat{u}_n(s)}) - f(\hat{u}(s), \mathcal{L}_{\hat{u}(s)}),
                              \hat{u}_n(s) - \hat{u}(s)
                      \rangle ds
            \right]  \nonumber\\
   \ge & \lambda_2 \widehat{\mathbb{E}}
             \left[ \int_{\tau}^{t}
                       \sum_{i \in \mathbb{Z}}
                         \left( |\hat{u}_{n,i}(s)|^{p-2} + |\hat{u}_{i}(s)|^{p-2}
                         \right) |\hat{u}_{n,i}(s) - \hat{u}_i(s)|^2 ds
             \right]  \nonumber\\
        & - 2 ( \|\phi_4\|_{\ell^\infty} + \|\phi_4\|_1)
            \int_{\tau}^{t}
               \widehat{\mathbb{E}} \left[ \|\hat{u}_{n}(s) - \hat{u}(s)\|^2 \right] ds.
 \end{align}

 For the second term on the right-hand side of \eqref{consys-1},
 by \eqref{F2}, we have
 \begin{align}\label{consys-3}
        & 2 \widehat{\mathbb{E}}
            \left[ \int_{\tau}^{t}
                      \langle F(\hat{u}_n(s-\rho), \mathcal{L}_{\hat{u}_n(s-\rho)}) - F(\hat{u}(s-\rho), \mathcal{L}_{\hat{u}(s-\rho)}),
                          \hat{u}_n(s) - \hat{u}(s)
                      \rangle ds
            \right]  \nonumber\\
   \le & 2 \left( \|L_{F}\|_{\ell^\infty}^2 + \|L_{F}\|^2 \right)
            \widehat{\mathbb{E}}
                     \left[ \int_{-\rho}^{0}
                                  \| \widetilde{ \varphi}_n(s) - \widetilde{\varphi}(s) \|^2 ds
                     \right]  \nonumber\\
        & + \left( 2 \|L_{F}\|_{\ell^\infty}^2 + 2 \|L_{F}\|^2 + 1 \right)
             \int_{\tau}^{t}
                   \widehat{\mathbb{E}} \left[ \| \hat{u}_n(s) - \hat{u}(s) \|^2 \right] ds.
 \end{align}

  For the last two terms on the right-hand side of \eqref{consys-1},
  by \eqref{sigma4} with $\tilde{\theta} = 1$, we obtain
 \begin{align}\label{consys-4}
        & \widehat{\mathbb{E}}
             \left[ \int_{\tau}^{t}
                       \sum_{k \in \mathbb{N} }
                             \| \sigma_k(\hat{u}_n(s),\mathcal{L}_{\hat{u}_n(s)}) - \sigma_k(\hat{u}(s),\mathcal{L}_{\hat{u}(s)})\|^2 ds
             \right]  \nonumber\\
        & + \widehat{\mathbb{E}}
               \left[ \int_{\tau}^{t}  \int_{
               y\in\mathbb{Y}
               }
                               \| \widetilde{\sigma}(\hat{u}_n(s),\mathcal{L}_{\hat{u}_n(s)},y)
                                   - \widetilde{\sigma}(\hat{u}(s),\mathcal{L}_{\hat{u}(s)},y)
                               \|^2 \nu(dy) ds
               \right]  \nonumber\\
      \le & \frac{\lambda_2}{4}
             \widehat{\mathbb{E}}
                \left[\int_{\tau}^{t} \sum_{i \in \mathbb{Z}}
                            \left( |\hat{u}_{n,i}(s)|^{p-2} + |\hat{u}_i(s)|^{p-2}
                            \right)
                            |\hat{u}_{n,i}(s) - \hat{u}_i(s)|^2 ds
                \right]  \nonumber\\
        & + \left(  C_1
             + 2 \| L \|^2
            \right)
            \widehat{\mathbb{E}}
                        \left[ \int_{\tau}^{t}
                                     \|\hat{u}_n(s) - \hat{u}(s)\|^2 ds
                        \right],
 \end{align}
 where $C_1 = C_1(p, q, \|L\|)>0$
 is a constant.
 Then it follows from \eqref{consys-1}-\eqref{consys-4} that
 \begin{align}\label{consys-5}
        & \widehat{\mathbb{E}} \left[ \|\hat{u}_n(t) - \hat{u}(t)\|^2 \right]
          + 2^{2-p} \widehat{\mathbb{E}}
              \left[ \int_{\tau}^{t}
                        \| B (\hat{u}_n(s) -  \hat{u}(s) ) \|_p^p ds
              \right]
          + 2 \lambda \widehat{\mathbb{E}}
              \left[ \int_{\tau}^{t} \|\hat{u}_n(s) - \hat{u}(s)\|^2 ds
              \right]  \nonumber\\
        & + \frac{3\lambda_2}{4} \widehat{\mathbb{E}}
             \left[ \int_{\tau}^{t}
                       \sum_{i \in \mathbb{Z}}
                         \left( |\hat{u}_{n,i}(s)|^{p-2} + |\hat{u}_{i}(s)|^{p-2}
                         \right) |\hat{u}_{n,i}(s) - \hat{u}_{n,i}(s)|^2 ds
             \right]  \nonumber\\
   \le & \widehat{\mathbb{E}} \left[ \|\widetilde{\varphi}_n(0) - \widetilde{\varphi}(0)\|^2 \right]
          + 2 \left( \|L_{F}\|_{\ell^\infty}^2 + \|L_{F}\|^2 \right)
            \widehat{\mathbb{E}}
                     \left[ \int_{-\rho}^{0}
                                  \| \widetilde{ \varphi}_n(s) - \widetilde{\varphi}(s) \|^2 ds
                     \right]   \nonumber\\
        & + \widetilde{C} \int_{\tau}^{t}
                                \widehat{\mathbb{E}} \left[ \|\hat{u}_{n}(s) - \hat{u}(s)\|^2 \right] ds,
 \end{align}
 where
      $\widetilde{C} = 2 \|\phi_4\|_{\ell^\infty} + 2 \|\phi_4\|_1
                                   + 2 \|L_{F}\|_{\ell^\infty}^2 + 2 \|L_{F}\|^2 + 1
                                   + C_1 + 2 \| L \|^2
      $.
Then \eqref{jan23a}
follows from 
  \eqref{consys-5} and Gronwall's inequality
  immediately.

 {\bf Step 2:}  Prove   the inequality:
 \begin{align}\label{jan23b}
       \widehat{\mathbb{E}}
             \left[ \sup_{t \in [\tau, \tau +T]}
                            \| \hat{u}_n(t) - \hat{u}(t) \|^2 
             \right]
  \le C_T
              \widetilde{\mathbb{E}} 
                   \left[ \| \widetilde{\varphi}_n(0) - \widetilde{\varphi}(0) \|^2  
                            + \| \widetilde{\varphi}_n - \widetilde{\varphi} \|^2_{L^2(-\rho, 0; \ell^2)}     
                  \right],
 \end{align} 
 where  $C_T > 0$
 is a constant depending on $T$ but not on $n$ or $\rho \in (0, 1]$.
 
  For any $K \in \mathbb{N}$, define 
  a  stopping time by:
    $\tau_{K,n}
     = \inf
           \left\{ r \ge \tau:
                     \| (\hat{u}_n)_{r} (\cdot)\|_{\infty}
                     \vee
                     \|\hat{u}_{r}(\cdot)\|_{\infty}
                     > K
           \right\}
       \wedge (\tau + T)$.
 Then for each fixed  $n \in \mathbb{N}$, $\lim_{K \to +\infty} \tau_{K,n} = \tau + T$,  $\widetilde{\mathbb{P}}$-almost surely.
 
 By \eqref{lmvlds-2} and It\^o's formula, we have
  \begin{align}\label{consys-6}
         & \widehat{\mathbb{E}}
           \bigg[ \sup_{s \in [\tau,t]}
                  \Big( \| \hat{u}_n(s \wedge \tau_{K,n}) - \hat{u}(s \wedge \tau_{K,n}) \|^2
                         + 2^{2-p} \int_{\tau}^{s \wedge \tau_{K,n}}
                                       \| B (\hat{u}_n(r) - \hat{u}(r)) \|_p^p dr \nonumber\\
         & \qquad\qquad \
                         + 2 \lambda \int_{\tau}^{s \wedge \tau_{K,n}} \|\hat{u}_n(r) - \hat{u}(r)\|^2 dr  \nonumber\\
         & \qquad\qquad \
                         + 2 \int_{\tau}^{s \wedge \tau_{K,n}}
                                   \langle f(\hat{u}_n(r), \mathcal{L}_{\hat{u}_n(r)})
                                            - f(\hat{u}(r), \mathcal{L}_{\hat{u}(r)}),
                                           \hat{u}_n(r) - \hat{u}(r)
                                   \rangle dr
                   \Big)
            \bigg] \nonumber\\
     \le & \widehat{\mathbb{E}}[\|\widetilde{\varphi}_n(0) - \widetilde{\varphi}(0)\|^2] \nonumber\\
          & + 2 \widehat{\mathbb{E}}
                \left[ \sup_{s \in [\tau,t]}
                       \int_{\tau}^{s \wedge \tau_{K,n}}
                          \langle F(\hat{u}_n(r-\rho), \mathcal{L}_{\hat{u}_n(r-\rho)})
                                  - F(\hat{u}(r-\rho), \mathcal{L}_{\hat{u}(r-\rho)}),
                                  \hat{u}_n(r) - \hat{u}(r)
                          \rangle dr
                \right] \nonumber\\
          & + 2 \widehat{\mathbb{E}}
                   \left[ \sup_{s \in [\tau,t]}
                          \Big| \int_{\tau}^{s \wedge \tau_{K,n}}
                                   \sum_{k \in \mathbb{N} }
                   \sqrt{\varepsilon}
                    \left\langle \sigma_k(\hat{u}_n(r),\mathcal{L}_{\hat{u}_n(r)}) - \sigma_k(\hat{u}(r),\mathcal{L}_{\hat{u}(r)}),
                                                   \hat{u}_n(r) - \hat{u}(r)
                                      \right\rangle dW_k(r)
                          \Big|
                   \right]  \nonumber\\
          & + \widehat{\mathbb{E}}
                 \left[ \int_{\tau}^{t \wedge \tau_{K,n}}
                           \sum_{k \in \mathbb{N} }
                              \|\sigma_k(\hat{u}_n(r),\mathcal{L}_{\hat{u}_n(r)}) - \sigma_k(\hat{u}(r),\mathcal{L}_{\hat{u}(r)})\|^2 dr
                 \right]  \nonumber\\
          & + 2 \widehat{\mathbb{E}}
                 \Big[ \sup_{s \in [\tau,t]}
                           \int_{\tau}^{s \wedge \tau_{K,n}}  \int_{            y\in\mathbb{Y}}
                               \sqrt{\varepsilon}  \big\langle \widetilde{\sigma}(\hat{u}_n(r-),\mathcal{L}_{\hat{u}_n(r)},y)
                                              - \widetilde{\sigma}(\hat{u}(r-),\mathcal{L}_{\hat{u}(r)},y),
                                                   \nonumber\\
          & \qquad\qquad\qquad\qquad\qquad\qquad\qquad
                                              \hat{u}_n(r-) - \hat{u}(r-)
                                 \big\rangle  \widetilde{N}(dr,dy)
                 \Big]  \nonumber\\
          & + \widehat{\mathbb{E}}
                 \left[ \int_{\tau}^{t \wedge \tau_{K,n}}  \int_{y\in\mathbb{Y}
                           }
                              \| \widetilde{\sigma}(\hat{u}_n(r-),\mathcal{L}_{\hat{u}_n(r)},y)
                                 - \widetilde{\sigma}(\hat{u}(r-),\mathcal{L}_{\hat{u}(r)},y)
                              \|^2 N(dr,dy)
                 \right].
  \end{align}
 
   For the fourth term on the left-hand side of \eqref{consys-6},
   by \eqref{f3}, we have
  \begin{align}\label{consys-7}
         & 2 \int_{\tau}^{s \wedge \tau_{K,n}}
                   \langle f(\hat{u}_n(r), \mathcal{L}_{\hat{u}_n(r)})
                           - f(\hat{u}(r), \mathcal{L}_{\hat{u}(r)}),
                           \hat{u}_n(r) - \hat{u}(r)
                   \rangle dr  \nonumber\\
       \ge & 
       2 \lambda_2 \int_{\tau}^{s \wedge \tau_{K,n}}
                      \sum_{i \in \mathbb{Z}}
                         \left( |\hat{u}_{n,i}(r)|^{p-2} + |\hat{u}_{i}(r)|^{p-2}
                         \right)
                         \left( \hat{u}_{n,i}(r) - \hat{u}_{n,i}(r) \right)^2 dr  \nonumber\\
         & - 2 \|\phi_4\|_{\ell^\infty} \int_{\tau}^{s \wedge \tau_{K,n}}
                       \|\hat{u}_{n}(r) - \hat{u}(r)\|^2 dr - 2 \|\phi_4\|_1
                       \int_{\tau}^{s} \widehat{\mathbb{E}}
                             \left[ \|\hat{u}_{n}(r) - \hat{u}(r)\|^2 \right] dr.
  \end{align}
 
  For the second term on the right-hand side of \eqref{consys-6},
  by \eqref{F2} and Young's inequality, we get
  \begin{align}\label{consys-8}
       & 2 \widehat{\mathbb{E}}
                \left[ \sup_{s \in [\tau,t]}
                          \int_{\tau}^{s \wedge \tau_{K,n}}
                              \langle F(\hat{u}_n(r-\rho), \mathcal{L}_{\hat{u}_n(r-\rho)})
                                      - F(\hat{u}(r-\rho), \mathcal{L}_{\hat{u}(r-\rho)}),
                                      \hat{u}_n(r) - \hat{u}(r)
                              \rangle dr
                \right] \nonumber\\
 \le & \widehat{\mathbb{E}}
               \left[ \int_{\tau}^{t \wedge \tau_{K,n}}
                            \| F(\hat{u}_n(r-\rho), \mathcal{L}_{\hat{u}_n(r-\rho)})
                               - F(\hat{u}(r-\rho), \mathcal{L}_{\hat{u}(r-\rho)})\|^2 dr
               \right]  \nonumber\\
      & + \widehat{\mathbb{E}}
                 \left[ \int_{\tau}^{t \wedge \tau_{K,n}}
                              \| \hat{u}_n(r) - \hat{u}(r) \|^2 dr
                 \right]  \nonumber\\
 \le & 2 \|L_F\|_{\ell^\infty}^2 \widehat{\mathbb{E}}
              \left[ \int_{\tau}^{t \wedge \tau_{K,n}}
                            \|\hat{u}_n(r-\rho) - \hat{u}(r-\rho)\|^2 dr
              \right]  \nonumber\\
      & + 2 \|L_F\|^2 \widehat{\mathbb{E}}
                \left[ \int_{\tau}^{t \wedge \tau_{K,n}}
                             \mathbb{W}_2^2(\mathcal{L}_{\hat{u}_n(r-\rho)},\mathcal{L}_{\hat{u}(r-\rho)} )dr
                \right]
          + \widehat{\mathbb{E}}
                 \left[ \int_{\tau}^{t \wedge \tau_{K,n}}
                              \| \hat{u}_n(r) - \hat{u}(r) \|^2 dr
                 \right]  \nonumber\\
 \le & 2 \left( \|L_F\|_{\ell^\infty}^2 + \|L_F\|^2 \right)
              \widehat{\mathbb{E}}
              \left[ \int_{-\rho}^{0}
                           \|\widetilde{\varphi}_n(r) - \widetilde{\varphi}(r)\|^2 dr
              \right]  \nonumber\\
      & + \left( 2 \|L_F\|_{\ell^\infty}^2  
          + 1 \right)
              \widehat{\mathbb{E}}
                    \left[ \int_{\tau}^{t \wedge \tau_{K,n}}
                              \| \hat{u}_n(r) - \hat{u}(r) \|^2 dr 
                    \right]  \nonumber\\
      & + 2 \|L_F\|^2  
                \int_{\tau}^{t  }
                    \widehat{\mathbb{E}}
                         \left[ \| \hat{u}_n(r) - \hat{u}(r) \|^2 \right] dr.
    \end{align}
 
  For the third and fourth terms on the right-hand side of \eqref{consys-6},
  by 
  the BDG inequality and \eqref{sigma4}, we obtain
  \begin{align}\label{consys-9}
         & 2 \widehat{\mathbb{E}}
                   \left[ \sup_{s \in [\tau,t]}
                          \Big| \int_{\tau}^{s \wedge \tau_{K,n}}
                                   \sum_{k \in \mathbb{N} }
                                      \left\langle \sigma_k(\hat{u}_n(r),\mathcal{L}_{\hat{u}_n(r)}) - \sigma_k(\hat{u}(r),\mathcal{L}_{\hat{u}(r)}),
                                                   \hat{u}_n(r) - \hat{u}(r)
                                      \right\rangle dW_k(r)
                          \Big|
                   \right]  \nonumber\\
          & + \widehat{\mathbb{E}}
                 \left[ \int_{\tau}^{t \wedge \tau_{K,n}}
                           \sum_{k \in \mathbb{N} }
                              \|\sigma_k(\hat{u}_n(r),\mathcal{L}_{\hat{u}_n(r)}) - \sigma_k(\hat{u}(r),\mathcal{L}_{\hat{u}(r)})\|^2 dr
                 \right]  \nonumber\\
     \le & \frac{1}{8} \widehat{\mathbb{E}}
                        \left[ \sup_{r \in [\tau, t]}
                                   \| \hat{u}_n(r \wedge \tau_{K,n}) - \hat{u}(r \wedge \tau_{K,n}) \|^2
                        \right]  \nonumber\\
          & + \left( 1 + 8 c_1^2 \right)
              \widehat{\mathbb{E}}
                    \left[ \int_{\tau}^{t \wedge \tau_{K,n}}
                              \sum_{k \in \mathbb{N} }
                                    \| \sigma_k(\hat{u}_n(r),\mathcal{L}_{\hat{u}_n(r)}) - \sigma_k(\hat{u}(r),\mathcal{L}_{\hat{u}(r)})\|^2 dr
                    \right]  \nonumber\\
     \le & \frac{1}{8} \widehat{\mathbb{E}}
                        \left[ \sup_{r \in [\tau, t]}
                                   \| \hat{u}_n(r \wedge \tau_{K,n}) - \hat{u}(r \wedge \tau_{K,n}) \|^2
                        \right]  \nonumber\\
          & + \frac{\lambda_2}{4} \widehat{\mathbb{E}}
                  \left[ \int_{\tau}^{t \wedge \tau_{K,n}}
                            \sum_{i \in \mathbb{Z}}
                                \Big( |\hat{u}_{n,i}(s)|^{p-2} + |\hat{u}_i(s)|^{p-2} \Big)
                                | \hat{u}_{n,i}(s) - \hat{u}_i(s) |^2 ds
                  \right]  \nonumber\\
          & + C_7 
                 \widehat{\mathbb{E}}
                    \left[ \int_{\tau}^{t \wedge \tau_{K,n}}
                                 \|\hat{u}_n(s) - \hat{u}(s)\|^2 ds
                    \right]  \nonumber\\
          & + 2 \left( 1 + 8 c_1^2 \right) \| L \|^2
                \int_{\tau}^{t} \mathbb{W}_2^2(\mathcal{L}_{\hat{u}_n(s)},\mathcal{L}_{\hat{u}(s)}) ds,
 \end{align}
 where $C_7 = C_7(p, q, \|L\|) > 0$ is a constant.
 
  For the fifth and sixth terms on the right-hand side of \eqref{consys-6},
  by the
  BDG inequality and \eqref{sigma4},
  we obtain
  \begin{align}\label{consys-10}
         & 2 \widehat{\mathbb{E}}
                 \bigg[ \sup_{s \in [\tau,t]}
                           \int_{\tau}^{s \wedge \tau_{K,n}} 
                               \int_{ {y \in Y} }
                                   \sqrt{\varepsilon}  
                                        \big\langle   
                                            \widetilde{\sigma}(\hat{u}_n(r-), \mathcal{L}_{\hat{u}_n(r)},y)
                                             - \widetilde{\sigma}(\hat{u}(r-), \mathcal{L}_{\hat{u}(r)},y),
                                                   \nonumber\\
          & \qquad\qquad\qquad\qquad\qquad\qquad\
                                              \hat{u}_n(r-) - \hat{u}(r-)
                                           \big\rangle \widetilde{N}(dr,dy)
                 \bigg]  \nonumber\\
          & + \widehat{\mathbb{E}}
                 \left[ \int_{\tau}^{t \wedge \tau_{K,n}} 
                              \int_{y \in \mathbb{Y} }
                                  \| \widetilde{\sigma}(\hat{u}_n(r-), \mathcal{L}_{\hat{u}_n(r)}, y)
                                     - \widetilde{\sigma}(\hat{u}(r-), \mathcal{L}_{\hat{u}(r)}, y)
                                  \|^2 N(dr,dy)
                 \right]  \nonumber\\
    \le & 2 c_2 \widehat{\mathbb{E}}
               \bigg[ \int_{\tau}^{t\wedge\tau_{K,n}}
                            \int_{y \in \mathbb{Y} }
                                \| \widetilde{\sigma}(\hat{u}_n(r-), \mathcal{L}_{\hat{u}_n(r)}, y)
                                    - \widetilde{\sigma}(\hat{u}(r-), \mathcal{L}_{\hat{u}(r)}, y)
                                \|^2 \nonumber\\
          & \qquad\qquad\qquad \ \ \ \ \ \ \ \
                     \| \hat{u}_n(r-) - \hat{u}(r-) \|^2 N(dr, dy)  
                \bigg]^{\frac{1}{2}}  \nonumber\\
          & + \widehat{\mathbb{E}}
                 \left[ \int_{\tau}^{t \wedge \tau_{K,n}} 
                              \int_{y \in \mathbb{Y} }
                                   \| \widetilde{\sigma}(\hat{u}_n(r-), \mathcal{L}_{\hat{u}_n(r)}, y)
                                       - \widetilde{\sigma}(\hat{u}(r-), \mathcal{L}_{\hat{u}(r)}, y)
                                   \|^2 \nu(dy) dr
                 \right]  \nonumber\\
     \le & \frac{1}{8} \widehat{\mathbb{E}}
                        \left[ \sup_{r \in [\tau, t]}
                                   \| \hat{u}_n(r \wedge \tau_{K,n}) - \hat{u}(r \wedge \tau_{K,n}) \|^2
                        \right]  \nonumber\\
             & + \left(1 + 8 c_2^2 \right) \widehat{\mathbb{E}}
                         \left[ \int_{\tau}^{t\wedge\tau_{K,n}}  
                                      \int_{y \in \mathbb{Y} }
                                          \| \widetilde{\sigma}(\hat{u}_n(r), \mathcal{L}_{\hat{u}_n(r)}, y)
                                             - \widetilde{\sigma}(\hat{u}(r), \mathcal{L}_{\hat{u}(r)}, y)
                                          \|^2
                                   \nu(dy)dr
                         \right]  \nonumber\\
     \le & \frac{1}{8} \widehat{\mathbb{E}}
                        \left[ \sup_{r \in [\tau, t]}
                                   \| \hat{u}_n(r \wedge \tau_{K,n}) - \hat{u}(r \wedge \tau_{K,n}) \|^2
                        \right]  \nonumber\\
          & + \frac{\lambda_2}{4} \widehat{\mathbb{E}}
                  \left[ \int_{\tau}^{t \wedge \tau_{K,n}}
                            \sum_{i \in \mathbb{Z}}
                                \Big( |\hat{u}_{n,i}(s)|^{p-2} + |\hat{u}_i(s)|^{p-2} \Big)
                                | \hat{u}_{n,i}(s) - \hat{u}_i(s) |^2 ds
                  \right]  \nonumber\\
          & + C_8
                  \widehat{\mathbb{E}}
                       \left[ \int_{\tau}^{t \wedge \tau_{K,n}}
                                    \|\hat{u}_n(s) - \hat{u}(s)\|^2 ds
                       \right]  
             + 2 \left( 1 + 8 c_2^2 \right) \| L \|^2
                 \int_{\tau}^{t} \mathbb{W}_2^2(\mathcal{L}_{\hat{u}_n(s)},\mathcal{L}_{\hat{u}(s)}) ds.
  \end{align}
 
  From \eqref{consys-6}-\eqref{consys-10}, it follows
  \begin{align}\label{consys-11}
         & \widehat{\mathbb{E}}
              \Big[ \sup_{s \in [\tau, t]}
                    \Big( \|\hat{u}_n(s \wedge \tau_{K,n}) - \hat{u}(s \wedge \tau_{K,n})\|^2 \nonumber\\
         &          +  2\lambda_2 \int_{\tau}^{s \wedge \tau_{K,n}}
                         \sum_{i \in \mathbb{Z}}
                            \Big( |\hat{u}_{n,i}(r)|^{p-2} + |\hat{u}_{i}(r)|^{p-2} \Big) |\hat{u}_{n,i}(r) - \hat{u}_{n,i}(r)|^2 dr \Big)
              \Big]  \nonumber\\
    \le & \frac{1}{4} \widehat{\mathbb{E}}
              \left[ \sup_{r \in [\tau, t]}
                          \| \hat{u}_n(r \wedge \tau_{K,n}) - \hat{u}(r \wedge \tau_{K,n}) \|^2
              \right]  \nonumber\\
         & + \widehat{\mathbb{E}}[\| \widetilde{\varphi}_n(0) - \widetilde{\varphi}(0)\|^2]
           + 2 \left( \|L_F\|_{\ell^\infty}^2 + \|L_F\|^2 \right)
               \widehat{\mathbb{E}}
                  \left[ \int_{-\rho}^{0}
                            \|{\widetilde{\varphi}}_n(s) - \widetilde{\varphi}(s)\|^2 ds
                  \right]  \nonumber\\
         & + \left( 2 \|\phi_4\|_{\ell^\infty}
                    + 2 \|L_F\|_{\ell^\infty}^2 + 1
                    +C_7
                    +C_8 
             \right)
             \widehat{\mathbb{E}}
                      \left[ \int_{\tau}^{t \wedge \tau_{K,n}}
                                \|\hat{u}_{n}(s) - \hat{u}(s)\|^2 ds
                      \right]  \nonumber\\
         & + \left[ 2 \|\phi_4\|_1
                    + 2 \|L_F\|^2
                    + 4 \| L \|^2 + 16 \left( c_1^2 + c_2^2 \right) \| L \|^2
             \right]
             \int_{\tau}^{t} \widehat{\mathbb{E}} [\|\hat{u}_{n}(s) - \hat{u}(s)\|^2] ds.  \nonumber\\
         & + \frac{\lambda_2}{2} \widehat{\mathbb{E}}
                \left[\int_{\tau}^{t \wedge \tau_{K,n}} \sum_{i \in \mathbb{Z}}
                         \Big( |\hat{u}_{n,i}(s)|^{p-2} + |\hat{u}_i(s)|^{p-2} \Big)
                         |\hat{u}_{n,i}(s) - \hat{u}_i(s)|^2 ds
                \right].
  \end{align}
  Then by \eqref{consys-11}, we obtain
  \begin{align}\label{consys-12}
         & \frac{1}{2} \widehat{\mathbb{E}}
              \left[ \sup_{s \in [\tau, t]}
                        \|\hat{u}_n(s \wedge \tau_{K,n}) - \hat{u}(s \wedge \tau_{K,n})\|^2
              \right]  \nonumber\\
    \le & 2 \widehat{\mathbb{E}} \left[ \| {\widetilde{\varphi}}_n(0) - \widetilde{\varphi}(0)\|^2 \right]
           + 4 \left( \|L_{F}\|_{\ell^\infty}^2 + \|L_F\|^2 \right)
               \widehat{\mathbb{E}}
                  \left[ \int_{-\rho}^{0}
                               \|{\widetilde{\varphi}}_n(s) - \widetilde{\varphi}(s)\|^2 ds
                  \right]  \nonumber\\
         & + 2 \left( 2 \|\phi_4\|_{\ell^\infty}
                      + 2 \|L_{F}\|_{\ell^\infty}^2 + 1
                      +C_7
                      +C_8
               \right)
             \widehat{\mathbb{E}}
                      \left[ \int_{\tau}^{t \wedge \tau_{K,n}}
                                \|\hat{u}_{n}(s) - \hat{u}(s)\|^2 ds
                      \right]  \nonumber\\
         & + \left[ 4 \|\phi_4\|_1
                    + 4 \|L_F\|^2
                    + 8 \| L\|^2 + 32 \left( c_1^2 + c_2^2 \right) \| L \|^2
             \right]
             \int_{\tau}^{t} \widehat{\mathbb{E}}
                   \left[ \|\hat{u}_{n}(s) - \hat{u}(s)\|^2 \right] ds.
 \end{align}
 Let
     $ \hat{C} = 2 \left( 2 \|\phi_4\|_{\ell^\infty}
                                     + 2 \|L_{F}\|_{\ell^\infty}^2 + 1
                                     + C_7 + C_8
                            \right)
                         + 4 \|\phi_4\|_1 + 4 \|L_F\|^2
                         + 8 \| L \|^2 + 32 \left( c_1^2 + c_2^2 \right) \| L \|^2
     $.
 Then by \eqref{jan23a} and Gronwall's inequality,
 we get from  \eqref{consys-12} that for all $ t\in [\tau, \tau+T]$,
 \begin{align*}
       & \widehat{\mathbb{E}}
              \left[ \sup_{s \in [\tau, t]}
                           \|\hat{u}_n(s \wedge \tau_{K,n}) - \hat{u}(s \wedge \tau_{K,n})\|^2
              \right]  \nonumber\\
 \le & \bigg( 4 \widetilde{\mathbb{E}} 
                                \left[ \|{\widetilde{\varphi}}_n(0) - \widetilde{\varphi}(0)\|^2 \right]
                  + 2 \hat{C} \widetilde{\mathbb{E}}
                                          \left[ \int_{-\rho}^{0}
                                                        \|{\widetilde{\varphi}}_n(s) - \widetilde{\varphi}(s)\|^2 ds
                                         \right]  \nonumber\\
      & \  \
                 + 2 \hat{C} \int_{\tau}^{\tau + T}
                    \widehat{\mathbb{E}}
                         \left[ \|\hat{u}_n(r) - \hat{u}(r)\|^2 \right] dr
           \bigg) e^{2 \hat{C}(t-\tau)} \nonumber\\
 \le & \left[ 4 + 2 \hat{C} 
                       + 2 \widehat{C} T(1 + \widetilde{C} ) e^{\widetilde{C} T}
           \right] 
           e^{2 \hat{C}(t-\tau)}
             \widetilde{\mathbb{E}} 
                  \left[ \| {\widetilde{\varphi}}_n(0) - \widetilde{\varphi}(0) \|^2  
                            + \| {\widetilde{\varphi}}_n - \widetilde{\varphi} \|^2_{L^2(-\rho, 0; \ell^2)}     
                  \right].
 \end{align*}
 Then letting $K \to +\infty$, we obtain that for all $t \in [\tau, \tau+T]$,
 \begin{align*}
      & \widehat{\mathbb{E}}
             \left[ \sup_{s \in [\tau, t]}
                            \| \hat{u}_n(s) - \hat{u}(s) \|^2
             \right] \\
\le & \left[ 4 + 2 \hat{C} 
                     + 2 \widehat{C} T(1 + \widetilde{C} ) e^{\widetilde{C} T}
           \right] 
           e^{2 \hat{C}T}
           \widetilde{\mathbb{E}} 
                \left[ \| {\widetilde{\varphi}}_n(0) - \widetilde{\varphi}(0) \|^2  
                         + \| {\widetilde{\varphi}}_n - \widetilde{\varphi} \|^2_{L^2(-\rho, 0; \ell^2)}     
                \right],
 \end{align*}
 which implies 
 \eqref{jan23b} .

 {\bf Step 3:} Prove that $ (\hat{u}_{n})_t(\cdot) \to \hat{u}_{t}(\cdot) $ in distribution in $(D_{\rho}, d^0)$ for 
 all $ t \in [\tau+ {{} \rho}, \tau+T]$ with
  {{}$T > \rho$.}
  
 By assumption, we know that
 $ {\widetilde{\varphi} }_n \to  \widetilde{\varphi}$ in $(D_\rho, d^0)$ 
 $\widetilde{\mathbb{P}}$-almost surely,
 and hence we have,
 $\widetilde{\mathbb{P}}$-almost surely,
 $$ \sup_{n \in \mathbb{N}}
            \sup_{t \in [-\rho,0]}
                 \| {\widetilde{\varphi}}_n (t) \|
      + \sup_{t\in [-\rho,0]}
              \| \widetilde{\varphi }(t) \|
     < \infty,
$$
\begin{align} \label{jan23d}
 \lim_{n \to \infty}
     \| {\widetilde{\varphi} }_n  (0) - \widetilde{\varphi} (0)\|^2 = 0,
\end{align}
and
\begin{align} \label{jan23e}
 \lim_{n \to \infty}
      \| {\widetilde{\varphi}}_n - \widetilde{\varphi} \|^2_{L^2(-\rho, 0; \ell^2)}
  = 0.
\end{align}

On the other hand,
since  
$ \widetilde{\mathbb{E}} 
          \left[ \| {\widetilde{\varphi}} \|^\theta_\infty
          \right] 
     \le R$ and
 $ \widetilde{\mathbb{E}} 
          \left[ \| {\widetilde{\varphi}}_n\|^\theta_\infty
          \right] 
     \le R$
for all $n \in \mathbb{N}$,
we find that
\begin{align} \label{jan23ga}
 \widetilde{\mathbb{E}} 
       \left[ \|{\widetilde{\varphi}}_n(0) - \widetilde{\varphi}(0)\|^\theta \right]
  \le  2^\theta R, \quad 
               \forall \ n \in \mathbb{N},
\end{align}
and 
\begin{align} \label{jan23gb}
 \widetilde{\mathbb{E}} 
       \left[ \| {\widetilde{\varphi}}_n 
                    - \widetilde{\varphi}\|^\theta_{L^2(-\rho, 0; \ell^2)}  
       \right]
 \le 2^\theta  \rho^{\frac \theta{2}} R, \quad 
               \forall \ n\in \mathbb{N}.
\end{align}
Since  $\theta>2$, it follows from
\eqref{jan23ga}-\eqref{jan23gb}
that
the sequence 
$\{\|\widetilde{\varphi_n}(0) - \widetilde{\varphi}(0)\|^2\}_{n=1}^\infty$
and the sequence 
 $\{ \| \widetilde{\varphi_n} - \widetilde{\varphi} 
       \|^2_{L^2(-\rho, 0; \ell^2)}
   \}_{n=1}^\infty$
are uniformly integrable,
which along with \eqref{jan23d}-\eqref{jan23e}
and the  Vitali  theorem 
implies that   
\begin{align} \label{jan23f}
 \lim_{n \to \infty}
      \widetilde{\mathbb{E}} 
             \left[ \| \widetilde{\varphi}_n(0) - \widetilde{\varphi}(0) \|^2 \right]
   = 0,
\end{align}
 and 
\begin{align} \label{jan23g}
 \lim_{n \to \infty}
      \widetilde{\mathbb{E}} 
       \left[ \| {\widetilde{\varphi}}_n - \widetilde{\varphi} \|^2_{L^2(-\rho, 0; \ell^2)}     
      \right] 
  = 0.
\end{align}
 By
 \eqref{jan23b}
 and
 \eqref{jan23f}-\eqref{jan23g}  we obtain
 that  
\begin{align}\label{feb23a}
 \lim_{n\to \infty}
 \widehat{\mathbb{E}}
       \left[ \sup_{t \in [\tau, \tau +T]}
                            \| \hat{u}_n(t) - \hat{u}(t) \|^2
       \right]  
  = 0,
\end{align}
which  completes the proof.
  \end{proof}

\begin{remark} \label{contskorohod}
  By  \eqref{feb23a},
  we see that 
  if     initial data
  converge
  at initial time
  $\tau$  in
  the  Skorohod 
  space
  $(D([-\rho, 0], \ell^2), d^0)$
    $\mathbb{P}$-almost surely,
    then the second moment
    of the solutions
    of \eqref{lmvlds-2}
    in $\ell^2$
    converges 
   uniformly on the interval
   $[\tau, \tau +T]$.
   Consequently,
   if the equation has no time delay
   (i.e.,  $\rho = 0$), 
  then the cocycle $\{ \Phi(t,\tau) \}_{t \ge 0, \tau \in \mathbb{R}}$ associated with \eqref{lmvlds-2} is continuous on bounded subsets of $\mathcal{P}_{\theta}(\ell^2)$ 
  {\bf for  all $t \ge 0$}. 
  However, 
  if the equation has time delay
  (i.e.,  $\rho > 0$), 
  then by 
   Theorem \ref{consys},
   we find that
     the cocycle $\{ \Phi(t,\tau) \}_{t \ge 0, \tau \in \mathbb{R}}$ is continuous on bounded subsets of $\mathcal{P}_{\theta}
     (D([-\rho, 0], \ell^2), d^0)
     $ 
     {\bf only for $t \ge \rho$}, but not for $t<\rho$.

     This is due to the fact that
     the  Skorohod metric
     $d^0$ on
      $ D([-\rho, 0], \ell^2)$
      is quite different from
      the 
      uniform metric  $\|\cdot\|_\infty$
      on  $ D([-\rho, 0], \ell^2)$.
      In particular,
        the convergence in $
        (D([-\rho, 0], \ell^2), d^0)$ does not imply the convergence in 
        $(D([s, t], \ell^2), d^0)$ for $-\rho < s < t \le 0$,
        which is in sharp contrast
        with the uniform metric $\|\cdot\|_\infty$.
        The same issue also arises
        from the Feller property
        of the solutions, see   \cite{RRG2006SPA}
        where
        the authors prove the
        Feller property
        only for $t\ge  \rho$
         for
         the  distribution independent delay   equations driven by L\'{e}vy processes.
 \end{remark}

 \subsection{Existence of pullback absorbing sets}

 This section is devoted to the existence of pullback absorbing sets for the cocycle $\Phi$ associated with \eqref{lmvlds-2}, which is useful for proving the existence of 
 measure attractors  of $\Phi$.
 To that end, we  further assume:
 there exists $ \theta \in \left(2, \tfrac{2(p-2)}{q-2} \right) $
 such that 
 \begin{align}\label{dissip1}
    \lambda 
     > \|\phi_1\|_{\ell^\infty}
   	     +  \|\phi_1\|_{1}
   	     + { \frac{4}{\theta} 
   	             (\theta-1)^{\frac{\theta - 1}{\theta} }   }
   	          	            \left( \|L_{F}\|_{\ell^\infty} + \|L_F\|
   	          	            \right).
 \end{align}
 
Note that \eqref{dissip1} implies  that 
 \begin{align}\label{31_1}
  \lambda > \|\phi_1\|_{\ell^\infty} 
             	    + \|\phi_1\|_1 
 	                + 2 \left( \|L_{F}\|_{\ell^\infty}
 	                                + \|L_F\|
 	                      \right).
 \end{align}  
 Indeed, if we set
   $\xi (\theta)
           =\ln 
               \left( \frac{1}{\theta} (\theta-1)^{\frac{\theta - 1}{\theta} } 
                \right) 
           = \frac{\theta - 1}{\theta} \ln(\theta -1) - \ln \theta
         $ 
 for $\theta \ge 2$,
 then   we have
    $ \frac{d \xi (\theta) }{d \theta} 
       = \frac{\ln(\theta-1) }{\theta^2} 
       \ge 0
    $ 
    for $\theta \ge 2$,
    and hence  
  $\frac{1}{\theta} (\theta-1)^{\frac{\theta - 1}{\theta} }$
 is   increasing on $  [2, \infty)$, which along with
  \eqref{dissip1} yields 
 \eqref{31_1}.

By \eqref{dissip1} and \eqref{31_1}, 
we find that 
  there exists a number $\beta \in (0,1)$ such that
 \begin{align}\label{dissip1-1}
   5 \beta - \theta \lambda
 	            + \theta \|\phi_1\|_{\ell^\infty}
 	            + \theta \|\phi_1\|_{1}
 	            + 4(\theta-1)^{\frac{\theta - 1}{\theta} } 
 	                      e^{\frac{\beta \rho}{\theta} }
 	          	                \left( \|L_{F}\|_{\ell^\infty} + \|L_F\|
 	          	                \right) 
    < 0,
 \end{align}
 and 
  {
 \begin{align}\label{31_1-1}     
      2\beta - 2\lambda 
                  + 2\|\phi_1\|_{\ell^\infty} 
                  + 4 \|L_{F}\|_{\ell^\infty} e^{\frac{\beta}{2} \rho}
	   < 0.
 \end{align}
}

 We now fix $\beta \in (0,1)$ satisfying \eqref{dissip1-1} and \eqref{31_1-1}, and further assume the following condition on the functions  $g, h$, and $\widetilde{h}$:
 for  all  $\tau \in \mathbb{R}$,
 \begin{align}\label{dissip2}
    \int_{-\infty}^{\tau}
          e^{- \beta (\tau - s) }
          \left( \|g(s)\|^{\theta} + \|h(s)\|^{\theta}
                    + \|\widetilde{h}(s)\|^{\theta}_{L^{ \theta}(\mathbb{Y}, \nu; \ell^2)}
                  + \|\widetilde{h}(s)\|^{\theta}_{L^{ 2}(\mathbb{Y}, \nu; \ell^2)}   
          \right) ds
    <   \infty,
 \end{align}
which implies 
\begin{align}\label{34_1}
	\int_{-\infty}^{\tau} 
	         e^{- \beta(\tau - s )} 
	         \left( \|g(s)\|^2 + \|h(s)\|^2 
	                  + \|\widetilde{h}(s)\|^2 _{L^2(\mathbb{Y}, \nu; \ell^2)} 
	         \right) ds 
	  < \infty.
\end{align}

Recall that
 $ \mathcal{U} \subseteq \mathcal{P}_\theta(D_{\rho}) $ is  bounded in $ \mathcal{P}_\theta(D_{\rho}) $,
 if
 \begin{align*}
    \|\mathcal{U}\|_{\mathcal{P}_{\theta}(D_{\rho})}
    = \sup_{\mu \in \mathcal{U}}
         \Big( \int_{D_{\rho}} \|x\|_{\infty}^\theta \mu(dx)
         \Big)^{\frac{1}{\theta}} <  \infty.
 \end{align*}

 Let $D = \{ D(\tau) \subseteq \mathcal{P}_{\theta}(D_{\rho}): \tau \in \mathbb{R} \}$
 be a family of bounded nonempty subsets in $\mathcal{P}_{\theta}(D_{\rho})$
 such that
 \begin{align} \label{temper}
    \lim_{\tau \to -\infty}
          e^{\beta \tau}
          \|D(\tau)\|^{\theta}_{\mathcal{P}_{\theta}(D_{\rho})}
     = 0.
 \end{align}
 Denote by $\mathcal{D}$ the collection of all families of bounded subsets of $\mathcal{P}_{\theta}(D_{\rho})$
 satisfying \eqref{temper}.

 We now establish the pullback uniform estimates of solutions of \eqref{lmvlds-2} in $ L^{\theta}(\Omega, D_{\rho}) $,
 which are also valid for $\theta =2$.

 \begin{lemma}\label{pue}
   Suppose that {\bf(H1)}-{\bf(H4)}, \eqref{thetaassum2},  \eqref{dissip1} and \eqref{dissip2} hold.  
    Then  there exists
     $ \varepsilon^{*} = \varepsilon^{*}(\theta) \in (0, 1) $
     such that 
     for every  
    $\tau \in \mathbb{R}$ and $D = \{ D(\tau): \tau \in \mathbb{R} \} \in \mathcal{D}$,
   there exists $T = T(\tau, D) \ge { 3} \rho$  
   such that 
   the solution $u$ to \eqref{lmvlds-2} satisfies,
    for all $t \ge T$,
   $r\in [\tau - {2} \rho, \tau]$
    and $\varepsilon \in (0, \varepsilon^{*})$, 
   \begin{align*}
	    & \mathbb{E}
	             \left[ \sup_{r \in [\tau - { 2}\rho,\tau]}
	                                   \|u(r;\tau-t,\varphi)\|^{\theta}
    	         \right]  
    	     + \mathbb{E}
	 	                         \left[ \int_{\tau-t}^{r} e^{\beta (s-r)}
	 	                                       \|u(s;\tau-t, \varphi)\|^{\theta} ds
	 	                         \right]  \nonumber\\
	    & + \mathbb{E}
	 	                     \left[ \int_{\tau-t}^{r} e^{\beta (s-r)}
	        	                              \|u(s;\tau-t, \varphi)\|^{\theta-2}
	 	                                  \|u(s;\tau-t, \varphi)\|_p^{\,p} ds
	 	                     \right]  \nonumber\\
   \le & C_{\beta,\theta, \rho}
                  \int_{-\infty}^{\tau} e^{\beta (s-\tau)}
	                  \Big( \|g(s)\|^{\theta}
	                           + \|h(s)\|^{\theta}
	                           + \|\widetilde{h}(s)\|^{\theta}_{L^{\theta}(\mathbb{Y}, \nu; \ell^2)}
	                           + \|\widetilde{h}(s)\|^{\theta}_{L^{2}(\mathbb{Y}, \nu; \ell^2)}
	                  \Big) ds                           
	         + C_{\beta,\theta, \rho},
   \end{align*} 
   where 
   $\varphi \in L^{\theta}(\Omega, \mathcal{F}_{\tau - t}; D_{\rho})$ with law $\mathcal{L}_{\varphi} \in D(\tau-t)$,
   and 
   $C_{\beta, \theta, \rho} > 0$ is a constant depending on $\beta$, $\theta$ and $\rho$, but not on $\varepsilon$.
\end{lemma}

\begin{proof}
	For convenience, we denote $u(r;\tau-t,\varphi)$ by $u(r)$.
	Applying It\^o's formula to \eqref{lmvlds-2}  we get 
	 that for all $r \ge \tau - t$,
\begin{align}\label{thetapue-1}
	& e^{\beta r} \|u(r)\|^\theta
	+ \theta \int_{\tau-t}^{r}
	e^{\beta s} \|u(s)\|^{\theta-2}
	\big\langle f\big(u(s), \mathcal{L}_{u(s)}\big), u(s) \big\rangle \, ds
	\nonumber\\
	\le\; &
	e^{\beta(\tau-t)} \|u(\tau-t)\|^{\theta}
	+ (\beta - \theta \lambda)
	\int_{\tau-t}^{r} e^{\beta s} \|u(s)\|^\theta \, ds
	\nonumber\\
	& \quad
	+ \theta \int_{\tau-t}^{r}
	e^{\beta s} \|u(s)\|^{\theta-2}
	\langle g(s), u(s) \rangle \, ds
	\nonumber\\
	& \quad
	+ \theta \int_{\tau-t}^{r}
	e^{\beta s} \|u(s)\|^{\theta-2}
	\big\langle F\big(u(s-\rho), \mathcal{L}_{u(s-\rho)}\big), u(s) \big\rangle
	\, ds
	\nonumber\\
	& \quad
	+ \sqrt{\varepsilon}\,\theta \int_{\tau-t}^{r}
	e^{\beta s} \|u(s)\|^{\theta-2}
	\big\langle
	\sum_{k \in \mathbb{N}}
	\big( \sigma_k(u(s), \mathcal{L}_{u(s)}) + h_k(s) \big) \, dW_k(s),
	u(s)
	\big\rangle
	\nonumber\\
	& \quad
	+ \frac{\theta(\theta-1)}{2}\,
	\varepsilon \int_{\tau-t}^{r}
	e^{\beta s} \|u(s)\|^{\theta-2}
	\sum_{k \in \mathbb{N}}
	\| \sigma_k(u(s), \mathcal{L}_{u(s)}) + h_k(s) \|^2 \, ds
	\nonumber\\
	& \quad
	+ \sqrt{\varepsilon}\,\theta 
	\int_{\tau-t}^{r} \int_{
	y\in\mathbb{Y}
	}
	e^{\beta s} \|u(s-)\|^{\theta-2}
	\big\langle
	\widetilde{\sigma}_k\big(u(s-), \mathcal{L}_{u(t)}, y\big)
	+ \widetilde{h}(s,y),
	u(s-)
	\big\rangle
	\,\widetilde{N}(ds,dy)
	\nonumber\\
	& \quad
	+  
	\int_{\tau-t}^{r} \int_{
	y\in\mathbb{Y}
	} e^{\beta s}
	\Big(
	\| u(s-)
	+ \sqrt{\varepsilon} \widetilde{\sigma}(u(s-), \mathcal{L}_{u(t)}, y)
	+ \sqrt{\varepsilon} \widetilde{h}(s,y) \|^\theta
	- \|u(s-)\|^\theta
	\nonumber\\
	& \qquad\qquad\qquad\qquad\quad
	- \theta \sqrt{\varepsilon}\, \|u(s-)\|^{\theta-2}
	\big\langle
	\widetilde{\sigma}(u(s-), \mathcal{L}_{u(t)}, y)
	+ \widetilde{h}(s,y),
	u(s-)
	\big\rangle
	\Big) N(ds,dy)
	\nonumber\\
	=\; &
	e^{\beta(\tau-t)} \|u(\tau-t)\|^{\theta}
	+ (\beta - \theta \lambda)
	\int_{\tau-t}^{r} e^{\beta s} \|u(s)\|^\theta \, ds
	+ \sum_{i=1}^{6} I_i .
\end{align}
 By \eqref{f1} and Young's inequality, we obtain 
 \begin{align}\label{moment_h1}
 	& \theta \int_{\tau-t}^{r}
 	e^{\beta s} \|u(s)\|^{\theta-2}
 	\big\langle f\big(u(s), \mathcal{L}_{u(s)}\big), u(s) \big\rangle \, ds
 	\nonumber\\
 	\ge\; &
 	\theta \int_{\tau-t}^{r}
 	e^{\beta s} \|u(s)\|^{\theta-2}
 	\sum_{i \in \mathbb{Z}}
 	\left(
 	\lambda_1 |u_i(s)|^p
 	- \phi_{1,i}
 	\left( 1 + |u_i(s)|^2 + \mathcal{L}_{u(s)}(\|\cdot\|^2) \right)
 	\right) ds
 	\nonumber\\
 	\ge\; &
 	\lambda_1 \theta
 	\int_{\tau-t}^{r}
 	e^{\beta s} \|u(s)\|^{\theta-2} \|u(s)\|_p^p \, ds
 	- \theta
 	\int_{\tau-t}^{r}
 	e^{\beta s} \|u(s)\|^{\theta-2} \|\phi_1\|_1
 	\bigl( 1 + \mathbb{E}[\|u(s)\|^2] \bigr) \, ds
 	\nonumber\\
 	& \quad
 	- \theta \|\phi_1\|_{\ell^\infty}
 	\int_{\tau-t}^{r}
 	e^{\beta s} \|u(s)\|^{\theta} \, ds
 	\nonumber\\
 	\ge\; &
 	\lambda_1 \theta
 	\int_{\tau-t}^{r}
 	e^{\beta s} \|u(s)\|^{\theta-2} \|u(s)\|_p^p \, ds
 	- 2 \left( \frac{\theta-2}{\beta} \right)^{\frac{\theta-2}{2}}
 	\|\phi_1\|_1^{\frac{\theta}{2}}
 	\frac{1}{\beta} e^{\beta r}
 	\nonumber\\
 	& \quad
 	- \left( \beta + \theta \|\phi_1\|_{\ell^\infty} \right)
 	\int_{\tau-t}^{r}
 	e^{\beta s} \|u(s)\|^{\theta} \, ds
 	- \theta \|\phi_1\|_1
 	\int_{\tau-t}^{r}
 	e^{\beta s} \|u(s)\|^{\theta-2}
 	\mathbb{E}\bigl[\|u(s)\|^{2}\bigr] \, ds .
 \end{align}

 Next, we will establish the estimates of $I_i$ in \eqref{thetapue-1}, where $i=1, \dots, 6$.
  For $I_1$, by Young's inequality, we have
\begin{align}\label{moment_h2}
	\mathbb{E}[I_1]
	& \le
	\frac{\beta}{2} \,
	\mathbb{E}\left[
	\int_{\tau-t}^{r}
	e^{\beta s} \|u(s)\|^{\theta} \, ds
	\right]
	+  \left( \frac{2(\theta-1)}{\beta} \right)^{\theta-1}
	\int_{\tau-t}^{r}
	e^{\beta s} \|g(s)\|^{\theta} \, ds .
\end{align}
  For $I_2$, by Young's inequality and \eqref{F3}, we obtain
 \begin{align}\label{moment_I2}
 	\mathbb{E}[I_2]
  \le & \theta \mathbb{E}
 	         \left[	\int_{\tau-t}^{r}
                        	e^{\beta s} \|u(s)\|^{\theta-1}
 	                      \big\| F\big( u(s-\rho), \mathcal{L}_{u(s-\rho)} \big) 
 	                      \big\| \, ds
          	\right]  \nonumber\\
 \le & \theta
 	         \mathbb{E}
 	                \left[ \int_{\tau-t}^{r}
 	                             e^{\beta s} \|u(s)\|^{\theta-1}
 	                             \Big( \sqrt{2}	\|F(0,\delta_0)\|
 	                                      + 2\|L_{F}\|_{\ell^\infty} \|u(s-\rho)\|
 	                                      +2 \|L_F\| \,\mathbb{W}_2\big(\mathcal{L}_{u(s-\rho)},\delta_0\big)
 	                             \Big) ds
                 	\right]  \nonumber\\
 \le & \frac{\beta}{2}
 	               \mathbb{E}
 	                      \left[ \int_{\tau-t}^{r}
 	                                    e^{\beta s} \|u(s)\|^{\theta} \, ds
 	                      \right]
 	         + \left( \frac{2(\theta-1)}{\beta} \right)^{\theta-1} 
 	               e^{\beta r} 2^{\frac{\theta}{2}} \frac{\|F(0,\delta_0)\|^{\theta}}{\beta}         \nonumber\\
      & + 2(\theta-1)^{\frac{\theta - 1}{\theta} } e^{\frac{\beta \rho}{\theta} }
     	        \|L_{F}\|_{\ell^\infty} 
 	            \mathbb{E}
 	                   \left[ \int_{\tau-t}^{r}
 	                                e^{\beta s} \|u(s)\|^{\theta} \, ds
 	                   \right]  \nonumber\\
 	  & + 2 [(\theta - 1) e^{-\beta \rho} ]^{\frac{\theta - 1}{\theta} } 
 	            \|L_{F}\|_{\ell^\infty}
 	            \mathbb{E}
 	                  \left[ \int_{\tau-t}^{r}
 	                               e^{\beta s} \|u(s-\rho)\|^{\theta} \, ds
                      \right]  \nonumber\\
 	  & + 2(\theta-1)^{\frac{\theta - 1}{\theta} } e^{\frac{\beta \rho}{\theta} } \|L_F\| 
 	         \mathbb{E}
 	                \left[ \int_{\tau-t}^{r} e^{\beta s} \|u(s)\|^{\theta} \, ds
 	               \right]  \nonumber\\
 	  & + 2 [(\theta - 1) e^{-\beta \rho} ]^{\frac{\theta - 1}{\theta} }  \|L_F\|
              \int_{\tau-t}^{r} e^{\beta s} 
                  \mathbb{E} \left[ \|u(s-\rho)\|^{\theta} \right] ds  \nonumber\\
 \le & \Big[ \tfrac{\beta}{2}
 	                + 4(\theta-1)^{\frac{\theta - 1}{\theta} } e^{\frac{\beta \rho}{\theta} }
 	                    \left( \|L_{F}\|_{\ell^\infty} + \|L_F\|
 	                    \right) 
 	       \Big]
 	        \mathbb{E}
 	               \left[ \int_{\tau-t}^{r}
 	                             e^{\beta s} \|u(s)\|^{\theta} \, ds
 	              \right]   \nonumber\\
 	   & + \left( \frac{2(\theta-1)}{\beta} \right)^{\theta-1}
 	           e^{\beta r} 2^{\frac{\theta}{2}} 
 	           \frac{\|F(0,\delta_0)\|^{\theta}}{\beta}   \nonumber\\
 	   & +  \frac{2 \rho}{\beta} 
 	           (\theta - 1)^{\frac{\theta - 1}{\theta} } 
 	           e^{\frac{\beta \rho}{\theta} } 
 	     	   \left( \|L_{F}\|_{\ell^\infty} + \|L_F\| \right)
 	     	    e^{\beta (\tau - t)} 
 	     	    \mathbb{E} \left[ \|\varphi\|_{\infty}^{\theta} \right].
 \end{align}

As for $I_4$, applying \eqref{sigma3} with $\tilde{\theta}$ replaced by $\theta-1$, we obtain
 \begin{align}\label{moment_h3}
 	\mathbb{E}[I_4]
 \le & \varepsilon \theta
          \mathbb{E}
                \left[ \sum_{k \in \mathbb{N}}
                            	\int_{\tau-t}^{r}
 	                                 e^{\beta s} \|u(s)\|^{\theta-2} (\theta-1)
                                	\| \sigma_k(u(s), \mathcal{L}_{u(s)}) \|^2 \, ds
 	            \right]   \nonumber\\
 	  & + \varepsilon \theta(\theta-1)
             \mathbb{E}
                   \left[ \int_{\tau-t}^{r}
                          	    e^{\beta s} \|u(s)\|^{\theta-2} \|h(s)\|^2 \, ds
 	               \right]  \nonumber\\
 \le & \varepsilon \theta
       	 \mathbb{E}
       	        \left[ \int_{\tau-t}^{r}
 	                         e^{\beta s} \|u(s)\|^{\theta-2}
 	                        \left( \frac{\lambda_1}{4} \|u(s)\|_{p}^{p}
 	                                  + 4(\theta-1)\|L\|^2 
 	                                      \mathbb{E} \left[ \|u(s)\|^2 \right] 
 	                                  + C_1
 	                       \right) ds
 	             \right]  \nonumber\\
 	   & + \varepsilon \theta(\theta-1)
 	          \mathbb{E}
 	                \left[ \int_{\tau-t}^{r}
 	                             e^{\beta s} \|u(s)\|^{\theta-2} \|h(s)\|^2 \, ds
 	                \right]   \nonumber\\
 \le & \frac{\lambda_1 \theta}{4}
           	\mathbb{E}
           	      \left[ \int_{\tau-t}^{r}
                             	e^{\beta s} \|u(s)\|^{\theta-2} \|u(s)\|_{p}^{p} \, ds
 	              \right]   \nonumber\\
 	     & + \varepsilon
 	            \Big( 4 \theta (\theta-1) \|L\|^2
 	                      + (\theta-2) C_1 
 	                      + (\theta-1)(\theta-2)
 	           \Big)
 	           \mathbb{E}
 	                  \left[ \int_{\tau-t}^{r}
 	                                e^{\beta s} \|u(s)\|^{\theta} \, ds
               	      \right]   \nonumber\\
 	     & + 2 C_1  \frac{e^{\beta r}}{\beta}
 	         + 2(\theta-1)
 	             \int_{\tau-t}^{r}
 	                 e^{\beta s} \|h(s)\|^{\theta} \, ds ,
 \end{align}
 where $C_1=C_1(p,q, L)>0$ is 
 a constant.
 
 For $I_6$, 
 by the argument of 
 \eqref{moment_gjh},
 we apply the Talor formula 
  to obtain the bound
 \[
 \mathbb{E}[I_6] \le \mathbb{E}[I_{6,1}] + \mathbb{E}[I_{6,2}],
 \]
 where $  I_{6,1}  $ and $ I_{6,2} $ 
 are given below:
 \begin{align}\label{moment_gjg}
	\mathbb{E}[I_{6,1}]
	:= & 
	\varepsilon C_\theta  \,
	\mathbb{E}\left[ 
	\int_{\tau-t}^{r}
	\int_{y\in\mathbb{Y}
	}
	e^{\beta s}
	\|u(s)\|^{\theta-2}
	\big\|
	\widetilde{\sigma}(u(s), \mathcal{L}_{u(s)}, y)
	+ \widetilde{h}(s,y)
	\big\|^2
	\nu(dy) \, ds
	\right]                                        \nonumber\\
	\le & 
	2 \varepsilon C_\theta \,
	\mathbb{E}\left[ 
	\int_{\tau-t}^{r}
	\int_{y\in\mathbb{Y}
	 }
	e^{\beta s}
	\|u(s)\|^{\theta-2}
	\big\|
	\widetilde{\sigma}(u(s), \mathcal{L}_{u(s)}, y)
	\big\|^2
	\nu(dy) \, ds
	\right]                                        \nonumber\\
	& + 2 \varepsilon C_\theta \,
	\mathbb{E}\left[
	\int_{\tau-t}^{r}
	e^{\beta s}
	\|u(s)\|^{\theta-2}
	\int_{y\in\mathbb{Y}}
	\|\widetilde{h}(s,y)\|^2 \,  \nu (dy) ds
	\right]                                      \nonumber\\[1ex]
	\le & 
	\frac{\theta \lambda_1}{8} \,
	\mathbb{E}\left[
	\int_{\tau-t}^{r}
	e^{\beta s}
	\|u(s)\|^{\theta-2} \|u(s)\|_p^p \, ds
	\right]  \nonumber\\
	& 
	+ 8 \varepsilon C_\theta \|L\|^2 \,
	\mathbb{E}\left[
	\int_{\tau-t}^{r}
	e^{\beta s}
	\|u(s)\|^{\theta-2}
	\mathbb{E}\bigl[\|u(s)\|^2\bigr] \, ds
	\right]                                      
	+ \varepsilon C_2   
	\mathbb{E}\left[
	\int_{\tau-t}^{r}
	e^{\beta s} \|u(s)\|^{\theta-2} \, ds
	\right]  \nonumber\\
	& + \varepsilon \frac{2 C_\theta (\theta-2)}{\theta} \,
	\mathbb{E}\left[
	\int_{\tau-t}^{r}
	e^{\beta s} \|u(s)\|^{\theta} \, ds
	\right]
	+ \frac{4 C_\theta}{\theta} \,
	\mathbb{E}\left[
	\int_{\tau-t}^{r}
	e^{\beta s} \|\widetilde{h}(s)\|^{\theta}
	_{L^2(\mathbb{Y}, \nu; \ell^2) 
	} \, ds
	\right]                                      \nonumber\\[1ex]
	\le &
	\frac{\theta \lambda_1}{8} \,
	\mathbb{E}\left[
	\int_{\tau-t}^{r}
	e^{\beta s} \|u(s)\|^{\theta-2} \|u(s)\|_p^p \, ds
	\right]  \nonumber\\
	& + \varepsilon \left(
	8 C_\theta \|L\|^2
	+ (\theta-2)
	\left(
	\frac{C_2}{\theta}
	+ \frac{2 C_\theta}{\theta}
	\right)
	\right)
	\mathbb{E}\left[
	\int_{\tau-t}^{r}
	e^{\beta s} \|u(s)\|^{\theta} \, ds
	\right]                                      \nonumber\\
	& + {\frac{ 2C_2}{\theta}}
	\frac{e^{\beta r}}{\beta}
	+ \frac{4 C_\theta}{\theta} \,
	\int_{\tau-t}^{r}
	e^{\beta s} \|\widetilde{h}(s)\|^{\theta}
	_{L^2(\mathbb{Y}, \nu; \ell^2) 
	}  \, ds ,
\end{align}
where $C_2=C_2 (p,q, L, \theta)>0
$ is a constant, and  we have used 
  Young's inequality and \eqref{sigma3} with
  $\lambda_0$ replaced by
  ${\frac 12} \theta \lambda_1$
  and  
$\tilde{\theta}$ replaced by 
$ 2C_{\theta}$ where $C_{\theta}$ is the same  
number as in \eqref{moment_gjh}.

By \eqref{h_22_simple} and Hölder's inequality, we  also have
\begin{align}\label{moment_h5}
	\mathbb{E}[I_{6,2}]
	:= &
	\varepsilon^{\frac{\theta}{2}} C_\theta
	\mathbb{E}\left[ 
	\int_{\tau-t}^{r}
	\int_{
	y\in\mathbb{Y}
	}
	e^{\beta s}
	\big\|
	\widetilde{\sigma}(u(s), \mathcal{L}_{u(s)}, y)
	+ \widetilde{h}(s,y)
	\big\|^\theta
	\nu(dy) \, ds
	\right] \nonumber\\
	\le &
	2^{\theta-1} C_{\theta}
	\int_{\tau-t}^{r}
	e^{\beta s} 
	\|\widetilde{h}(s)\|^{\theta}_{L^{\theta}(\mathbb{Y}, \nu; \ell^2)} ds
	 \nonumber\\
	& + \varepsilon^{\frac{\theta}{2}}
	2^{\theta-1} C_{\theta} 
	\mathbb{E}\left[
	\int_{\tau-t}^{r} 
	\int_{y \in \mathbb{Y} }
	e^{\beta s}
	\big\|
	\widetilde{\sigma}(u(s), \mathcal{L}_{u(s)}, y)
	\big\|^{\theta}
	\nu (dy) ds
	\right] \nonumber\\
	\le &
	2^{\theta-1} C_{\theta}
	\int_{\tau-t}^{r}
	e^{\beta s} 
	\|\widetilde{h}(s)\|^{\theta}_{L^{\theta}(\mathbb{Y}, \nu; \ell^2)} ds
	 \nonumber\\
	& + \varepsilon^{\frac{\theta}{2}}
	6^{\theta-1} 2^{\frac{\theta}{2}} C_{\theta}
	\left(
	  \int_{
	  	y\in\mathbb{Y}
	  }
	\left(
	\sum_{i \in \mathbb{Z}}
	\big( L_{\tilde{\sigma},i}(y) \big)^{\frac{\theta}{\theta-1}}
	\right)^{\theta-1}
	\nu(dy)
	\right)
	\mathbb{E}\left[
	\int_{\tau-t}^{r}
	e^{\beta s}
	\sum_{i \in \mathbb{Z}}
	|u_i(s)|^{\frac{q\theta}{2}} ds
	\right]  \nonumber\\
	& + \varepsilon^{\frac{\theta}{2}}
	6^{\theta-1} 2^{\frac{\theta}{2}} C_{\theta}
	\left(
	 \int_{y \in \mathbb{Y} }
	\left(
	\sum_{i \in \mathbb{Z}}
	L_{\tilde{\sigma}, i}(y)
	\right)^{\theta}
	\nu(dy)
	\right)
	\mathbb{E}\left[
	\int_{\tau-t}^{r}
	e^{\beta s} \|u(s)\|^{\theta} ds
	\right]  \nonumber\\
	& + \frac{1}{\beta} e^{\beta r}
	6^{\theta-1} 2^{\frac{\theta}{2}} C_{\theta}
	  \int_{y \in \mathbb{Y} }
	\left( \sum_{i \in \mathbb{Z} }
	\big|\widetilde{\sigma}_{i}(0,\delta_0,y)\big|
	\right)^{\theta}
	\nu(dy) .
\end{align}
As for the second term on the right-hand side of \eqref{moment_h5}, 
for $a = \tfrac{\theta(q-2)}{2(p-2)}$ and $b = (1-a) \theta$,
by Young's inequality we have 
\begin{align}\label{moment_h6}
	& \varepsilon^{\frac{\theta}{2}}
	6^{\theta-1} 2^{\frac{\theta}{2}} C_{\theta}
	\left(
	  \int_{
	  y\in\mathbb{Y}
	  }
	\left(
	\sum_{i \in \mathbb{Z}}
	\big( L_{\tilde{\sigma},i}(y)
	\big)^{\frac{\theta}{\theta-1}}
	\right)^{\theta-1} \nu(dy)
	\right)
	\mathbb{E}\left[
	\int_{\tau-t}^{r}
	e^{\beta s} 
	|u_i(s)|^{\frac{q\theta}{2}} ds
	\right] \nonumber\\
	\le &
	\varepsilon^{\frac{\theta}{2}}
	6^{\theta-1} 2^{\frac{\theta}{2}} C_{\theta}
	  \int_{
	  	  y\in\mathbb{Y}
	  }
	\left(
	\sum_{i \in \mathbb{Z}}
	\big( L_{\tilde{\sigma},i}(y)
	\big)^{\frac{\theta}{\theta-1}}
	\right)^{\theta-1} \nu(dy) \nonumber\\
	& \quad \cdot
	\mathbb{E}\left[
	\int_{\tau-t}^{r}
	e^{\beta s}
	\sum_{i \in \mathbb{Z}}
	\big( |u_i(s)|^{\theta-2} |u_i(s)|^{p} \big)^{a}
	|u_i(s)|^{b} ds
	\right]  \nonumber\\
	\le &
	\frac{\lambda_1 \theta}{8}
	\mathbb{E}\left[
	\int_{\tau-t}^{r}
	e^{\beta s}
	\|u(s)\|^{\theta-2} \|u(s)\|_p^{p} ds
	\right] \nonumber\\
	& \quad
	+ \varepsilon^{\frac{\theta}{2}} (1-a)
	\left(
	6^{\theta-1} 2^{\frac{\theta}{2}} C_{\theta}
  \int_{
  	  y\in\mathbb{Y}
  }
	\left(
	\sum_{i \in \mathbb{Z}}
	\big(L_{\tilde{\sigma},i}(y)\big)^{\frac{\theta}{\theta-1}}
	\right)^{\theta-1} \nu(dy)
	\right)^{\frac{1}{1-a}}  \nonumber\\
	& \quad \cdot
	\left( \frac{\lambda_1 \theta}{8 a} \right)^{-\frac{a}{1-a}}
	\mathbb{E}\left[
	\int_{\tau-t}^{r}
	e^{\beta s} \|u(s)\|^{\theta} ds
	\right].
\end{align}

 For convenience,  denote  by 
 \begin{align}\label{K-theta-eps}
	K(\theta,\varepsilon)
	:= & - \theta \lambda
	         + \theta \|\phi_1\|_{\ell^\infty}
	         + \theta \|\phi_1\|_{1}
	         + 4(\theta-1)^{\frac{\theta - 1}{\theta} } e^{\frac{\beta \rho}{\theta} }
	          	 \left( \|L_{F}\|_{\ell^\infty} + \|L_F\|
	          	 \right)   \nonumber\\
	& + 4 \varepsilon  \theta(\theta-1)\|L\|^2
	+ \varepsilon (\theta-2)
	 C_1
	+ \varepsilon (\theta-1)(\theta-2)                            \nonumber\\
	& + 8 \varepsilon C_{\theta} \|L\|^2
	+ \varepsilon (\theta-2)
	\left(
	\frac{C_2}
	 {\theta}
	+ \frac{2 C_{\theta}}{\theta}
	\right)                                                 \nonumber\\
	& + \varepsilon^{\frac{\theta}{2}} (1-a)
	\left(
	6^{\theta-1} 2^{\frac{\theta}{2}} C_{\theta}
  \int_{
  y\in \mathbb{Y}
  }
	\left(
	\sum_{i \in \mathbb{Z}}
	(L_{\tilde{\sigma},i}(y))^{\frac{\theta}{\theta-1}}
	\right)^{\theta-1}
	\nu(dy)
	\right)^{\frac{1}{1-a}}
	\left(
	\frac{\lambda_1 \theta}{8 a}
	\right)^{-\frac{a}{1-a}}                                   \nonumber\\
	& + \varepsilon^{\frac{\theta}{2}}
	6^{\theta-1} 2^{\frac{\theta}{2}} C_{\theta}
	\left(
	  \int_{ y\in \mathbb{Y}
	  }
	\left(
	\sum_{i \in \mathbb{Z}} L_{\tilde{\sigma},i}(y)
	\right)^{\theta}
	\nu(dy)
	\right).
 \end{align}
 Then by using a  stopping time if necessary,
 we obtain from \eqref{thetapue-1}-\eqref{moment_h6}  that 
 for $r \in [\tau - { 3} \rho, \tau]$ and $t \ge { 3} \rho$,
 \begin{align}\label{3.12_4}
 	   & \mathbb{E}
 	              \left[ e^{\beta r} \|u(r)\|^{\theta} \right]
 	       + \left( -3\beta - K(\theta,\varepsilon) \right)
              \int_{\tau-t}^{r}
 	              \mathbb{E}
 	                     \left[ e^{\beta s} \|u(s)\|^{\theta}
 	                     \right] ds  \nonumber\\
 	   & + \frac{\lambda_1 \theta}{2}
 	               \mathbb{E}
 	                     \left[ \int_{\tau-t}^{r}
 	                                  e^{\beta s}
                    	              \|u(s)\|^{\theta-2} \|u(s)\|_p^{p} \, ds
 	                     \right]  \nonumber\\
\le & e^{\beta (\tau-t)}
 	     \mathbb{E} \left[ \|\varphi(0)\|^{\theta} \right]
 	     + \frac{2 \rho}{\beta} 
 	        (\theta - 1)^{\frac{\theta - 1}{\theta} } 
 	      	e^{\frac{\beta \rho}{\theta} } 
 	      	 \left( \|L_{F}\|_{\ell^\infty} + \|L_F\| \right)
 	      	 e^{\beta (\tau - t)} 
 	      	 \mathbb{E} \left[ \|\varphi\|_{\infty}^{\theta} \right]  \nonumber\\
     & + C_3 
      	     \int_{\tau-t}^{r} e^{\beta s}
 	             \left( \|g(s)\|^{\theta} 
 	                      + \|h(s)\|^{\theta} 
 	                      + \| \widetilde{h}(s)\|_{L^{\theta}(\mathbb{Y},\nu; \ell^2)}^\theta
 	                        + \| \widetilde{h}(s)\|_{L^{2}(\mathbb{Y},\nu; \ell^2)}^\theta
 	             \right) ds
 	     + C_3  e^{\beta r},
 \end{align}
 where $C_3=C_3(\beta, \theta, p, q, L) > 0$ is a constant independent of $\varphi$,
 $\varepsilon$ and $\rho$.
 
 By \eqref{dissip1-1} and \eqref{K-theta-eps}, 
 we infer that there exists
 $\varepsilon^* =\varepsilon^* (\theta, \beta) \in (0,1)$ independent of $\rho$
 such that for all $\varepsilon \in (0, \varepsilon^*)$,
 \be\label{jan24a}
   5\beta + K(\theta, \varepsilon) <0.
\ee
   
 On the other hand, 
  since the distribution law $\mathcal{L}_{\varphi} \in D(\tau - t)$, we obtain
 \begin{align*}
 	& 
 	\mathbb{E} \left[   \|\varphi \|^{\theta}
 	_\infty  \right] e^{-\beta t} 
 	\le 
 	\mathbb{E} \left[ \|D(\tau-t)\|^{\theta}_{\mathcal{P}_{\theta}(D_{\rho})} \right] e^{-\beta t} \to 0, \ \text{as}\ t \to  \infty.
 \end{align*}
 Then, 
 it follows
 from \eqref{3.12_4}
 and \eqref{jan24a} that 
 there exists $T = T(\tau, D)  \ge { 3}\rho$ such that for all $t \ge T$,
 $\varepsilon \in (0, \varepsilon^*)$
  and $r \in [\tau - {3} \rho, \tau]$,  
\begin{align}\label{pue-11}
	  & \mathbb{E}
	            \left[ \|u(r; \tau-t, \varphi)\|^{\theta}
        	   \right]
	      + \beta \mathbb{E}
	                         \left[ \int_{\tau-t}^{r} e^{\beta (s-r)}
	                                       \|u(s;\tau-t, \varphi)\|^{\theta} ds
	                         \right]  \nonumber\\
	  & + \frac{\lambda_1 \theta}{2}
	              \mathbb{E}
	                     \left[ \int_{\tau-t}^{r} e^{\beta (s-r)}
       	                              \|u(s;\tau-t, \varphi)\|^{\theta-2}
	                                  \|u(s;\tau-t, \varphi)\|_p^{\,p} ds
	                     \right]   \nonumber\\
 \le & 2C_3  + C_3
	                       \int_{-\infty}^{\tau} e^{\beta (s-r)}
      	                                  \Big(	\|g(s)\|^{\theta}
	                                                + \|h(s)\|^{\theta}
	                                                + \|\widetilde{h}(s)\|^{\theta}_{L^{\theta}(\mathbb{Y}, \nu; \ell^2)}
	                                                  + \|\widetilde{h}(s)\|^{\theta}_{L^{2}(\mathbb{Y}, \nu; \ell^2)}
	                                      \Big) ds.                                      
\end{align}

 Similarly to  \eqref{thetapue-1} with $\beta =0$, 
 we have for all 
 $ r \in [\tau- { 2}\rho,\tau]$,
 \begin{align}\label{pue-9}
 	& 
 	\|u(r)\|^{\theta}
 	+ \theta \int_{\tau- { 2}\rho}^{r}
 	\|u(s)\|^{\theta-2}
 	\big\langle f\big(u(s), \mathcal{L}_{u(s)}\big), u(s) \big\rangle
 	\, ds                                                \nonumber\\
 	\le &
 	\|u(\tau- 2\rho)\|^{\theta}
    + \theta \int_{\tau - 2\rho}^{\tau}
 	\|u(s)\|^{\theta-2}
 	\big| \langle g(s), u(s) \rangle \big| \, ds       \nonumber\\
 	& + \theta \int_{\tau - 2\rho}^{\tau}
 	\|u(s)\|^{\theta-2}
 	\big|
 	\big\langle
 	F\big(u(s - \rho), \mathcal{L}_{u(s-\rho)}\big),
 	u(s)
 	\big\rangle
 	\big| \, ds                                        \nonumber\\
 	& + \sup_{r \in [\tau - 2\rho,\tau]}
 	\left|
 	\theta \int_{\tau - 2\rho}^{r}
 	\|u(s)\|^{\theta-2}
 	\big\langle
 	\sum_{k \in \mathbb{N}}
 	\big( \sigma_k(u(s), \mathcal{L}_{u(s)}) + h_k(s) \big)
 	\, dW_k(s),
 	u(s)
 	\big\rangle
 	\right|                                              \nonumber\\
 	& + \frac{\theta(\theta-1)}{2}
 	\int_{\tau - 2\rho}^{\tau}
 	\|u(s)\|^{\theta-2}
 	\sum_{k \in \mathbb{N}}
 	\|\sigma_k(u(s), \mathcal{L}_{u(s)}) + h_k(s)\|^2 \, ds
 	\nonumber\\
 	& + \sup_{r \in [\tau - 2\rho,\tau]}
 	\Bigg| 
 	\int_{\tau - 2\rho}^{r} \int_{
 	y\in \mathbb{Y}
 	}
 	\|u(s-)\|^{\theta-2}
 	\big\langle
 	\widetilde{\sigma}\big(u(s-), \mathcal{L}_{u(t)}, y\big)
 	+ \widetilde{h}(s,y),
 	u(s-)
 	\big\rangle
 	\,\widetilde{N}(ds,dy)
 	\Bigg|                                               \nonumber\\
 	& + \int_{\tau - 2\rho}^{\tau} \int_{y \in \mathbb{Y} } 
 	 \Big|
 	                \big\| u(s-)
 	                          + \sqrt{\varepsilon} 
 	                               \widetilde{\sigma}\big( u(s-), \mathcal{L}_{u(t)}, y \big)
 	                         + \sqrt{\varepsilon} \widetilde{h}(s,y)
 	                \big\|^{\theta}
 	               - \|u(s-)\|^{\theta}   \nonumber\\
 	& \qquad\qquad\qquad
 	               - \theta \|u(s-)\|^{\theta-2}
 	                       \sqrt{\varepsilon}
 	                           \big\langle
 	                                      \widetilde{\sigma}\big(u(s-), \mathcal{L}_{u(t)}, y\big)
 	                                     + \widetilde{h}(s,y),
 	                                     u(s-)
 	                           \big\rangle
 	  \Big| N(ds,dy) .
 \end{align}
 For the third term on the right-hand side of \eqref{pue-9}, 
 by \eqref{F3} we have for $t\ge 3\rho$,
 \begin{align}
 	& \mathbb{E}\Biggl[
 	\theta \int_{\tau - 2\rho}^{\tau}  \|u(s)\|^{\theta-2}
 	\bigl|
 	\langle F(u(s-\rho), \mathcal{L}_{u(s-\rho)}), u(s) \rangle 
 	\bigr| ds
 	\Biggr]
 	\nonumber \\
 	\le & \theta 
 	\mathbb{E}\Biggl[
 	\int_{\tau - 2\rho}^{\tau}
 	\|u(s)\|^{\theta-1}
 	\bigl\|F(u(s-\rho),\mathcal{L}_{u(s-\rho)})\bigr\|
 	ds\Biggr]
 	\nonumber \\
 	\le &
 	\theta \mathbb{E}\Biggl[
 	\int_{\tau - 2\rho}^{\tau} 
 	\|u(s)\|^{\theta-1}
 	\Bigl(
 \sqrt{2}	\|F(0,\delta_0)\|
 	+2 \|L_{F}\|_{\ell^\infty} \|u(s-\rho)\|
 	+2 \|L_{F}\|\, \mathbb{W}_2(\mathcal{L}_{u(s-\rho)}, \delta_0) 
 	\Bigr) ds
 	\Biggr]
 	\nonumber \\
 	\le &
 	\beta  
 	\mathbb{E}\Biggl[
 	\int_{\tau - 2\rho}^{\tau} \|u(s)\|^{\theta} ds \Biggr]
 	+ \Bigl( \tfrac{2(\theta - 1)}{\beta} \Bigr)^{\theta - 1}
 2^{\frac{\theta}{2}}	\|F(0,\delta_0)\|^{\theta}\,2\rho
 	\nonumber \\
 	& + \Bigl( \tfrac{2(\theta - 1)}{\beta} \Bigr)^{\theta - 1}
 2^{\theta}	\|L_{F}\|_{\ell^\infty}^{\theta}
 	\mathbb{E}\Biggl[
 	\int_{\tau - 2\rho}^{\tau} \|u(s-\rho)\|^{\theta} ds
 	\Biggr]
 	\nonumber \\
 	& + 2 (\theta-1)\|L_F\| \,
 	\mathbb{E}\Biggl[
 	\int_{\tau - 2\rho}^{\tau} \|u(s)\|^{\theta} ds \Biggr]
 	+ 2\|L_F\| \,
 	\mathbb{E}\Biggl[
 	\int_{\tau - 2\rho}^{\tau} \|u(s-\rho)\|^{\theta} ds \Biggr]
 	\nonumber \\
 	\le &
 	\bigl(\beta + 2(\theta -1)\|L_F\|\bigr)
 	\mathbb{E}\Biggl[
 	\int_{\tau - 2\rho}^{\tau} \|u(s)\|^{\theta} ds \Biggr]
    + \Bigl( \tfrac{2(\theta - 1)}{\beta} \Bigr)^{\theta - 1}2^{\frac{\theta}{2}}
 	\|F(0,\delta_0)\|^{\theta}\, 2\rho
 	\nonumber \\
 	& + \Bigl(
 	\Bigl( \tfrac{2(\theta - 1)}{\beta} \Bigr)^{\theta - 1}
 2^{\theta}	\|L_{F}\|_{\ell^\infty}^{\theta}
 	+2 \|L_{F}\|
 	\Bigr)
 	\mathbb{E}\Biggl[
 	\int_{\tau - 3\rho}^{\tau - \rho} \|u(s)\|^{\theta} ds \Biggr] .
 \end{align}
 By Young's inequality,
 the BDG inequality
 and 
 \eqref{sigma3},  
 we obtain the following
 estimate for the fourth term on the right-hand side of \eqref{pue-9}:
 \begin{align}
 	& \mathbb{E}\Biggl[
 	\sup_{r \in [\tau - 2\rho,\tau]}
 	\Biggl| \theta \int_{\tau - 2\rho}^{r} \|u(s)\|^{\theta-2}
 	\big\langle
 	\sum_{k \in \mathbb{N}}
 	\bigl( \sigma_k(u(s),\mathcal{L}_{u(s)}) + h_k(s) \bigr) dW_k(s),
 	u(s)
 	\big\rangle
 	ds \Biggr|
 	\Biggr]
 	\nonumber \\
 	&\le
 	\theta C_4 \,
 	\mathbb{E}\Biggl[
 	\biggl(
 	\int_{\tau - 2\rho}^{\tau}
 	\|u(s)\|^{2\theta-2}
 	\sum_{k \in \mathbb{N}}
 	\bigl\|\sigma_k(u(s),\mathcal{L}_{u(s)}) + h_k(s)\bigr\|^{2}
 	ds
 	\biggr)^{\frac12}
 	\Biggr]
 	\nonumber \\
 	&\le
 	\theta \,
 	\mathbb{E}\Biggl[
 	\biggl(
 	\int_{\tau - 2\rho}^{\tau}
 	\|u(s)\|^{2\theta-2}
 	\Bigl(
 	\tfrac{\lambda_1}{4\theta} \|u(s)\|_{p}^{p}
 	+ 8C_4 \|L\|^{2}\,\mathbb{E}\bigl[\|u(s)\|^{2}\bigr]
 	+ C_5
 	 + 2C_4 \|h(s)\|^{2}
 	\Bigr)
 	ds
 	\biggr)^{\frac12}
 	\Biggr]
 	\nonumber \\
 	&\le
 	\theta \,
 	\mathbb{E}\Biggl[
 	\Bigl(
 	\sup_{s \in [\tau- 2\rho,\tau]}
 	\|u(s)\|^{\theta}
 	\Bigr)^{\frac12}
 	\nonumber \\
 	& \qquad \qquad  \times
 	\biggl(
 	\int_{\tau - 2\rho}^{\tau}
 	\|u(s)\|^{\theta-2}
 	\Bigl(
 	\tfrac{\lambda_1}{4 \theta} \|u(s)\|_{p}^{p}
 	+ 8C_4 \|L\|^{2}\,\mathbb{E}\bigl[\|u(s)\|^{2}\bigr]
 	+ C_5 
 	+ 2C_4 \|h(s)\|^{2}
 	\Bigr)
 	ds
 	\biggr)^{\frac12}
 	\Biggr]
 	\nonumber \\
 	&\le
 	\frac{1}{4} \,
 	\mathbb{E}\Biggl[
 	\sup_{s \in [\tau - 2\rho,\tau]}
 	\|u(s)\|^{\theta}
 	\Biggr]
 	\nonumber \\
 	&\quad
 	+{\theta}^2
 	\mathbb{E}\Biggl[
 	\int_{\tau - 2\rho}^{\tau}
 	\|u(s)\|^{\theta-2}
 	\Bigl(
 	\tfrac{\lambda_1}{4 \theta} \|u(s)\|_{p}^{p}
 	+ 8C_4 \|L\|^{2}\,\mathbb{E}\bigl[\|u(s)\|^{2}\bigr]
 	+ C_5
 	+ 2C_4 \|h(s)\|^{2}
 	\Bigr)
 	ds
 	\Biggr]
 	\nonumber \\
 	&\le
 	\frac{1}{4} \,
 	\mathbb{E}\Biggl[
 	\sup_{s \in [\tau - 2\rho,\tau]}
 	\|u(s)\|^{\theta}
 	\Biggr]
 	+
 	\frac{\lambda_1 \theta}{4 }
 	\mathbb{E}\Biggl[
 	\int_{\tau - 2\rho}^{\tau}
 	\|u(s)\|^{\theta-2} \|u(s)\|_{p}^{p} ds
 	\Biggr]
 	\nonumber \\
 	&\quad
 	+ 8\theta^2 C_4 \|L\|^{2}
 	\mathbb{E}\Biggl[
 	\int_{\tau - 2\rho}^{\tau}
 	\|u(s)\|^{\theta} ds
 	\Biggr]
 	+ \theta  (\theta-2)
 	\Bigl(
 	C_5 + 2C_4
 	\Bigr)
 	\mathbb{E}\Biggl[
 	\int_{\tau - 2\rho}^{\tau}
 	\|u(s)\|^{\theta} ds
 	\Biggr]
 	\nonumber \\
 	&\quad
 	+ 4 \theta C_5
 	  \rho
 	+ 4 \theta C_4
 	\int_{\tau - 2\rho}^{\tau}
 	\|h(s)\|^{\theta} ds,
 \end{align}
 where we have used
 \eqref{sigma3}
 with  $\tilde{\theta} = 2C_4$.
By the BDG inequality
and  \eqref{sigma3}, we have
\begin{align}\label{pue-10}
	& \mathbb{E}\Biggl[
	\sup_{r \in [\tau - 2\rho,\tau]}
	\Bigl|
	\theta 
	 \int_{\tau - 2\rho}^{r} \int_{
	 \mathbb{Y}
	 } 
	\|u(s-)\|^{\theta-2}
	\big\langle
	\widetilde{\sigma}(u(s-),\mathcal{L}_{u(t)}, y) 
	+ \widetilde{h}(s),
	u(s-)
	\big\rangle 
	\widetilde{N}(ds,dy)
	\Bigr|
	\Biggr]
	\nonumber \\
	&\le
	\theta C_{6 }
	\mathbb{E}\Biggl[
	\Biggl(
	\int_{\tau - 2\rho}^{\tau} \int_{
	 \mathbb{Y}
	} 
	\|u(s-)\|^{2\theta-2}
	\bigl\|
	\widetilde{\sigma}(u(s-),\mathcal{L}_{u(s)}, y) 
	+ \widetilde{h}(s,y)
	\bigr\|^{2}
	N(ds, dy)
	\Biggr)^{\frac{1}{2}}
	\Biggr]
	\nonumber \\
	&   
	\le \frac{1}{4} 
	          \mathbb{E}
	            \left[\sup_{[\tau - 2\rho, \tau]} \|u(s)\|^\theta
	            \right]
	     + \theta^2 C_6^2
	                \mathbb{E}
	                       \Biggl[ \int_{\tau - 2\rho}^{\tau} 
	                                       \int_{ \mathbb{Y} } 
	                                                 \|u(s)\|^{\theta-2}
	                                                 \bigl\| \widetilde{\sigma}(u(s),\mathcal{L}_{u(s)}, y) 
	                                                                + \widetilde{h}(s,y)
                                                   	 \bigr\|^{2}
	                                        \nu(dy) ds
	                      \Biggr]   
	\nonumber \\	
	&\le 
	        \frac{1}{4}
	             \mathbb{E}
	                    \Biggl[ \sup_{s \in [\tau - 2\rho,\tau]}
  	                                     \|u(s)\|^{\theta}
	                    \Biggr]
	       + \frac{\lambda_1 \theta}{16 }
	                \mathbb{E}
	                       \Biggl[ \int_{\tau - 2\rho}^{\tau}
	                                       \|u(s)\|^{\theta-2}
	                                       \|u(s)\|_{p}^{p} ds
	                       \Biggr]  
	                       \nonumber \\
	&\quad  
	   + 8\theta^2 C_6^2 \|L\|^2 
	        \mathbb{E}
	             \Biggl[ \int_{\tau - 2\rho}^{\tau}
                             	\|u(s)\|^{\theta} ds
	            \Biggr]
	   + \theta(\theta-2)
	           \Bigl( C_7 + 2 C_6^2
	           \Bigr)
	           \mathbb{E}
	                 \Biggl[ \int_{\tau - 2\rho}^{\tau}
	                                 \|u(s)\|^{\theta} ds
	                \Biggr]   
	                \nonumber \\
	&\quad  
	   + 2 \theta C_7 \rho
	   + 4 \theta C_6^2
	         \int_{\tau - 2\rho}^{\tau}
	             \| \widetilde{h}(s) \|^{\theta}_{L^2(\mathbb{Y}, \nu; \ell^2) } 
	             ds,   
\end{align}
where we have used
 \eqref{sigma3}
 with  $\tilde{\theta} = 2C_6^2$.

Following the argument of  \eqref{moment_h1}-\eqref{moment_h2} and \eqref{moment_h3}-\eqref{moment_h6},
  by \eqref{pue-9}-\eqref{pue-10}  we obtain that for all $t \ge {  3}\rho$,  
 \begin{align}\label{pue-12}
 	&  
 	\mathbb{E}\left[
 	\sup_{r \in [\tau - 2\rho,\tau]} \|u(r)\|^{\theta}
 	\right] 
 	\le  C_8
 	\mathbb{E}\left[ \|u(\tau - 2\rho)\|^{\theta} \right]
 	+ C_8 
 	\mathbb{E}\left[
 	\int_{\tau- 3\rho}^{\tau}
 	\|u(s)\|^{\theta} ds
 	\right]
 	 \nonumber\\
 	&\quad
 	+C_8
 	+C_8 
 	   \int_{\tau - 2\rho}^{\tau}
 	\big(
 	\|g(s)\|^{\theta}
 	+ \|h(s)\|^{\theta}
 	+ \|\widetilde{h}(s)\|^{\theta}_{L^\theta(\mathbb{Y}, \nu; \ell^2)}
 	+ \|\widetilde{h}(s)\|^{\theta}_{L^{2}(\mathbb{Y}, \nu; \ell^2)}
 	\big) ds,
 	 \end{align}
 where $C_8
 =C_8(\theta, \beta, p,q, L)>0$ 
 is a constant.
 
 By \eqref{pue-11} and \eqref{pue-12}
 we infer that
 there exists $T_1 = T_1(\tau, D) 
 \ge T$ such that for all $t \ge T_1$ and $r \in [\tau -2\rho, \tau]$,
\begin{align} \label{pue-13}
	& \mathbb{E}\left[
	\sup_{r \in [\tau - 2\rho,\tau]}
	\|u(r;\tau-t,\varphi)\|^{\theta}
	\right]   \nonumber\\
	\le\;&
	C_9 (1+\rho)
  \Bigg(
	1
	+ \int_{-\infty}^{\tau} e^{\beta(s-r)}
	\Big(
	\|g(s)\|^{\theta}
	+ \|h(s)\|^{\theta}
	+ \|\widetilde{h}(s)\|^{\theta}_{L^{\theta}(\mathbb{Y}, \nu; \ell^2)}
	+ \|\widetilde{h}(s)\|^{\theta}_{L^{2}(\mathbb{Y}, \nu; \ell^2)}
	\Big) ds                       
	\Bigg)   \nonumber\\
	& \quad
  +C_8 
 	   \int_{\tau - 2\rho}^{\tau}
 	\big(
 	\|g(s)\|^{\theta}
 	+ \|h(s)\|^{\theta}
 	+ \|\widetilde{h}(s)\|^{\theta}_{L^{\theta}(\mathbb{Y}, \nu; \ell^2)}
 	+ \|\widetilde{h}(s)\|^{\theta}_{L^{2}(\mathbb{Y}, \nu; \ell^2)}
 	\big) ds,
 	 \end{align}
 where $C_9
 =C_9(\theta, \beta, p,q, L)>0$ 
 is a constant independent of $\rho$,
 $\varepsilon$, $D$  and $\tau$.
 By  \eqref{pue-11} and \eqref{pue-13}, we complete 
  the proof.
\end{proof}

 The next lemma
 is concerned 
 with  the existence of $ \mathcal{D} $-pullback absorbing set
 for the non-autonomous cocycle $\Phi$ associated with \eqref{lmvlds-2}.

\begin{lemma}\label{pas}
   Suppose that {\bf(H1)}-{\bf(H4)}, \eqref{thetaassum2}, \eqref{dissip1} and \eqref{dissip2} hold.
   Then, for all $\varepsilon \in (0,\varepsilon^{*})$, the non-autonomous cocycle $\Phi$ has a closed $\mathcal{D}$-pullback absorbing set
       $K = \{ K(\tau): \tau \in \mathbb{R} \} \in \mathcal{D}$
   which is given by
   \begin{align*}
      K(\tau) =
                \left\{ \mu \in \mathcal{P}_\theta(D_{\rho}):
                                \int_{D_{\rho}} \| \xi \|_{\infty}^{\theta} \mu(d \xi)
                                \le R_{\tau}
                \right\},
   \end{align*}
   where $  	R_{\tau} >0$ is a  number given by 
        \begin{align}
        	R_{\tau} = &C_{\beta,\theta,
        	\rho}+ C_{\beta,\theta, \rho} \int_{-\infty}^{\tau} e^{\beta (s-\tau) }
        \left( \|g(s)\|^{\theta} + \|h(s)\|^{\theta} 
       + \|\widetilde{h}(s)\|^{\theta}_{L^{\theta}(\mathbb{Y}, \nu; \ell^2)} 
        \right) ds  , 
        \end{align}
   and $C_{\beta,\theta, \rho}$ and $\varepsilon^*$ are 
   the same
    constants as in Lemma \ref{pue}.
\end{lemma}

\begin{proof}
This is an immediate consequence of
  Lemma \ref{pue}, and  the
  details  are omitted.
 \end{proof}

\subsection{Pullback asymptotic compactness}
 
 In this subsection,   we 
  prove the $\mathcal{D}$-pullback asymptotic compactness of 
  the cocycle $\Phi$, 
  which along with
  the existence of
  $\mathcal{D}$-pullback
  absorbing sets implies
  the existence of
  measure attractors.
   
 From now on, we always assume
 $2 < \theta < {\tfrac {2(p-2)}{q-2} }$ and
 $\varepsilon \in (0, \varepsilon^*)$
 where $\varepsilon^*$ is the  same constant  as in Lemma \ref{pue}.

 Next, we show the uniform pullback estimates on the tails of solutions to \eqref{lmvlds-2} in $L^\infty(\tau-2\rho, \tau; L^2(\Omega, \ell^2) )$
 with $\tau
 \in\mathbb{R}$.

\begin{lemma}\label{upte}
   Suppose that {\bf(H1)}-{\bf(H4)}, \eqref{thetaassum2}, \eqref{dissip1} and \eqref{dissip2} hold.
   Then for any $\tau \in \mathbb{R}$, $\delta > 0$ and $D = \{D(t): t \in \mathbb{R} \} \in \mathcal{D}$,
   there exist constants $T = T(\tau, \delta, D)  \ge { 3}\rho$ 
   and $m_0 = m_0(\tau, \delta) \ge 1$
   such that for all $t \ge T$ and $m \ge m_0$, the solution $u$ of \eqref{lmvlds-2} satisfies
   \begin{align*}
      \sup_{r \in [\tau-2\rho, \tau]} \sum_{|i| \ge m}
            \mathbb{E} \left[ |u_i(r; \tau-t, \varphi)|^2 \right]
      < \delta,
   \end{align*}
   where $\varphi \in L^\theta(\Omega, \mathcal{F}_{\tau-t}; D_{\rho})$ with law $\mathcal{L}_{\varphi} \in D(\tau-t)$.
\end{lemma}

\begin{proof}
 Consider a smooth function $\vartheta: \mathbb{R} \to [0,1]$ satisfying
 \begin{align}\label{smoothf}
    \vartheta(s) = 0\ \text{for}\ |s| \le 1; \ \text{and}\
    \vartheta(s) = 1\ \text{for}\ |s| \ge 2.
 \end{align}
 For given $m \in \mathbb{N}$,
 denote by $\vartheta_m = ( \vartheta(\frac{i}{m}) )_{i \in \mathbb{Z}}$
 and $\vartheta_m u = (\vartheta(\frac{i}{m}) u_i)_{i \in \mathbb{Z}}$
 for $u = (u_i)_{i \in \mathbb{Z}} \in \ell^2$.
 Let $\vartheta_m u(r) = \vartheta_m u(r; \tau - t,\varphi)$ for $r \ge \tau-t$.
 Then by \eqref{lmvlds-2} we have
 \begin{align}\label{upte-1}
       & d \vartheta_m u(r)
          + \vartheta_m Au(r) dr
          + \lambda \vartheta_m u(r) dr
          + \vartheta_m f(u(r),\mathcal{L}_{u(r)}) dr \nonumber \\
   = & \vartheta_m g(r) dr
          + \vartheta_m F( u(r - \rho), \mathcal{L}_{u(r - \rho)} ) dr
          + \sqrt{\varepsilon} 
                 \sum_{k \in \mathbb{N} } 
                     \left( \vartheta_m \sigma_k(u(r), \mathcal{L}_{u(r)}) 
                               + \vartheta_m h_k(r) 
                    \right) dW_k(r)  \nonumber \\
      & + \sqrt{\varepsilon}
                  \int_{y \in \mathbb{Y} }
                   \left( \vartheta_m \widetilde{\sigma}(u(r-), \mathcal{L}_{u(r)}, y) 
                            + \vartheta_m \widetilde{h}(r, y) 
                   \right) \widetilde{N}(dr,dy).
 \end{align}
 Applying It\^o's formula to \eqref{upte-1}, we get that for all $r \in [\tau-2\rho, \tau]$ and $t \ge 3\rho$,
 \begin{align}\label{upte-2}
        & e^{\beta r} \mathbb{E} \left[ \| \vartheta_m  u(r)\|^2 \right]
          + 2 \mathbb{E}
                  \left[ \int_{\tau-t}^{r} e^{\beta s}
                               \langle \vartheta_m A u(s), \vartheta_m u(s) \rangle ds
                  \right]  \nonumber\\
        & + 2 \mathbb{E}
                 \left[ \int_{\tau-t}^{r} e^{\beta s}
                              \langle \vartheta_m f(u(s), \mathcal{L}_{u(s)}),
                                      \vartheta_m u(s)
                              \rangle ds
                 \right]  \nonumber\\
      = & e^{\beta (\tau-t)} \mathbb{E} \left[ \|\vartheta_m \varphi(0)\|^2 \right]
          + (\beta - 2 \lambda) \mathbb{E}
                \left[ \int_{\tau-t}^{r} e^{\beta s} \|\vartheta_m  u(s)\|^2 ds \right]
          + 2 \mathbb{E}
                 \left[ \int_{\tau-t}^{r} e^{\beta s}
                              \langle \vartheta_m g(s), \vartheta_m u(s) \rangle ds
                 \right]  \nonumber\\
        & + 2 \mathbb{E}
                 \left[ \int_{\tau-t}^{r} e^{\beta s}
                              \langle \vartheta_m F(u(s-\rho), \mathcal{L}_{u(s-\rho)}),
                                      \vartheta_m u(s)
                              \rangle ds
                 \right]  \nonumber \\
        & + \varepsilon \mathbb{E}
               \left[ \sum_{k \in \mathbb{N} } \int_{\tau-t}^{r} e^{\beta s}
                            \|\vartheta_m \sigma_k(u(s), \mathcal{L}_{u(s)})
                              + \vartheta_m h_k(s)\|^2 ds
               \right]  \nonumber \\
        & + \varepsilon \mathbb{E}
               \left[ \int_{\tau-t}^{r} \int_{y \in \mathbb{Y} } e^{\beta s}
                            \|\vartheta_m \widetilde{\sigma} (u(s), \mathcal{L}_{u(s)},y)
                                        + \vartheta_m \widetilde{h}(s, y)\|^2
                       \nu(dy) ds
               \right].
 \end{align}

 For the second term on the left-hand side of \eqref{upte-2},
 by the Young inequality, we have
 \begin{align}\label{upte-3}
        & 2 \mathbb{E}
               \left[ \int_{\tau-t}^{r}  e^{\beta s}
                            \langle \vartheta_m A u(s), \vartheta_m u(s) \rangle ds
               \right]  \nonumber\\
      = & 2 \mathbb{E}
               \left[ \int_{\tau-t}^{r}  e^{\beta s}
                  \sum_{i \in \mathbb{Z}} |(Bu(s))_i|^{p-2} (Bu(s))_i
                        (B (\vartheta_m^2 u(s) ) )_i ds
               \right]  \nonumber\\
      = & 2 \mathbb{E}
               \left[ \int_{\tau-t}^{r}  e^{\beta s}
                  \sum_{i \in \mathbb{Z}}
                        |u_{i+1}(s)- u_i(s)|^{p-2}
                        (u_{i+1}(s)- u_i(s) )
                        ({\vartheta}^2(\frac{i+1}{m}) u_{i+1}(s) -{\vartheta}^2(\frac{i}{m}) u_{i}(s) ) ds
               \right]  \nonumber\\
   \ge & 2 \mathbb{E}
               \left[ \int_{\tau-t}^{r}  e^{\beta s}
                  \sum_{i \in \mathbb{Z}}
                        |u_{i+1}(s)- u_i(s)|^{p-2}
                        (u_{i+1}(s)- u_i(s) ) u_{i+1}(s)
                        ({\vartheta}^2(\frac{i+1}{m}) -{\vartheta}^2(\frac{i}{m}) ) ds
               \right]  \nonumber\\
   \ge & - 4 \mathbb{E}
               \left[ \int_{\tau-t}^{r}  e^{\beta s}
                  \sum_{i \in \mathbb{Z}}
                        |u_{i+1}(s)- u_i(s)|^{p-1}
                        |u_{i+1}(s)|
                        | \vartheta(\frac{i+1}{m}) - \vartheta(\frac{i}{m}) | ds
               \right]  \nonumber\\
   \ge & -\frac{2^{p+1} c_1}{m}
              \int_{\tau-t}^{r} e^{\beta s} \mathbb{E} \left[ \|u(s)\|_p^p \right] ds,
 \end{align}
 where $c_1>0$ is
 an upper  bound of the derivative of $\theta(\cdot)$ and is independent of $m$.

 For the third term on the left-hand side of \eqref{upte-2},
 by \eqref{f1}, we have
 \begin{align}\label{upte-4}
        & 2\mathbb{E}
               \left[ \int_{\tau-t}^{r} e^{\beta s}
                           \langle \vartheta_m f(u(s), \mathcal{L}_{u(s)}), \vartheta_m u(s) \rangle ds
               \right]  \nonumber\\
      = & 2\mathbb{E}
               \left[ \int_{\tau-t}^{r} e^{\beta s}
                            \langle{\vartheta}^2_m f(u(s), \mathcal{L}_{u(s)}),  u(s) \rangle ds
               \right]  \nonumber\\
      = & 2\mathbb{E}
               \left[ \int_{\tau-t}^{r} e^{\beta s}
                         \sum_{i \in \mathbb{Z}}
                              {\vartheta}^2(\frac{i}{m}) f_i(u_i(s), \mathcal{L}_{u(s)}) u_i(s) ds
               \right]  \nonumber\\
   \ge & 2\mathbb{E}
               \left[ \int_{\tau-t}^{r} e^{\beta s}
                            \sum_{i \in \mathbb{Z}}
                                 {\vartheta}^2(\frac{i}{m})
                                  \Big( \lambda_1 |u_i(s)|^p - \phi_{1,i} \left( 1 + |u_i(s)|^2 + \mathcal{L}_{u(s)}(\|\cdot\|^2) \right)
                                  \Big) ds
               \right]  \nonumber\\
   \ge & 2\lambda_1 \mathbb{E}
              \left[ \int_{\tau-t}^{r} e^{\beta s}
                        \sum_{i \in \mathbb{Z}}
                             {\vartheta}^2(\frac{i}{m}) |u_i(s)|^p ds
              \right]
          - \frac{2}{\beta} (e^{\beta r} - e^{\beta(\tau-t)}) \sum_{|i| \ge m} |\phi_{1,i}|  \nonumber\\
        & - 2 \| \phi_1 \|_{\ell^\infty}
              \mathbb{E}
                 \left[ \int_{\tau-t}^{r}  e^{\beta s} \|\vartheta_m u(s)\|^2 ds
                 \right]
          - 2 \sum_{|i| \ge m} |\phi_{1,i}|
                   \mathbb{E}
                      \left[ \int_{\tau-t}^{r} e^{\beta s} \|u(s)\|^2 ds
                      \right].
 \end{align}

 For the third term on the right-hand side of \eqref{upte-2}, by Young's inequality, we get
 \begin{align}\label{upte-5}
        2 \mathbb{E}
             \left[ \int_{\tau-t}^{r} e^{\beta s}
                        \langle \vartheta_m g(s), \vartheta_m u(s) \rangle ds
             \right]
    \le \beta \int_{\tau-t}^{r}  e^{\beta s}
                  \mathbb{E} \left[ \|\vartheta_m u(s)\|^2 \right] ds
          + \frac{1}{\beta} \int_{\tau-t}^{r} e^{\beta s} \sum_{|i| \ge m} |g_i(s)|^2 ds.
 \end{align}

 For the fourth term on the right-hand side of \eqref{upte-2},
 by \eqref{F1}, we obtain that for $t \ge 3\rho$,
 \begin{align}\label{upte-6}
        & 2 \mathbb{E}
                 \left[ \int_{\tau-t}^{r} e^{\beta s}
                              \langle \vartheta_m F(u(s-\rho), \mathcal{L}_{u(s-\rho)}),
                                      \vartheta_m u(s)
                              \rangle ds
                 \right]  \nonumber\\
   \le & 2 \|L_{F}\|_{\ell^\infty} e^{\frac{\beta}{2} \rho}
          \mathbb{E}
            \left[ \int_{\tau-t}^{r} e^{\beta s} \|\vartheta_m u(s)\|^2 ds
            \right]  \nonumber\\
        & + \frac{1}{ 2 \|L_{F}\|_{\ell^\infty} } e^{-\frac{\beta \rho}{2} }
            \mathbb{E}
               \left[ \int_{\tau-t}^{r} e^{\beta s}
                           \| \vartheta_m F(u(s-\rho),\mathcal{L}_{u(s-\rho)})\|^2 ds
               \right]  \nonumber\\
   \le & 2 \|L_{F}\|_{\ell^\infty} e^{\frac{\beta}{2} \rho}
             \mathbb{E}
                    \left[ \int_{\tau-t}^{r} e^{\beta s} \|\vartheta_m u(s)\|^2 ds
                    \right]
             + \frac{1}{\beta \|L_{F}\|_{\ell^\infty} } 
                     e^{-\frac{\beta \rho}{2} } 
                          \|\vartheta_m F(0,\delta_0)\|^2
                               (e^{\beta r} - e^{\beta(\tau-t)})  \nonumber\\
        & + 2 \|L_{F}\|_{\ell^\infty} e^{-\frac{\beta}{2} \rho}
            \mathbb{E}
               \left[ \int_{\tau-t}^{r} e^{\beta s} \|\vartheta_m u(s-\rho)\|^2 ds 
               \right]  \nonumber\\
        & + \frac{2}{\|L_{F}\|_{\ell^\infty} } 
                   e^{- \frac{\beta}{2} \rho} \|\vartheta_m L_{F}\|^2
                           \int_{\tau-t}^{r} e^{\beta s}
                               \mathbb{E} \left[ \|u(s-\rho)\|^2 \right] ds  \nonumber\\
   \le & 4 \|L_{F}\|_{\ell^\infty} e^{\frac{\beta}{2} \rho}
               \mathbb{E}
                      \left[ \int_{\tau-t}^{r} e^{\beta s} \|\vartheta_m u(s)\|^2 ds
                      \right]
           + \frac{1}{\beta \|L_{F}\|_{\ell^\infty} } 
                  e^{- \frac{\beta \rho}{2} } 
                           \|\vartheta_m F(0,\delta_0)\|^2
                                 (e^{\beta r} - e^{\beta(\tau-t)})  \nonumber\\
        & + 2 \|L_{F}\|_{\ell^\infty} e^{\beta(\frac{\rho}{2}+\tau-t) }
            \mathbb{E}
               \left[ \int_{-\rho}^{0} e^{\beta s} \|\vartheta_m \varphi(s)\|^2 ds \right]  \nonumber\\
        & + \frac{2}{\|L_{F}\|_{\ell^\infty} } 
                   e^{\frac{\beta \rho}{2} } 
                         \|\vartheta_m L_{F}\|^2
                             \int_{\tau-t}^{r} e^{\beta s}
                                  \mathbb{E} \left[ \|u(s)\|^2 \right] ds  \nonumber\\            
        & + \frac{2}{ \beta \|L_{F}\|_{\ell^\infty} } 
                     e^{\beta (\frac{\rho}{2} + \tau -t)} 
                            \|\vartheta_m L_{F}\|^2
                                  \mathbb{E}
                                         \left[ \sup_{r \in [-\rho, 0]}\|\varphi(r)\|^2 
                                         \right].
 \end{align}

 For the last two terms on the right-hand side of \eqref{upte-2},
 by \eqref{sigma3} with $\tilde{\theta} = 2$, we obtain
 \begin{align}\label{upte-6}
        & \varepsilon \mathbb{E}
               \left[ \sum_{k \in \mathbb{N} } \int_{\tau-t}^{r} e^{\beta s}
                            \|\vartheta_m \sigma_k(u(s), \mathcal{L}_{u(s)})
                              + \vartheta_m h_k(s)\|^2 ds
               \right]  \nonumber \\
        & + \varepsilon \mathbb{E}
               \left[ \int_{\tau-t}^{r} \int_{y \in \mathbb{Y} } e^{\beta s}
                            \|\vartheta_m \widetilde{\sigma}(u(s), \mathcal{L}_{u(s)},y)
                                        + \vartheta_m \widetilde{h}(s, y)\|^2
                       \nu(dy) ds
               \right]  \nonumber \\
  \le & 2 \sum_{|i|\ge m}
                 \int_{\tau-t}^{r} e^{\beta s}
                     \left( \sum_{k \in \mathbb{N} } |h_{k,i}(s)|^2 
                               + \int_{y \in \mathbb{Y} } |\widetilde{h}_{i}(s,y)|^2 \nu(dy)
                     \right) ds
         + 2 \sum_{|i| \ge m} 
                      \left( \sum_{k \in \mathbb{N} } \alpha_{k,i}^2
                                + \beta_i^2
                      \right)
                      \int_{\tau-t}^{r} e^{\beta s} ds   \nonumber\\
       & + \frac{\lambda_1}{4} \int_{\tau-t}^{r} e^{\beta s}
                                  \mathbb{E}
                                     \left[ \sum_{i \in \mathbb{Z}}
                                                 {\vartheta}^2(\frac{i}{m}) |u_i(s)|^p
                                     \right] ds  \nonumber\\
       & + 2^{\frac{3p+2q}{p-q} } \frac{p-q}{p}
                        \sum_{|i| \ge m} 
                              \Big( \sum_{k \in \mathbb{N} } L^2_{\sigma,k,i} 
                                       + \int_{y \in \mathbb{Y} } (L_{\tilde{\sigma},i}(y) )^2 \nu(dy)
                              \Big)^{ \frac{p}{p-q} }
                              \Big( \frac{\lambda_1 p}{q} \Big)^{-\frac{q}{p-q} }
                        \int_{\tau-t}^{r} e^{\beta s} ds  \nonumber\\
       & + 8 \sum_{|i| \ge m} 
                       \left( \sum_{k \in \mathbb{N} } L_{\sigma, k,i}^2 
                                 + \int_{y \in \mathbb{Y} } (L_{\tilde{\sigma},i}(y) )^2 \nu(dy)
                       \right)
               \int_{\tau-t}^{r} e^{\beta s} \mathbb{E} \left[ \|u(s)\|^2 \right] ds  \nonumber\\                        
  \le & 2 \sum_{|i|\ge m}
                   \int_{\tau-t}^{r} e^{\beta s}
                       \left( \sum_{k \in \mathbb{N} } |h_{k,i}(s)|^2 
                                 + \int_{y \in \mathbb{Y} } |\widetilde{h}_{i}(s,y)|^2 \nu(dy)
                       \right) ds  
            + \frac{\lambda_1}{4} \int_{\tau-t}^{r} e^{\beta s}
                                  \mathbb{E}
                                     \left[ \sum_{i \in \mathbb{Z}}
                                                 {\vartheta}^2(\frac{i}{m}) |u_i(s)|^p
                                     \right] ds  \nonumber\\
       & + C_{p,q} \sum_{|i| \ge m} 
                                 \left[ \sum_{k \in \mathbb{N} } \alpha_{k,i}^2
                                           + \beta_i^2
                                           + \Big( \sum_{k \in \mathbb{N} } L^2_{\sigma,k,i} 
                                                        + \int_{y \in \mathbb{Y} } (L_{\tilde{\sigma},i}(y) )^2 \nu(dy)
                                              \Big)^{ \frac{p}{p-q} }
                                 \right]
              ( e^{\beta r} - e^{\beta(\tau-t)} )  \nonumber\\
       & + 8 \sum_{|i| \ge m} 
                                \left( \sum_{k \in \mathbb{N} } L_{\sigma, k,i}^2 
                                          + \int_{y \in \mathbb{Y} } (L_{\tilde{\sigma},i}(y) )^2 \nu(dy)
                                \right)
                       \int_{\tau-t}^{r} e^{\beta s} \mathbb{E} \left[ \|u(s)\|^2 \right] ds,
 \end{align}
 where $C_{p,q} = \frac{1}{\beta}
                  \Big(2 + 2^{\frac{3p+2q}{p-q} } \frac{p-q}{p} 
                                      \Big( \frac{\lambda_1 p}{q} \Big)^{-\frac{q}{p-q} }
                  \Big)$.

 From \eqref{upte-2}-\eqref{upte-6}, it follows that for all $r \in [\tau-2\rho, \tau]$ and $t \ge 3\rho$,
 \begin{align}\label{upte-7}
        & \mathbb{E} \left[ \| \vartheta_m  u(r)\|^2 \right]
           + \frac{7}{4} \lambda_1 \mathbb{E}
                   \left[ \int_{\tau-t}^{r} e^{\beta (s-r)}
                                \sum_{i \in \mathbb{Z}}
                                      \vartheta^2(\frac{i}{m}) |u_i(s)|^p ds
                   \right]  \nonumber\\
   \le & e^{\beta (\tau-t-r)} \mathbb{E} \left[ \| \vartheta_m \varphi(0) \|^2 \right]
            + 2 \|L_{F}\|_{\ell^\infty} e^{\beta(\frac{\rho}{2}+\tau-t) }
                      \mathbb{E}
                         \left[ \int_{-\rho}^{0} e^{\beta (s-r)} \|\vartheta_m \varphi(s)\|^2 ds \right]  \nonumber\\
        & + \frac{2^{p+1} c_1}{m}
              \int_{\tau-t}^{r} e^{\beta (s-r)} \mathbb{E} \left[ \|u(s)\|_p^p \right] ds  \nonumber\\
        & + \left( 2\beta - 2\lambda 
                                    + 2\|\phi_1\|_{\ell^\infty} 
                                    + 4 \|L_{F}\|_{\ell^\infty} e^{\frac{\beta}{2} \rho} 
               \right)
               \int_{\tau-t}^{r} e^{\beta (s-r)} \mathbb{E}
                    \left[ \|\vartheta_m u(s)\|^2 \right] ds  \nonumber\\
        & + \frac{2}{\beta} (1 - e^{\beta(\tau-t-r)}) \sum_{|i| \ge m} |\phi_{1,i}|
            + \frac{1}{\beta \|L_{F}\|_{\ell^\infty} } 
                           e^{- \frac{\beta \rho}{2} } 
                                    \|\vartheta_m F(0,\delta_0)\|^2
                                          (1 - e^{\beta(\tau-t-r)}) \nonumber\\
        & + C_{p,q} \sum_{|i| \ge m}
               \left[ \sum_{k \in \mathbb{N} } \alpha_{k,i}^2
                        + \beta_i^2
                        + \Big( \sum_{k \in \mathbb{N} } L^2_{\sigma,k,i} 
                                     + \int_{y \in \mathbb{Y} } (L_{\tilde{\sigma},i}(y) )^2 \nu(dy)
                           \Big)^{ \frac{p}{p-q} }
               \right]
              (1 - e^{\beta(\tau-t-r)} )  \nonumber\\
        & + \frac{1}{\beta} \int_{\tau-t}^{r} e^{\beta (s-r)} \sum_{|i| \ge m} |g_i(s)|^2 ds
            + 2 \sum_{|i|\ge m}
                 \int_{\tau-t}^{r} e^{\beta (s-r)} 
                          \left( \sum_{k \in \mathbb{N} } |h_{k,i}(s)|^2 
                                    + \int_{y \in \mathbb{Y} } |\widetilde{h}_{i}(s,y)|^2 \nu(dy)
                          \right) ds   \nonumber\\
        & + \sum_{|i| \ge m} 
                     \bigg( 2 |\phi_{1,i}|
                                + \frac{2}{\|L_{F}\|_{\ell^\infty} } 
                                                 e^{\frac{\beta \rho}{2} } 
                                                        L_{F,i}^2
                                + 8\sum_{k \in \mathbb{N} } L_{\sigma, k,i}^2 
                                + 8 \int_{y \in \mathbb{Y} } (L_{\tilde{\sigma},i}(y) )^2 \nu(dy)
                     \bigg)
                  \mathbb{E} \left[ \int_{\tau-t}^{r}  e^{\beta (s-r)} \| u(s)\|^2 ds \right]  \nonumber\\
        & + \frac{2}{ \beta \|L_{F}\|_{\ell^\infty} } 
                   e^{\beta (\frac{5\rho}{2} - t)} 
                          \|\vartheta_m L_{F}\|^2
                                 \mathbb{E} \left[ \sup_{r \in [-\rho, 0]}\|\varphi(r)\|^2 \right].
 \end{align}

 By \eqref{31_1-1}, Lemma \ref{pue} with $\theta =2$ and \eqref{upte-7},
 we obtain that there exists $T_1 = T_1(\tau, D) \ge 3\rho$ such that for all $t \ge T_1$ and $r \in [\tau-2\rho, \tau]$,
 \begin{align}\label{upte-8}
        & \mathbb{E} \left[ \| \vartheta_m  u(r)\|^2 \right]
          + \frac{7}{4} \lambda_1 \mathbb{E}
                          \left[ \int_{\tau-t}^{r} e^{\beta (s-r)}
                                    \sum_{i \in \mathbb{Z}}
                                         {\vartheta}^2(\frac{i}{m}) |u_i(s)|^p ds
                          \right]  \nonumber\\
   \le & \left( 1 + \frac{2}{\beta} \|L_{F}\|_{\ell^\infty} e^{\frac{\beta \rho}{2} }
                        + \frac{2}{ \beta \|L_{F}\|_{\ell^\infty} } e^{\frac{\beta \rho}{2} } \|L_{F}\|^2
            \right)
            e^{\beta (2\rho-t)}
                    \mathbb{E} \left[ \sup_{r \in [-\rho, 0]}\|\varphi(r)\|^2 \right]  \nonumber\\
        & + \frac{2}{\beta} \sum_{|i| \ge m} |\phi_{1,i}|
            + \frac{1}{\beta \|L_{F}\|_{\ell^\infty} } \sum_{|i| \ge m} |F_i(0,\delta_0)|^2
                   \nonumber\\
        & + C_{p,q} \sum_{|i| \ge m}
                \left[ \sum_{k \in \mathbb{N} } \alpha_{k,i}^2
                                   + \beta_i^2
                                   + \Big( \sum_{k \in \mathbb{N} } L^2_{\sigma,k,i} 
                                                + \int_{y \in \mathbb{Y} } (L_{\tilde{\sigma},i}(y) )^2 \nu(dy)
                                      \Big)^{ \frac{p}{p-q} }
                \right]  \nonumber\\
        & + e^{\beta \rho} \int_{-\infty}^{\tau} e^{\beta (s-\tau)}
               \sum_{|i| \ge m}
                     \left( \frac{1}{\beta} |g_i(s)|^2 ds
                               + 2 \sum_{k \in \mathbb{N} } |h_{k,i}(s)|^2 
                               + 2 \int_{y \in \mathbb{Y} } |\widetilde{h}_{i}(s,y)|^2 \nu(dy)
                     \right) ds  \nonumber\\
        & + \bigg[ \frac{2^{p+1} c_1}{m}
                          + \sum_{|i| \ge m} 
                                   \bigg( 2 |\phi_{1,i}|
                                              + \frac{2}{\|L_{F}\|_{\ell^\infty} } 
                                                                 e^{\frac{\beta \rho}{2} } 
                                                                        L_{F,i}^2
                                              + 8 \sum_{k \in \mathbb{N} } L_{\sigma, k,i}^2 
                                              + 8 \int_{y \in \mathbb{Y} } (L_{\tilde{\sigma},i}(y) )^2 \nu(dy)
                                   \bigg)
                 \bigg]  \nonumber\\
        &\quad
            \cdot \left( C_{\beta, 2, \rho} + C_{\beta, 2, \rho} \int_{-\infty}^{\tau}
                                                          e^{\beta (s-\tau) }
                                                          \left( \|g(s)\|^2 + \|h(s)\|^2 
                                                                    + \|\widetilde{h}(s)\|^2_{L^2(\mathbb{Y}, \nu; \ell^2)}
                                                         \right) ds
                       \right).
 \end{align}

 Since $\mathcal{L}_{\varphi} \in D(\tau-t)$,
 it follows from \eqref{upte-8} that for every $\delta > 0$,
 there exist constants $T = T(\tau, \delta, D) \ge T_1$ and $ m_0 = m_0(\tau, \delta) \ge 1$
 such that for all $t \ge T$ and $m \ge m_0$,
 \begin{align*}
    \sup_{r \in [\tau-2\rho, \tau]}
          \sum_{|i| \ge 2m} \mathbb{E}\left[ |u_i(r; \tau-t, \varphi)|^2 \right]
    \le \sup_{r \in [\tau-2\rho, \tau]}
             \mathbb{E} \left[ \|\vartheta_m u(r; \tau-t, \varphi)\|^2 \right]
    < \delta,
 \end{align*}
 as desired.
\end{proof}

 Based on Lemma \ref{upte}, we obtain the following pullback tail-estimates in $L^2(\Omega, D_\rho) \bigcap L^p(\Omega, L^p(\tau-\rho, \tau; \ell^p) )$.

\begin{lemma}\label{momentpte}
   Suppose that {\bf(H1)}-{\bf(H4)}, \eqref{thetaassum2}, \eqref{dissip1} and \eqref{dissip2} hold.
   Then for every 
   $\tau \in \mathbb{R}$, $\delta > 0$ and $D = \{ D(t): t \in \mathbb{R} \} \in \mathcal{D}$,
   there exist constants $T = T(\tau, \delta, D) \ge 3\rho$ and $ m_1 = m_1(\tau, \delta) \ge 1$
   such that for all $t \ge T$ and $m \ge m_1$,
   the solution $u$ of \eqref{lmvlds-2} satisfies
   \begin{align*}
      \mathbb{E}
         \Big[ \sup_{s \in [\tau-\rho, \tau]}
                   \sum_{|i| \ge m}
                      |u_i(s; \tau-t, \varphi)|^2
         \Big]
      + \mathbb{E}
           \Big[ \int_{\tau-\rho}^{\tau}
                    \sum_{|i| \ge m} |u_i(s; \tau-t, \varphi)|^p ds
           \Big]
      < \delta,
   \end{align*}
   where $\varphi \in L^\theta(\Omega, \mathcal{F}_{\tau-t};D_{\rho})$ with law $\mathcal{L}_{\varphi} \in D(\tau-t)$.
\end{lemma}

\begin{proof}
 For convenience, we denote by $u_i(r) = u_i(r; \tau-t, \varphi)$ for any $r \ge \tau -t$.
 It follows from \eqref{upte-1} and It\^o's formula that for $t \ge 
 3 \rho$
 and $r\in [\tau -\rho, \tau]$,
   \begin{align}\label{momentpte-1}
        &    \|\vartheta_m u(r)\|^2
                       + 2 \int_{\tau-\rho}^{r}
                              \langle f(u(s),\mathcal{L}_{u(s)} ),
                                     {\vartheta}^2_m u(s)
                              \rangle ds
                \nonumber\\ 
   \le &  \|\vartheta_m u(\tau-\rho)\|^2 
          + 2  \sup_{r \in [\tau-\rho,\tau]}
                     \left (   \int_{\tau-\rho}^{r}
                     \langle -\vartheta_m A u(s), \vartheta_m u(s)
                     \rangle ds
                     \right )
              \nonumber \\  
        & + 2 
                  \int_{\tau-\rho}^{\tau}
                              \left|
                                     \langle \vartheta_m g(s), \vartheta_m u(s)
                                     \rangle
                              \right| ds
         + 2  \int_{\tau-\rho}^{\tau}
                                  \left|
                                         \langle \vartheta_m F(u(s-\rho), \mathcal{L}_{u(s-\rho)}),
                                                 \vartheta_m u(s)
                                         \rangle
                                  \right| ds
                        \nonumber\\
         & + 2   \sup_{r \in [\tau-\rho,\tau]}
                                \Big| \int_{\tau-\rho}^{r} \sum_{k \in \mathbb{N} }
                                      \langle \vartheta_m \sigma_k(u(s), \mathcal{L}_{u(s)}) + \vartheta_m h_k(s),
                                              \vartheta_m u(s)
                                      \rangle dW_k(s)
                                \Big|
             \nonumber\\
        & + \int_{\tau-\rho}^{\tau}
                         \sum_{k \in \mathbb{N} }
                            \|\vartheta_m \sigma_k(u(s), \mathcal{L}_{u(s)})
                               + \vartheta_m h_k(s) \|^2 ds
             \nonumber\\
        & + 2  
        \sup_{r\in [\tau-\rho,\tau]}
                         \Big| \int_{\tau-\rho}^{r}
                                      \int_{y \in \mathbb{Y} }
                                        \langle \vartheta_m \widetilde{\sigma}(u(s-), \mathcal{L}_{u(s)}, y) + \vartheta_m \widetilde{h}(s, y) ),
                                                \vartheta_m u(s-)
                                        \rangle \widetilde{N}(ds,dy)
                         \Big|
              \nonumber\\
        & +   \int_{\tau-\rho}^{\tau}
                            \int_{y \in \mathbb{Y} }
                            \| \vartheta_m \widetilde{\sigma}(u(s-), \mathcal{L}_{u(s)}, y) 
                                + \vartheta_m \widetilde{h}(s, y) \|^2 N(ds,dy)  .
 \end{align}
By \eqref{f1}, we have  for all $r \in [\tau-\rho, \tau]$ and $t \ge 
3\rho$,
 \begin{align}\label{momentpte-2}
       2 \int_{\tau-\rho}^{r}
                  \langle f(u(s), \mathcal{L}_{u(s)}),
                         {\vartheta}^2_m u(s)
                  \rangle ds 
   \ge & 2 \int_{\tau-\rho}^{r}
                     \sum_{i \in \mathbb{Z}}
                          {\vartheta}^2(\frac{i}{m}) f_i(u_i(s), \mathcal{L}_{u(s)} ) u_i(s)
                 ds  \nonumber\\
  \ge & 2 \lambda_1
           \int_{\tau-\rho}^{r}
               \sum_{i \in \mathbb{Z} }
                    {\vartheta}^2(\frac{i}{m}) |u_i(s)|^p ds
           - 2 \|\phi_1\|_{\ell^\infty}
                       \int_{\tau-\rho}^{r} \|\vartheta_m u(s)\|^2 ds  \nonumber\\
       & - 2 \sum_{|i| \ge m} |\phi_{1,i}|
                      \int_{\tau-\rho}^{r} \mathbb{E} \left[ \|u(s)\|^2 \right] ds
          - 2 \rho \sum_{|i| \ge m} |\phi_{1,i}|.
 \end{align}
 By \eqref{momentpte-1} and \eqref{momentpte-2}, we 
 obtain that
 for all $r \in [\tau-\rho, \tau]$ and $t \ge 3\rho$,
 \begin{align}\label{momentpte-3}
        & \mathbb{E}
          \left[ \sup_{r \in [\tau-\rho,\tau]}
                 \Big( \|\vartheta_m u(r)\|^2
                       + 2 \lambda_1
                           \int_{\tau-\rho}^{r}
                              \sum_{i \in \mathbb{Z} }
                                  {\vartheta}^2(\frac{i}{m}) |u_i(s)|^p ds
                 \Big)
          \right]  \nonumber\\
   \le & \mathbb{E} \left[ \|\vartheta_m u(\tau-\rho)\|^2 \right]
          + 2 \|\phi_1\|_{\ell^\infty}
              \int_{\tau-\rho}^{\tau} \mathbb{E} \left[ \|\vartheta_m u(s)\|^2 \right] ds
          + 2 \sum_{|i| \ge m} |\phi_{1,i}|
              \int_{\tau-\rho}^{\tau} \mathbb{E} \left[ \|u(s)\|^2 \right] ds  \nonumber\\
        & + 2 \rho \sum_{|i| \ge m} |\phi_{1,i}|
          + 2 \mathbb{E}
              \left[ \sup_{r \in [\tau-\rho,\tau]}
                        \int_{\tau-\rho}^{r}
                     \langle -\vartheta_m A u(s), \vartheta_m u(s)
                     \rangle ds
              \right]  \nonumber \\
        & + 2 \mathbb{E}
                 \left[ \int_{\tau-\rho}^{\tau}
                              \left|
                                     \langle \vartheta_m g(s), \vartheta_m u(s)
                                     \rangle
                              \right| ds
                 \right]
          + 2 \mathbb{E}
                        \left[ \int_{\tau-\rho}^{\tau}
                                  \left|
                                         \langle \vartheta_m F(u(s-\rho), \mathcal{L}_{u(s-\rho)}),
                                                 \vartheta_m u(s)
                                         \rangle
                                  \right| ds
                        \right]  \nonumber\\
        & + 2 \mathbb{E}
                 \left[ \sup_{r \in [\tau-\rho,\tau]}
                                \Big| \int_{\tau-\rho}^{r} \sum_{k \in \mathbb{N} }
                                      \langle \vartheta_m \sigma_k(u(s),\mathcal{L}_{u(s)}) + \vartheta_m h_k(s),
                                              \vartheta_m u(s)
                                      \rangle dW_k(s)
                                \Big|
                 \right]  \nonumber\\
        & + \mathbb{E}
               \left[ \int_{\tau-\rho}^{\tau}
                         \sum_{k \in \mathbb{N} }
                            \|\vartheta_m \sigma_k(u(s),\mathcal{L}_{u(s)})
                               + \vartheta_m h_k(s) \|^2 ds
               \right]  \nonumber\\
        & + 2 \mathbb{E}
               \left[ \sup_{r\in [\tau-\rho,\tau]}
                         \Big| \int_{\tau-\rho}^{r}
                                      \int_{y \in \mathbb{Y} }
                                          \langle \vartheta_m 
                                                             \widetilde{\sigma}_k(u(s-), \mathcal{L}_{u(s)}, y)
                                                       + \vartheta_m \widetilde{h}(s, y),
                                                \vartheta_m u(s-)
                                        \rangle \widetilde{N}(ds,dy)
                         \Big|
              \right]  \nonumber\\
        & + \mathbb{E}
               \left[ \int_{\tau-\rho}^{\tau}
                            \int_{y \in \mathbb{Y} }
                                \| \vartheta_m \widetilde{\sigma}(u(s-), \mathcal{L}_{u(s)}, y) 
                                   + \vartheta_m \widetilde{h}(s, y) 
                                \|^2 N(ds,dy)
               \right].
 \end{align}

 For the first two terms on the right-hand side of \eqref{momentpte-3},
 by Lemma \ref{upte}, we find that for any $\delta > 0$,
 there exist constants $T_1 = T_1(\tau, \delta, D) \ge 3\rho$ and $m_0 = m_0(\tau, \delta) > 0$ such that for all $t \ge T_1$ and $m \ge m_0$,
 \begin{align}\label{momentpte-4}
       & \mathbb{E} \left[ \|\vartheta_m u(\tau-\rho)\|^2 \right]
                 + 2 \|\phi_1\|_{\ell^\infty}
                     \int_{\tau-\rho}^{\tau} \mathbb{E} \left[ \|\vartheta_m u(s)\|^2 \right] ds
                          \nonumber\\
  \le & \left( 1+ 2\rho \|\phi_1\|_{\ell^\infty} \right)
           \sup_{s \in[\tau - 2\rho, \tau]}
                \sum_{|i| \ge m}
                       \mathbb{E} \left[ |u_i(s)|^2 \right]
    < \left( 1+ 2\rho \|\phi_1\|_{\ell^\infty} \right) \delta.
 \end{align}

 We now consider the third term on the right-hand side of \eqref{momentpte-3}.
 It follows from Lemma \ref{pue} that there exists $T_2 = T_2(\tau, D) \ge 3\rho$ such that for all $t \ge T_2$,
 \begin{align}\label{momentpte-5}
    \mathbb{E}
       \left[ \sup_{r \in [\tau-2\rho, \tau]} \|u(r; \tau-t, \varphi)\|^2 \right]
    + \mathbb{E}
         \left[ \int_{\tau-2\rho}^{\tau} \|u(s; \tau-t, \varphi)\|^p_p ds \right]
   \le (1 + e^{2\rho\beta}) R(\tau),
 \end{align}
 where $R(\tau) = C_{\beta,2,\rho} + C_{\beta, 2, \rho} \int_{-\infty}^{\tau}
                               e^{\beta (s-\tau) }
                               \left( \|g(s)\|^2 + \|h(s)\|^2 
                                                        + \|\widetilde{h}(s)\|_{L^2(\mathbb{Y}, \nu; \ell^2)}^2
                               \right) ds$,
from which we get
 \begin{align}\label{momentpte-6}
    2 \sum_{|i| \ge m} |\phi_{1,i}|
         \int_{\tau-\rho}^{\tau} \mathbb{E} \left[ \|u(s)\|^2 \right] ds
    \le 2 \rho (1 + e^{2\rho\beta}) R(\tau) \sum_{|i| \ge m} |\phi_{1,i}|.
 \end{align}

 For the fifth term on the right-hand side of \eqref{momentpte-3},
 by the  argument 
 of  \eqref{upte-3}, 
 we obtain from   \eqref{momentpte-5} 
  that for all $t \ge T_2$,
 \begin{align}\label{momentpte-7}
       2 \mathbb{E}
                 \left[ \sup_{r \in [\tau-\rho,\tau]}
                               \int_{\tau-\rho}^{r}
                                   \langle -\vartheta_m A u(s), \vartheta_m u(s)
                                   \rangle ds
                 \right] 
   \le & \frac{2^{p+1} c_1}{m}
                \mathbb{E} \left[ \int_{\tau-\rho}^{\tau} \|u(s)\|_p^p ds \right]  \nonumber\\
   \le & \frac{2^{p+1} c_1}{m} (1 + e^{2\rho\beta}) R(\tau).
 \end{align}

 For the sixth term on the right-hand side of \eqref{momentpte-3},
 by \eqref{momentpte-4}, we obtain that for all $t \ge T_1$ and $m \ge m_0$,
 \begin{align}\label{momentpte-8}
          2 \mathbb{E}
                 \left[ \int_{\tau-\rho}^{\tau}
                              \left|
                                     \langle \vartheta_m g(s), \vartheta_m u(s)
                                     \rangle
                              \right| ds
                 \right]
   \le & \int_{\tau-\rho}^{\tau}
             \sum_{|i| \ge m}
                \mathbb{E} \left[ |u_i(s)|^2 \right] ds
          + \int_{\tau-\rho}^{\tau}
               \sum_{|i| \ge m} |g_i(s)|^2 ds  \nonumber\\
   \le & \rho \delta
          + \int_{\tau-\rho}^{\tau} \sum_{|i|\ge m} |g_i(s)|^2 ds.
 \end{align}

 For the seventh term on the right-hand side of \eqref{momentpte-3},
 it follows from \eqref{F1}, \eqref{momentpte-4} and \eqref{momentpte-5}
 that for all $t \ge T_1 \vee T_2$ and $m \ge m_0$,
 \begin{align}\label{momentpte-9}
       & 2 \mathbb{E}
               \left[ \int_{\tau-\rho}^{\tau}
                         \left|
                               \langle \vartheta_m F(u(s-\rho), \mathcal{L}_{u(s-\rho)}),
                                       \vartheta_m u(s)
                               \rangle
                         \right| ds
               \right]  \nonumber\\
  \le & \mathbb{E}
                  \left[ \int_{\tau -\rho}^{\tau} \|\vartheta_m u(s)\|^2 ds \right]
           + \mathbb{E}
                      \left[ \int_{\tau -\rho}^{\tau} \vartheta_m^2
                                    \|F(u(s-\rho),\mathcal{L}_{u(s-\rho)})\|^2 ds 
                      \right]  \nonumber\\
  \le & \mathbb{E}
            \left[ \int_{\tau -\rho}^{\tau} \|\vartheta_m u(s)\|^2 ds \right]
           + 2 \rho \sum_{|i| \ge m} | F_i(0,\delta_0)|^2  \nonumber\\
       & + 4 \|L_F\|_{\ell^\infty}^2 \mathbb{E}
             \left[ \int_{\tau -\rho}^{\tau} \|\vartheta_m u(s-\rho)\|^2 ds \right]
           + 4 \|\vartheta_mL_F\|^2 \int_{\tau -\rho}^{\tau}
                  \mathbb{E} \left[ \|u(s-\rho)\|^2 \right] ds \nonumber\\
  \le & \rho \left( 1 + 4 \|L_F\|_{\ell^\infty}^2 \right) \delta
           + 2 \rho \sum_{|i| \ge m} | F_i(0,\delta_0)|^2
           + 4 \rho (1 + e^{2\rho\beta}) R(\tau) \sum_{|i| \ge m} |L_{F,i}|^2.
 \end{align}

 For the eighth and ninth terms on the right-hand side of \eqref{momentpte-3},
 by BDG's inequality, Young's inequality, \eqref{sigma3} and \eqref{momentpte-5},
 we obtain that for all $t \ge T_2$,
 \begin{align}\label{momentpte-10}
        & 2 \mathbb{E}
               \left[ \sup_{r \in [\tau-\rho,\tau]}
                         \Big| \int_{\tau-\rho}^{r} \sum_{k \in \mathbb{N} }
                                     \langle \vartheta_m \sigma_k(u(s),\mathcal{L}_{u(s)}) + \vartheta_m h_k(s),
                                             \vartheta_m u(s)
                                     \rangle dW_k(s)
                         \Big|
                 \right]  \nonumber\\
        & + \mathbb{E}
               \left[ \int_{\tau-\rho}^{\tau}
                         \sum_{k \in \mathbb{N} }
                            \|\vartheta_m \sigma_k(u(s),\mathcal{L}_{u(s)})
                               + \vartheta_m h_k(s) \|^2 ds
               \right]  \nonumber\\
  \le & \frac{1}{8} \mathbb{E}
            \left[ \sup_{s \in [\tau-\rho,\tau]} \|\vartheta_m u(s)\|^2 \right]
         + (1 + 8c_1^2)
           \mathbb{E}
              \left[ \int_{\tau-\rho}^{\tau}
                        \sum_{k \in \mathbb{N} }
                           \| \vartheta_m \sigma_k(u(s),\mathcal{L}_{u(s)})
                               + \vartheta_m h_k(s) \|^2 ds
              \right]  \nonumber\\
  \le & \frac{1}{8} \mathbb{E}
            \left[ \sup_{s \in [\tau-\rho,\tau]} \|\vartheta_m u(s)\|^2 \right]
         + 2(1 + 8c_1^2) \int_{\tau-\rho}^{\tau}
                            \sum_{k \in \mathbb{N} }
                               \sum_{|i| \ge m} |h_{k,i}(s)|^2 ds  \nonumber\\
       & + \frac{\lambda_1}{4} \int_{\tau-\rho}^{\tau}
                                  \mathbb{E}
                                     \left[ \sum_{i \in \mathbb{Z}}
                                                 {\vartheta}^2(\frac{i}{m}) |u_i(s)|^p
                                     \right] ds
         + 2(1 + 8c_1^2) \rho \sum_{k \in \mathbb{N} } \sum_{|i| \ge m} \alpha_{k,i}^2  \nonumber\\
       & + 8(1 + 8c_1^2) \rho (1 + e^{2\rho\beta}) R(\tau)
           \sum_{k \in \mathbb{N} }
              \sum_{|i| \ge m} L_{\sigma,k,i}^2  \nonumber\\
       & + \rho \frac{p-q}{p} 
              \Big[ 8(1 + 8c_1^2) \Big]^{ \frac{p}{p-q} }
              \Big( \frac{\lambda_1 p}{4q} \Big)^{-\frac{q}{p-q} }
                   \sum_{|i| \ge m} 
                         \Big( \sum_{k \in \mathbb{N} } L^2_{\sigma,k,i} 
                         \Big)^{ \frac{p}{p-q} }.
 \end{align}

 For the last two terms on the right-hand side of \eqref{momentpte-3},
 similar to the argument of \eqref{momentpte-10}, 
 by \eqref{sigma2} and \eqref{siglinear-2}, 
 we can get that for all $t \ge T_2$,
 \begin{align}\label{momentpte-11}
       & 2 \mathbb{E}
               \left[ \sup_{r\in [\tau-\rho,\tau]}
                         \Big| \int_{\tau-\rho}^{r}
                                      \int_{y \in \mathbb{Y} }
                                        \langle \vartheta_m \widetilde{\sigma}(u(s-), \mathcal{L}_{u(s)}, y) 
                                                    + \vartheta_m \widetilde{h}(s, y),
                                                    \vartheta_m u(s-)
                                        \rangle \widetilde{N}(ds,dy)
                         \Big|
              \right]  \nonumber\\
       & + \mathbb{E}
               \left[ \int_{\tau-\rho}^{\tau}
                             \int_{y \in \mathbb{Y} }
                            \| \vartheta_m \widetilde{\sigma}(u(s-), \mathcal{L}_{u(s)}, y) 
                               + \vartheta_m \widetilde{h}(s, y) \|^2 N(ds,dy)
               \right] \nonumber\\
  \le &  \frac{1}{8} \mathbb{E}
              \left[ \sup_{s \in [\tau-\rho,\tau]} \|\vartheta_m u(s)\|^2 \right]
                   \nonumber\\
       & + (1 + 8c_2^2) \mathbb{E}
                           \left[ \int_{\tau-\rho}^{\tau}
                                         \int_{y \in \mathbb{Y} }
                                        \| \vartheta_m \widetilde{\sigma}(u(s-), \mathcal{L}_{u(s)}, y) 
                                           + \vartheta_m \widetilde{h}(s, y) \|^2 N(ds,dy)
                           \right] \nonumber\\                
   \le & \frac{1}{8} \mathbb{E}
                 \left[ \sup_{s \in [\tau-\rho,\tau]} \|\vartheta_m u(s)\|^2 \right]
           + 2 (1 + 8c_2^2) 
               \int_{\tau-\rho}^{\tau}
                   \int_{y \in \mathbb{Y} }
                           \| \vartheta_m \widetilde{h}(s, y) \|^2  \nu(dy) ds  \nonumber\\
        &  + 2 (1 + 8c_2^2) 
                        \int_{\tau-\rho}^{\tau}
                            \int_{y \in \mathbb{Y} }
                                \|\vartheta_m \widetilde{\sigma}(u(s), \mathcal{L}_{u(s)}, y) \| ^2 \nu(dy)  ds \nonumber\\
  \le & \frac{1}{8} \mathbb{E}
                   \left[ \sup_{s \in [\tau-\rho,\tau]} \|\vartheta_m u(s)\|^2 \right]
           + 2 (1 + 8c_2^2) 
                \int_{\tau-\rho}^{\tau}
                     \int_{y \in \mathbb{Y} }
                             \| \vartheta_m \widetilde{h}(s, y) \|^2  \nu(dy) ds  \nonumber\\
       & + 2 (1 + 8c_2^2) \rho \| \vartheta_m \beta \|^2  \nonumber\\                      
       & + 8 (1 + 8c_2^2) 
                                \int_{\tau-\rho}^{\tau}   
                                    \sum_{i \in \mathbb{Z}} \vartheta^2(\frac{i}{m})
                                          \int_{y\in \mathbb{Y} }
                                                           \left( L_{ \widetilde{\sigma}, i}(y)
                                                           \right)^2 \nu(dy)   
                                   \left(  |u_i(s)|^q  
                                             + \mathbb{E}\left[ \|u(s)\|^2 \right]
                                   \right)
                                ds
                \nonumber\\ 
  \le & \frac{1}{8} \mathbb{E}
            \left[ \sup_{s \in [\tau-\rho,\tau]} \|\vartheta_m u(s)\|^2 \right]
            + 2 (1 + 8c_2^2) 
                \int_{\tau-\rho}^{\tau}
                    \sum_{|i| \ge m}
                               \int_{y \in \mathbb{Y} }
                                       | \widetilde{h}_i(s, y) |^2  \nu(dy) ds  \nonumber\\
       & + \frac{\lambda_1}{4} \int_{\tau-\rho}^{\tau}
                                  \mathbb{E}
                                     \left[ \sum_{i \in \mathbb{Z}}
                                                 {\vartheta}^2(\frac{i}{m}) |u_i(s)|^p
                                     \right] ds
           + 2(1 + 8c_2^2) \rho \sum_{|i| \ge m} \beta_i^2  \nonumber\\
       & + 8(1 + 8c_2^2) \rho (1 + e^{2\rho\beta}) R(\tau)
               \int_{y \in \mathbb{Y} } 
                    \sum_{|i| \ge m} 
                            \left( L_{\tilde{\sigma}, i}(y) \right)^2 \nu(dy)  \nonumber\\
       & + \rho \frac{p-q}{p} 
              \left[ 8(1 + 8c_2^2) \right]^{\frac{p}{p-q} }
                 \Big( \frac{\lambda_1 p}{4q} \Big)^{-\frac{q}{p-q} }
                  \sum_{|i| \ge m} 
                        \Big( \int_{y \in \mathbb{Y} }
                                     \left( L_{\tilde{\sigma}, i}(y) \right)^2 
                                           \nu(dy)
                        \Big)^{ \frac{p}{p-q} }.
 \end{align}

 Then From \eqref{momentpte-3}, \eqref{momentpte-4} and \eqref{momentpte-6}-\eqref{momentpte-11},
 it follows that for all $ t \ge T_1 \vee T_2 $ and $ m \ge m_0$,
  \begin{align}\label{momentpte-12}
        & \frac{1}{2} \mathbb{E}
             \Big[ \sup_{r \in [\tau-\rho,\tau]} \|\vartheta_m u(r)\|^2
             \Big]
          + \lambda_1
            \mathbb{E}
               \Big[ \int_{\tau-\rho}^{\tau}
                         \sum_{i \in \mathbb{Z} }
                               {\vartheta}^2(\frac{i}{m}) |u_i(s)|^p ds
               \Big]  \nonumber\\
   \le & 2 \left( 1+ 2\rho \|\phi_1\|_{\ell^\infty} \right) \delta
            + 4 \rho (1 + e^{2\rho\beta}) R(\tau) \sum_{|i| \ge m} |\phi_{1,i}|
            + 4 \rho \sum_{|i| \ge m} |\phi_{1,i}|  \nonumber\\
        & + \frac{2^{p+2} c_1}{m} (1 + e^{2\rho\beta}) R(\tau)
            + 2 \rho \delta
            + 2 \int_{\tau-\rho}^{\tau} \sum_{|i|\ge m} |g_i(s)|^2 ds  \nonumber\\
        & + 2 \rho \left( 1 + 4 \|L_F\|_{\ell^\infty}^2 \right) \delta
            + 4 \rho \sum_{|i| \ge m} | F_i(0,\delta_0)|^2
            + 8 \rho (1 + e^{2\rho\beta}) R(\tau) \sum_{|i| \ge m} |L_{F,i}|^2  \nonumber\\               
        & + 4(1 + 8c_1^2 + 8c_2^2) \int_{\tau-\rho}^{\tau}
                   \sum_{|i| \ge m} 
                         \left( |h_{k,i}(s)|^2 
                                   + \int_{y \in \mathbb{Y} } |\widetilde{h}_i(s,y)|^2 \nu(dy) 
                         \right) ds  \nonumber\\
        & + 4(1 + 8c_1^2 + 8c_2^2) \rho  
               \sum_{|i| \ge m} 
                     \Big( \sum_{k \in \mathbb{N} } \alpha_{k,i}^2 
                               + \beta_i^2
                     \Big)
            + 16(1 + 8c_1^2) \rho (1 + e^{2\rho\beta}) R(\tau)
                       \sum_{k \in \mathbb{N} }
                             \sum_{|i| \ge m} L_{\sigma,k,i}^2   \nonumber\\
        & + 16(1 + 8c_2^2) \rho (1 + e^{2\rho\beta}) R(\tau)
                       \int_{y \in \mathbb{Y} } 
                            \sum_{|i| \ge m} 
                                    \left( L_{\tilde{\sigma}, i}(y) \right)^2 \nu(dy)  \nonumber\\
        & + 2\rho \frac{p-q}{p} 
                      \Big[ 8(1 + 8c_1^2) \Big]^{ \frac{p}{p-q} }
                      \Big( \frac{\lambda_1 p}{4q} \Big)^{-\frac{q}{p-q} }
                           \sum_{|i| \ge m} 
                                 \Big( \sum_{k \in \mathbb{N} } L^2_{\sigma,k,i} 
                                 \Big)^{ \frac{p}{p-q} } \nonumber\\
        & + 2\rho \frac{p-q}{p} 
                      \left[ 8(1 + 8c_2^2) \right]^{\frac{p}{p-q} }
                         \Big( \frac{\lambda_1 p}{4q} \Big)^{-\frac{q}{p-q} }
                          \sum_{|i| \ge m} 
                                \Big( \int_{y \in \mathbb{Y} }
                                             \left( L_{\tilde{\sigma}, i}(y) \right)^2 
                                                   \nu(dy)
                                \Big)^{ \frac{p}{p-q} }.
 \end{align}
 By \eqref{momentpte-12}, we  
  find that there exists $m_1 = m_1(\tau, \delta) \ge m_0$
 such that for all $t \ge T_1 \vee T_2$ and $m \ge m_1$,
 \begin{align*}
    \frac{1}{2} \mathbb{E}
             \Big[ \sup_{r \in [\tau-\rho,\tau]} \|\vartheta_m u(r)\|^2
             \Big]
     + \lambda_1
          \mathbb{E}
             \Big[ \int_{\tau-\rho}^{\tau}
                       \sum_{i \in \mathbb{Z} }
                             {\vartheta}^2(\frac{i}{m}) |u_i(s)|^p ds
             \Big]
    \le C(\tau) \delta,
 \end{align*}
 where $C(\tau)>0$ is a constant  depending
 only  on $\tau$,
 which completes the proof.
\end{proof}

 We are now 
 in a position to
   show the pullback asymptotic compactness of $\Phi$,  
   which is
    different from  that  in \cite{SSL2024arXiv, LW2024JDE, CWZ2025JNS, SSW2026JDE, BCS2026JDE, mll2024AMO}.

\begin{lemma}\label{pac}
   Suppose that {\bf(H1)}-{\bf(H4)}, \eqref{thetaassum2}, \eqref{dissip1} and \eqref{dissip2}  hold.
   Then 
   the cocyle $\Phi$ is $\mathcal{D}$-pullback asymptotically compact in $\mathcal{P}_{\theta}(D_{\rho})$.
\end{lemma}

\begin{proof}
 Let $\tau \in \mathbb{R}$, $D \in \mathcal{D}$
 and  $t_n \to  \infty$.
 Without loss of generality, 
 we  assume $t_n \ge 3\rho$ for all $n \in \mathbb{N}$.
 Given  $\mu_n \in D(\tau - t_n)$,
 let $\varphi_n \in L^{\theta}(\Omega, \mathcal{F}_{\tau-t_n}; D_{\rho} )$ such that its law $\mathcal{L}_{\varphi_n} = \mu_n$.
 We will prove the sequence
    $\{ \Phi(t_n, \tau-t_n) \mu_n \}_{n=1}^{\infty}
    = \{ \mathcal{L}_{u_\tau(\cdot; \tau-t_n, \varphi_n) } \}_{n=1}^{\infty}$
 is precompact  in
    $(\mathcal{P}_{\theta}(D_{\rho}), d_{\mathcal{P}(D_{\rho})})$.

      For every
  $k \in \mathbb{N}$, we decompose the segment process of
  the  solution $u$ of  \eqref{lmvlds-2}
  with initial value $\varphi$ at initial time $\tau-t$
   as
 \begin{align}\label{feb24a}
    u_{\tau}(\cdot; \tau-t,\varphi)
    = \widetilde{u}_\tau^k(\cdot; \tau-t,\varphi)
      + \widehat{u}_\tau^k(\cdot; \tau-t,\varphi),
 \end{align}
 where
 \begin{align}\label{feb24b}
    \widetilde{u}_\tau^k(\cdot; \tau-t,\varphi)
    = \left( {1}_{[-k,k]}(i) u_{\tau,i}(\cdot; \tau-t,\varphi) \right)_{i \in \mathbb{Z}},
 \end{align}
 and
 \begin{align}\label{feb24c}
    \widehat{u}_\tau^k(\cdot;\tau-t,\varphi)
    = \left( (1 - {1}_{[-k,k]}(i)) u_{\tau,i}(\cdot; \tau-t,\varphi) \right)_{i \in \mathbb{Z}}.
 \end{align}
     We first
     prove for every  fixed $k\in \mathbb{N}$,
     the sequence
     $ \{ \mathcal{L}_{ 
      \widetilde{u}_\tau^k
     (\cdot; \tau-t_n, \varphi_n) } \}_{n=1}^{\infty}$
 is precompact  in
    $(\mathcal{P}_{\theta}(D_{\rho}), d_{\mathcal{P}(D_{\rho})})$.
 
 \textbf{Step 1:} Prove that for every
  $k \in \mathbb{N}$ and 
  $\delta> 0$,
 there exists a compact subset $\mathcal{E}^\delta_k$ of $\ell^2$
 such that for all $n \in \mathbb{N}$,
    \begin{align}\label{feb25a}
         \mathbb{P} 
                   \left( 
                    \widetilde{u}_\tau^k
                   (r; \tau-t_n, \varphi_n) \in \mathcal{E}^\delta_k, \  
                             \forall \  r \in [-\rho, 0]
                   \right) 
           > 1 - \delta.
\end{align}
  
 By
  Lemma \ref{pue}  we see 
  that  there exists $T_1 =T_1 (\tau, D) \ge 3\rho$
  such that for all $t\ge T_1$,
 \begin{align} \label{feb25a_0}
       \mathbb{E}
              \left[ \sup_{ r \in [-2\rho, 0]} 
                          \| u_\tau(r; \tau-t, \varphi) \|^2 
                      + \int^\tau_{\tau - \rho}
                             ( \| u(s; \tau-t, \varphi)\|^2  
                               + \| u(s; \tau-t, \varphi)\|^p_p
                              ) ds 
              \right]
       \le c_1,
 \end{align}
 where $\varphi \in L^{\theta}(\Omega, \mathcal{F}_{\tau-t}; D_{\rho})$ with  law $\mathcal{L}_{\varphi} = \mu$, and 
 $c_1= c_1(\tau, \rho)>0$ is a constant
 independent of $D$ and 
 $ \varphi$.
 Since $t_n \to \infty$,
 there exists $N_1 =N_1(\tau, D)\ge 1$ such that
 $t_n \ge T_1$ for all $n\ge N_1$.
 Then by \eqref{feb25a_0} we have,
 for all $n\ge N_1$,
 \begin{align} \label{feb25a_1}
   \mathbb{E}
          \left[ \sup_{ r \in [-2\rho, 0]} 
                          \|  u_\tau(r; \tau-t_n, \varphi_n) \|^2 
                   + \int^\tau_{\tau - \rho}
                          ( \| u(s; \tau-t_n, \varphi_n) \|^2  
                             + \| u(s; \tau-t_n, \varphi_n)\|^p_p
                           ) ds
         \right]
   \le c_1.
\end{align}
 It follows from \eqref{feb25a_1}
 that
  for 
 each  $\delta > 0$,
  there exists  $R_{
 \delta} = R(\delta, \tau, \rho) > 0$
 such that
  for all $n\ge N_1$,
 \begin{align}\label{feb25c}
    \mathbb{P}
       \Big(
            \big\{ \sup_{ r \in [-\rho, 0]} 
                          \|  {u}_\tau(r; \tau-t_n, \varphi_n)\|
                       > R_{\delta}
           \big\}
       \Big)
   <  \delta. 
 \end{align}
 By \eqref{feb24b} and \eqref{feb25c} we see that
 for all  $k\in \mathbb{N}$ and $n\ge N_1$,
 \begin{align}\label{feb25d}
    \mathbb{P}
       \Big(
            \big\{ \sup_{ r \in [-\rho, 0]} 
                          \|  \widetilde{u}
                          ^k_\tau(r; \tau-t_n, \varphi_n)\|
                       > R_{\delta}
           \big\}
       \Big)
   <  \delta. 
 \end{align}
 Denote by
 $$ \mathcal{E}_{1, k, \delta}
 =   \left\{ z =(z_i)_{i\in
 \mathbb{Z}}
  \in \ell^2 : \; 
                    z_i = 0 \  \text{for} \ | i | >  k
                    \  \text{and} \
                    \| z \| \le R_{\delta} 
      \right\}.
      $$
      Then
      $ \mathcal{E}_{1, k, \delta}$
      is a compact subset of $\ell^2$ and by
      \eqref{feb25d} we have  for all  $n\ge N_1$,
 \begin{align}\label{feb25e}
    \mathbb{P}
       \Big(
            \big\{     \widetilde{u}
                ^k_\tau(r; \tau-t_n, \varphi_n)\in
                          \mathcal{E}_{1, k, \delta},\
                          \forall \ r\in [-\rho, 0]
           \big\}
       \Big)
   > 1-  \delta. 
 \end{align}
 We now consider the 
 terms $ \widetilde{u}
                ^k_\tau(r; \tau-t_n, \varphi_n)$
                for $n \le  N_1$.
                Note that for any $n, m\in \mathbb{N}$,
                we have
                $$
                \left \{ \omega \in \Omega: 
                \sup_{r\in [-\rho, 0]}
                  \|  {u}_\tau(r; \tau-t_n, \varphi_n)
                  (\omega) \|
                  \le m \right \}
                  \subseteq  \left  \{  \omega \in \Omega: 
                 \sup_{r\in [-\rho, 0]}
                  \|  {u}_\tau(r; \tau-t_n, \varphi_n) (\omega) \|
                  \le m+1\right \},
                  $$
                  and
                  $$
                  \bigcup_{m=1}^\infty
                      \left \{  \omega \in \Omega:  \sup_{r\in [-\rho, 0]}
                  \|  {u}_\tau(r; \tau-t_n, \varphi_n)(\omega) \|
                  \le m \right \}
                  =\Omega,
                  $$
                  which shows that for every
                  fixed $n\in \mathbb{N}$,
\be\label{feb25f}
                  \lim_{m
                  \to \infty}
                  \mathbb{P}
                   \big 
                  (
                      \big  \{
                        \sup_{r\in [-\rho, 0]}
                  \|  {u}_\tau(r; \tau-t_n, \varphi_n)\|
                  \le m
                 \big  \}
                 \big  ) =1.
               \ee
 By \eqref{feb25f} we infer that
 for every  $\delta>0$, there exists
 $m_\delta=m_\delta (\delta, \tau, D)
  \ge 1 $ such that
  for all  $n \le N_1$,
  \be\label{feb25g}
                  \mathbb{P}
                  \big 
                  (
                         \big  \{
                        \sup_{r\in [-\rho, 0]}
                  \|  {u}_\tau(r; \tau-t_n, \varphi_n)\|
                  \le m_\delta
                    \big  \}
                    \big  ) >1-\delta.
               \ee
               By \eqref{feb24b}
               and \eqref{feb25g} we have
               for all $k\in \mathbb{N}$
               and $n\le N_1$,
                \be\label{feb25h}
                  \mathbb{P}
                  \big
                  (
                       \big \{
                        \sup_{r\in [-\rho, 0]}
                  \|  \widetilde{u}^k_\tau(r; \tau-t_n, \varphi_n)\|
                  \le m_\delta
                    \big  \}
                      \big ) >1-\delta.
               \ee
 Denote by
 $ \mathcal{E}_{2, k, \delta}
     = \left\{ z =(z_i)_{i \in \mathbb{Z}} \in \ell^2 : \; 
                                          z_i = 0 \  \text{for} \ | i | >  k
                                          \  \text{and} \
                                          \| z \| \le m_\delta
         \right\}.
  $
      Then
      $ \mathcal{E}_{2, k, \delta}$
      is a compact subset of $\ell^2$, and by
      \eqref{feb25h} we have  for all  $n\le N_1$,
 \begin{align}\label{feb25i}
    \mathbb{P}
       \Big(
            \big\{     \widetilde{u}
                ^k_\tau(r; \tau-t_n, \varphi_n)\in
                          \mathcal{E}_{2, k, \delta},\
                          \forall \ r\in [-\rho, 0]
           \big\}
       \Big)
   > 1-  \delta. 
 \end{align}
 Let
 $\mathcal{E}^\delta_k
 =  \mathcal{E}_{1, k, \delta}\cup
   \mathcal{E}_{2, k, \delta}$.
   Then $\mathcal{E}^\delta_k$ is a compact subset
   of $\ell^2$
   and  for all  $  n\in\mathbb{N}$,
$$
    \mathbb{P}
       \Big(
            \big\{     \widetilde{u}
                ^k_\tau(r; \tau-t_n, \varphi_n)\in
                          \mathcal{E}^\delta_k ,\
                          \forall \ r\in [-\rho, 0]
           \big\}
       \Big)
   > 1-  \delta,
$$ and hence \eqref{feb25a} is valid.

      {\bf Step 2:}
      Prove that  for every 
      $k\in \mathbb{N}$, $\xi \in \ell^2$,
      any stopping time $\theta_n \in [\tau - \rho, \tau]$ and 
      any $\delta_n \in [0, \rho]$ with $\delta_n \to 0$,
 \begin{align} \label{feb25j}
       \lim_{n \to \infty} 
            \langle \widetilde{u}^k(\theta_n + \delta_n; \tau-t_n, \varphi_n) 
                          - \widetilde{u}^k(\theta_n; \tau-t_n, \varphi_n), 
                          \xi 
            \rangle = 0, \ \ \  \text{in probability},
 \end{align}
 where $\theta_n + \delta_n = \tau$ if $\theta_n + \delta_n > \tau$.

   For convenience, we write
   $u^n (t) = u(t; \tau -t_n, \varphi_n)$. 
By
  \eqref{lmvlds-2} and It\^o's formula, we obtain
 \begin{align} \label{pac-9}
       & \langle u^n( \theta_n + \delta_n )
                       - u^n(\theta_n), 
                       \xi
           \rangle
           + \int_{\theta_n}^{\theta_n + \delta_n} 
                   \langle A u^n(s), \xi \rangle ds  \nonumber\\
       & + \lambda 
                  \int_{\theta_n}^{\theta_n + \delta_n} 
                     \langle u^n(s), \xi \rangle ds
           + \int_{\theta_n}^{\theta_n + \delta_n} 
                   \langle f(u^n(s), \mathcal{L}_{u^n(s)}), 
                               \xi 
                   \rangle ds  \nonumber \\
   = & \int_{\theta_n}^{\theta_n + \delta_n} 
                 \langle g(s), \xi \rangle ds
           + \int_{\theta_n}^{\theta_n + \delta_n} 
                 \langle F(u^n(s-\rho), \mathcal{L}_{u^n(s-\rho)}), 
                              \xi 
                 \rangle ds  \nonumber \\
       & + \sqrt{\varepsilon} \sum_{k \in \mathbb{N} } 
                   \int_{\theta_n}^{\theta_n + \delta_n} 
                        \langle \sigma_k(u^n(s), \mathcal{L}_{u^n(s)}) + h_k(s),
                                     \xi
                        \rangle dW_k(s)  \nonumber \\
       & + \sqrt{\varepsilon} 
                   \int_{\theta_n}^{\theta_n + \delta_n} 
                        \int_{y \in \mathbb{Y} }
                            \langle
                                        \widetilde{\sigma}(u^n(s-), \mathcal{L}_{u^n(s)}, y) 
                                        +  \widetilde{h}(s, y), 
                                        \xi 
                            \rangle
                       \widetilde{N}(ds, dy).
 \end{align}
 
 For the second term on the left-hand side of \eqref{pac-9}, 
by the embedding
$\ell^p \subseteq \ell^{2p-2}$ 
for $p \ge 2$, we have
 \begin{align} \label{pac-10}
       \mathbb{E} 
                 \left[ \Big|
                              \int_{\theta_n}^{\theta_n + \delta_n} 
                                     \langle A u^n(s), \xi \rangle ds 
                           \Big|
                 \right] 
    = & \mathbb{E} 
                  \left[ \Big|
                                 \int_{\theta_n}^{\theta_n + \delta_n} 
                                        \sum_{i \in \mathbb{Z} }
                                            | (B u^n)_i(s)|^{p-2} (B u^n)_i(s) (B\xi)_i  ds 
                              \Big|
                    \right]   \nonumber\\
  \le & \mathbb{E} 
                  \left[ \int_{\theta_n}^{\theta_n + \delta_n} 
                                        \left( \sum_{i \in \mathbb{Z} }
                                                     | (B u^n)_i(s)|^{2p-2} 
                                        \right)^{\frac{1}{2} }
                                        \| B\xi\|  ds
                    \right]   \nonumber\\     
   = & \mathbb{E} 
                  \left[ \int_{\theta_n}^{\theta_n + \delta_n} 
                                        \| B u^n(s) \|_{2p-2}^{p - 1}
                                        \| B\xi\|  ds
                    \right]   \nonumber\\  
 \le & \mathbb{E} 
                   \left[ \int_{\theta_n}^{\theta_n + \delta_n} 
                                         \| B u^n(s) \|_p^{p - 1}
                                         \| B\xi\|  ds
                     \right]   \nonumber\\                                                                              
    \le & \|B\xi \| 
             \left( \mathbb{E} 
                             \bigg[ \int_{\theta_n}^{\theta_n + \delta_n} 
                                         \|B u^n(s)\|_p^p ds 
                              \bigg] 
             \right)^{\frac{p-1}{p} } 
             \delta_n^{\frac{1}{p} }  \nonumber\\ 
     \le &2^{p-1} \|B\xi \| 
                  \left( \mathbb{E} 
                                  \bigg[ \int_{\theta_n}^{\theta_n + \delta_n} 
                                               \|u^n(s)\|_p^p ds 
                                   \bigg] 
                  \right)^{\frac{p-1}{p} } 
                  \delta_n^{\frac{1}{p} }  \nonumber\\
     \le &2^{p-1} \|B\xi \| 
                  \left( \mathbb{E} 
                                  \bigg[ \int_{\tau - \rho}^{\tau} 
                                                 \|u^n(s)\|_p^p ds 
                                   \bigg] 
                  \right)^{\frac{p-1}{p} } 
                  \delta_n^{\frac{1}{p} }.                                                        
 \end{align}
 By \eqref{feb25a_1} and
 \eqref{pac-10} we have,
 for all $n\ge N_1$,
  \begin{align} \label{pac-12}
        \mathbb{E} 
                  \left[ \Big|
                               \int_{\theta_n}^{\theta_n + \delta_n} 
                                      \langle A u^n(s), \xi \rangle ds 
                            \Big|
                  \right] 
  \le & 2^{p-1} 
          c_1 ^{\frac{p-1}{p} } 
               \delta_n^{\frac{1}{p} }\|B\xi \| .                               
 \end{align}
  For the third term on the left-hand side of \eqref{pac-9}, 
 by \eqref{feb25a_1}, we obtain that for all $n \ge N_1$, 
 \begin{align} \label{pac-13}
      & \lambda 
             \mathbb{E} 
                    \left[ 
                             \Big| \int_{\theta_n}^{\theta_n + \delta_n} 
                                          \langle u^n(s), \xi \rangle ds
                             \Big|
                    \right]  
  \le   \lambda 
                \mathbb{E} 
                       \left[ \int_{\theta_n}^{\theta_n + \delta_n} 
                                         \| u^n(s) \|  \| \xi \| ds
                       \right]     \nonumber\\
 & \le   \lambda \| \xi \|
                \mathbb{E} 
                       \left[ \int_{\theta_n}^{\theta_n + \delta_n} 
                                         \| u^n(s) \| ds
                       \right]  
  \le   \lambda \| \xi \|
                \left( \mathbb{E} 
                                \left[ \int_{\tau - \rho}^\tau
                                                  \| u^n(s) \|^2 ds
                                \right]
                \right)^{\frac{1}{2} }
                    \delta_n^{\frac{1}{2} }  
  \le   \lambda \| \xi \| c_1^{\frac{1}{2} }
                    \delta_n^{\frac{1}{2} }.                                                                                 
 \end{align} 
 For the fourth term on the left-hand side of \eqref{pac-9}, 
 by \eqref{f4}, we have
 \begin{align} \label{pac-14}
       & \mathbb{E}
                  \left[ 
                           \Big| \int_{\theta_n}^{\theta_n + \delta_n} 
                                           \langle f(u^n(s), \mathcal{L}_{u^n(s)}), 
                                                        \xi 
                                           \rangle ds 
                            \Big|
                  \right]  \nonumber \\
  \le & \mathbb{E}
                   \left[ \int_{\theta_n}^{\theta_n + \delta_n} 
                                           \| f(u^n(s), \mathcal{L}_{u^n(s)}) \|
                                           \| \xi \| 
                              ds 
                    \right]  \nonumber \\   
  \le & \| \xi \| 
           \mathbb{E}
                  \left[ 
                           \int_{\theta_n}^{\theta_n + \delta_n} 
                                \left( \sum_{i \in \mathbb{Z} }
                                               \left( | \phi_{2,i} | 
                                                          \big( \big| (u^n)_i(s) \big| 
                                                                   + \big| (u^n)_i(s) \big|^{p-1} 
                                                          \big) 
                                                         + | \phi_{3,i} | 
                                                             \sqrt{\mathbb{E} \left[ \| u^n(s) \|^2\right] } 
                                                \right)^2
                                \right)^{\frac{1}{2} }                                     
                            ds 
                  \right]  \nonumber \\ 
  \le & \| \xi \| 
           \mathbb{E}
                  \left[ \int_{\theta_n}^{\theta_n + \delta_n} 
                               \left( 3 \| \phi_2 \|_{\ell^\infty}^2 \| u^n(s) \|^2 
                                        + 3 \sum_{i \in \mathbb{Z} } 
                                                     | \phi_{2,i} |^{2}
                                                     \big| (u^n)_i(s) \big|^{2p-2} 
                                        + 3 \|\phi_{3}\|^2 
                                           \mathbb{E} \left[ \| u^n(s) \|^2\right]                                     
                                \right)^{\frac{1}{2} } ds 
                  \right]  \nonumber \\  
  \le & \sqrt{3} \| \xi \| 
           \mathbb{E}
                  \left[ \int_{\theta_n}^{\theta_n + \delta_n} 
                               \| \phi_2 \|_{\ell^\infty} \| u^n(s) \| ds
                  \right]  
            + \sqrt{3} \| \xi \| 
                    \mathbb{E}
                           \bigg[ \int_{\theta_n}^{\theta_n + \delta_n}
                                                \left( \sum_{i \in \mathbb{Z} } 
                                                                  | \phi_{2,i} |^{2}
                                                                  \big| (u^n)_i(s) \big|^{2p-2} 
                                               \right)^{\frac{1}{2} }
                                       ds
                            \bigg]  \nonumber\\
        & + \sqrt{3} \| \xi \| 
                     \mathbb{E}
                            \left[ \int_{\theta_n}^{\theta_n + \delta_n}
                                             \|\phi_{3}\| 
                                             \sqrt{\mathbb{E} \left[ \| u^n(s) \|^2\right] }                                    
                                        ds 
                             \right]  \nonumber \\  
  \le & \sqrt{3} \| \xi \|  \| \phi_2 \|_{\ell^\infty}
           \left( \mathbb{E}
                           \left[ \int_{\theta_n}^{\theta_n + \delta_n} 
                                    \| u^n(s) \|^2 ds
                           \right] 
           \right)^{\frac{1}{2} }
           \delta_n^{\frac{1}{2} } 
           + \sqrt{3} \| \xi \| 
                    \mathbb{E}
                           \left[ \int_{\theta_n}^{\theta_n + \delta_n}
                                             \| \phi_{2} \|_{\ell^\infty}
                                             \| u^n(s) \|_{2p-2}^{p - 1}                        
                                         ds
                           \right]  \nonumber\\
       & + \sqrt{3} \|\phi_{3}\| \| \xi \| 
                    \left( \mathbb{E}
                                        \left[ \int_{\theta_n}^{\theta_n + \delta_n}
                                                      \mathbb{E} \left[ \| u^n(s) \|^2\right]                                 
                                                  ds 
                                        \right]
                     \right)^{\frac{1}{2} } 
                     \delta_n^{\frac{1}{2} }  \nonumber\\  
  \le & \sqrt{3} \| \xi \|  \| \phi_2 \|_{\ell^\infty}
                \left( \mathbb{E}
                                \left[ \int_{\theta_n}^{\theta_n + \delta_n} 
                                        \| u^n(s) \|^2 ds
                                \right] 
                \right)^{\frac{1}{2} }
                \delta_n^{\frac{1}{2} } 
           + \sqrt{3} \| \xi \|  \| \phi_{2} \|_{\ell^\infty}                                   
                      \mathbb{E}
                             \left[ \int_{\theta_n}^{\theta_n + \delta_n}
                                               \| u^n(s) \|_{p}^{p - 1}                        
                                     ds
                             \right]  \nonumber\\
       & + \sqrt{3} \| \xi \| \|\phi_{3}\| 
                    \left( \mathbb{E}
                                    \left[ \int_{\theta_n}^{\theta_n + \delta_n}
                                                         \mathbb{E} \left[ \| u^n(s) \|^2\right]                                 
                                                ds 
                                    \right]
                    \right)^{\frac{1}{2} } 
                                  \delta_n^{\frac{1}{2} }  \nonumber\\    
  \le & \sqrt{3} \| \xi \|  
                \left( \| \phi_2 \|_{\ell^\infty}
                           + \|\phi_{3}\|
                \right)        
                    \left( \mathbb{E}
                                    \left[ \int_{\tau - \rho}^\tau
                                                 \| u^n(s) \|^2 ds
                                    \right] 
                    \right)^{\frac{1}{2} }
                        \delta_n^{\frac{1}{2} } \nonumber\\
       & + \sqrt{3} \| \xi \| \| \phi_{2} \|_{\ell^\infty}
                    \left( \mathbb{E}
                                    \left[ \int_{\tau - \rho}^\tau
                                                          \| u^n(s) \|_{p}^{p}                        
                                                  ds
                                    \right] 
                    \right)^{\frac{p-1}{p} }
                        \delta_n^{\frac{1}{p} }.                                                                                       
  \end{align} 
 By \eqref{feb25a_1} and \eqref{pac-14}, we obtain that for all $n \ge N_1$,
 \begin{align} \label{pac-15}
     &  \mathbb{E}
                  \left[ 
                           \Big| \int_{\theta_n}^{\theta_n + \delta_n} 
                                           \langle f(u^n(s), \mathcal{L}_{u^n(s)}), 
                                                        \xi 
                                           \rangle ds 
                            \Big|
                  \right]  \nonumber\\
  \le & \sqrt{3} \| \xi \|  
                \left( \| \phi_2 \|_{\ell^\infty}
                           + \|\phi_{3}\|
                \right)        
                c_1^{\frac{1}{2} }
                \delta_n^{\frac{1}{2} }  
        + \sqrt{3} \| \xi \| \| \phi_{2} \|_{\ell^\infty}
          c_1^{\frac{p-1}{p} }
                        \delta_n^{\frac{1}{p} }.                                                                                       
  \end{align}  
 By \eqref{F3} and \eqref{feb25a_1},
 we have for all $n\ge N_1$,
 \begin{align} \label{pac-16}   
       & \mathbb{E} 
                 \left[ 
                         \Big| \int_{\theta_n}^{\theta_n + \delta_n} 
                                      \langle g(s), \xi \rangle ds
                         \Big|
                 \right]
           + \mathbb{E} 
                      \left[ 
                               \Big| \int_{\theta_n}^{\theta_n + \delta_n} 
                                             \langle F(u^n(s-\rho), \mathcal{L}_{u^n(s-\rho)}), 
                                                         \xi 
                                            \rangle ds
                               \Big|
                      \right]  \nonumber\\
  \le & \| \xi \| 
              \mathbb{E} 
                     \left[ \int_{\theta_n}^{\theta_n + \delta_n} 
                                         \| g(s) \|ds
                     \right]
            + \| \xi \| 
                 \mathbb{E} 
                        \left[ \int_{\theta_n}^{\theta_n + \delta_n} 
                                                \| F(u^n(s-\rho), \mathcal{L}_{u^n(s-\rho)}) \| 
                                            ds
                        \right]  \nonumber\\
  \le & \| \xi \| 
              \left( \mathbb{E} 
                              \left[ \int_{\tau - \rho}^\tau \| g(s) \|^2 ds
                              \right]  
              \right)^{\frac{1}{2} }
                  \delta_n^{\frac{1}{2} } 
           + \| \xi \| 
                \left( \mathbb{E} 
                                \left[ \int_{\theta_n}^{\theta_n + \delta_n} 
                                                  \| F(u^n(s-\rho), \mathcal{L}_{u^n(s-\rho)}) \|^2 
                                            ds
                                \right]
               \right)^{\frac{1}{2} }  
                    \delta_n^{\frac{1}{2} }  \nonumber\\   
  \le & \| \xi \| 
              \left( \mathbb{E} 
                              \left[ \int_{\tau - \rho}^\tau \| g(s) \|^2 ds
                              \right]  
              \right)^{\frac{1}{2} }
                  \delta_n^{\frac{1}{2} }   \nonumber\\
       & + \| \xi \| 
                \left( \mathbb{E} 
                                \left[ \int_{\theta_n}^{\theta_n + \delta_n} 
                                              \left( 2 \|F(0,\delta_0)\|^2
                                                       + 4 \|L_{F}\|_{\ell^\infty}^2 \| u^n(s - \rho) \|^2
                                                       + 4 \|L_{F}\|^2 \mathbb{E} \left[ \| u^n(s - \rho) \|^2 \right]
                                              \right) 
                                           ds
                                \right]
               \right)^{\frac{1}{2} }  
                    \delta_n^{\frac{1}{2} }  \nonumber\\ 
  \le & \| \xi \| 
              \left( \mathbb{E} 
                              \left[ \int_{\tau - \rho}^\tau \| g(s) \|^2 ds
                              \right]  
              \right)^{\frac{1}{2} }
                  \delta_n^{\frac{1}{2} } 
           + \sqrt{2 } \| \xi \|  \|F(0,\delta_0)\| \delta_n  \nonumber\\
       & + 2 \| \xi \| \|L_{F}\|_{\ell^\infty}
                  \left( \mathbb{E} 
                                   \left[ \int_{\tau - \rho}^\tau
                                                 \| u^n(s - \rho) \|^2 ds
                                   \right]
                   \right)^{\frac{1}{2} }
                       \delta_n^{\frac{1}{2} } 
           + 2 \| \xi \| \|L_{F}\| 
                  \left(  \int_{\tau - \rho}^\tau
                                 \mathbb{E} \left[ \| u^n(s - \rho) \|^2 \right] ds
                  \right)^{\frac{1}{2} }  
                       \delta_n^{\frac{1}{2} }  \nonumber\\      
  \le & \| \xi \| 
              \left( \mathbb{E} 
                              \left[ \int_{\tau - \rho}^\tau \| g(s) \|^2 ds
                              \right]  
              \right)^{\frac{1}{2} }
                  \delta_n^{\frac{1}{2} } 
           + \sqrt{2 } \| \xi \|  \|F(0,\delta_0)\| \delta_n  \nonumber\\
       & + 2 \| \xi \| 
                 \left( \|L_{F}\|_{\ell^\infty}
                          + \|L_{F}\|
                 \right)  \rho^{\frac{1}{2} }
                 \left( \sup_{s \in [\tau - 2\rho, \tau - \rho] } 
                                      \mathbb{E} \left[ \| u^n(s) \|^2 \right]
                 \right)^{\frac{1}{2} }  
                     \delta_n^{\frac{1}{2} }  \nonumber  \\     
 \le & \| \xi \| 
              \left( \mathbb{E} 
                              \left[ \int_{\tau - \rho}^\tau \| g(s) \|^2 ds
                              \right]  
              \right)^{\frac{1}{2} }
                  \delta_n^{\frac{1}{2} } 
           + \sqrt{2 } \| \xi \|  \|F(0,\delta_0)\| \delta_n  
     + 2 \| \xi \| 
                 \left( \|L_{F}\|_{\ell^\infty}
                          + \|L_{F}\|
                 \right)  \rho^{\frac{1}{2} }
             c_1^{\frac{1}{2} }    
                      \delta_n^{\frac{1}{2} }.                                                                                            
 \end{align}
 
 For the third term on the right-hand side of \eqref{pac-9},
 by the BDG inequality, \eqref{siglinear-1}, and
 the  H\"{o}lder inequality,
 we have 
 \begin{align} \label{pac-17}  
       & \sqrt{\varepsilon} 
           \mathbb{E} 
                  \left[ 
                           \Big| \sum_{k \in \mathbb{N} } 
                                          \int_{\theta_n}^{\theta_n + \delta_n} 
                                               \langle \sigma_k(u^n(s), \mathcal{L}_{u^n(s)}) + h_k(s),
                                                           \xi
                                               \rangle dW_k(s)
                           \Big|
                   \right]  \nonumber \\
 \le & \sqrt{\varepsilon} c_1
           \mathbb{E} 
                  \left[ 
                          \Big( \sum_{k \in \mathbb{N} } 
                                         \int_{\theta_n}^{\theta_n + \delta_n} 
                                              \| \sigma_k(u^n(s), \mathcal{L}_{u^n(s)}) + h_k(s) \|^2
                                              \| \xi \|^2 ds
                          \Big)^\frac{1}{2}
                  \right]  \nonumber \\     
 \le & \sqrt{2\varepsilon} c_1  \| \xi \|
           \mathbb{E} 
                  \left[ 
                          \Big( \sum_{k \in \mathbb{N} } 
                                         \int_{\theta_n}^{\theta_n + \delta_n} 
                                             \left( \| \sigma_k(u^n(s), \mathcal{L}_{u^n(s)}) \|^2
                                                      + \| h_k(s) \|^2
                                             \right)
                                         ds
                          \Big)^\frac{1}{2}
                  \right]  \nonumber \\ 
 \le & \sqrt{2\varepsilon} c_1  \| \xi \|
                \left( \mathbb{E} 
                               \left[ \sum_{k \in \mathbb{N} } 
                                               \int_{\theta_n}^{\theta_n + \delta_n} 
                                                        \| \sigma_k(u^n(s), \mathcal{L}_{u^n(s)}) \|^2
                                                ds
                               \right] 
                  \right)^{\frac{1}{2} }  \nonumber\\
       & + \sqrt{2\varepsilon} c_1  \| \xi \|
                    \left( \mathbb{E}
                                    \left[ \sum_{k \in \mathbb{N} } 
                                                     \int_{\theta_n}^{\theta_n + \delta_n} 
                                                              \| h_k(s) \|^2
                                                      ds
                                    \right]
                    \right)^{\frac{1}{2} }  \nonumber \\       
 \le & \sqrt{2\varepsilon} c_1  \| \xi \|
                \left( \mathbb{E} 
                                \left[ \| \alpha \|^2 \delta_n
                                         + \sum_{k \in \mathbb{N} } 
                                                   \int_{\theta_n}^{\theta_n + \delta_n} 
                                                        \sum_{i \in \mathbb{Z} }
                                                              4 L_{\sigma, k, i}^2
                                                                  \left( |(u^n)_i(s)|^q
                                                                           + \mathbb{E} \left[ \| u^n(s) \|^2 \right]
                                                                  \right)
                                                    ds
                                \right] 
                  \right)^{\frac{1}{2} }  \nonumber\\
       & + \sqrt{2\varepsilon} c_1  \| \xi \|
                    \left( \mathbb{E}
                                    \left[ \int_{\theta_n}^{\theta_n + \delta_n} 
                                                      \| h(s) \|^2
                                               ds
                                    \right]
                    \right)^{\frac{1}{2} }  \nonumber \\ 
 \le & \sqrt{2\varepsilon} c_1  \| \xi \|
                \| \alpha \| \delta_n^{\frac{1}{2} } 
          + \sqrt{2\varepsilon} c_1  \| \xi \|
                   \left( \mathbb{E} 
                                   \left[ \sum_{k \in \mathbb{N} } 
                                                  \int_{\theta_n}^{\theta_n + \delta_n} 
                                                       \sum_{i \in \mathbb{Z} }
                                                            4 L_{\sigma, k, i}^2
                                                             |(u^n)_i(s)|^q
                                                    ds
                                   \right]  
                  \right)^{\frac{1}{2} }   \nonumber\\ 
       & + \sqrt{2\varepsilon} c_1  \| \xi \|
                    \left( \mathbb{E} 
                                    \left[ \sum_{k \in \mathbb{N} } 
                                                   \int_{\theta_n}^{\theta_n + \delta_n} 
                                                            \sum_{i \in \mathbb{Z} }
                                                                  4 L_{\sigma, k, i}^2
                                                                  \mathbb{E} \left[ \| u^n(s) \|^2 \right]
                                                     ds
                                    \right]                           
                   \right)^{\frac{1}{2} }  \nonumber\\
       & + \sqrt{2\varepsilon} c_1  \| \xi \|
                    \left( \mathbb{E}
                                    \left[ \int_{\theta_n}^{\theta_n + \delta_n} 
                                                      \| h(s) \|^2
                                               ds
                                    \right]
                    \right)^{\frac{1}{2} }  \nonumber \\ 
 \le & \sqrt{2\varepsilon} c_1  \| \xi \|
                \| \alpha \| \delta_n^{\frac{1}{2} } 
          + 2 \sqrt{2\varepsilon} c_1  \| \xi \|
                   \left( \mathbb{E} 
                                   \left[ \int_{\theta_n}^{\theta_n + \delta_n} 
                                                 \bigg( \sum_{i \in \mathbb{Z} }
                                                                   \Big( \sum_{k \in \mathbb{N} } 
                                                                                       L_{\sigma, k, i}^2
                                                                    \Big)^{\frac{p}{p-q} }
                                                 \bigg)^\frac{p-q}{p} 
                                                  \| u^n(s) \|_p^q
                                              ds
                                   \right]  
                  \right)^{\frac{1}{2} }   \nonumber\\                       
       & + 2 \sqrt{2\varepsilon} c_1  \| \xi \| \|L\|
                    \left( \mathbb{E} 
                                    \left[ \int_{\theta_n}^{\theta_n + \delta_n} 
                                                      \mathbb{E} \left[ \| u^n(s) \|^2 \right]
                                              ds
                                    \right]                           
                   \right)^{\frac{1}{2} }  \nonumber\\
       & + \sqrt{2\varepsilon} c_1  \| \xi \|
                    \left( \mathbb{E}
                                    \left[ \int_{\theta_n}^{\theta_n + \delta_n} 
                                                      \| h(s) \|^2
                                               ds
                                    \right]
                    \right)^{\frac{1}{2} }  \nonumber \\    
 \le & \sqrt{2\varepsilon} c_1  \| \xi \|
                \| \alpha \| \delta_n^{\frac{1}{2} } 
          + 2 \sqrt{2\varepsilon} c_1  \| \xi \| 
                      \bigg( \sum_{i \in \mathbb{Z} }
                                       \Big( \sum_{k \in \mathbb{N} } L_{\sigma, k, i}^2
                                       \Big)^{\frac{p}{p-q} }
                      \bigg)^\frac{p-q}{2p} 
                           \left( \mathbb{E} 
                                        \left[ \int_{\tau -\rho}^\tau
                                                          \| u^n(s) \|_p^p
                                                  ds
                                        \right]  
                          \right)^{\frac{q}{2p} } 
                              \delta_n^{\frac{p-q}{2p} }  \nonumber\\                        
       & + 2 \sqrt{2\varepsilon} c_1  \| \xi \| \|L\|
                      \Big( \sup_{s \in [\tau-\rho, \tau] } \mathbb{E} \left[ \| u^n(s) \|^2 \right]
                      \Big)^{\frac{1}{2} }
                           \delta_n^{\frac{1}{2} }  
          + \sqrt{2\varepsilon} c_1  \| \xi \|
                    \left( \mathbb{E} 
                                    \left[ \int_{\theta_n}^{\theta_n + \delta_n} 
                                                      \| h(s) \|^2
                                               ds
                                    \right]
                    \right)^{\frac{1}{2} }.                           
 \end{align}
 By \eqref{feb25a_1} and \eqref{pac-17}, we obtain that for all $n \ge N_1$,
 \begin{align} \label{pac-18}  
       & \sqrt{\varepsilon} 
           \mathbb{E} 
                  \left[ \Big| \sum_{k \in \mathbb{N} } 
                                          \int_{\theta_n}^{\theta_n + \delta_n} 
                                               \langle \sigma_k(u^n(s), \mathcal{L}_{u^n(s)}) + h_k(s),
                                                           \xi
                                               \rangle dW_k(s)
                           \Big|
                   \right]  \nonumber \\
 \le & \sqrt{2\varepsilon} c_1  \| \xi \|
                \| \alpha \| \delta_n^{\frac{1}{2} } 
          + 2 \sqrt{2\varepsilon} c_1  \| \xi \| 
                   \bigg( \sum_{i \in \mathbb{Z} }
                                    \Big( \sum_{k \in \mathbb{N} } L_{\sigma, k, i}^2
                                    \Big)^{\frac{p}{p-q} }
                   \bigg)^\frac{p-q}{2p} 
                   c_1^{\frac{q}{2p} } 
                   \delta_n^{\frac{p-q}{2p} }  \nonumber\\                        
       & + 2 \sqrt{2\varepsilon} c_1  \| \xi \| \|L\|
                 c_1^{\frac{1}{2} } 
                           \delta_n^{\frac{1}{2} }  
           + \sqrt{2\varepsilon} c_1  \| \xi \|
                    \left( \mathbb{E} 
                                    \left[ \int_{\theta_n}^{\theta_n + \delta_n} 
                                                      \| h(s) \|^2
                                               ds
                                    \right]
                    \right)^{\frac{1}{2} }.                           
 \end{align} 
 
 For the fourth term on the right-hand side of \eqref{pac-9},
 by the BDG inequality, \eqref{siglinear-2}, and 
 the H\"{o}lder inequality,
 we have 
 \begin{align} \label{pac-19}  
       & \sqrt{\varepsilon} 
           \mathbb{E} 
                  \left[ 
                          \Big| \int_{\theta_n}^{\theta_n + \delta_n} 
                               \int_{y \in \mathbb{Y} }
                                   \langle
                                          \widetilde{\sigma}(u^n(s-), \mathcal{L}_{u^n(s)}, y) 
                                          +  \widetilde{h}(s, y), 
                                          \xi 
                                  \rangle
                           \widetilde{N}(ds, dy) 
                       \Big| 
                   \right]  \nonumber\\
  \le & \sqrt{\varepsilon} c_2
                \mathbb{E} 
                       \left[
                                \Big( \int_{\theta_n}^{\theta_n + \delta_n} 
                                             \int_{y \in \mathbb{Y} }
                                                      \| \widetilde{\sigma}(u^n(s-), \mathcal{L}_{u^n(s)}, y)
                                                         + \widetilde{h}(s, y)
                                                      \|^2 
                                                      \| \xi \|^2 
                                             N(ds, dy) 
                               \Big)^{\frac{1}{2} }
                      \right] \nonumber\\
  \le & \sqrt{\varepsilon} c_2
                 \left( \mathbb{E} 
                                   \left[ \int_{\theta_n}^{\theta_n + \delta_n} 
                                                 \int_{y \in \mathbb{Y} }
                                                          \| \widetilde{\sigma}(u^n(s-), \mathcal{L}_{u^n(s)}, y)
                                                             + \widetilde{h}(s, y)
                                                          \|^2 
                                                          \| \xi \|^2 
                                                   \nu(dy) ds
                                  \right] 
                  \right)^{\frac{1}{2} }  \nonumber\\
  \le & \sqrt{2 \varepsilon} c_2 \| \xi \|
                 \left( \mathbb{E} 
                                   \left[ \int_{\theta_n}^{\theta_n + \delta_n} 
                                                 \int_{y \in \mathbb{Y} }
                                                          \| \widetilde{\sigma}(u^n(s-), \mathcal{L}_{u^n(s)}, y)
                                                          \|^2       
                                                   \nu(dy) ds 
                                  \right] 
                  \right)^{\frac{1}{2} }  \nonumber\\    
        & + \sqrt{2 \varepsilon} c_2 \| \xi \|
                  \left( \mathbb{E} 
                                   \left[ \int_{\theta_n}^{\theta_n + \delta_n} 
                                                 \int_{y \in \mathbb{Y} }
                                                          \| \widetilde{h}(s, y) \|^2        
                                                   \nu(dy) ds 
                                  \right] 
                  \right)^{\frac{1}{2} }  \nonumber\\   
  \le & \sqrt{2 \varepsilon} c_2 \| \xi \|
                 \left( \mathbb{E} 
                                   \left[ \int_{\theta_n}^{\theta_n + \delta_n} 
                                               \sum_{i \in \mathbb{Z} }
                                                     \left( \beta_i^2
                                                              + 4 \int_{y \in \mathbb{Y} }
                                                                  \left( L_{ \widetilde{\sigma}, i}(y) \right)^2 \nu(dy)
                                                                  \left( | (u^n)_i(s) |^q 
                                                                           + \mathbb{E} \left[ \|u^n(s)\|^2 \right] 
                                                                  \right) 
                                                    \right)     
                                              ds 
                                  \right] 
                  \right)^{\frac{1}{2} }  \nonumber\\    
        & + \sqrt{2 \varepsilon} c_2 \| \xi \|
                     \left( \mathbb{E} 
                                     \left[ \int_{\theta_n}^{\theta_n + \delta_n} 
                                                       \| \widetilde{h}(s) \|_{L^2(\mathbb{Y}, \nu; \ell^2) }^2 
                                                ds 
                                    \right] 
                  \right)^{\frac{1}{2} }  \nonumber\\                     
  \le & \sqrt{2 \varepsilon} c_2 \| \xi \| 
                \| \beta \| \delta_n^{\frac{1}{2} }
           + 2 \sqrt{2 \varepsilon} c_2 \| \xi \|
                       \left( \mathbb{E} 
                                       \left[ \int_{\theta_n}^{\theta_n + \delta_n} 
                                                          \| L_{ \widetilde{\sigma}} \|_{L^2(\mathbb{Y}, \nu; \ell^2) }^2 
                                                               \mathbb{E} \left[ \|u^n(s)\|^2 \right]     
                                                 ds 
                                      \right] 
                      \right)^{\frac{1}{2} }  \nonumber\\                
        & + 2 \sqrt{2 \varepsilon} c_2 \| \xi \| 
                       \left( \mathbb{E} 
                                       \left[ \int_{\theta_n}^{\theta_n + \delta_n} 
                                                   \sum_{i \in \mathbb{Z} }
                                                          \int_{y \in \mathbb{Y} }
                                                                    \left( L_{ \widetilde{\sigma}, i}(y) \right)^2 \nu(dy)
                                                                     | (u^n)_i(s) |^q      
                                                 ds 
                                      \right] 
                       \right)^{\frac{1}{2} }  \nonumber\\ 
        & + \sqrt{2 \varepsilon} c_2 \| \xi \|
                  \left( \mathbb{E} 
                                   \left[ \int_{\theta_n}^{\theta_n + \delta_n} 
                                                     \| \widetilde{h}(s) \|_{L^2(\mathbb{Y}, \nu; \ell^2) }^2  
                                             ds 
                                  \right] 
                  \right)^{\frac{1}{2} }  \nonumber\\   
  \le & \sqrt{2 \varepsilon} c_2 \| \xi \| 
                \| \beta \| \delta_n^{\frac{1}{2} }
            + 2 \sqrt{2 \varepsilon} c_2 \| \xi \| \|L\|
                         \left( \sup_{s \in [\tau -\rho, \tau] } 
                                       \mathbb{E} \left[ \|u^n(s)\|^2 \right]  
                         \right)^{\frac{1}{2} }
                              \delta_n^{\frac{1}{2} }  \nonumber\\                    
       & + 2 \sqrt{2 \varepsilon} c_2 \| \xi \| 
                  \bigg( \sum_{i \in \mathbb{Z} }
                                   \Big( \int_{y \in \mathbb{Y} }
                                                 \left( L_{ \widetilde{\sigma}, i}(y) \right)^2 \nu(dy)
                                   \Big)^{\frac{p}{p-q} }
                  \bigg)^{\frac{p-q}{2p} }
                       \left( \mathbb{E} 
                                        \left[ \int_{\theta_n}^{\theta_n + \delta_n} 
                                                   \| u^n(s) \|_p^q      
                                                 ds 
                                       \right] 
                       \right)^{\frac{1}{2} }  \nonumber\\  
       & + \sqrt{2 \varepsilon} c_2 \| \xi \|
                    \left( \mathbb{E} 
                                    \left[ \int_{\theta_n}^{\theta_n + \delta_n} 
                                                      \| \widetilde{h}(s) \|_{L^2(\mathbb{Y}, \nu; \ell^2) }^2 
                                              ds 
                                    \right] 
                    \right)^{\frac{1}{2} }  \nonumber\\    
  \le & \sqrt{2 \varepsilon} c_2 \| \xi \| 
                \| \beta \| \delta_n^{\frac{1}{2} }
           + 2 \sqrt{2 \varepsilon} c_2 \| \xi \| \|L\|
                         \left( \sup_{s \in [\tau -\rho, \tau] } 
                                       \mathbb{E} \left[ \|u^n(s)\|^2 \right]  
                         \right)^{\frac{1}{2} }
                              \delta_n^{\frac{1}{2} }  \nonumber\\                
       & + 2 \sqrt{2 \varepsilon} c_2 \| \xi \| 
                  \bigg( \sum_{i \in \mathbb{Z} }
                                   \Big( \int_{y \in \mathbb{Y} }
                                                 \left( L_{ \widetilde{\sigma}, i}(y) \right)^2 \nu(dy)
                                   \Big)^{\frac{p}{p-q} }
                  \bigg)^{\frac{p-q}{2p} }
                       \left( \mathbb{E} 
                                        \left[ \int_{\tau - \rho}^\tau 
                                                      \| u^n(s) \|_p^p    
                                                 ds 
                                       \right] 
                       \right)^{\frac{q}{2p} }
                       \delta_n^{\frac{p-q}{2p} }  \nonumber\\      
       & + \sqrt{2 \varepsilon} c_2 \| \xi \|
                  \left( \mathbb{E} 
                                  \left[ \int_{\theta_n}^{\theta_n + \delta_n} 
                                                    \| \widetilde{h}(s) \|_{L^2(\mathbb{Y}, \nu; \ell^2) }^2        
                                             ds 
                                  \right]
                  \right)^{\frac{1}{2} }.        
 \end{align} 
 By \eqref{feb25a_1} and \eqref{pac-19}, we obtain that for all $n \ge N_1$, 
 \begin{align} \label{pac-20}  
       & \sqrt{\varepsilon} 
           \mathbb{E} 
                  \left[ 
                          \Big| \int_{\theta_n}^{\theta_n + \delta_n} 
                               \int_{y \in \mathbb{Y} }
                                   \langle
                                          \widetilde{\sigma}(u^n(s-), \mathcal{L}_{u^n(s)}, y) 
                                          +  \widetilde{h}(s, y), 
                                          \xi 
                                  \rangle
                           \widetilde{N}(ds, dy) 
                       \Big| 
                   \right]  \nonumber\\   
  \le & \sqrt{2 \varepsilon} c_2 \| \xi \| 
                \| \beta \| \delta_n^{\frac{1}{2} }
           + 2 \sqrt{2 \varepsilon} c_2 \| \xi \| \|L\|
                        c_1^{\frac{1}{2} }
                              \delta_n^{\frac{1}{2} }  \nonumber\\                
       & + 2 \sqrt{2 \varepsilon} c_2 \| \xi \| 
                  \bigg( \sum_{i \in \mathbb{Z} }
                                   \Big( \int_{y \in \mathbb{Y} }
                                                 \left( L_{ \widetilde{\sigma}, i}(y) \right)^2 \nu(dy)
                                   \Big)^{\frac{p}{p-q} }
                  \bigg)^{\frac{p-q}{2p} }
                      c_1^{\frac{q}{2p} }
                       \delta_n^{\frac{p-q}{2p} }  \nonumber\\      
       & + \sqrt{2 \varepsilon} c_2 \| \xi \|
                  \left( \mathbb{E} 
                                 \left[ \int_{\theta_n}^{\theta_n + \delta_n} 
                                                   \| \widetilde{h}(s) \|_{L^2(\mathbb{Y}, \nu; \ell^2) }^2        
                                            ds
                                 \right]
                  \right)^{\frac{1}{2} }. 
 \end{align} 
 
 From \eqref{pac-9}, \eqref{pac-12}, \eqref{pac-13}, \eqref{pac-15}, \eqref{pac-16},
 \eqref{pac-18} and \eqref{pac-20}, 
we obtain  that for all $n \ge N_1$, 
  \begin{align} \label{pac-21}
        & \mathbb{E} 
                  \left[ 
                           \Big|  
                                    \langle u^n( \theta_n + \delta_n)
                                                 - u^n(\theta_n), 
                                                 \xi
                                   \rangle
                           \Big| 
                  \right]  
  \le  c_3 \| \xi\|
           \left( \delta_n^{\frac{1}{p} } + \delta_n^{\frac{p-q}{2p} } \right)
           \nonumber\\
          &+ c_3  \| \xi \|
                     \left( \mathbb{E} 
                                     \left[  \int_{\theta_n}^{\theta_n + \delta_n} 
                                                        \| h(s) \|^2
                                                ds 
                                     \right]
                     \right)^{\frac{1}{2} }    
       + c_3 \| \xi \|
                  \left( \mathbb{E} 
                                  \left[ \int_{\theta_n}^{\theta_n + \delta_n} 
                                                    \| \widetilde{h}(s) \|_{L^2(\mathbb{Y}, \nu; \ell^2) }^2        
                                            ds
                                  \right]
                   \right)^{\frac{1}{2} },                                                                             
 \end{align} 
 where $c_3> 0$ is a constant depending 
   $\tau$ but  not on  $\xi$ or $n$.

 By the  absolute continuity of the Lebesgue integral,
 we obtain
 \begin{align*}
       \lim_{n\to \infty }
              \mathbb{E} 
                     \left[ \int_{\theta_n}^{\theta_n + \delta_n} 
                                                              \| h(s) \|^2
                                                       ds  
                              + \int_{\theta_n}^{\theta_n + \delta_n} 
                                                               \| \widetilde{h}(s) \|_{L^2(\mathbb{Y}, \nu; \ell^2) }^2        
                                                       ds  
                     \right]
            = 0,   
 \end{align*} 
 which along with \eqref{pac-21}  implies that
 \begin{align} \label{pac-22}
       \lim_{n\to \infty} 
                 \mathbb{E} 
                        \left[ 
                                 \Big|  
                                          \langle u (
             \theta_n + \delta_n; \tau-t_n,
             \varphi_n  )
                        - u(\theta_n;
                        \tau-t_n, \varphi_n  ), 
                                                       \xi
                                          \rangle
                                 \Big| 
                        \right]  
          = 0.
 \end{align}
 Given $\xi =(\xi_i)_{i\in \mathbb{Z}} \in \ell^2$,  
 we get from \eqref{pac-22} with $\xi$ replaced by 
 $(1_{[-k,k]} (i) \xi_i)_{i \in \mathbb{Z} } \in \ell^2$,
  $$
       \lim_{n\to \infty} 
                 \mathbb{E} 
                        \left[ 
                                 \Big|  
                                          \langle \widetilde{u}^k
                                          (
             \theta_n + \delta_n; \tau -t_n,
             \varphi_n  )
                        - u^n(\theta_n;
                        \tau- t_n,
                        \varphi_n   ), 
                                                       \xi
                                          \rangle
                                 \Big| 
                        \right]  
          = 0,
          \quad \forall \ k\in\mathbb{N},
$$
 which implies \eqref{feb25j}.
 
 {\bf Step 3:} Prove
 for every  fixed  $k\in
 \mathbb{N}$,   
 $\{ \mathcal{L}_{
 \widetilde{u}^k_\tau
                (\cdot; \tau -t_n,
             \varphi_n  )
 }\}_{n=1}^\infty$
 is precompact
 in $ (\mathcal{P}
 (D_\rho), d_{
 \mathcal{P}
 (D_\rho)
 })$.

 Since $\ell^2$ is
 a
  separable Hilbert space, the dual space 
  of $\ell^2$ separates the points of $\ell^2$.
  Then by \eqref{feb25a},
  \eqref{feb25a_1}
  and \eqref{feb25j}
  it follows from   \cite[Theorem 3.1]{J1986AIHP} and \cite[Theorem 1]{A1978AOP}  that 
  the sequence 
  $\{ \mathcal{L}_{
 \widetilde{u}^k_\tau
                (\cdot; \tau -t_n,
             \varphi_n  )
 }\}_{n=1}^\infty$
 of probability measures
 on $(D([-\rho, 0],
 \ell^2), d^0)$ 
 is  tight,
 as desired.

  {\bf Step 4:} Prove  the sequence 
 $\{ \mathcal{L}_{
  {u}_\tau
                (\cdot; \tau -t_n,
             \varphi_n  )
 }\}_{n=1}^\infty$
 is precompact
 in $ (\mathcal{P}_\theta
 (D_\rho), d_{
 \mathcal{P}
 (D_\rho)
 })$.
 
     By 
  Lemma \ref{momentpte},  we 
  find   that
  for every  
  $\delta>0$,  
 there exist 
       $N_2 = N_2(\delta , \tau, D)  \ge 1$ 
 and 
      $k_{0} = k(\delta , \tau) \ge 1$
 such that for all $n \ge N_2$,
 \begin{align}\label{feb26a2}
    \mathbb{E}
       \left[ \sup_{r \in [-\rho,0]}
                    \| \widehat{u}_\tau^{k_{0} }(r; \tau-t_n, \varphi_n) \|^2
       \right]
    \le {\frac   19} \delta^2  .
 \end{align}
    By Step 3,
      we see that the sequence
  $\{ \mathcal{L}_{
 \widetilde{u}^{k_0}_\tau
                (\cdot; \tau -t_n,
             \varphi_n  )
 }\}_{n= N_2}^\infty$
 is tight, and hence  
   there exist
 $n_1,\cdots , n_l\ge N_2$ such that
 \be\label{feb26a3}
 \{ \mathcal{L}_{
 \widetilde{u}^{k_0}_\tau
                (\cdot; \tau -t_n,
             \varphi_n  )
 }\}_{n= N_2}^\infty
 \subseteq
\bigcup_{i=1}^{l}
 B\left (
 \mathcal{L}_{
 \widetilde{u}^{k_0}_\tau
                (\cdot; \tau -t_{n_i},
             \varphi_{n_i}  )},\ 
 \ {\frac 13} \delta
 \right ),
 \ee
where
$B  (
 \mathcal{L}_{
 \widetilde{u}^{k_0}_\tau
                (\cdot; \tau -t_{n_i},
             \varphi_{n_i}  )},
   {\frac 13} \delta
   )
  $ is the neighborhood in
    $ (\mathcal{P}
 (D_\rho), d_{
 \mathcal{P}
 (D_\rho)
 })$
    with center 
 $\mathcal{L}_{
 \widetilde{u}^{k_0}_\tau
                (\cdot; \tau -t_{n_i},
             \varphi_{n_i}  )}$
             and  
             radius 
             ${\frac 13} \delta$, where
               $ (\mathcal{P}
 (D_\rho), d_{
 \mathcal{P}
 (D_\rho)
 })$
 is the space of
 all probability measures
 in 
  the Skorohov space $(D([-\rho, 0],
 \ell^2), d^0)$.

  Next, we prove:
  \be\label{feb26a4}
 \{ \mathcal{L}_{
  {u} _\tau
                (\cdot; \tau -t_n,
             \varphi_n  )
 }\}_{n=N_2}^\infty
 \subseteq
\bigcup_{i=1}^{l}
 B\left (
 \mathcal{L}_{
 {u}_\tau
                (\cdot; \tau -t_{n_i},
             \varphi_{n_i}  )},\ 
 \   \delta
 \right ).
 \ee
   If  $n\ge N_2$, 
  then by
\eqref{feb26a3} we find that
there exists $i\in \{1,\cdots, l\}$ such that
  \be\label{feb26a5} 
 d_{
 \mathcal{P}
 (D_\rho)
 }
  \left (
   \mathcal{L}_{
 \widetilde{u}^{k_0}_\tau
                (\cdot; \tau -t_n,
             \varphi_n  )
 }, \
 \mathcal{L}_{
 \widetilde{u}^{k_0}_\tau
                (\cdot; \tau -t_{n_i},
             \varphi_{n_i}  )}
  \right )
  <{\frac 13}\delta.
  \ee  
   By \eqref{feb24a}-\eqref{feb24c},
  \eqref{feb26a2}
 and \eqref{feb26a5}
 we get  
\begin{align*}
 & d_{
 \mathcal{P}
 (D_\rho)
 }
 \left (
 \mathcal{L}_{   u_\tau   ( \cdot; \tau - t_n,  \varphi_n   )},
 \
 \mathcal{L}_{  u _\tau  ( \cdot; \tau - t_{n_i},  \varphi_{n_i}   )}
 \right )  \nonumber\\
 = &
 \sup_{\psi \in L_b
 (D_\rho), \|\psi\|_{L_b}
 \le 1}
 \left |
 \int_ {D_\rho}
 \psi   d \mathcal{L}_{   u_\tau   ( \cdot; \tau - t_n,  \varphi_n   )}
 -\int_ {D_\rho}
 \psi  d\mathcal{L}_{  u_\tau   ( \cdot; \tau - t_{n_i},  \varphi_{n_i}   )}
 \right |  \nonumber\\
 = &
 \sup_{\psi \in L_b(
 {D_\rho}
 ), \|\psi\|_{L_b}
 \le 1}
 \left |
 \E \left[
 \psi    (  u _\tau  ( \cdot; \tau - t_n,  \varphi_n   ) )
 \right]
 -\E \left[
 \psi  (  u_\tau   ( \cdot; \tau - t_{n_i},  \varphi_{n_i}   ))
 \right]
 \right |  \nonumber\\
 \le &
 \sup_{\psi \in L_b( {D_\rho}), \|\psi\|_{L_b}
 \le 1}
 \left |
 \E \left[
 \psi    (  u_\tau   ( \cdot; \tau - t_n,  \varphi_n   ) )
 \right]
 -\E \left[
 \psi  (  \widetilde{u}
 _\tau^{k_0}   ( \cdot; \tau - t_{n},  \varphi_{n}   ))
 \right]
 \right |  \nonumber\\
& +
 \sup_{\psi \in L_b( {D_\rho}), \|\psi\|_{L_b}
 \le 1}
 \left |
 \E \left[
 \psi( \widetilde{ u}
 _\tau^{k_0}   ( \cdot; \tau - t_n,  \varphi_n   ) )
 \right]
 -\E \left[
 \psi ( \widetilde{u}_\tau^{k_0}( \cdot; \tau - t_{n_i},  \varphi_{n_i} ) )
 \right]
 \right |  \nonumber\\
& +
 \sup_{\psi \in L_b(
 {D_\rho}
 ), \|\psi\|_{L_b}
 \le 1}
 \left |
 \E \left[
 \psi ( \widetilde{ u}_\tau^{k_0}
    ( \cdot; \tau - t_{n_i},  \varphi_{n_i}   ))
 \right]
 -
 \E \left[
 \psi( u_\tau    ( \cdot; \tau - t_{n_i},  \varphi_{n_i}   ))
 \right]
 \right |  \nonumber\\
 \le &
 \E \left[
d^0  (  u_\tau   ( \cdot; \tau - t_n,  \varphi_n  )
,\     \widetilde{u}
 _\tau^{k_0}   ( \cdot; \tau - t_{n},  \varphi_{n}    )
 \right]  \nonumber\\
& +  
d_{
 \mathcal{P}
 (D_\rho)
 }
 \left (
 \mathcal{L}_{   \widetilde{u}
 ^{k_0}
 _\tau   ( \cdot; \tau - t_n,  \varphi_n   )},
 \
 \mathcal{L}_{  \widetilde{u}
 ^{k_0}
  _\tau  ( \cdot; \tau - t_{n_i},  \varphi_{n_i}   )}
 \right )  \nonumber\\
& + 
 \E \left[
d^0  ( \widetilde{ u}_\tau^{k_0}
    ( \cdot; \tau - t_{n_i},  \varphi_{n_i}   ),
    \    u_\tau    ( \cdot; \tau - t_{n_i},  \varphi_{n_i}  ))
 \right]  \nonumber\\
 \le &
 \E \bigg[
 \sup_{r\in [-\rho, 0]}
\| u_\tau   ( r; \tau - t_n,  \varphi_n  )
-   \widetilde{u}
 _\tau^{k_0}   ( r; \tau - t_{n},  \varphi_{n}   )\|
 \bigg]  \nonumber\\
& +  
d_{ \mathcal{P}(D_\rho) }
 \left (
 \mathcal{L}_{ \widetilde{u}^{k_0}_\tau( \cdot; \tau - t_n,  \varphi_n   )},
 \
 \mathcal{L}_{ \widetilde{u}^{k_0}_\tau( \cdot; \tau - t_{n_i},  \varphi_{n_i}   )}
 \right )  \nonumber\\
& + 
 \E \bigg[
  \sup_{r\in [-\rho, 0]}
\| \widetilde{ u}_\tau^{k_0}
    (r; \tau - t_{n_i},  \varphi_{n_i}   )
    -
      u_\tau    ( r; \tau - t_{n_i},  \varphi_{n_i}  )\|
 \bigg]  \nonumber\\
= &
 \E \bigg[
 \sup_{r\in [-\rho, 0]}
\|  
   \widehat{u}
 _\tau^{k_0}   ( r; \tau - t_{n},  \varphi_{n}   )\|
 \bigg] 
  +  
d_{
 \mathcal{P}
 (D_\rho)
 }
 \left (
 \mathcal{L}_{   \widetilde{u}
 ^{k_0}
 _\tau   ( \cdot; \tau - t_n,  \varphi_n   )},
 \
 \mathcal{L}_{  \widetilde{u}
 ^{k_0}
  _\tau  ( \cdot; \tau - t_{n_i},  \varphi_{n_i}   )}
 \right )  \nonumber\\
& + 
 \E \bigg[
  \sup_{r\in [-\rho, 0]}
\| \widehat{ u}_\tau^{k_0}
    (r; \tau - t_{n_i},  \varphi_{n_i}   )
     \|
 \bigg] 
  <{\frac 13} \delta
 + {\frac 13} \delta +
 {\frac 13} \delta =\delta,
\end{align*}
and hence  \eqref{feb26a4}
is valid, which shows that
   the sequence
    $\{ \mathcal{L}_{
       u_\tau   ( \cdot; \tau - t_n,  \varphi_n   )  }
       \}_{n=1}^\infty$
 is tight in $ {\mathcal{P} (D_\rho)} $.
 Therefore, 
 there exists $\nu \in \mathcal{P} (D_\rho)$ such that,
 up to a subsequence,
 \be\label{feb26a10}
 \mathcal{L}_{
       u _\tau  ( \cdot; \tau - t_n,  \varphi_n   )  }
 \to \nu \ \text{weakly}.
 \ee

We now prove   $\nu \in \mathcal{P}_\theta (D_\rho)$.
 Let $K=\{K(\tau): \tau \in \R \}$
 be the closed $\cald$-pullback
 absorbing set of $\Phi$ given by 
 Theorem \ref{pas}.
  Then there exists
 $N_3=N_3 (\tau, D) \ge 1$ such that for all
 $n\ge N_3$,
 \be\label{feb26a11}
 \mathcal{L}_{  u_\tau (\cdot; \tau -t_n, \varphi_n)}
 \in  K(\tau).
\ee
By \eqref{feb26a10}-\eqref{feb26a11}
and the closedness of 
 $K(\tau)$,   we infer that 
  $\nu \in K(\tau)$,
  which implies 
 $\nu \in \mathcal{P}_\theta (D_\rho)$.
Consequently, we find that
the sequence
$\{\mathcal{L}_{
       u _\tau  ( \cdot; \tau - t_n,  \varphi_n   )  }
       \}_{n=1}^\infty$
       has a convergent subsequence
       in  $(\mathcal{P}_{\theta}(D_{\rho}), d_{\mathcal{P}(D_{\rho})})$,
       and thus the cocycle $\Phi$
       is asymptotically compact,
       which completes the proof.
 \end{proof}

 \begin{remark}
   If the stochastic equation
    \eqref{lmvlds-1} is   driven by Brownian motion
    only without jump noise, 
    then the tightness of solutions can be 
    proved  by the equi-continuity 
    of solutions in probability
    based on 
    the  continuity 
    of nonlinear terms 
     as in \cite{CWZ2025JNS, wwm2019SCL, ll2022DCDSB, csw2025PRSEA}. 
     However, 
    when  system  \eqref{lmvlds-1} 
    is driven by  multiplicative L\'{e}vy noise,
    it is difficult
    to establish the equi-continuity 
    of solution in probability
    by 
  the arguments
  of the aforementioned articles
  because, in this case,
  the necessary
  uniform estimates
  of the  higher-order moments of 
  the L\'{e}vy integral are unavailable.
   In the present paper,
   we employ  the tightness
   arguments of  \cite{J1986AIHP, A1978AOP} to 
   prove the asymptotic
   tightness of  a sequence of solutions in
   terms of the  Skorohod
   topology    driven by superlinear L\'{e}vy noise, 
  which  are   distinct from 
  the arguments  in \cite{RRG2006SPA, bc2016Stoch} 
  for the  delay   equations with linearly
  growing   L\'{e}vy noise. 
  The arguments of this paper are 
 also applicable  
 to establish the 
  pullback asymptotic compactness
  of 
 other types
   of stochastic delay equations driven by Brownian motion 
   as well as   L\'{e}vy noise  
    without using higher-order moment estimates.
 \end{remark}

 Next, we 
 present the  main result
 of the paper  on 
 the existence of 
 measure attractors
  of 
  the cocycle
  $ \Phi$ in
  the space  $(\mathcal{P}_{\theta}(D_{\rho}),
  d_{\mathcal{P}( D_{\rho})})$.

\begin{theorem}\label{pis}
   Suppose that {\bf(H1)}-{\bf(H4)}, \eqref{thetaassum2}, \eqref{dissip1} and \eqref{dissip2} hold
   with  $\theta \in \left(2,\, \tfrac{2(p-2)}{q-2}\right)$.
   Then 
   the cocycle $ \Phi $ associated with
   system  \eqref{lmvlds-2} has
   a unique $\mathcal{D}$-pullback measure attractor
   $\mathcal{A} = \{ \mathcal{A}(\tau) : \tau \in \mathbb{R} \} \in \mathcal{D}$ 
   in   $(\mathcal{P}_{\theta}(D_{\rho}),
  d_{\mathcal{P}( D_{\rho})})$
  in the sense that:
  
   (i)   $\mathcal{A}(\tau)$ is nonempty and
   compact for every
   $ \tau \in \mathbb{R}$.
   
   (ii)  {$ \mathcal{A} $ is invariant:
   $\Phi (t, \tau )  \mathcal{A}(\tau) 
   = \mathcal{A}(t+ \tau)$ for all $t \ge 0$
   and $\tau \in \mathbb{R}$. }
   
   (iii) $ \mathcal{A} $ pullback  attracts   every element in $\mathcal{D}$.

\end{theorem}

\begin{proof}
  By Lemma \ref{pas},  we know that
  $\Phi$ has a closed $\mathcal{D}$-pullback absorbing set
          $K = \{ K(\tau): \tau \in \mathbb{R} \} \in \mathcal{D}$.
  For  every  $\tau \in \mathbb{R}$, 
  let 
  $$
  \mathcal{A}(\tau)
     = \bigcap_{s \ge  0 } \overline{\bigcup_{t \ge s} \Phi(t, \tau-t) K(\tau - t)},
     $$
     where the closure
     is taken
     in  
   the space $(\mathcal{P}_{\theta}(D_{\rho}), 
   d_{\mathcal{P}(D_{\rho} ) } )$. 
   Note that  $ \mathcal{A}
   =\{  \mathcal{A}(\tau): \tau \in
   \mathbb{R}\}$ is the 
   $\mathit{\Omega}$-limit set of $K$ in 
    $(\mathcal{P}_{\theta}(D_{\rho}), 
   d_{\mathcal{P}(D_{\rho} ) } )$.  
   By the arguments of  \cite{bwang2012JDE}, 
   we find that
   for every
   $\tau \in \mathbb{R}$,
    $\mathcal{A}(\tau)$ is nonempty and
   compact, and  for every  
   $ D
   =\{  D(\tau): \tau \in
   \mathbb{R}\} \in \mathcal{D}$,
   $$
   \lim_{t\to \infty}
   d_{\mathcal{P}(D_{\rho} )}
   (\Phi (t, \tau -t) D(\tau-t), \
   \mathcal{A}(\tau) ) =0,
   $$ 
   where the distance between
   two sets is understood as
   the Hausdorff 
   semi-distance,
   which shows that
   $\mathcal{A}$
    satisfies conditions (i) and (iii).
    
     Next, we prove  $ \mathcal{A} $ 
      satisfies conditions (ii).
 Given $\tau \in \mathbb{R}$ and $t \ge 0$, 
 we will prove $\Phi(t, \tau) \mathcal{A}(\tau) = \mathcal{A}(\tau+t)$.
 
 Firstly, we need to show $\Phi(t, \tau) \mathcal{A}(\tau) \subseteq \mathcal{A}(\tau+t)$.
 For any $\mu \in \mathcal{A}(\tau)$, by the definition of $\mathcal{A}(\tau)$,
 we obtain that there exist $t_n \to +\infty$ and $\mu_n \in K(\tau - t_n) $ such that
 \begin{align}\label{invar-1}
       \Phi(t_n , \tau  - t_n) \mu_n \rightarrow \mu,
       \quad \text{as } \ n \to +\infty.
  \end{align}
 Then by \eqref{invar-1} and the cocycle property of $\Phi$, 
 we obtain that for any $t \ge 0$, 
  \begin{align} \label{invar-2}
     \Phi(t, \tau) \mu
     = & \Phi(t, \tau)  
              \lim_{n \to +\infty} \Phi(t_n , \tau  - t_n) \mu_n \nonumber\\
      =& \Phi(t, \tau)  
             \lim_{n \to +\infty} 
                   \Phi(\rho, \tau-\rho) \circ \Phi(t_n -\rho , \tau  - t_n) \mu_n.
 \end{align}
 
 By Lemma \ref{pac}, 
 we see that  $\{ \Phi(t_n -\rho , \tau  - t_n) \mu_n \}_{n \in \mathbb{N} }$ is precompact, 
 and hence there exists a subsequence 
  $\{ \Phi(t_{n_m} - \rho , \tau  - t_{n_m}) \mu_{n_m} \}_{m \in \mathbb{N} }$ and $\nu \in \mathcal{A}(\tau - \rho)$ such that
 $\Phi(t_{n_m} - \rho , \tau - t_{n_m}) \mu_{n_m}$ converges to $\nu$ in $(\mathcal{P}_{\theta}(D_{\rho}), d_{\mathcal{P}(D_{\rho} ) } )$ as $m \to +\infty$.
 
%
%
 
 By \eqref{invar-2}
 and the continuity of $\Phi(\rho, \tau-\rho)$ and $\Phi(t+\rho, \tau-\rho)$
 with $t\ge 0$  due to Theorem \ref{consys}, 
 we have
 \begin{align*}
       \Phi(t, \tau) \mu    
      =& \Phi(t, \tau) \circ 
                  \Phi(\rho, \tau-\rho) 
                      \lim_{m \to +\infty} \Phi(t_{n_m} -\rho , \tau - t_{n_m} ) \mu_{n_m}  \nonumber\\
     = & \Phi(t+\rho, \tau-\rho) 
                   \lim_{m \to +\infty} \Phi(t_{n_m} -\rho , \tau - t_{n_m} ) \mu_{n_m}  \nonumber\\       
     = & \lim_{m \to +\infty}
                \Phi(t+\rho, \tau-\rho) 
                     \circ \Phi(t_{n_m} -\rho , \tau - t_{n_m} ) \mu_{n_m} \nonumber\\   
     = & \lim_{m \to +\infty}
                \Phi(t_{n_m} +t, \tau - t_{n_m}) \mu_{n_m}  \nonumber\\                     
     = & \lim_{m \to +\infty}
                \Phi(t_{n_m} +t, \tau +t - (t_{n_m} + t) ) \mu_{n_m}
                \nonumber\\
                = & \lim_{m \to +\infty}
                \Phi(r_m, \tau +t - r_m ) \mu_{n_m},                                                                                                                    
                \ \ r_m \to  +\infty,
 \end{align*}
 which shows $\Phi(t, \tau) \mu \in  \mathcal{A}(\tau+t)$, 
 and hence 
 \begin{align} \label{invar-3}
       \Phi(t, \tau) \mathcal{A}(\tau) \subseteq \mathcal{A}(\tau+t),
             \quad \text{for all} \  t \ge 0 \  \text{and} \  \tau \in \mathbb{R}.
 \end{align}

  Next, we show the converse of \eqref{invar-3}.
 Note that for any $\bar{\mu} \in \mathcal{A}(\tau+t)$,
 there exists $t_n \to +\infty$ and $\bar{\mu}_n \in K(\tau + t - t_n)$ such that
 \begin{align}\label{invar-4}
       \Phi(t_n , \tau + t - t_n) \bar{\mu}_n \rightarrow \bar{\mu},
       \quad \text{as } \ n \to +\infty.
 \end{align}
  By the cocycle property of $\Phi$, we have
 \begin{align}\label{invar-5}
       \Phi(t_n , \tau + t - t_n) \bar{\mu}_n 
       = \Phi(t+\rho, \tau-\rho) \circ  \Phi(t_n - t - \rho , \tau + t - t_n) \bar{\mu}_n.
 \end{align}
   By Lemma \ref{pac},  
  $\{ \Phi(t_n - t -\rho, \tau + t - t_n) \bar{\mu}_n \}_{n \in \mathbb{N} }$ is precompact
  and hence
   there exist a subsequence  $\{ \Phi(t_{n_m} - t - \rho , \tau + t - t_{n_m}) \bar{\mu}_{n_m} \}_{m \in \mathbb{N} }$ and $\bar{\nu} \in \mathcal{A}(\tau - \rho)$ such that
  $\Phi(t_{n_m} - t - \rho , \tau + t - t_{n_m}) \bar{\mu}_{n_m}$ converges to $\bar{\nu}$ in $(\mathcal{P}_{\theta}(D_{\rho}), d_{\mathcal{P}(D_{\rho} ) } )$ as $m \to +\infty$.
  Then by \eqref{invar-5} and the continuity of $\Phi(t+\rho, \tau-\rho)$
  with $t\ge 0$  due to Theorem \ref{consys}, we have
 \begin{align}\label{invar-6}
       \lim_{m \to +\infty} \Phi(t_{n_m}, \tau + t - t_{n_m}) \bar{\mu}_{n_m}
       = & \Phi(t+\rho, \tau-\rho) 
                \lim_{m \to +\infty} 
                     \Phi(t_{n_m} - t - \rho , \tau + t - t_{n_m}) \bar{\mu}_{n_m}  \nonumber\\
       = & \Phi(t+\rho, \tau-\rho) \bar{\nu}  \nonumber\\
       = & \Phi(t, \tau) \circ \Phi(\rho, \tau - \rho) \bar{\nu}.                   
 \end{align}  
 By \eqref{invar-4} and \eqref{invar-6}, we obtain 
 $\bar{\mu} = \Phi(t, \tau) \circ \Phi(\rho, \tau - \rho) \bar{\nu}$.
 
 Noting $\bar{\nu} \in \mathcal{A}(\tau - \rho)$,  
 we find by \eqref{invar-3} that 
       $\Phi(\rho, \tau - \rho) \bar{\nu} \in \mathcal{A}(\tau)$.
 Then $\bar{\mu} \in \Phi(t, \tau) \mathcal{A}(\tau)$,
 which shows $\mathcal{A}(\tau+t) \subseteq \Phi(t, \tau) \mathcal{A}(\tau)$, and thus
  $\Phi(t, \tau) \mathcal{A}(\tau) = \mathcal{A}(\tau+t)$ for all $\tau \in \mathbb{R}$ and $t \ge 0$.
 \end{proof}

\begin{remark} \label{difference}
 Due to  time delay in \eqref{lmvlds-1},
 the solution $u(t; \tau, \varphi)$ of \eqref{lmvlds-1} is continuous 
 in initial values $\varphi$ in bounded subsets 
 of $\mathcal{P}_{\theta}(D_{\rho})$
 only for   $t \ge \tau + \rho$ as stated 
  in Theorem \ref{consys}.
 Consequently,  \cite[Proposition 2.2]{SSW2026JDE} 
 can not be directly
  applied  to obtain the existence of 
 measure attractors for \eqref{lmvlds-1}  in  
  $\mathcal{P}_{\theta}(D_{\rho})$ .
   
 If the equation \eqref{lmvlds-1} doesn't have time delay (i.e., $\rho =0$), 
 then the solutions are continuous in initial values in bounded subsets of $\mathcal{P}_{\theta}(\ell^2)$ for all $t \ge \tau$, 
 and thus \eqref{lmvlds-1} has a unique $\mathcal{D}$-pullback measure attractor in space $\mathcal{P}_\theta(\ell^2)$ as in \cite{BCS2026JDE}.
 \end{remark}

\section{Vanishing noise limit and optimal convergence rate of measure attractors} \label{5}

 In this section, we investigate the limiting behavior of measure attractors of
 the stochastic equation \eqref{lmvlds-2} as the noise intensity approaches zero.

 The limiting system of \eqref{lmvlds-2} for $\varepsilon = 0$ is given by:
 \begin{equation}\label{lds}
  \left\{
  \begin{aligned}
        & d u (t)
              + Au(t) dt
              + \lambda u(t) dt
              + f(u(t),\mathcal{L}_{u(t)}) dt 
          = g(t) dt + F( u(t-\rho), \mathcal{L}_{u(t-\rho) } ) dt, \ \ \  t > \tau,  \\
       & u(s) = \varphi(s-\tau) := ( \varphi_i(s-\tau) )_{i \in \mathbb{Z}},\  s \in  [\tau-\rho, \tau].
  \end{aligned}
  \right.
 \end{equation}

 Based on the results of the previous sections, 
 we know that for every $T>0$ and
       $\xi \in L^{\theta}(\Omega,\mathcal{F}_\tau; D_\rho)$, 
 \eqref{lmvlds-2} with $\varepsilon \in [0, 1]$ has a unique solution
       $u^\varepsilon(t; \tau, \xi) \in L^{\theta}(\Omega, \mathcal{F}_t; \ell^2)$
 defined for all $t \in [\tau, \tau +T]$,
 from which one can define a non-autonomous cocycle $\Phi^\varepsilon$ as in Section \ref{4} by 
        $\Phi^\varepsilon(t,\tau) \mu = P^{\ast,\varepsilon}_{\tau, \tau + t} \mu$
  for $\mu \in \mathcal{P}_\theta(D_\rho)$.
 Moreover, for small noise intensity $\varepsilon \in [0, \varepsilon^*)$ where $\varepsilon^*$ is the same constant as in Lemma \ref{pue},
 $\Phi^\varepsilon$ has a unique $\mathcal{D}$-pullback measure attractor
     $\mathcal{A}^\varepsilon = \{ \mathcal{A}^\varepsilon(\tau) : \tau \in \mathbb{R} \}$
 in $\mathcal{P}_{\theta}(D_\rho)$.
 
 In what follows,
 we will investigate the limiting behavior of $\mathcal{A}^\varepsilon$ as   $\varepsilon
 \to 0$   as well as the convergence rate
 in the case 
 where  $\mathcal{A}^\varepsilon$
 is a singleton.

\subsection{Vanishing noise limit of measure attractors}

 We first establish the convergence rate of solutions of \eqref{lmvlds-2} as $\varepsilon \to 0$.
 
 \begin{lemma}\label{solcon}
   Suppose that ({\bf H1})-({\bf H4}) hold.
   Then for any $\tau \in \mathbb{R}$, $T > 0$, and $R > 0$,
   there exists a constant $C = C(\tau, T, R) > 0$ such that 
   for all $t \in [\tau, \tau+T]$ and $\varepsilon \in (0,1]$,
   \begin{align*}
      \sup_{\varphi \in \mathbb{K}_R}
         \mathbb{E} 
                \left[ \| u^\varepsilon_{t}(\cdot; \tau, \varphi) 
                             - u^0_{t}(\cdot; \tau, \varphi) \|_{\infty}^2 
                \right]
      \le C \varepsilon,
   \end{align*}
   where
        $\mathbb{K}_R
          = \{ \varphi \in L^{2}(\Omega, \mathcal{F}_{\tau}; D_{\rho}) :
                  \mathbb{E} \left[ \|\varphi\|^{2}_{\infty} \right] \le R
            \}$.
 \end{lemma}
 
 \begin{proof}
  It follows from Theorem \ref{wp-sol} that for every
      $\tau \in \mathbb{R}$, $T> 0$ and $R > 0$,
  there exists a constant $C_1 = C_1(\tau, T, R) > 0$
  such that for all $\varphi \in \mathbb{K}_R$ and $\varepsilon \in [0,1]$,
  \begin{align} \label{scp-1}
        \mathbb{E}
              \left[ \sup_{s \in [\tau,\tau+T]}
                           \| u^\varepsilon(s; \tau, \varphi) \|^2
              \right]
           + \mathbb{E}
                \left[ \int_{\tau}^{\tau+T}
                             \| u^\varepsilon(s; \tau, \varphi) \|^p_p ds
                \right]
    \le C_1.
  \end{align}
 
  For simplicity,
  we write $u^\varepsilon(s; \tau, \varphi)$ as $u^\varepsilon(s)$ and $u^0(s; \tau, \varphi)$ as $u^0(s)$, respectively.
  By \eqref{lmvlds-2} and \eqref{lds} we have
  \begin{align}\label{scp-2}
     & d \Big( u^\varepsilon(t) - u^0(t) \Big)
       + \nu \left( A u^\varepsilon(t) - A u^0(t) \right) dt \nonumber\\
     & + \lambda \left(u^\varepsilon(t) - u(t) \right)  dt
       + \left( f(u^\varepsilon(t),\mathcal{L}_{u^\varepsilon(t)})
                 - f(u^0(t),\mathcal{L}_{u^0(t)} ) \right)dt \nonumber \\
   = & \left( F(u^\varepsilon(t -\rho), \mathcal{L}_{u^\varepsilon(t -\rho)} )
               - F(u^0(t-\rho), \mathcal{L}_{u^0(t-\rho)} ) \right) dt \nonumber\\
      & + \sqrt{\varepsilon} \sum_{k \in \mathbb{N} }
                \left( \sigma_k(u^\varepsilon(t), \mathcal{L}_{u^\varepsilon(t)})
                         - \sigma_k(u^0(t), \mathcal{L}_{u^0(t)})
                \right) dW_k(t)  \nonumber\\
      & + \sqrt{\varepsilon} \sum_{k \in \mathbb{N} }
                   \left( h_k(t)
                            + \sigma_k(u^0(t), \mathcal{L}_{u^0(t)}) 
                   \right) dW_k(t)  \nonumber \\
      & + \sqrt{\varepsilon} \int_{y \in \mathbb{Y} }
                  \left( \widetilde{\sigma}(u^\varepsilon(t-), \mathcal{L}_{u^\varepsilon(t)}, y)
                            - \widetilde{\sigma}(u^0(t), \mathcal{L}_{u^0(t)}, y)
                  \right) \widetilde{N}(dt,dy)  \nonumber \\
      & + \sqrt{\varepsilon} \int_{y \in \mathbb{Y} }
                    \left( \widetilde{h}(t,y) 
                             + \widetilde{\sigma}(u^0(t), \mathcal{L}_{u^0(t)}, y)
                    \right) \widetilde{N}(dt,dy).
  \end{align}
 
  Applying It\^o's formula to \eqref{scp-2} and then by the BDG inequality, we obtain that for all $t \in [\tau, \tau+T]$ and $\varepsilon \in (0,1]$,
  \begin{align}\label{scp-3}
         & \mathbb{E}
              \left[ \sup_{r \in [\tau, t]}
                        \bigg( \| u^\varepsilon(r) - u^0(r) \|^2 
                              +2 \int_{\tau}^{r}
                                   \langle f(u^\varepsilon(s), \mathcal{L}_{u^\epsilon(s)}) - f(u^{0}(s), \mathcal{L}_{u^{0}(s)}),
                                           u^\varepsilon(s) - u^{0}(s)
                                   \rangle ds
                        \bigg)
              \right]  \nonumber\\
   \le & 2 \mathbb{E}
               \left[ \int_{\tau}^{t}
                      \left| \langle F( u^\varepsilon(s-\rho), \mathcal{L}_{u^\varepsilon(s-\rho)}) - F( u^{0}(s-\rho), \mathcal{L}_{u^{0}(s-\rho)}),
                                     u^\varepsilon(s) - u^{0}(s)
                             \rangle
                       \right| ds
               \right]  \nonumber\\
        & + \frac{1}{4} \mathbb{E} \left[ \sup_{r \in [\tau, t]} \| u^\varepsilon(r) - u^0(r) \|^2 \| \right]  \nonumber\\
        & + (1+16c_1^2)\varepsilon \mathbb{E}
            \left[ \int_{\tau}^{t}
                      \sum_{k \in \mathbb{N} }
                            \|\sigma_k(u^\varepsilon(s), \mathcal{L}_{u^\varepsilon(s)}) - \sigma_k(u^{0}(s), \mathcal{L}_{u^{0}(s)}) \|^2 ds
            \right]  \nonumber\\
        & + (1+16c_1^2) \varepsilon \mathbb{E}
            \left[ \int_{\tau}^{t}
                      \sum_{k \in \mathbb{N} }
                            \| h_k(s) + \sigma_k(u^0(s), \mathcal{L}_{u^0(s)}) \|^2 ds
            \right]  \nonumber\\
        & + (1+16c_2^2)\varepsilon \mathbb{E}
            \left[ \int_{\tau}^{t}
                         \int_{y \in \mathbb{Y} }
                            \| \widetilde{\sigma}(u^\varepsilon(s-), \mathcal{L}_{u^\varepsilon(s)},y)
                              - \widetilde{\sigma}(u^{0}(s), \mathcal{L}_{u^{0}(s)},y) \|^2 \nu(dy) ds
            \right]  \nonumber\\
        & + (1+16c_2^2)\varepsilon \mathbb{E}
            \left[ \int_{\tau}^{t}
                         \int_{y \in \mathbb{Y} }
                            \| \widetilde{h}(s,y) 
                               + \widetilde{\sigma}(u^0(s), \mathcal{L}_{u^0(t)}, y) \|^2 \nu(dy) ds
            \right].
   \end{align}
 
  Similar to argument of \eqref{consys-12},
  by \eqref{scp-3} we can deduce that for all $t \in [\tau, \tau+T]$ and $\varepsilon \in (0,1]$,
  \begin{align}\label{scp-4}
         & \frac{1}{2} \mathbb{E} \left[ \sup_{r \in [\tau, t]} \| u^\varepsilon(r) - u^0(r) \|^2 \right]  \nonumber\\
         & + \lambda_2 \mathbb{E}
              \left[ \int_{\tau}^{t}
                        \sum_{i \in \mathbb{Z}}
                          \left( |u^\varepsilon_i(s)|^{p-2} + |u^0_i(s)|^{p-2}
                          \right) |u^\varepsilon_i(s) - u^0_i(s)|^2 ds
              \right]  \nonumber\\
 \le & C_2 \int_{\tau}^{t}
                            \mathbb{E} \left[ \|u^\varepsilon(s) -  u^0(s)\|^2 \right] ds  \nonumber\\
        & + 2(1+16c_1^2) \varepsilon \mathbb{E}
                  \left[ \int_{\tau}^{t}
                               \sum_{k \in \mathbb{N} }
                                     \| h_k(s) + \sigma_k(u^{0}(s), \mathcal{L}_{u^{0}(s)}) \|^2 ds
                  \right]  \nonumber\\
        & + 2(1+16c_2^2) \varepsilon \mathbb{E}
                \left[ \int_{\tau}^{t}
                             \int_{y \in \mathbb{Y} }
                                                      \| \widetilde{h}(s,y) 
                                                         + \widetilde{\sigma}(u^0(t), \mathcal{L}_{u^0(t)}, y)
                                                      \|^2 \nu(dy) ds
                \right],
  \end{align}
  where $ {C}_2 =C_2 (L, L_F)>0$ is a constant.
 
  By \eqref{sigma3} and \eqref{scp-1}, we get that for all $t \in [\tau, \tau+T]$,
  \begin{align}\label{scp-5}
         &2(1+16c_1^2) \varepsilon \mathbb{E}
              \left[ \int_{\tau}^{t}
                        \sum_{k \in \mathbb{N} }
                             \| h_k(s) + \sigma_k(u^0(s), \mathcal{L}_{u^0(s)}) \|^2 ds
              \right]  \nonumber \\
         & + 2(1+16c_2^2) \varepsilon \mathbb{E}
               \left[ \int_{\tau}^{t}  
                            \int_{y \in \mathbb{Y} }
                                 \| \widetilde{h}(s,y) 
                                    + \widetilde{\sigma}(u^0(s), \mathcal{L}_{u^0(s)},y) \|^2
                       \nu(dy) ds
               \right]  \nonumber \\
    \le & \frac{\lambda_1}{2} \varepsilon
              \mathbb{E} \left[ \int_{\tau}^{t} \| u^0(s) \|^p_p ds \right]
           + 16 \left( 1 + 16c_1^2 + 16c_2^2 \right) \varepsilon \|L\|^2
                  \int_{\tau}^{t} \mathbb{E} \left[ \|u^0(s)\|^2 \right]
                  ds  \nonumber\\
         & + 2 C_3 \varepsilon (t - \tau)  
             + 4 (1 + 16c_1^2 + 16c_2^2) \varepsilon 
                       \int_{\tau}^{t}
                           \left( \| h(s) \|^2 + \| \widetilde{h}(s) \|_{L^2(\mathbb{Y}, \nu; \ell^2) }^2
                           \right) ds  \nonumber \\                
    \le &C_4  \varepsilon
             + 4 (1 + 16c_1^2 + 16c_2^2) \varepsilon 
                        \int_{\tau}^{\tau+T}
                            \left( \| h(s) \|^2 + \| \widetilde{h}(s) \|_{L^2(\mathbb{Y}, \nu; \ell^2) }^2
                            \right) ds,
  \end{align}
  where $C_3 = C_3(\lambda_1, L) > 0$ and $C_4 = C_4(\tau, T, R, \lambda_1, L)>0 $ are constants.
 
  From \eqref{scp-4}, \eqref{scp-5} and Gronwall's inequality,
  it follows that for all $t \in [\tau, \tau+T]$ and $\varepsilon \in (0,1]$,
  \begin{align*}
         & \mathbb{E} \left[ \sup_{r \in [\tau, t]} \| u^\varepsilon(r) - u^0(r) \|^2 \right]  \nonumber\\
    \le &2 \left[ C_4  
                                 + 4 (1 + 16c_1^2 + 16c_2^2)
                                     \int_{\tau}^{\tau+T}
                                         \left( \| h(s) \|^2
                                                  + \| \widetilde{h}(s) \|_{L^2(\mathbb{Y}, \nu; \ell^2) }^2
                                        \right) ds
                \right] \varepsilon e^{2 C_2 t},
  \end{align*}
  as desired.
 \end{proof}
  
 Now, we show the convergence of measure attractors as the noise intensity $\varepsilon$ goes to $0$.

\begin{theorem}\label{uppercont}
   Suppose that {\bf(H1)}-{\bf(H4)}, \eqref{thetaassum2}, \eqref{dissip1} and \eqref{dissip2} hold.
   Then for every $\tau \in \mathbb{R}$,
   \begin{align*}
      \lim_{\varepsilon \to 0}
         d_{\mathcal{P}(D_\rho)}
           \left( \mathcal{A}^\varepsilon(\tau), \mathcal{A}^0(\tau) \right) = 0,
   \end{align*}
   where $d_{\mathcal{P}(D_\rho)}$ is the Hausdorff semi-metric between subsets of $\mathcal{P}(D_\rho)$.
\end{theorem}

\begin{proof}
 Let $ {K} = \{K(\tau): \tau \in \mathbb{R} \}$
 with
           $ K(\tau)      = \{\mu \in \mathcal{P}_{\theta}(D_\rho):
                               \int_{D_\rho} \| \xi \|_{\infty}^{\theta} \mu(d \xi)
                               \le R_{\tau}
                    \}$ 
 where $R_\tau$ is given by Lemma \ref{pas}.
 Then $ {K}$ is also a $\mathcal{D}$-pullback absorbing set of $\Phi^0$ in $
 (\mathcal{P}_{\theta}(D_\rho), d_{\mathcal{P}(D_\rho)})$.
 Since $\mathcal{A}^0$ is the $\mathcal{D}$-pullback measure attractor of $\Phi^0$,
 it follows that for any $\varepsilon' > 0$, there exists $T = T(\varepsilon',\tau) > 0$ such that
 \begin{align}\label{uppercont-1}
    d_{ \mathcal{P}(D_\rho) }
          \left( \Phi^0(T,\tau-T) K(\tau-T),
                 \mathcal{A}^0(\tau)
          \right)
    < \varepsilon'.
 \end{align}
 Since $ \mathcal{A}^\varepsilon(\tau - T) \subseteq K(\tau-T)$,
 it follows from \eqref{uppercont-1} that 
 \begin{align}\label{uppercont-2}
    \sup_{\mu \in \mathcal{A}^\varepsilon(\tau-T)}
          d_{\mathcal{P}(D_\rho) }
                 (\Phi^0(T, \tau - T) \mu, \mathcal{A}^0(\tau))
    \le \varepsilon'.
 \end{align}
 By the invariance of $\mathcal{A}^\varepsilon$ and \eqref{uppercont-2},
 we obtain
 \begin{align} \label{uppercont-3}
        d_{\mathcal{P}(D_\rho) }
              \left( \mathcal{A}^\varepsilon(\tau),
                     \mathcal{A}^0(\tau)
              \right)
    = & d_{ \mathcal{P}(D_\rho) }
              \left( \Phi^\varepsilon(T,\tau-T) \mathcal{A}^\varepsilon(\tau-T),
                     \mathcal{A}^0(\tau)
              \right)  \nonumber\\
    = & \sup_{\nu^\varepsilon \in \mathcal{A}^\varepsilon(\tau-T) }
              d_{\mathcal{P}(D_\rho) }
                  \left( \Phi^\varepsilon(T,\tau-T) \nu^\varepsilon,
                         \mathcal{A}^0(\tau)
                  \right)  \nonumber\\
 \le & \sup_{ \nu^\varepsilon \in \mathcal{A}^\varepsilon(\tau-T) }
                     d_{ \mathcal{P}(D_\rho) }
                              \left( \Phi^\varepsilon(T,\tau-T) \nu^\varepsilon, \Phi^0(T,\tau-T) \nu^\varepsilon
                              \right)  \nonumber\\
      & + \sup_{ \nu^\varepsilon \in \mathcal{A}^\varepsilon(\tau-T) }
                     d_{ \mathcal{P}(D_\rho) }
                              \left( \Phi^0(T,\tau-T) \nu^\varepsilon, \mathcal{A}^0(\tau)
                              \right)  \nonumber\\
 \le & \sup_{ \nu^\varepsilon \in \mathcal{A}^\varepsilon(\tau-T) }
                     d_{ \mathcal{P}(D_\rho) }
                              \left( \Phi^\varepsilon(T,\tau-T) \nu^\varepsilon, \Phi^0(T,\tau-T) \nu^\varepsilon
                              \right)
        + \varepsilon'.
 \end{align}

 For any $\mu^\varepsilon \in \mathcal{A}^\varepsilon(\tau-T)$,
 let $\xi^\varepsilon \in L^{\theta}(\Omega,\mathcal{F}_{\tau-T}; D_\rho)$ such that 
 its  law $\mathcal{L}_{\xi^\varepsilon} = \mu^\varepsilon$.
 Then
 $\mathbb{E} \left[ \| \xi^\varepsilon \|_{\infty}^{\theta} \right] \le R_{\tau-T}$.
By  Lemma \ref{solcon}, there exists a constant $C = (\tau, T, \varepsilon') > 0$ independent of 
$\varepsilon$
and $\mu^\varepsilon$ 
such that
for all $\varepsilon>0$
and
 $\mu^\varepsilon \in \mathcal{A}^\varepsilon(\tau-T)$,
 \begin{align*}
        & d_{\mathcal{P}(D_\rho)}
             \left( \Phi^\varepsilon(T, \tau-T) \mu^\varepsilon,
                    \Phi^{0}(T, \tau-T) \mu^\varepsilon
             \right) \nonumber\\
      = & \sup_{\psi \in L_b(D_\rho),\|\psi\|_{L_b} \le 1}
             \left| \langle \psi,\ \Phi^\varepsilon(T, \tau-T) \mu^\varepsilon \rangle
                    - \langle \psi,\Phi^{0}(T, \tau-T) \mu^\varepsilon \rangle
             \right| \nonumber\\
   \le & \sup_{\psi \in L_b(D_\rho),\|\psi\|_{L_b} \le 1}
           \mathbb{E}
              \left[ | \psi(u^\varepsilon_\tau(\cdot; \tau-T, \xi^\varepsilon) )
                       - \psi(u^0_\tau(\cdot; \tau-T, \xi^\varepsilon) ) |
              \right]  \nonumber\\
   \le & \mathbb{E}
              \left[ \| (u^\varepsilon_\tau(\cdot; \tau-T, \xi^\varepsilon)
                          - u^0_\tau(\cdot; \tau-T, \xi^\varepsilon)
                       \|_{\infty}
              \right] 
     < \sqrt{C \varepsilon},
 \end{align*}
 which implies that for all $\varepsilon>0$,
 \begin{align} \label{uppercont-4}
    \sup_{ \nu^\varepsilon \in \mathcal{A}^\varepsilon(\tau-T) }
       d_{\mathcal{P}(D_\rho) }
          \left( \Phi^\varepsilon(T, \tau-T) \mu^\varepsilon,
                    \Phi^{0}(T, \tau-T) \mu^\varepsilon
          \right)
    \le \sqrt{C \varepsilon}.
 \end{align}
 By \eqref{uppercont-3} and \eqref{uppercont-4},
 we obtain that for every
 $\varepsilon'>0$ and
 $\varepsilon \in (0, \varepsilon^{*})$,
 \begin{align*}
    d_{\mathcal{P}(D_\rho) }
          \left( \mathcal{A}^\varepsilon(\tau),
                 \mathcal{A}^0(\tau)
          \right)
    \le \sqrt{C \varepsilon} + \varepsilon',
 \end{align*}
which shows that
 $\limsup_{\varepsilon \to 0 }d_{ \mathcal{P}(D_\rho) }
          \left( \mathcal{A}^\varepsilon(\tau),
                 \mathcal{A}^0(\tau)
          \right)
    \le \varepsilon',
 $
as desired.
\end{proof}

  Next,
  we consider the 
   optimal rate of convergence
   of measure attractors
   as $\varepsilon \to 0$
   in the  case where 
   every measure attractor
   $ \mathcal{A}^\varepsilon$
   is a singleton, which   implies the optimal convergence rate of an evolution family of measures for \eqref{lmvlds-2}
   (see the related definitions of the evolution family of measures in \cite{DR2007, LL2025CMP}).

\subsection{Optimal convergence rate of singleton measure attractors}
In this subsection,
under a stronger dissipative
condition, 
 we first prove 
 the measure attractor 
 of \eqref{lmvlds-2} is a singleton,
 and then establish the   optimal convergence rate of 
 measure  attractors
 as  the  noise intensity
 approaches zero. 
 
 We  now further assume:
 \begin{align}\label{dissip1-2}
	\lambda
	 > \|\phi_{4}\|_{\ell^\infty} + \|\phi_{4}\|_1
	    + \sqrt{2}\|L_{F}\|_{\ell^\infty}
	    + \sqrt{2}\|L_{F}\|.
 \end{align}
  Then for given $\theta \in \left(2, \tfrac{2(p-2)}{q-2} \right)$, 
  there exist $\beta > 0$ and $\delta \in (0, 1)$ such that
  for all $\rho \in [0,1]$,
   \begin{align}\label{dissip1-2-1}
                \beta - 2\lambda + 2\|\phi_{4}\|_{\ell^\infty} + 2\|\phi_{4}\|_1
   	                     + \sqrt{2}( \|L_{F}\|_{\ell^\infty} + \|L_{F}\| )
   	                         (1+e^{\beta \rho} )
   	                    < 0,
  \end{align}
  and 
  \begin{align}\label{dissip1-3}
 	   & \theta \lambda 
 	                       - 2 \beta 
 	                       - \theta \left( \|\phi_4\|_{\ell^\infty} +  \|\phi_4\|_1 \right) 
 	                       - \sqrt{2}
 	                                   \big( \|L_{F}\|_{\ell^\infty} + \|L_{F}\| \big) 
 	                                   \big( e^{\beta \rho} + (\theta - 1) \big)  \nonumber\\
 	    & - \frac{(\theta - 1)(\theta - 2)}{2} (1 + 2^{\theta-2}) \delta
 	     > 0.
  \end{align}
 
 We now fix a constant $\beta$ satisfying \eqref{dissip1-1}-\eqref{dissip2}, \eqref{dissip1-2-1} and \eqref{dissip1-3}.
  Next, we show the continuity of solutions of \eqref{lmvlds-2} in $L^2(\Omega, D_\rho)$ with respect to initial data.

\begin{lemma}\label{sdsec}
 Suppose that ({\bf H1})-({\bf H4}), \eqref{31_1}, \eqref{34_1}  and \eqref{dissip1-2} hold.
 Then there exists $\varepsilon_0 \in (0, \varepsilon^*]$ such that for every 
  $\varepsilon \in (0, \varepsilon_0)$, 
  $\tau \in \mathbb{R}$,
 $\rho \in [0,1]$ 
 and $t \ge 2\rho$, 
 the solution $u^{\varepsilon} (\cdot; \tau-t, \varphi_i)$ of \eqref{lmvlds-2} satisfies
 \begin{align*}
       \mathbb{E} 
    	             \left[ \| u^\varepsilon_\tau(\cdot; \tau -t, \varphi_1) 
    	                         - u^\varepsilon_\tau(\cdot; \tau -t, \varphi_2)  
    	                      \|_{\infty}^2
    	             \right] 
   \le \mathfrak{C}
                 e^{ - \beta t} 
    	              \mathbb{E}
    	                     \left[ \| \varphi_1 - \varphi_2 \|_\infty^2
    	                     \right],
 \end{align*}
 where $\varphi_i \in L^2(\Omega, \mathcal{F}_{\tau-t}; D_\rho )$,
 and $\mathfrak{C}$ is a positive constant independent of $\varepsilon$, $\varphi_i$, $\tau$, $t$ and $\rho \in [0,1]$.
\end{lemma}

\begin{proof}
 For convenience, denote by 
        $u^{i}(r) = u^{\varepsilon}(r; \tau-t, \varphi_i)$, $i = 1,2$.
 By \eqref{lmvlds-2} and It\^{o}'s formula, 
 we obtain that for any $r \in [\tau-2\rho, \tau]$ and $t \ge 2\rho$,
 \begin{align}\label{sdsec-1}
       & e^{\beta r} \mathbb{E} \left[ \| u^1(r) - u^2(r) \|^2 \| \right]
          + 2 \mathbb{E} \left[ \int_{\tau-t}^{r} e^{\beta s}
                                 \|B(u^1(s) - u^2(s))\|_p^p ds \right] \nonumber\\
       & + (2 \lambda -\beta  \mathbb{E}
                     \left[ \int_{\tau-t}^{r} e^{\beta s} \|u^1(s) - u^2(s)\|^2 ds \right]  \nonumber\\
       & + 2 \mathbb{E}
              \left[ \int_{\tau-t}^{r} e^{\beta s}
                           \langle f(u^1(s), \mathcal{L}_{u^1(s) } ) - f(u^2(s), \mathcal{L}_{u^2(s) } ),
                                       u^1(s) - u^2(s)
                           \rangle ds
              \right]  \nonumber\\
  \le & e^{\beta(\tau-t)} \mathbb{E} \left[ \|  \varphi_1
  (0) - \varphi_2 (0) \|^2 \right]  \nonumber \\
       & + 2 \mathbb{E}
                   \left[ \int_{\tau-t}^{r} e^{\beta s}
                                 \langle F( u^1(s-\rho), \mathcal{L}_{u^1(s-\rho)}) - F(  u^2(s), \mathcal{L}_{u^2(s)}),
                                             u^1(s) - u^2(s)
                                 \rangle ds
                   \right]  \nonumber\\
       & + 2 \varepsilon \mathbb{E}
                  \left[ \int_{\tau-t}^{r} e^{\beta s}
                              \sum_{k \in \mathbb{N} }
                                     \|\sigma_k(u^1(s), \mathcal{L}_{u^1(s)}) -  \sigma_k(u^2(s), \mathcal{L}_{u^2(s)}) \|^2 ds
                  \right]  \nonumber\\
       & + 2 \varepsilon \mathbb{E}
                      \left[ \int_{\tau-t}^{r} e^{\beta s}
                                   \int_{ y \in \mathbb{Y} }
                                          \| \widetilde{\sigma}(u^1(s),  \mathcal{L}_{u^1(s)},y)
                                             - \widetilde{\sigma}(u^2(s), \mathcal{L}_{u^2(s)},y) 
                                           \|^2 \nu(dy) ds
                      \right].
 \end{align}

 For the fourth term on left-hand side of \eqref{sdsec-1},
 by \eqref{f3}, we obtain
 \begin{align}\label{sdsec-2}
       & 2 \mathbb{E}
             \left[ \int_{\tau-t}^{r} e^{\beta s}
                        \langle f(u^1(s), \mathcal{L}_{u^1(s)}) - f(u^2(s), \mathcal{L}_{u^2(s)}),
                                u^1(s) - u^2(s)
                        \rangle ds
             \right]  \nonumber\\
  \ge & \lambda_2 \mathbb{E}
             \left[ \int_{\tau-t}^{r} e^{\beta s} 
                          \sum_{ i \in \mathbb{Z}} 
             	                                    \left( |u_{i}^1(s)|^{p-2} + |u_{i}^2(s)|^{p-2} \right)    
             	                                    |u_{i} ^1(s) - u_{i}^2(s)|^2 ds
             \right]  \nonumber\\
         & - 2 \left( \|\phi_{4}\|_{\ell^\infty} + \|\phi_{4}\|_1 \right)
             \int_{\tau-t}^{r} e^{\beta s} \mathbb{E} \left[ \|u^1(s) - u^2(s) \|^2 \right] ds.
\end{align}

For the second term on right-hand side of \eqref{sdsec-1},
by \eqref{F2}, we obtain
\begin{align}\label{sdsec-3}
	& 2 \mathbb{E}
	            \left[ \int_{\tau-t}^{r} e^{\beta s}
	                         \langle F( u^1(s-\rho), \mathcal{L}_{u^2(s-\rho)}) 
	                                    - F(u^2(s-\rho), \mathcal{L}_{u^2(s-\rho)}),
	                                    u^1(s) - u^2(s)
	                         \rangle ds
	            \right]  \nonumber\\
	\le & 2 \mathbb{E}
	                  \left[ \int_{\tau-t}^{r} e^{\beta s}
	                               \|F( u^1(s-\rho), \mathcal{L}_{u^1(s-\rho)}) 
	                                  - F( u^2(s-\rho), \mathcal{L}_{u^2(s-\rho)}) \|
	                                  \|u^1(s) - u^2(s)\| ds
	                  \right]  \nonumber\\ 
	 \le & 
	          2  \mathbb{E}
	                \bigg[ \int_{\tau-t}^{r}  e^{\beta s}
	                        \left( 2 \|L_{F}\|_{\ell^\infty}^2  \| u^1(s-\rho) - u^2(s-\rho) \|^2 
	                        + 2 \|L_{F}\|^2\mathbb{E} \left[ \| u^1(s-\rho) - u^2(s-\rho) \|^2 \right]
	                        \right) ^{\frac{1}{2}}
	                       \nonumber \\
	              & \qquad\quad\ \
	               \cdot      \| u^1(s) - u^2(s) \|
	                        ds
	                \bigg]  \nonumber\\                          
	    = & \sqrt{2}( \|L_{F}\|_{\ell^\infty} +\|L_{F}\|)  e^{\beta (\rho+\tau-t)}
                    \mathbb{E}
                         \left[ \int_{-\rho}^{0}
                                      \| \varphi_1(s) - \varphi_2(s) \|^2 ds
                         \right]  \nonumber\\
           & + \sqrt{2}(\|L_{F}\|_{\ell^\infty}+\|L_{F}\|) (1+e^{\beta \rho})
                     \int_{\tau-t}^{r} e^{\beta^{*} s}
                          \mathbb{E}
                               \left[ \| u^1(s) - u^2(s) \|^2 
                               \right] 
                     ds. 
\end{align} 

For the third and fourth terms on right-hand side of \eqref{sdsec-1},
by \eqref{sigma4}, we obtain
\begin{align}\label{sdsec-4}
	  & 2 \varepsilon \mathbb{E}
               \left[ \int_{\tau-t}^{r} e^{\beta s}
            	            \sum_{k \in \mathbb{N} }
	                               \| \sigma_k(u^1(s), \mathcal{L}_{u^1(s)}) 
	                                  -  \sigma_k(u^2(s), \mathcal{L}_{u^2(s)}) \|^2 ds
         	   \right]  \nonumber\\
	  & + 2 \varepsilon \mathbb{E}
             	  \left[ \int_{\tau-t}^{r} e^{\beta s}
	                           \int_{y \in \mathbb{Y} }
	                                 \| \widetilde{\sigma}(u^1(s), \mathcal{L}_{u^1(s)},y)
	                                    - \widetilde{\sigma}(u^2(s), \mathcal{L}_{u^2(s)},y) \|^2 
	                                 \nu(dy) ds
	              \right]  \nonumber\\
 \le &  \frac{\lambda_2 \varepsilon}{2} 
	         \mathbb{E}
  	                \left[  \int_{\tau-t}^{r} e^{\beta s}
	                              \sum_{ i \in \mathbb{Z}} 
	                                    \left( |u_{i}^1(s)|^{p-2} + |u_{i}^2(s)|^{p-2} \right)    
	                                    |u_{i} ^1(s) - u_{i}^2(s)|^2 ds
	                \right]   \nonumber \\
	  & + 4 \varepsilon \| L\|^2 
	              \int_{\tau-t}^{r} e^{\beta s}
	                  \mathbb{W}_2^2(\mathcal{L}_{u^1(s)},  \mathcal{L}_{u^2(s)} ) ds
	        + 2 \varepsilon C_3
	              \mathbb{E}
	                    \left[  \int_{\tau-t}^{r} e^{\beta s} 
	                                   \|u^{\varepsilon_1}(s) - u^{\varepsilon_2}(s) \|^2 ds
	                    \right]  \nonumber\\
 \le &  \frac{\lambda_2}{2} 
	        \mathbb{E}
	              \left[  \int_{\tau-t}^{r} e^{\beta s}
	                            \sum_{ i \in \mathbb{Z}} 
	                                  \left( |u_{i}^1(s)|^{p-2} + |u_{i}^2(s)|^{p-2} \right)    
	                                  |u_{i} ^1(s) - u_{i}^2(s)|^2 ds
	             \right]   \nonumber \\
	  & + \varepsilon 
	              \left( 4 \| L \|^2 + 2 C_3
                 \right)
	                 \mathbb{E}
	                        \left[ \int_{\tau-t}^{r} e^{\beta s} 
	                                     \|u^1(s) - u^2(s) \|^2 ds
	                        \right].	               
\end{align}
 
 From \eqref{sdsec-1}-\eqref{sdsec-4}, it follows that
 \begin{align}\label{sdsec-5}
 	   & e^{\beta r} \mathbb{E}
 	            \left[ \| u^1(r) - u^2(r) \|^2 \| \right]
 	       + \frac{\lambda_2}{2} \mathbb{E}
                 	\left[ \int_{\tau-t}^{r} e^{\beta s}
                 	             \sum_{ i \in \mathbb{Z}} 
         	                            \left( |u_{i}^1(s)|^{p-2} + |u_{i}^2(s)|^{p-2} \right)    
         	                            |u_{i}^1(s) - u_{i}^2(s)|^2
         	                  ds
 	                \right]  \nonumber\\
 \le & e^{\beta (\tau-t)} \mathbb{E} \left[ \|  \varphi_1
 (0)  - \varphi_2 (0)  \|^2 \right]
            + \sqrt{2}( \|L_{F}\|_{\ell^\infty}+ \|L_{F}\| ) 
               e^{ \beta(\rho+\tau-t)}
                     \mathbb{E}
                             \left[ \int_{-\rho}^{0} \| \varphi_1(s) - \varphi_2(s) \|^2 ds
                            \right]  \nonumber\\
        & + \left[ \beta - 2\lambda + 2\|\phi_{4}\|_{\ell^\infty} + 2\|\phi_{4}\|_1
 	                     + \sqrt{2}( \|L_{F}\|_{\ell^\infty} + \|L_{F}\| )
 	                         (1+e^{\beta \rho} )
 	                     + \varepsilon
 	                              \left( 4 \| L \|^2 + 2 C_3
 	                              \right)   
 	            \right]  \nonumber\\
 	    & \quad  
 	            \cdot \int_{\tau-t}^{r} e^{\beta s} 
 	                          \mathbb{E} 
                                     \left[ \|u^1(s) - u^2(s) \|^2 \right] ds.
 \end{align}
 
 By \eqref{dissip1-2-1} and \eqref{sdsec-5}, 
 we obtain that there exists $\varepsilon_0 \in (0, \varepsilon^*]$ such that  for all 
       $\varepsilon \in (0, \varepsilon_0)$, 
       $\rho \in [0,1]$,
       $r \in [\tau-2\rho, \tau]$ and $t \ge 2\rho$,
 \begin{align}\label{sdsec-6}
       & \mathbb{E} \left[ \| u^1(r) - u^2(r) \|^2 \| \right]   
           + \frac{\lambda_2}{2} \mathbb{E}
                   \left[ \int_{\tau-t}^{r} e^{\beta (s-r)}
                                \sum_{ i \in \mathbb{Z}} 
                                       \left( |u_{i}^1(s)|^{p-2} + |u_{i}^2(s)|^{p-2} \right)    
                                       |u_{i} ^1(s) - u_{i}^2(s)|^2
                             ds
                   \right]  \nonumber\\
\le & e^{\beta (\tau-t-r)} \mathbb{E} \left[ \|  \varphi_1
(0)  - \varphi_2 (0) \|^2 \right]
 	       + \sqrt{2}( \|L_{F}\|_{\ell^\infty} + \|L_{F}\|)
 	          e^{ \beta (\rho+\tau - t - r)} 
        	       \mathbb{E}
           	              \left[ \int_{-\rho}^{0}
 	                                    \| \varphi_1(s) - \varphi_2(s) \|^2 ds
             	         \right]  \nonumber\\        
 \le & \mathfrak{C}_1 e^{ -\beta t} 
  	           \mathbb{E}
  	                 \left[ \| \varphi_1 - \varphi_2 \|_{\infty}^2
  	                 \right],   
 \end{align}
 where $ \mathfrak{C}_1$ is positive constant independent of $\varepsilon$, $D, \varphi_i$, $\tau$, $t$, $r$ and $\rho \in [0,1]$.
 
 By \eqref{lmvlds-1}, It\^{o}'s formula and BDG's inequality, 
 we obtain for all $t \ge 2\rho$,
 \begin{align}\label{sdsec-8}
 	   & \mathbb{E} 
 	           \bigg[ \sup_{r \in [\tau-\rho, \tau] } 
 	                    \Big( \| u^1(r) - u^2(r) \|^2  \nonumber\\
 	   & \qquad\qquad\quad                  
 	                             + 2 \int_{\tau-\rho}^{r}
 	                                   \langle f(u^1(s), \mathcal{L}_{u^1(s)}) 
 	                                                - f(u^2(s), \mathcal{L}_{u^2(s)}),
 	                                               u^1(s) - u^2(s)
 	                                  \rangle ds
 	                     \Big)
 	             \bigg]  \nonumber\\
\le & \mathbb{E} 
                  \left[ \|  u^1(\tau-\rho) - u^2(\tau-\rho) \|^2 \right]  \nonumber\\
       & + 2 \mathbb{E}
               \left[ \int_{\tau-\rho}^{\tau}
                            \left|
                                 \langle F( u^1(s-\rho),\mathcal{L}_{u^1(s-\rho)})
                                             - F( u^2(s), \mathcal{L}_{u^2(s)}),
                                             u^1(s) - u^2(s)
                                \rangle
                            \right| ds
                \right]  \nonumber\\
       & + \frac{1}{4} \mathbb{E} 
                     \left[ \sup_{r \in [\tau-\rho, \tau] }  \|u^1(r) - u^2(r) \|^2 
                     \right]  \nonumber\\
       & + \left(8 c_1^2 + 1\right) \varepsilon
                \mathbb{E}
                      \left[ \int_{\tau-\rho}^{\tau}
                                   \sum_{k \in \mathbb{N} }
                                               \| \sigma_k(u^1(s),  \mathcal{L}_{u^1(s)}) 
                                                    - \sigma_k(u^2(s),  \mathcal{L}_{u^2(s)})
                                               \|^2 
                                ds
                      \right]   \nonumber\\
       & + \left(8 c_2^2 + 1 \right) \varepsilon 
                  \mathbb{E}
                       \left[ \int_{\tau-\rho}^{\tau}
                                    \int_{ y \in \mathbb{Y} }
                                        \| \widetilde{\sigma}_k(u^{1}(s),  \mathcal{L}_{u^{1}(s)}, y) 
                                           - \widetilde{\sigma}_k(u^{2}(s),  \mathcal{L}_{u^{_2}(s)}, y) \|^2
                                 \nu(dy) ds
                       \right].     	               
 \end{align}
 
 By \eqref{F2}, \eqref{sdsec-2} and \eqref{sdsec-8}, we get that for all $t \ge 2 \rho$,
  \begin{align}\label{sdsec-9}
  	    & \mathbb{E} 
  	           \bigg[ \sup_{r \in [\tau-\rho, \tau] } 
  	                    \Big( \| u^1(r) - u^2(r) \|^2  \nonumber\\
  	    & \qquad\qquad\quad \                 
  	                             + \lambda_2 \int_{\tau-t}^{r}
  	                                                        \sum_{ i \in \mathbb{Z}} 
  	                                                               \left( |u_{i}^1(s)|^{p-2} + |u_{i}^2(s)|^{p-2} \right)    
  	                                                                |u_{i} ^1(s) - u_{i}^2(s)|^2 
  	                                                     ds
  	                     \Big)
  	             \bigg]  \nonumber\\
 \le & \mathbb{E} 
                   \left[ \|  u^1(\tau-\rho) - u^2(\tau-\rho) \|^2 \right] 
             + 2 \left( \|\phi_{4}\|_{\ell^\infty} + \|\phi_{4}\|_1 \right)
                     	                                          \int_{\tau-\rho}^{\tau} 
                     	                                               \mathbb{E} \left[ \|u^1(s) - u^2(s) \|^2 \right] 
                     	                                           ds  \nonumber\\
      & + 2 \left( \|L_{F}\|_{\ell^\infty}^2 + \|L_{F}\|^2 \right)
                              \int_{\tau-\rho}^{\tau} 
                                   \mathbb{E}
                                        \left[ \| u^1(s-\rho) - u^2(s-\rho) \|^2 
                                        \right] 
                              ds 
           +  \int_{\tau-\rho}^{\tau} 
                                             \mathbb{E}
                                                     \left[ \| u^1(s) - u^2(s) \|^2 
                                                     \right] 
                   ds  \nonumber\\
       & + \frac{1}{4} \mathbb{E} 
                     \left[ \sup_{r \in [\tau-\rho, \tau] }  \|u^1(r) - u^2(r) \|^2 
                     \right]  \nonumber\\
       & + \left(8 c_1^2 + 1 \right) \varepsilon
                \mathbb{E}
                      \left[ \int_{\tau-\rho}^{\tau}
                                   \sum_{k \in \mathbb{N} }
                                               \| \sigma_k(u^1(s), \mathcal{L}_{u^1(s)}) 
                                                    - \sigma_k(u^2(s), \mathcal{L}_{u^2(s)})
                                               \|^2 
                                ds
                      \right]   \nonumber\\
       & + \left(8 c_2^2 + 1 \right) \varepsilon 
                  \mathbb{E}
                       \left[ \int_{\tau-\rho}^{\tau}
                                    \int_{ y \in \mathbb{Y} }
                                        \| \widetilde{\sigma}(u^1(s), \mathcal{L}_{u^1(s)}, y) 
                                           - \widetilde{\sigma}(u^2(s), \mathcal{L}_{u^2(s)}, y) \|^2
                                 \nu(dy) ds
                       \right].
  \end{align}
 
  Similar to argument of \eqref{sdsec-4}, we have
 \begin{align}\label{sdsec-10}
 	   & \left(8 c_1^2 + 1 \right) \varepsilon \mathbb{E}
                \left[ \int_{\tau-\rho}^{\tau} 
             	            \sum_{k \in \mathbb{N} }
 	                               \| \sigma_k(u^1(s), \mathcal{L}_{u^1(s)}) 
 	                                  -  \sigma_k(u^2(s), \mathcal{L}_{u^2(s)}) \|^2 ds
          	   \right]  \nonumber\\
 	   & + (8 c_2^2 + 1 ) \varepsilon \mathbb{E}
              	  \left[ \int_{\tau-\rho}^{\tau}
 	                           \int_{ y \in \mathbb{Z} }
 	                                 \| \widetilde{\sigma}(u^1(s), \mathcal{L}_{u^1(s)}, y)
 	                                    - \widetilde{\sigma}(u^2(s), \mathcal{L}_{u^2(s)}, y) \|^2 
 	                                 \nu(dy) ds
 	              \right]  \nonumber\\
 \le &  \frac{\lambda_2 \varepsilon}{2} 
 	         \mathbb{E}
   	                \left[  \int_{\tau-\rho}^{\tau}
 	                              \sum_{ i \in \mathbb{Z}} 
 	                                    \left( |u_{i}^1(s)|^{p-2} + |u_{i}^2(s)|^{p-2} \right)    
 	                                    |u_{i} ^1(s) - u_{i}^2(s)|^2 ds
 	                \right]   \nonumber \\
 	    & + \varepsilon 
 	           \left[ 8 c_1^2 + 8 c_2^2 + 2 ) \| L \|^2 
 	                    + 2C_3
               \right]
 	          \mathbb{E}
 	                \left[  \int_{\tau-\rho}^{\tau} 
 	                              \|u^1(s) - u^2(s) \|^2 ds
 	                \right].
 \end{align}
 
 By \eqref{sdsec-9}-\eqref{sdsec-10}, we obtain that for all $t \ge 2\rho$,
 \begin{align} \label{sdsec-12}
  	   & \frac{1}{2} \mathbb{E} 
  	             \left[ \sup_{r \in [\tau-\rho, \tau] } 
  	                         \| u^1(r) - u^2(r) \|^2  
  	             \right]  \nonumber\\
\le & 2 \mathbb{E} 
                   \left[ \|  u^1(\tau-\rho) - u^2(\tau-\rho) \|^2 \right] 
             + 4 \left( \|\phi_{4}\|_{\ell^\infty} + \|\phi_{4}\|_1 \right)
                     	                                          \int_{\tau-\rho}^{\tau} 
                     	                                               \mathbb{E} \left[ \|u^1(s) - u^2(s) \|^2 \right] 
                     	                                           ds  \nonumber\\
        & + 4 \left( \|L_{F}\|_{\ell^\infty}^2 + \|L_{F}\|^2 \right)
                              \int_{\tau-\rho}^{\tau} 
                                   \mathbb{E}
                                        \left[ \| u^1(s-\rho) - u^2(s-\rho) \|^2 
                                        \right] 
                              ds 
            + 2 \int_{\tau-\rho}^{\tau} 
                                                \mathbb{E}
                                                     \left[ \| u^1(s) - u^2(s) \|^2 
                                                     \right] 
                   ds  \nonumber\\     
  	   & + 2 \varepsilon 
  	    	           \left[ 8 c_1^2 + 8 c_2^2 + 2 ) \| L \|^2 
  	    	                    + 2C_3
  	                  \right]
  	           \mathbb{E}
  	                \left[  \int_{\tau-\rho}^{\tau} 
  	                              \|u^1(s) - u^2(s) \|^2 ds
  	                \right].               
 \end{align}
 Then by \eqref{sdsec-6} and \eqref{sdsec-12}, we obtain that for all $t \ge 2\rho$ and $\rho \in [0,1]$, 
  \begin{align*}
   	    \mathbb{E} 
   	             \left[  \sup_{r \in [\tau-\rho, \tau] } \| u^1(r) - u^2(r) \|^2  
   	             \right] 
   \le \mathfrak{C}_2 
                e^{ -\beta t} 
    	                    \mathbb{E}
    	                           \left[ \| \varphi_1 - \varphi_2 \|_\infty^2
    	                           \right],           
  \end{align*}
 where $\mathfrak{C}_2$ is positive constant independent of $\varepsilon$, $\varphi_i$, $\tau$, $t$ and $\rho \in [0,1]$.
\end{proof}

As
an applications of Lemma \ref{sdsec}, 
we now 
prove that the  measure attractor of \eqref{lmvlds-2}
is a singleton, 
which is an evolution family of 
probability measures in $\mathcal{P}_{\theta}(D_\rho)$ with pullback mixing property.

\begin{theorem}\label{sin-attractor}
   Suppose that {\bf(H1)}-{\bf(H4)}, \eqref{thetaassum1}, \eqref{thetaassum2}, \eqref{dissip1}, \eqref{dissip2} 
   and \eqref{dissip1-2} hold.
   Then for every $\varepsilon \in (0, \varepsilon_0)$,
   the 
   $\mathcal{D}$-pullback measure attractor $\mathcal{A}^\varepsilon$ of \eqref{lmvlds-2} 
   is a singleton, 
   which is actually an evolution family of 
   probability measures $\{ \mu_\tau^\varepsilon \}_{\tau \in \mathbb{R} }  \subseteq \mathcal{P}_{\theta}(D_\rho)$.
   Moreover, for 
   every  $\tau \in \mathbb{R}$, $\mu_\tau^\varepsilon$ is pullback mixing in the sense that
   $ \lim_{t \to \infty}
          \mathbb{W}_2
                     \left( P_{\tau-t, \tau}^{\ast,\varepsilon} \nu,
                              \mu_\tau^\varepsilon
                     \right)
      = 0 $
   for all $\nu \in \mathcal{P}_{\theta}(D_\rho)$.
\end{theorem}

\begin{proof}
 Given $\mu_\tau, \nu_\tau \in \mathcal{A}^\varepsilon(\tau)$ with $\varepsilon \in (0, \varepsilon_0)$,
 it follows from the invariance of measure attractors that
 for any sequence $t_m \to \infty$,
 there are $\mu_{\tau-t_m}$ and $\nu_{\tau-t_m} \in \mathcal{A}^\varepsilon(\tau-t_m)$
 such that
     $\mu_\tau = \Phi (t_m, \tau-t_m) \mu_{\tau-t_m} 
                       = P_{\tau-t_m, \tau}^{\ast, \varepsilon} \mu_{\tau-t_m}$
 and
     $\nu_\tau = \Phi (t_m, \tau-t_m) \nu_{\tau-t_m} 
                      = P_{\tau-t_m, \tau}^{\ast,\varepsilon} \nu_{\tau-t_m}.$

 Assume $\xi_{m}, \zeta_{m} \in L^{\theta}(\Omega, \mathcal{F}_{\tau-t_m}; D_\rho)$
 such that $\mathcal{L}_{\xi_m} = \mu_{\tau-t_m}$ and $\mathcal{L}_{\zeta_m} = \nu_{\tau-t_m}$.
 Let $u^{\varepsilon}(\tau; \tau-t_m, \xi_m)$ and $v^{\varepsilon}(\tau; \tau-t_m, \zeta_m)$
 be 
 the solutions of \eqref{lmvlds-2}
 with initial data $\xi_m$ and $\zeta_m$ at initial time $\tau-t_m$, respectively.
 Then
     $\mu_\tau = \mathcal{L}_{u^\varepsilon_\tau(\cdot; \tau-t_m, \xi_m) }$
 and
     $\nu_\tau = \mathcal{L}_{v^\varepsilon_\tau(\cdot; \tau-t_m, \zeta_m) }$.

 From Lemma \ref{sdsec}, 
 it follows that  
 there exists
 $N=N(\tau, \mathcal{A}^\varepsilon)>0$
 such that for all $m\ge N$,
 \begin{align} \label{sa-1}
      d_{\mathcal{P}(D_\rho)} \left( \mu_\tau, \nu_\tau \right)
    = & \sup_{ \psi \in L_b(D_\rho), \|\psi\|_{L_b} \le 1 }
                \left| \mathbb{E}
                               \left[ \psi( u^\varepsilon_\tau(\cdot; \tau-t_m, \xi_m) )
                               \right]
                         - \mathbb{E}
                                  \left[ \psi( v^\varepsilon_\tau(\cdot; \tau-t_m, \zeta_m) )
                                  \right]
               \right|  \nonumber\\
 \le & \left( \mathbb{E}
                         \left[ \| u^\varepsilon_\tau(\cdot; \tau-t_m, \xi_m )
                                      - v^\varepsilon_\tau(\cdot; \tau-t_m, \zeta_m )
                                  \|_{\infty}^2
                         \right]
              \right)^{\frac{1}{2} }  \nonumber\\
  \le & \sqrt{\mathfrak{C} } 
                 \left( \mathbb{E} \left[ \|  \xi_m- \zeta_m \|_{\infty}^2 \right] 
                 \right)^{\frac{1}{2} }  
                    e^{-\frac{\beta }{2} t_m}   \nonumber\\
  \le & \sqrt{2 \mathfrak{C} } 
               e^{-\frac{\beta }{2} t_m} 
                     \left( \int_{D_\rho} \|y\|_{\infty}^2 \mu_{\tau-t_m}(dy)
                               + \int_{D_\rho} \|z\|_{\infty}^2 \nu_{\tau-t_m}(dz)
                     \right)^{\frac{1}{2} }  \nonumber\\
  \le & 2 \sqrt{2 \mathfrak{C} }
                  e^{-\frac{\beta}{\theta} t_m} 
                         \| \mathcal{A}^\varepsilon(\tau-t_m) \|_{\mathcal{P}_{\theta}(D_\rho) }  \nonumber\\ 
    = & 2 \sqrt{2 \mathfrak{C} }
                   e^{-\frac{\beta }{\theta} \tau} e^{\frac{\beta}{\theta} (\tau-t_m)} 
                      \| \mathcal{A}^\varepsilon(\tau-t_m) \|_{\mathcal{P}_{\theta}(D_\rho) }
                      \to 0, \ \text{ as } m \to \infty, 
 \end{align} 
 which shows that
   $\mu_\tau = \nu_\tau$ for all $\tau\in\mathbb{R}$, and hence 
 $\mathcal{A}^\varepsilon$ is a singleton,  denote  
 $\{\mu_\tau^\varepsilon\}_{\tau\in\mathbb{R}}$. 

 Next, we will show $\mu_\tau^\varepsilon$ is pullback mixing in the sense of the  Wasserstein metric.
 For any $\mu \in \mathcal{P}_{\theta}(D_\rho)$ and $t > 0$,
 let $\xi$ and $\zeta \in L^{\theta}(\Omega, \mathcal{F}_{\tau-t}; D_\rho)$ such that the laws
 $\mathcal{L}_{\xi} = \mu$ and $\mathcal{L}_{\zeta} = \mu_{\tau-t}^\varepsilon$.
 Then
   by Lemma \ref{sdsec},
 we  find that
 there exists
 $T=T(\tau,  \mu, \mathcal{A}^\varepsilon)>0$
 such that for all $t\ge T$,
 \begin{align*} 
   \mathbb{W}_2
                              \left( P_{\tau-t, \tau}^{\ast,\varepsilon} \mu,
                                       \mu_\tau^\varepsilon
                              \right)
 \le & \mathbb{E}
                   \left[ \| u^\varepsilon_\tau(\cdot; \tau-t, \xi) 
                               - u^\varepsilon_\tau(\cdot; \tau-t, \zeta)
                            \|_\infty^2
                   \right]^{\frac{1}{2}}  \nonumber \\
 \le & \sqrt{\mathfrak{C} } e^{-\frac{\beta }{2} t}
                  \left( \mathbb{E}
                                  \left[ \| \xi - \zeta \|_{\infty}^2
                                  \right]
                  \right)^{\frac{1}{2} }  \nonumber\\
 \le & \sqrt{2 \mathfrak{C} } 
               e^{-\frac{\beta }{\theta} t}
                     \left(
                              \Bigl( \int_{D_\rho} \|y\|_{\infty}^2 \mu(dy) 
                              \Bigr)^{\frac{1}{2}}
                              + \| \mathcal{A}^\varepsilon(\tau-t) \|_{\mathcal{P}_{\theta}(D_\rho) }
                    \right)  \nonumber\\
 \rightarrow & \ 0, \ \ \ \text{as} \  t \to \infty.
 \end{align*}
 This completes the proof.
\end{proof}

 Since 
 $\{ \mu_\tau^\varepsilon \}_{\tau \in \mathbb{R} }  
            \subset \mathcal{P}_{\theta}(D_\rho)$,
 it is natural to investigate the convergence of 
    $\{ \mu_\tau^\varepsilon \}_{\tau \in \mathbb{R} }$ 
 in Wasserstein metric of order $\theta$
 as the noise intensity $\varepsilon \to 0$.  
To that  end,
 we will show the convergence of the segments of solutions to \eqref{lmvlds-1} in $L^\theta(\Omega, D_\rho)$ with respect to initial data and parameter $\varepsilon \in (0, \varepsilon^*)$.
 The main difficulty
 is to deal   with
 the  superlinear noise terms.
 
\begin{lemma}\label{scr}
 Suppose that {\bf (H1)}-{\bf (H4)}, \eqref{31_1}, \eqref{34_1}  and \eqref{dissip1-2} hold.
 Then for every $\tau \in \mathbb{R}$ and and $\theta \in \left(2, \tfrac{2(p-2)}{q-2} \right)$,
 there exists $\varepsilon_0 = \varepsilon(\theta)  \in (0, \varepsilon^*)$ such that for all $\varepsilon \in (0, \varepsilon_0)$, $\rho \in [0,1]$ and $t \ge 3\rho$,
 the solution $u^{\varepsilon} (r; \tau-t, \varphi)$ of \eqref{lmvlds-1} with parameter $\varepsilon$ and the solution $u^0(\tau; \tau - t, \psi)$ of \eqref{lds} satisfy
 \begin{align}\label{jul12a1}
       & \mathbb{E} 
              \left[ \sup_{s \in [-\rho, 0] } 
                            \| u^{\varepsilon}_\tau(s; \tau - t, \varphi) 
                               - u^0_\tau(s; \tau - t, \psi) 
                            \|^\theta 
              \right]  \nonumber\\             
  \le & \mathfrak{C}_{\theta} e^{-\beta t}
                                          \mathbb{E}
                                                 \left[ \| \varphi - \psi \|_\infty^\theta
                                                 \right]  
            + \mathfrak{C}_{\theta} \varepsilon^{\frac{\theta}{2} }   
                      e^{-\beta t}
                             \mathbb{E} \left[ \|\varphi\|_{\infty}^{\theta} \right] 
            + \mathfrak{C}_\theta \varepsilon^{\frac{\theta}{2} } \nonumber\\
       & + \mathfrak{C}_\theta \varepsilon^{\frac{\theta}{2} } 
              \int_{-\infty}^{\tau} e^{\beta (s-\tau)}
                  \left( \|g(s)\|^{\theta} 
                           + \|h(s)\|^{\theta} 
                           + \| \widetilde{h}(s)\|_{L^{\theta}(\mathbb{Y},\nu; \ell^2)}^\theta
                           + \| \widetilde{h}(s)\|_{L^{2}(\mathbb{Y},\nu; \ell^2)}^\theta
                  \right) ds,                                                                    
 \end{align}  
 where $\varphi, \psi \in L^\theta(\Omega, \mathcal{F}_{\tau-t}; D_\rho)$, 
 and $\mathfrak{C}_{\theta} > 0$ is a constant depending on $\theta$.  
\end{lemma}

\begin{proof}
 For convenience, denote by 
        $u^{\varepsilon}(s) = u^{\varepsilon}(s; \tau - t, \varphi)$
 and
        $u^0(s) = u^0(s; \tau - t, \psi)$.
        
 By \eqref{lmvlds-1} and It\^{o}'s formula, 
 we obtain that for any $r \in [\tau-2\rho, \tau]$ and $t \ge 2\rho$,
 \begin{align}\label{scr-1}
       & e^{\beta r} \mathbb{E} \left[ \| u^{\varepsilon} (r) - u^0(r) \|^\theta \right]
          + \theta \mathbb{E} 
                           \left[ \int_{\tau-t}^{r} e^{\beta s} 
                                       \| u^{\varepsilon}(s) - u^0(s) \|^{\theta - 2}
                                       \|B(u^{\varepsilon}(s) - u^0(s) )\|_p^p ds 
                           \right] \nonumber\\
       & + (\theta \lambda - \beta )  \mathbb{E}
                     \left[ \int_{\tau-t}^{r} e^{\beta s} \|u^{\varepsilon}(s) - u^0(s)\|^\theta ds \right]  \nonumber\\
       & + \theta \mathbb{E}
              \bigg[ \int_{\tau-t}^{r} e^{\beta s}
                            \| u^{\varepsilon}(s) - u^0(s) \|^{\theta - 2}
                            \langle f(u^{\varepsilon}(s), \mathcal{L}_{u^{\epsilon}(s)}) 
                                         - f(u^0(s), \mathcal{L}_{u^0(s) } ),
                                        u^{\varepsilon}(s) - u^0(s)
                           \rangle ds
              \bigg]  \nonumber\\
 \le & e^{\beta(\tau-t)} 
                \mathbb{E} 
                      \left[ \| u^{\varepsilon} (\tau-t) - u^0(\tau-t)\|^\theta \right]  \nonumber \\
      & + \theta \mathbb{E}
                   \bigg[ \int_{\tau-t}^{r} e^{\beta s}
                                 \| u^{\varepsilon}(s) - u^0(s) \|^{\theta - 2}
                                 \langle F( u^{\varepsilon}(s-\rho), \mathcal{L}_{u^{\varepsilon}(s-\rho)}) 
                                              - F( u^0(s-\rho), \mathcal{L}_{u^0(s-\rho)}),
                                             u^{\varepsilon}(s) - u^0(s)
                                 \rangle ds
                   \bigg]  \nonumber\\
       & + \frac{\theta(\theta - 1)}{2} \varepsilon \mathbb{E}
                  \bigg[ \int_{\tau-t}^{r} e^{\beta s}
                              \sum_{k \in \mathbb{N} }
                                     \| u^{\varepsilon}(s) - u^0(s) \|^{\theta - 2}
                                     \| h_k(s) + \sigma_k(u^{\varepsilon}(s), \mathcal{L}_{u^{\varepsilon}(s)}) \|^2 ds
                  \bigg]  \nonumber\\
      & + \mathbb{E}
                    \bigg[ \int_{\tau-t}^{r} 
                                 \int_{y \in \mathbb{Y} } e^{\beta s}
                                     \Bigl( \| u^{\varepsilon}(s) - u^0(s)
	                                              + \sqrt{\varepsilon}\widetilde{h}(s, y) 
	                                              + \sqrt{\varepsilon}
	                                                        \widetilde{\sigma}
	                                                              (u^{\varepsilon}(s), \mathcal{L}_{u^{\varepsilon}(s)}, y)
	                                           \|^{\theta}
      	                                       - \| u^{\varepsilon}(s) - u^0(s) \|^{\theta}  \nonumber\\
	  & \qquad\qquad\quad
	                                           - \sqrt{ \varepsilon} \theta 
	                                                                             \| u^{\varepsilon}(s) - u^0(s) \|^{\theta - 2}
	                                           \langle u^{\varepsilon}(s) - u^0(s),
	                                                        \widetilde{h}(s,y) 
	                                                        + \widetilde{\sigma}
	                                                               (u^{\varepsilon}(s), \mathcal{L}_{u^{\varepsilon}(s)}, y)
	                                           \rangle
	                                 \Bigr) \nu(dy) ds   
	                  \bigg].
 \end{align}  
 
 For the fourth term on the left-hand side of \eqref{scr-1},
 by \eqref{f3} and H\"{o}lder's inequality, we obtain
 \begin{align} \label{scr-2}
       &  \theta \mathbb{E}
               \left[ \int_{\tau-t}^{r} e^{\beta s}
                             \| u^{\varepsilon}(s) - u^0(s) \|^{\theta - 2}
                             \langle f(u^{\varepsilon}(s), \mathcal{L}_{u^{\epsilon}(s)}) 
                                          - f(u^0(s), \mathcal{L}_{u^0(s) } ),
                                         u^{\varepsilon}(s) - u^0(s)
                            \rangle ds
               \right]  \nonumber\\
  \ge & \lambda_2 \theta
               \mathbb{E}
                      \left[ \int_{\tau-t}^{r} e^{\beta s}
                                     \| u^{\varepsilon}(s) - u^0(s) \|^{\theta - 2}
                                    \sum_{ i \in \mathbb{Z}} 
   	                                      \left( |u_{i}^{\varepsilon}(s)|^{p-2} + |u_{i}^0(s)|^{p-2} \right)    
   	                                      |u_{i} ^{\varepsilon}(s) - u_{i}^0(s)|^2 
   	                          ds
   	                  \right]  \nonumber\\
   	       & - \theta \left( \|\phi_4\|_{\ell^\infty} +  \|\phi_4\|_1  \right)
   	                    \int_{\tau-t}^{r} e^{\beta s}
   	                        \mathbb{E}
   	                           	   \left[ \| u^{\varepsilon}(s) - u^0(s) \|^{\theta}
   	                               \right]
   	                     ds. 
 \end{align}
 
 For the second term on the right-hand side of \eqref{scr-1},
 by \eqref{F2}, we obtain
 \begin{align}\label{scr-3}
 	&  \theta \mathbb{E}
 	                    \left[ \int_{\tau-t}^{r} e^{\beta s}
 	                                  \| u^{\varepsilon}(s) - u^0(s) \|^{\theta - 2}
 	                                  \langle F( u^{\varepsilon}(s-\rho), \mathcal{L}_{u^{\varepsilon}(s-\rho)}) 
 	                                               - F( u^0(s-\rho), \mathcal{L}_{u^0(s-\rho)}),
 	                                              u^{\varepsilon}(s) - u^0(s)
 	                                  \rangle ds
 	                    \right]  \nonumber\\
 	\le & \theta \mathbb{E}
 	                  \left[ \int_{\tau-t}^{r} e^{\beta s}
 	                               \| u^{\varepsilon}(s) - u^0(s) \|^{\theta - 1}
 	                               \|F( u^{\varepsilon}(s-\rho), \mathcal{L}_{u^{\varepsilon}(s-\rho)}) 
 	                                  - F( u^0(s-\rho), \mathcal{L}_{u^0(s-\rho)}) \|
 	                            ds
 	                  \right]  \nonumber\\ 
   \le & \sqrt{2} \theta \|L_{F}\|_{\ell^\infty}  
                 \mathbb{E}
 	                   \left[ \int_{\tau-t}^{r}  e^{\beta s}
 	                                 \| u^{\varepsilon}(s) - u^0(s) \|^{\theta - 1} 
 	                                 \| u^{\varepsilon}(s-\rho) - u^0(s-\rho) \| 
 	                             ds 
 	                   \right] \nonumber\\
  	      & + \sqrt{2} \theta \|L_{F}\| 
  	                 \mathbb{E}
  	       	                \left[ \int_{\tau-t}^{r}  e^{\beta s}
  	       	                               \| u^{\varepsilon}(s) - u^0(s) \|^{\theta - 1}  
 	                                       \sqrt{\mathbb{E} 
 	                                                       \left[ \| u^{\varepsilon}(s-\rho) - u^0(s-\rho) \|^2 \right]}
 	                                  ds    
 	                       \right]  \nonumber\\                                      
   \le & \sqrt{2}
 	                \big( \|L_{F}\|_{\ell^\infty} + \|L_{F}\| \big) e^{\beta (\rho+\tau-t) }
                        \mathbb{E}
                               \left[ \int_{-\rho}^{0}
                                             \| \varphi(s) - \psi(s) \|^\theta ds
                               \right]  \nonumber\\
          & + \sqrt{2}
                        \big( \|L_{F}\|_{\ell^\infty} + \|L_{F}\| \big) 
                        \big( e^{\beta \rho} + (\theta - 1) \big)
                        \int_{\tau-t}^{r} e^{\beta s}
                             \mathbb{E}
                                   \left[ \| u^{\varepsilon}(s) - u^0(s) \|^\theta
                                   \right] 
                       ds. 
 \end{align} 
 
 For the third term on the right-hand side of \eqref{scr-1},  
 by \eqref{siglinear-1} and using the constant $\delta$ in \eqref{dissip1-3}, we obtain 
 \begin{align} \label{scr-4}      
       & \frac{\theta(\theta - 1)}{2} \varepsilon \mathbb{E}
                  \left[ \int_{\tau-t}^{r} e^{\beta s}
                              \sum_{k \in \mathbb{N} }
                                     \| u^{\varepsilon}(s) - u^0(s) \|^{\theta - 2}
                                     \| h_k(s) 
                                        + \sigma_k(u^{\varepsilon}(s), \mathcal{L}_{u^{\varepsilon}(s)}) \|^2 ds
                  \right]  \nonumber\\
 \le & \frac{(\theta - 1)(\theta - 2)}{2} \delta \mathbb{E}
                   \left[ \int_{\tau-t}^{r} e^{\beta s}
                                \| u^{\varepsilon}(s) - u^0(s) \|^\theta
                            ds
                   \right]  \nonumber\\   
      & + (\theta - 1) \delta^{-\frac{\theta - 2}{2} } \varepsilon^{\frac{\theta}{2} } 
                \mathbb{E}
                      \left[ \int_{\tau-t}^{r} e^{\beta s}
                                   \left( \sum_{k \in \mathbb{N} }
                                                         \| h_k(s) 
                                                             + \sigma_k(u^{\varepsilon}(s),   \mathcal{L}_{u^{\varepsilon}(s)}) 
                                                         \|^2 
                                         \right)^{\frac{\theta}{2} } ds
                      \right]  \nonumber\\  
 \le & \frac{(\theta - 1)(\theta - 2)}{2} \delta
               \mathbb{E}
                      \left[ \int_{\tau-t}^{r} e^{\beta s}
                                   \| u^{\varepsilon}(s) - u^0(s) \|^\theta
                               ds
                      \right]  \nonumber\\   
      & + 2^{\frac{\theta}{2} } (\theta - 1) \delta^{-\frac{\theta - 2}{2} } 
             \varepsilon^{\frac{\theta}{2} } 
                   \mathbb{E}
                         \bigg[ \int_{\tau-t}^{r} e^{\beta s}
                                      \Big( \| h(s) \|^2 + \| \alpha\|^2
                                               + 4 \sum_{i \in \mathbb{Z} }
                                                         \sum_{k \in \mathbb{N} }                                     
                                                               L_{\sigma, k,i}^2 |u^{\varepsilon}_i(s)|^q
                                               + 4 \|L_\sigma\|^2 \mathbb{E} \left[ \|u^{\varepsilon}(s) \|^2 \right]
                                     \Big)^{\frac{\theta}{2} } ds
                        \bigg]  \nonumber\\   
   \le & \frac{(\theta - 1)(\theta - 2)}{2} \delta
                 \mathbb{E}
                        \left[ \int_{\tau-t}^{r} e^{\beta s}
                                     \| u^{\varepsilon}(s) - u^0(s) \|^\theta
                                 ds
                        \right]  \nonumber\\   
        & + 2^{\theta -2} 2^{\frac{\theta}{2} } (\theta - 1) 
               \delta^{-\frac{\theta - 2}{2} } \varepsilon^{\frac{\theta}{2} } 
                     \int_{\tau-t}^{r} e^{\beta s}
                                  \| h(s) \|^\theta
                      ds
            +  \frac{1}{\beta} 2^{\theta -2} 2^{\frac{\theta}{2} } (\theta - 1) 
                \delta^{-\frac{\theta - 2}{2} } \varepsilon^{\frac{\theta}{2} }  
                           e^{\beta r} \| \alpha \|^\theta  \nonumber\\   
        & + 2^{2\theta -2} 2^{\frac{\theta}{2} } (\theta - 1) 
                      \Big( \sum_{i \in \mathbb{Z} }
                                    \sum_{k \in \mathbb{N} }                                     
                                          L_{\sigma, k,i}^2
                      \Big)^{\frac{\theta}{2} } 
                \delta^{-\frac{\theta - 2}{2} } \varepsilon^{\frac{\theta}{2} } 
                \mathbb{E}
                      \bigg[ \int_{\tau-t}^{r} e^{\beta s}
                                         \Big( a \|u^{\varepsilon}(s)\|^{\theta -2} \|u^{\varepsilon}(s)\|^p 
                                                  + \frac{b}{\theta} \|u^{\varepsilon}(s)\|^\theta
                                         \Big)        
                                     ds
                     \bigg]  \nonumber\\  
        & + 2^{2\theta -2} 2^{\frac{\theta}{2} } (\theta - 1) \|L_\sigma\|^\theta 
               \delta^{-\frac{\theta - 2}{2} } \varepsilon^{\frac{\theta}{2} }  
                    \int_{\tau-t}^{r} e^{\beta s}
                             \mathbb{E} \left[ \|u^{\varepsilon}(s) \|^\theta \right]
                    ds,                                                                            
 \end{align}
 where 
  $ a = \tfrac{\theta(q-2)}{2(p-2)} \in (0,1)$
  and 
  $b = (1-a) \theta$ 
  satisfying $a (\theta-2+p) + b = \tfrac{q \theta}{2}$.  
 
 For the fourth term on the right-hand side of \eqref{scr-1},
 Together with argument of \eqref{moment_gjh}, we use the constant $\delta$ in \eqref{dissip1-3} to obtain the estimate:
 \begin{align} \label{scr-5}
       & \mathbb{E}
                    \bigg[ \int_{\tau-t}^{r} e^{\beta s}
                                 \int_{y \in \mathbb{Y} }
                                     \Bigl( \| u^{\varepsilon}(s) - u^0(s)
	                                              + \sqrt{\varepsilon}\widetilde{h}(s, y) 
	                                              + \sqrt{\varepsilon}
	                                                        \widetilde{\sigma}
	                                                              (u^{\varepsilon}(s), \mathcal{L}_{u^{\varepsilon}(s)}, y)
	                                           \|^{\theta}
      	                                       - \| u^{\varepsilon}(s) - u^0(s) \|^{\theta}  \nonumber\\
	   & \qquad\qquad\quad
	                                           - \sqrt{ \varepsilon} \theta 
	                                                                             \| u^{\varepsilon}(s) - u^0(s) \|^{\theta - 2}
	                                           \langle u^{\varepsilon}(s) - u^0(s),
	                                                        \widetilde{h}(s,y) 
	                                                        + \widetilde{\sigma}
	                                                               (u^{\varepsilon}(s), \mathcal{L}_{u^{\varepsilon}(s)}, y)
	                                           \rangle
	                                 \Bigr) \nu(dy) ds   
	                  \bigg]  \nonumber\\
 \le & C_{\theta} \varepsilon
 	         \mathbb{E}
 	              \left[ \int_{\tau-t}^{r} e^{\beta s}
                               \| u^{\varepsilon}(s) - u^0(s) \|^{\theta-2}
                               \int_{y \in \mathbb{Y} }
                                         \| \widetilde{\sigma}(u^{\varepsilon}(s),        
                                                                             \mathcal{L}_{u^{\varepsilon}(s) }, y)
 	                                         + \widetilde{h}(s,y)
 	                                     \|^{2}
 	                             \nu(dy) ds
 	              \right]  \nonumber\\
 	  & + C_{\theta} \varepsilon^{\frac{\theta}{2} }
         	\mathbb{E}
 	               \left[ \int_{\tau-t}^{r} e^{\beta s}
                                \int_{y \in \mathbb{Y} }
 	                                \| \widetilde{\sigma}(u^\varepsilon(s), \mathcal{L}_{u^\varepsilon(s)}, y)
 	                                    + \widetilde{h}(s,y)
 	                                \|^{\theta}
 	                            \nu(dy) ds
 	               \right] \nonumber\\
 \le & \frac{\theta-2}{\theta} C_{\theta} \delta
 	         \mathbb{E}
 	              \left[ \int_{\tau-t}^{r} e^{\beta s}
                               \| u^{\varepsilon}(s) - u^0(s) \|^{\theta} ds
                  \right]  \nonumber\\                  
      & + \frac{2}{\theta} C_{\theta} 
             \delta^{-\frac{\theta - 2}{2} } \varepsilon^{\frac{\theta}{2} }
          	         \mathbb{E}
          	              \bigg[ \int_{\tau-t}^{r} e^{\beta s}
                                       \left( \int_{y \in \mathbb{Y} }
                                                    \| \widetilde{\sigma}(u^{\varepsilon}(s),        
                                                                                        \mathcal{L}_{u^{\varepsilon}(s) }, y)
 	                                                    + \widetilde{h}(s,y)
 	                                                \|^{2}
 	                                            \nu(dy)
 	                                     \right)^{\frac{\theta}{2} } ds
 	                     \bigg]  \nonumber\\
 	  & + C_{\theta} \varepsilon^{\frac{\theta}{2} }
                  \mathbb{E}
 	                     \left[ \int_{\tau-t}^{r} e^{\beta s}
                                       \int_{y \in \mathbb{Y} }
 	                                        \| \widetilde{\sigma}(u^{\varepsilon}(s), \mathcal{L}_{u^{\varepsilon}(s)}, y)
 	                                            + \widetilde{h}(s,y)
 	                                        \|^{\theta}
 	                                   \nu(dy) ds
 	                     \right],
 \end{align}
 where
 $C_{\theta} = \theta(\theta-1) 2^{\theta-3}$.        

For the second term on the right-hand side of \eqref{scr-5},  
by \eqref{siglinear-2}, H\"{o}lder's inequality and Young's inequality, we have    
\begin{align} \label{scr-6}
      & \frac{2}{\theta} C_{\theta} 
          \delta^{-\frac{\theta - 2}{2} } \varepsilon^{\frac{\theta}{2} }
          	         \mathbb{E}
          	              \bigg[ \int_{\tau-t}^{r} e^{\beta s}
                                       \Big( \int_{y \in \mathbb{Y} }
                                                    \| \widetilde{\sigma}(u^{\varepsilon}(s),        
                                                                                        \mathcal{L}_{u^{\varepsilon}(s) }, y)
 	                                                    + \widetilde{h}(s,y)
 	                                                \|^{2}
 	                                            \nu(dy)
 	                                    \Big)^{\frac{\theta}{2} }  ds
 	                     \bigg]  \nonumber\\
\le & \frac{2^{\theta} }{\theta} C_{\theta} 
         \delta^{-\frac{\theta - 2}{2} } \varepsilon^{\frac{\theta}{2} }
          	         \mathbb{E}
          	              \bigg[ \int_{\tau-t}^{r} e^{\beta s}
                                       \Big( \int_{y \in \mathbb{Y} }
                                                    \| \widetilde{\sigma}(u^{\varepsilon}(s),        
                                                                                        \mathcal{L}_{u^{\varepsilon}(s) }, y)
                                                    \|^2
                                               \nu(dy)  
                                        \Big)^{\frac{\theta}{2} }     
                                     ds
                        \bigg]   \nonumber\\                     
 	   & + \frac{2^{\theta}}{\theta} C_{\theta} 
 	          \delta^{-\frac{\theta - 2}{2} } \varepsilon^{\frac{\theta}{2} }
 	             	         \int_{\tau-t}^{r} e^{\beta s}
 	                                          \Big( \int_{y \in \mathbb{Y} }
 	                                                       \| \widetilde{h}(s,y) \|^{2}
 	                                                    \nu(dy)
 	                                           \Big)^{\frac{\theta}{2} } ds  \nonumber\\
\le & \frac{2^{\theta} }{\theta} C_{\theta} 
        \delta^{-\frac{\theta - 2}{2} } \varepsilon^{\frac{\theta}{2} }
          	         \mathbb{E}
          	              \bigg[ \int_{\tau-t}^{r} e^{\beta s}
                                       \Big( \sum_{i \in \mathbb{Z} } \beta_{i}^2
                                                + 4 \sum_{i \in \mathbb{Z} } 
                                                             \int_{y\in \mathbb{Y} }
                                                                     \left( L_{ \widetilde{\sigma}, i}(y)
                                                                    \right)^2 \nu(dy)
                                                            \left( |u_i^{\varepsilon}(s) |^{q}
                                                                      + \mathbb{E}\left[ \|u^{\varepsilon}(s)\|^2 \right]
                                                            \right)
                                        \Big)^{\frac{\theta}{2} }     
                                     ds
                        \bigg]   \nonumber\\                     
 	   & + \frac{2^{\theta}}{\theta} C_{\theta} 
 	          \delta^{-\frac{\theta - 2}{2} } \varepsilon^{\frac{\theta}{2} }
 	             	         \int_{\tau-t}^{r} e^{\beta s}
 	                                           \| \widetilde{h}(s) \|_{L^2(\mathbb{Y}, \nu; \ell^2)}^\theta
 	                         ds  \nonumber\\ 	                  
\le & 2^{\frac{\theta}{2} -1} 
               \frac{2^{\theta} }{\theta} C_{\theta} 
               \delta^{-\frac{\theta - 2}{2} } \varepsilon^{\frac{\theta}{2} }
          	      \int_{\tau-t}^{r} e^{\beta s}
                            \big( \sum_{i \in \mathbb{Z} } \beta_{i}^2
                            \big)^{\frac{\theta}{2} } 
                  ds  
        + \frac{2^{\theta}}{\theta} C_{\theta} 
           \delta^{-\frac{\theta - 2}{2} } \varepsilon^{\frac{\theta}{2} }
 	             	      \int_{\tau-t}^{r} e^{\beta s}
 	                               \| \widetilde{h}(s) \|_{L^2(\mathbb{Y}, \nu; \ell^2)}^\theta
 	                       ds  \nonumber\\
     & + 2^{\frac{\theta}{2} - 1} 2^{\frac{3\theta}{2} -1} 
                \frac{2^{\theta} }{\theta} C_{\theta} 
                \delta^{-\frac{\theta - 2}{2} } \varepsilon^{\frac{\theta}{2} }
               	         \mathbb{E}
               	              \bigg[  \int_{\tau-t}^{r} e^{\beta s}           	                          
               	                             \Big( \sum_{i \in \mathbb{Z} } 
                                                            \| L_{ \widetilde{\sigma}, i} \|_{L^2(\mathbb{Y}, \nu; \mathbb{R} ) }^{\frac{2\theta}{\theta -2} } 
                                             \Big)^{\frac{\theta -2}{2} } 
                                                           \sum_{i \in \mathbb{Z} }
                                                                  |u^{\varepsilon}_i(s) |^{\frac{q \theta} {2} }   
                                        ds
                                \bigg]  \nonumber\\
     & + 2^{\frac{\theta}{2} - 1} 2^{\frac{3\theta}{2} -1} 
                \frac{2^{\theta} }{\theta} C_{\theta} 
                \delta^{-\frac{\theta - 2}{2} } \varepsilon^{\frac{\theta}{2} }  
                    \int_{\tau-t}^{r} e^{\beta s}                                                                              
                         \Big( \sum_{i \in \mathbb{Z} } 
                                        \| L_{ \widetilde{\sigma}, i} \|_{L^2(\mathbb{Y}, \nu; \mathbb{R} ) }^2 
                                        \mathbb{E}\left[ \|u^{\varepsilon}(s)\|^2 \right]
                          \Big)^{\frac{\theta}{2} }                                             
                     ds  \nonumber\\      
\le & 2^{\frac{\theta}{2} -1} 
               \frac{2^{\theta} }{\theta} C_{\theta} 
               \delta^{-\frac{\theta - 2}{2} } \varepsilon^{\frac{\theta}{2} }
               \big( \sum_{i \in \mathbb{Z} } \beta_{i}^2
               \big)^{\frac{\theta}{2} }  
               \frac{1}{\beta} e^{\beta r} 
        + \frac{2^{\theta}}{\theta} C_{\theta} 
           \delta^{-\frac{\theta - 2}{2} } \varepsilon^{\frac{\theta}{2} }
 	             	         \int_{\tau-t}^{r} e^{\beta s}
 	                                   \| \widetilde{h}(s) \|_{L^2(\mathbb{Y}, \nu; \ell^2)}^\theta
 	                         ds  \nonumber\\
     & + 2^{\frac{\theta}{2} - 1} 2^{\frac{3\theta}{2} -1} 
                \frac{2^{\theta} }{\theta} C_{\theta} 
                \delta^{-\frac{\theta - 2}{2} } \varepsilon^{\frac{\theta}{2} }
                \Big( \sum_{i \in \mathbb{Z} } 
                              \| L_{ \widetilde{\sigma}, i} \|_{L^2(\mathbb{Y}, \nu; \mathbb{R} ) }^{\frac{2\theta}{\theta -2} } 
                \Big)^{\frac{\theta -2}{2} }  \nonumber\\
     & \quad 
                \cdot \mathbb{E}
               	                 \bigg[ \int_{\tau-t}^{r} e^{\beta s}       
                                                 \Big( a \|u^{\varepsilon}(s)\|^{\theta-2}   
                                                             \|u^{\varepsilon}(s)\|_p^{p} 
                                                   	      + \frac{b}{\theta} 
                                              	                  \|u^{\varepsilon}(s)\|^{\theta}  
                                                 \Big)
                                            ds
                                \bigg]  \nonumber\\
     & + 2^{\frac{\theta}{2} - 1} 2^{\frac{3\theta}{2} -1} 
                \frac{2^{\theta} }{\theta} C_{\theta} 
                \delta^{-\frac{\theta - 2}{2} } \varepsilon^{\frac{\theta}{2} }
                    \int_{\tau-t}^{r} e^{\beta s}  
                        \mathbb{E} \left[ \| u^{\varepsilon}(s) \|^\theta \right]                                                   
                     ds  
                      \Big( \sum_{i \in \mathbb{Z} } 
                                     \| L_{ \widetilde{\sigma}, i} \|_{L^2(\mathbb{Y}, \nu; \mathbb{R} ) }^2 
                      \Big)^{\frac{\theta}{2} }.                                                                     
\end{align}     

 For the third term on the right-hand side of \eqref{scr-5},  
 by \eqref{sigma2}, we obtain     
 \begin{align} \label{scr-7}        
       & C_{\theta} \varepsilon^{\frac{\theta}{2} }
                  \mathbb{E}
 	                     \left[ \int_{\tau-t}^{r} e^{\beta s}
                                       \int_{y \in \mathbb{Y} }
 	                                        \| \widetilde{\sigma}(u^{\varepsilon}(s), \mathcal{L}_{u^{\varepsilon}(s)}, y)
 	                                            + \widetilde{h}(s,y)
 	                                        \|^{\theta}
 	                                   \nu(dy) ds
 	                     \right]  \nonumber\\
 \le & 2^{\theta-1} C_{\theta} \varepsilon^{\frac{\theta}{2} }
	            \int_{\tau-t}^{r} e^{\beta s} 
	                     \| \widetilde{h}(s) \|^{\theta}_{L^{\theta}(\mathbb{Y}, \nu; \ell^2)} ds \nonumber\\
	  & + 2^{\theta-1} C_{\theta} \varepsilon^{\frac{\theta}{2} }
       	            \mathbb{E}
       	                  \left[ \int_{\tau-t}^{r} 
	                                   \int_{y \in \mathbb{Y} }
       	                                  e^{\beta s}
	                                      \| \widetilde{\sigma}(u^{\varepsilon}(s), \mathcal{L}_{u^{\varepsilon}(s)}, y) \|^{\theta}
	                                   \nu(dy) ds
	                      \right]  \nonumber\\
 \le & 2^{\theta-1} C_{\theta} \varepsilon^{\frac{\theta}{2} }
              \int_{\tau-t}^{r} e^{\beta s} 
 	              \| \widetilde{h}(s) \|^{\theta}_{L^{\theta}(\mathbb{Y}, \nu; \ell^2)} ds  
 	     + \frac{1}{\beta} e^{\beta r}
 	               	   6^{\theta-1} 2^{\frac{\theta}{2}} C_{\theta} \varepsilon^{\frac{\theta}{2} }
 	     	                 \int_{y \in \mathbb{Y} }
 	     	                      \Big( \sum_{i \in \mathbb{Z} }
 	                                      	     |\widetilde{\sigma}_{i}(0,\delta_0,y)|
 	     	                      \Big)^{\theta}
 	     	                 \nu(dy)  \nonumber\\
	  & + 6^{\theta-1} 2^{\frac{\theta}{2}} C_{\theta} \varepsilon^{\frac{\theta}{2} }
  	            \int_{y \in \mathbb{Y} }
	                         \Big( \sum_{i \in \mathbb{Z} }
	                                      \big| L_{\tilde{\sigma},i}(y) \big|^{\frac{\theta}{\theta-1} }
	                         \Big)^{\theta-1}
	                     \nu(dy)  \nonumber\\
	  & \quad 
	        \cdot \mathbb{E}
  	                  \left[ \int_{\tau-t}^{r} e^{\beta s}
	                               \Big( a \|u^{\varepsilon}(s)\|^{\theta-2} \|u^{\varepsilon}(s)\|_p^{p} 
	                                        + \frac{b}{\theta} \|u^{\varepsilon}(s)\|^{\theta}  
	                               \Big) ds
                   	  \right]  \nonumber\\
  	  & + 6^{\theta-1} 2^{\frac{\theta}{2}} C_{\theta} \varepsilon^{\frac{\theta}{2} }
	                   \int_{y \in \mathbb{Y} }
	                                \Big( \sum_{i \in \mathbb{Z}}
	                                               L_{\tilde{\sigma}, i}(y)
	                                \Big)^{\theta}
	                   \nu(dy)
	                   \mathbb{E}
	                          \left[ \int_{\tau-t}^{r} e^{\beta s} 
	                                             \|u^{\varepsilon}(s)\|^{\theta} 
	                                    ds
	                          \right].  	                     	                 	                     
 \end{align}
 By \eqref{scr-5}-\eqref{scr-7}, 
 we obtain that there exists a constant $\widetilde{C}_\theta > 0$ independent of $\rho \in [0,1]$, $\varphi$ and $\psi$, such that for all $r \in [\tau-2\rho, \tau]$ and $\varepsilon \in (0,1]$,     
 \begin{align} \label{scr-8}
       & \mathbb{E}
                    \bigg[ \int_{\tau-t}^{r} e^{\beta s}
                                 \int_{y \in \mathbb{Y} }
                                     \Bigl( \| u^{\varepsilon}(s) - u^0(s)
	                                              + \sqrt{\varepsilon}\widetilde{h}(s, y) 
	                                              + \sqrt{\varepsilon}
	                                                        \widetilde{\sigma}
	                                                              (u^{\varepsilon}(s), \mathcal{L}_{u^{\varepsilon}(s)}, y)
	                                           \|^{\theta}
      	                                       - \| u^{\varepsilon}(s) - u^0(s) \|^{\theta}  \nonumber\\
	   & \qquad\qquad\quad
	                                           - \sqrt{ \varepsilon} \theta 
	                                                                             \| u^{\varepsilon}(s) - u^0(s) \|^{\theta - 2}
	                                           \langle u^{\varepsilon}(s) - u^0(s),
	                                                        \widetilde{h}(s,y) 
	                                                        + \widetilde{\sigma}
	                                                               (u^{\varepsilon}(s), \mathcal{L}_{u^{\varepsilon}(s)}, y)
	                                           \rangle
	                                 \Bigr) \nu(dy) ds   
	                  \bigg]  \nonumber\\
 \le & \frac{\theta-2}{\theta} C_{\theta} \delta
 	         \mathbb{E}
 	              \left[ \int_{\tau-t}^{r} e^{\beta s}
                               \| u^{\varepsilon}(s) - u^0(s) \|^{\theta} ds
                  \right]  \nonumber\\ 
      & + e^{\beta r} \widetilde{C}_\theta \varepsilon^{\frac{\theta}{2} }
          + \widetilde{C}_\theta \varepsilon^{\frac{\theta}{2} }  
                    \int_{\tau-t}^{r} e^{\beta s}
           	                \big( \| \widetilde{h}(s) \|_{L^2(\mathbb{Y}, \nu; \ell^2) }^\theta
           	                         + \| \widetilde{h}(s) \|^{\theta}_{L^{\theta}(\mathbb{Y}, \nu; \ell^2) }
           	                \big)  
           	        ds  \nonumber\\
      &  + \widetilde{C}_\theta \varepsilon^{\frac{\theta}{2} } 
                   \Big( \sum_{i \in \mathbb{Z} } 
                                           \| L_{ \widetilde{\sigma}, i} \|_{L^2(\mathbb{Y}, \nu; \mathbb{R} ) }^{\frac{2\theta}{\theta -2} } 
                   \Big)^{\frac{\theta -2}{2} } 
                   \mathbb{E}
                         	              \left[  \int_{\tau-t}^{r} e^{\beta s}       
                                                      \big( \|u^{\varepsilon}(s)\|^{\theta-2} \|u^{\varepsilon}(s)\|_p^{p} 
                                                             	+ \|u^{\varepsilon}(s)\|^{\theta}  
                                                      \big)
                                                  ds
                                          \right]  \nonumber\\ 	
     & + \widetilde{C}_\theta \varepsilon^{\frac{\theta}{2} } 
                 \Big( \sum_{i \in \mathbb{Z} } 
                                          \| L_{ \widetilde{\sigma}, i} \|_{L^2(\mathbb{Y}, \nu; \mathbb{R} ) }^2 
                \Big)^{\frac{\theta}{2} }
                \int_{\tau-t}^{r} e^{\beta s}  
                        \mathbb{E} \left[ \| u^{\varepsilon}(s) \|^\theta \right]                                                   
                ds  \nonumber\\
 	 & + \widetilde{C}_\theta \varepsilon^{\frac{\theta}{2} }  
 	                   	            \int_{y \in \mathbb{Y} }
 	                 	                         \Big( \sum_{i \in \mathbb{Z} }
 	                 	                                      \big| L_{\tilde{\sigma},i}(y) \big|^{\frac{\theta}{\theta-1} }
 	                 	                         \Big)^{\theta-1}
 	                 	                     \nu(dy)
                                     \mathbb{E}
 	                   	                  \left[ \int_{\tau-t}^{r} e^{\beta s}
 	                 	                               \big( \|u^{\varepsilon}(s)\|^{\theta-2} \|u^{\varepsilon}(s)\|_p^{p} 
 	                 	                                        + \|u^{\varepsilon}(s)\|^{\theta}  
 	                 	                               \big) ds
 	                                    	  \right]  \nonumber\\
  	  & + \widetilde{C}_\theta \varepsilon^{\frac{\theta}{2} }
	                   \int_{y \in \mathbb{Y} }
	                                \Big( \sum_{i \in \mathbb{Z}}
	                                               L_{\tilde{\sigma}, i}(y)
	                                \Big)^{\theta}
	                   \nu(dy)
	                   \mathbb{E}
	                          \left[ \int_{\tau-t}^{r} e^{\beta s} 
	                                             \|u^{\varepsilon}(s)\|^{\theta} 
	                                    ds
	                          \right].  	                                    	                                   	                     
 \end{align}     

 From \eqref{scr-1}-\eqref{scr-4} and \eqref{scr-8}, 
 it follows that there exists a constant $\widehat{C}_\theta > 0$ depending on $\theta$ such that for all $r \in [\tau-2\rho, \tau]$ and $\varepsilon \in (0, 1]$,
 \begin{align}\label{scr-9}
       & e^{\beta r} \mathbb{E} \left[ \| u^{\varepsilon}(r) - u^0(r) \|^\theta \| \right]  \nonumber\\
       & + \bigg[ \theta \lambda 
                       - \beta 
                       - \theta \left( \|\phi_4\|_{\ell^\infty} +  \|\phi_4\|_1 \right) 
                       - \sqrt{2}
                                   \big( \|L_{F}\|_{\ell^\infty} + \|L_{F}\| \big) 
                                   \big( e^{\beta \rho} + (\theta - 1) \big)  \nonumber\\
       & \quad \ \ 
                       - \frac{(\theta - 1)(\theta - 2)}{2} \delta
                       - \frac{\theta-2}{\theta} C_{\theta} \delta
              \bigg] 
              \mathbb{E}
                     \left[ \int_{\tau-t}^{r} e^{\beta s} \|u^{\varepsilon}(s) - u^0(s)\|^\theta ds \right]  \nonumber\\
  \le & e^{\beta(\tau-t)} \mathbb{E} \left[ \| \varphi(0) - \psi(0) \|^\theta \right] 
            + \sqrt{2}
 	                \big( \|L_{F}\|_{\ell^\infty} + \|L_{F}\| \big)  e^{\beta (\rho+\tau-t) }
                        \mathbb{E}
                               \left[ \int_{-\rho}^{0}
                                             \| \varphi(s) - \psi(s) \|^\theta ds
                               \right]  \nonumber\\
        & + e^{\beta r} \widehat{C}_\theta \varepsilon^{\frac{\theta}{2} }
            + \widehat{C}_\theta \varepsilon^{\frac{\theta}{2} } 
                     \int_{\tau-t}^{r} e^{\beta s}
                          \big( \| h(s) \|^\theta
                                   + \| \widetilde{h}(s) \|_{L^2(\mathbb{Y}, \nu; \ell^2) }^\theta
                                   + \| \widetilde{h}(s) \|^{\theta}_{L^{\theta}(\mathbb{Y}, \nu; \ell^2) }
                          \big)
                      ds  \nonumber\\
       & + \widehat{C}_\theta \varepsilon^{\frac{\theta}{2} }            
                    \mathbb{E}
                          	              \left[  \int_{\tau-t}^{r} e^{\beta s}       
                                                       \big( \|u^{\varepsilon}(s)\|^{\theta-2} \|u^{\varepsilon}(s)\|_p^{p} 
                                                              	+ \|u^{\varepsilon}(s)\|^{\theta}  
                                                       \big)
                                                   ds
                                           \right].                                                                    
 \end{align}

 By \eqref{3.12_4} and \eqref{jan24a} in 
 the proof of Lemma \ref{pue}, 
 there is a constant 
        $\bar{C}_\theta = C(\beta, \theta, p, q, L) > 0$ 
 independent of $\varphi$  and $\rho \in [0,1]$,
 such that for all $r \in [\tau-2\rho, \tau]$, $t \ge 3\rho$ and $\varepsilon \in (0, \varepsilon^*)$,
 \begin{align} \label{scr-10}
       & \mathbb{E}
                  \left[  \int_{\tau-t}^{r} e^{\beta (s-r) }       
                                                              \big( \|u^{\varepsilon}(s)\|^{\theta-2}  
                                                                       \|u^{\varepsilon}(s)\|_p^{p} 
                                                                       + \|u^{\varepsilon}(s)\|^{\theta}  
                                                              \big)
                             ds
                  \right]  \nonumber\\
  \le & \bar{C}_\theta e^{\beta (\tau-t-r)}
  	            \mathbb{E} \left[ \|\varphi\|_{\infty}^{\theta} \right] 
  	       + \bar{C}_\theta  \nonumber\\
       & + \bar{C}_\theta  
       	     \int_{\tau-t}^{r} e^{\beta (s-r)}
  	             \left( \|g(s)\|^{\theta} 
  	                      + \|h(s)\|^{\theta} 
  	                      + \| \widetilde{h}(s)\|_{L^{\theta}(\mathbb{Y},\nu; \ell^2)}^\theta
  	                        + \| \widetilde{h}(s)\|_{L^{2}(\mathbb{Y},\nu; \ell^2)}^\theta
  	             \right) ds.
  \end{align}

 By \eqref{dissip1-3}, \eqref{scr-9} and \eqref{scr-10}, 
 we obtain that there exists $\varepsilon_0 = \varepsilon(\theta) \in (0, \varepsilon^*)$ such that for all $r \in [\tau-2\rho, \tau]$, $t \ge 3\rho$ and $\varepsilon \in (0, \varepsilon_0)$,
 \begin{align} \label{scr-11}
       & \mathbb{E} \left[ \| u^{\varepsilon} (r) - u^0(r) \|^\theta \| \right] 
           + \beta \mathbb{E}
                               \left[ \int_{\tau-t}^{r} e^{\beta (s-r) } 
                                                \| u^{\varepsilon}(s) - u^0(s) \|^\theta  
                                        ds 
                               \right]  \nonumber\\              
  \le & e^{\beta (\tau-t-r)} \mathbb{E} \left[ \| \varphi(0) - \psi(0) \|^\theta \right] 
            + \sqrt{2}
 	                \big( \|L_{F}\|_{\ell^\infty} + \|L_{F}\| \big)  e^{\beta (\rho+\tau-t-r) }
                        \mathbb{E}
                               \left[ \int_{-\rho}^{0}
                                             \| \varphi(s) - \psi(s) \|^\theta ds
                               \right]  \nonumber\\
        & + \widehat{C}_\theta \varepsilon^{\frac{\theta}{2} }
            +\widehat{C}_\theta \varepsilon^{\frac{\theta}{2} } 
                     \int_{\tau-t}^{r} e^{\beta (s-r) }
                          \big( \| h(s) \|^\theta
                                   + \| \widetilde{h}(s) \|_{L^2(\mathbb{Y}, \nu; \ell^2) }^\theta
                                   + \| \widetilde{h}(s) \|^{\theta}_{L^{\theta}(\mathbb{Y}, \nu; \ell^2) }
                          \big)
                      ds  \nonumber\\            
        & + \widehat{C}_\theta  \bar{C}_\theta \varepsilon^{\frac{\theta}{2} }   
                e^{\beta (\tau-t-r)}
          	            \mathbb{E} \left[ \|\varphi\|_{\infty}^{\theta} \right] 
          	       +  \widehat{C}_\theta  \bar{C}_\theta \varepsilon^{\frac{\theta}{2} }   \nonumber\\
               & +  \widehat{C}_\theta \bar{C}_\theta \varepsilon^{\frac{\theta}{2} } 
               	     \int_{\tau-t}^{r} e^{\beta (s-r)}
          	             \left( \|g(s)\|^{\theta} 
          	                      + \|h(s)\|^{\theta} 
          	                      + \| \widetilde{h}(s)\|_{L^{\theta}(\mathbb{Y},\nu; \ell^2)}^\theta
          	                        + \| \widetilde{h}(s)\|_{L^{2}(\mathbb{Y},\nu; \ell^2)}^\theta
          	             \right) ds  \nonumber\\
  \le & \left( 1 + \sqrt{2}
 	                            \big( \|L_{F}\|_{\ell^\infty} + \|L_{F}\| \big)  e^{\beta \rho} \rho
 	       \right) e^{\beta (\tau-t-r) }
                            \mathbb{E}
                                   \left[ \| \varphi - \psi \|_\infty^\theta
                                   \right]  \nonumber\\
        & + \widehat{C}_\theta \bar{C}_\theta \varepsilon^{\frac{\theta}{2} }   
                e^{\beta (\tau-t-r)}
          	            \mathbb{E} \left[ \|\varphi\|_{\infty}^{\theta} \right] 
           + \widehat{C}_\theta (1 + \bar{C}_\theta) \varepsilon^{\frac{\theta}{2} }          
          	         \nonumber\\
       & +  \widehat{C}_\theta  (1 + \bar{C}_\theta) \varepsilon^{\frac{\theta}{2} } 
               	     \int_{\tau-t}^{r} e^{\beta (s-r)}
          	             \left( \|g(s)\|^{\theta} 
          	                      + \|h(s)\|^{\theta} 
          	                      + \| \widetilde{h}(s)\|_{L^{\theta}(\mathbb{Y},\nu; \ell^2)}^\theta
          	                        + \| \widetilde{h}(s)\|_{L^{2}(\mathbb{Y},\nu; \ell^2)}^\theta
          	             \right) ds \nonumber\\
  \le & \breve{C}_\theta e^{- \beta t}
                            \mathbb{E}
                                   \left[ \| \varphi - \psi \|_\infty^\theta
                                   \right]  
            + \breve{C}_\theta e^{- \beta t} \varepsilon^{\frac{\theta}{2} }   
                   \mathbb{E} \left[ \|\varphi\|_{\infty}^{\theta} \right] 
            + \breve{C}_\theta \varepsilon^{\frac{\theta}{2} }          
          	         \nonumber\\
        & + \breve{C}_\theta \varepsilon^{\frac{\theta}{2} } 
               	     \int_{-\infty}^{\tau} e^{\beta (s-\tau)}
          	             \left( \|g(s)\|^{\theta} 
          	                      + \|h(s)\|^{\theta} 
          	                      + \| \widetilde{h}(s)\|_{L^{\theta}(\mathbb{Y},\nu; \ell^2)}^\theta
          	                        + \| \widetilde{h}(s)\|_{L^{2}(\mathbb{Y},\nu; \ell^2)}^\theta
          	             \right) ds.                
 \end{align}

 By It\^{o} formula, we obtain that for all $t \ge 2\rho$ and $\varepsilon \in (0, 1)$,
 \begin{align}\label{scr-12}
       & \mathbb{E} 
                  \left[ \sup_{r \in [\tau -\rho, \tau] } 
                               \| u^{\varepsilon} (r) - u^0(r) \|^\theta
                               \right]  \nonumber\\
 \le & \mathbb{E} 
                      \left[ \| u^{\varepsilon} (\tau-\rho) - u^0(\tau-\rho)\|^\theta \right]  \nonumber \\
       & + \theta \mathbb{E}
              \bigg[ \sup_{r \in [\tau -\rho, \tau] }
                             \int_{\tau-\rho}^{r} 
                            - \| u^{\varepsilon}(s) - u^0(s) \|^{\theta - 2}
                            \langle f(u^{\varepsilon}(s), \mathcal{L}_{u^{\epsilon}(s)}) 
                                         - f(u^0(s), \mathcal{L}_{u^0(s) } ),
                                        u^{\varepsilon}(s) - u^0(s)
                           \rangle ds
              \bigg]  \nonumber\\                      
      & + \theta \mathbb{E}
                   \bigg[ \int_{\tau-\rho}^{\tau} 
                                 \| u^{\varepsilon}(s) - u^0(s) \|^{\theta - 1}
                                 \| F( u^{\varepsilon}(s-\rho), \mathcal{L}_{u^{\varepsilon}(s-\rho)}) 
                                    - F( u^0(s-\rho), \mathcal{L}_{u^0(s-\rho)})
                                 \| ds
                   \bigg]  \nonumber\\
       & + \frac{\theta(\theta - 1)}{2} \varepsilon \mathbb{E}
                  \bigg[ \int_{\tau-\rho}^{\tau} 
                              \sum_{k \in \mathbb{N} }
                                     \| u^{\varepsilon}(s) - u^0(s) \|^{\theta - 2}
                                     \| h_k(s) + \sigma_k(u^{\varepsilon}(s), \mathcal{L}_{u^{\varepsilon}(s)}) \|^2 ds
                  \bigg]  \nonumber\\
      & + \theta \sqrt{\varepsilon} \mathbb{E}
                        \bigg[ \sup_{r \in [\tau -\rho, \tau] } 
                                    \int_{\tau-\rho}^{r} 
                                    \sum_{k \in \mathbb{N} }
                                           \| u^{\varepsilon}(s) - u^0(s) \|^{\theta - 2}
                                            \langle h_k(s) + \sigma_k(u^{\varepsilon}(s), \mathcal{L}_{u^{\varepsilon}(s) } ),  
                                             u^{\varepsilon}(s) - u^0(s) 
                                            \rangle dW_k(s)
                        \bigg]  \nonumber\\
      & + \mathbb{E}
                    \bigg[ \int_{\tau-\rho}^{\tau} 
                                 \int_{y \in \mathbb{Y} } 
                                    \Big |  \Bigl( \| u^{\varepsilon}(s) - u^0(s)
	                                              + \sqrt{\varepsilon}\widetilde{h}(s, y) 
	                                              + \sqrt{\varepsilon}
	                                                        \widetilde{\sigma}
	                                                              (u^{\varepsilon}(s), \mathcal{L}_{u^{\varepsilon}(s)}, y)
	                                           \|^{\theta}
      	                                       - \| u^{\varepsilon}(s) - u^0(s) \|^{\theta}  \nonumber\\
	  & \qquad\qquad\quad
	                                           - \sqrt{ \varepsilon} \theta 
	                                                                             \| u^{\varepsilon}(s) - u^0(s) \|^{\theta - 2}
	                                           \langle u^{\varepsilon}(s) - u^0(s),
	                                                        \widetilde{h}(s,y) 
	                                                        + \widetilde{\sigma}
	                                                               (u^{\varepsilon}(s), \mathcal{L}_{u^{\varepsilon}(s)}, y)
	                                           \rangle
	                                 \Bigr) \Big |
	                                  \nu(dy) ds   
	                  \bigg] \nonumber\\
      & + \theta \sqrt{\varepsilon} 
               \mathbb{E}
                        \bigg[ \sup_{r \in [\tau -\rho, \tau] }
                                     \int_{\tau-\rho}^{r} 
                                        \int_{y \in \mathbb{Y} } 
                                           \| u^{\varepsilon}(s-) - u^0(s-) \|^{\theta - 2}  \nonumber\\
      &\qquad\qquad\qquad\qquad\qquad\qquad                              
                                            \langle \widetilde{h}(s,y) 
                                            	         + \widetilde{\sigma}(u^{\varepsilon}(s-), \mathcal{L}_{u^{\varepsilon}(s)}, y),  
                                                         u^{\varepsilon}(s-) - u^0(s-) 
                                            \rangle  \widetilde{N}(ds, dy)
                        \bigg].
 \end{align}   
 
  For the second term on the right-hand side of \eqref{scr-12},
  by \eqref{f3} and H\"{o}lder's inequality, we obtain
  \begin{align} \label{scr-13}
        &  \theta \mathbb{E}
                \left[ \sup_{r \in [\tau -\rho, \tau] }
                          \int_{\tau-\rho}^{r} 
                              - \| u^{\varepsilon}(s) - u^0(s) \|^{\theta - 2}
                                 \langle f(u^{\varepsilon}(s), \mathcal{L}_{u^{\epsilon}(s)}) 
                                              - f(u^0(s), \mathcal{L}_{u^0(s) } ),
                                             u^{\varepsilon}(s) - u^0(s)
                                 \rangle ds
                \right]  \nonumber\\
   \le & \lambda_2 \theta
                \mathbb{E}
                       \left[ \sup_{r \in [\tau -\rho, \tau] }
                                   \int_{\tau-\rho}^{r} 
                                     - \| u^{\varepsilon}(s) - u^0(s) \|^{\theta - 2}
                                        \sum_{ i \in \mathbb{Z}} 
    	                                      \left( |u_{i}^{\varepsilon}(s)|^{p-2} + |u_{i}^0(s)|^{p-2} \right)    
    	                                      |u_{i} ^{\varepsilon}(s) - u_{i}^0(s)|^2 
    	                          ds
    	                  \right]  \nonumber\\
        & +  \theta \|\phi_4\|_{\ell^\infty} 
    	                   \mathbb{E}
    	                         \left[ \int_{\tau-\rho}^{\tau} 
    	                                      \| u^{\varepsilon}(s) - u^0(s) \|^{\theta}
    	                                  ds
    	                         \right]  \nonumber\\
        & + \theta \|\phi_4\|_1 \mathbb{E}
    	               	                         \left[ \int_{\tau-\rho}^{\tau} 
    	               	                                      \| u^{\varepsilon}(s) - u^0(s) \|^{\theta-2}
    	               	                                      \mathbb{E} \left[ \| u^{\varepsilon}(s) - u^0(s) \|^2 \right]
    	               	                                  ds
    	               	                         \right]  \nonumber\\
   \le & \theta \left( \|\phi_4\|_{\ell^\infty} + \|\phi_4\|_1  \right)
    	                    \int_{\tau-\rho}^{\tau}
    	                        \mathbb{E}
    	                           	   \left[ \| u^{\varepsilon}(s) - u^0(s) \|^{\theta}
    	                               \right]
    	                     ds. 
  \end{align}
 
 For the third term on the right-hand side of \eqref{scr-12},
 by \eqref{F2}, we obtain
 \begin{align}\label{scr-14}
 	&  \theta \mathbb{E}
 	                    \left[ \int_{\tau-\rho}^{\tau}
 	                                  \| u^{\varepsilon}(s) - u^0(s) \|^{\theta - 2}
 	                                  \langle F( u^{\varepsilon}(s-\rho), \mathcal{L}_{u^{\varepsilon}(s-\rho)}) 
 	                                               - F( u^0(s-\rho), \mathcal{L}_{u^0(s-\rho)}),
 	                                              u^{\varepsilon}(s) - u^0(s)
 	                                  \rangle ds
 	                    \right]  \nonumber\\
 	\le & \theta \mathbb{E}
 	                  \left[ \int_{\tau-\rho}^{\tau}
 	                               \| u^{\varepsilon}(s) - u^0(s) \|^{\theta - 1}
 	                               \|F( u^{\varepsilon}(s-\rho), \mathcal{L}_{u^{\varepsilon}(s-\rho)}) 
 	                                  - F( u^0(s-\rho), \mathcal{L}_{u^0(s-\rho)}) \|
 	                            ds
 	                  \right]  \nonumber\\ 	                 	                	                
   \le & \sqrt{2} (\theta - 1)
 	                \big( \|L_{F}\|_{\ell^\infty} + \|L_{F}\| \big) 
                        \mathbb{E}
                               \left[ \int_{\tau-\rho}^{\tau}
                                             \| u^{\varepsilon}(s) - u^0(s) \|^\theta ds
                               \right]  \nonumber\\
          & + \sqrt{2} 
                        \big( \|L_{F}\|_{\ell^\infty} + \|L_{F}\| \big) 
                        \int_{\tau-\rho}^{\tau} 
                             \mathbb{E}
                                   \left[ \| u^{\varepsilon}(s-\rho) - u^0(s-\rho) \|^\theta
                                   \right] 
                          ds  \nonumber\\
   \le & \sqrt{2} \theta
 	                \big( \|L_{F}\|_{\ell^\infty} + \|L_{F}\| \big) 
                        \int_{\tau-2\rho}^{\tau}
                            \mathbb{E}
                                  \left[ \| u^{\varepsilon}(s) - u^0(s) \|^\theta 
                                  \right]
                        ds. 
 \end{align}  

 For the fourth term on the right-hand side of \eqref{scr-12},
 similar the argument in \eqref{scr-4}, we obtain
 \begin{align}\label{scr-15}  
       & \frac{\theta(\theta - 1)}{2} \varepsilon \mathbb{E}
                  \bigg[ \int_{\tau-\rho}^{\tau} 
                              \sum_{k \in \mathbb{N} }
                                     \| u^{\varepsilon}(s) - u^0(s) \|^{\theta - 2}
                                     \| h_k(s) + \sigma_k(u^{\varepsilon}(s), \mathcal{L}_{u^{\varepsilon}(s)}) \|^2 ds
                  \bigg]  \nonumber\\ 
    \le & \frac{(\theta - 1)(\theta - 2)}{2}
                  \mathbb{E}
                         \left[ \int_{\tau-\rho}^{\tau} 
                                      \| u^{\varepsilon}(s) - u^0(s) \|^\theta
                                  ds
                         \right]  \nonumber\\   
         & + 2^{\theta -2} 2^{\frac{\theta}{2} } (\theta - 1) 
                   \varepsilon^{\frac{\theta}{2} } 
                      \int_{\tau-\rho}^{\tau} 
                                   \| h(s) \|^\theta
                       ds
             + 2^{\theta -2} 2^{\frac{\theta}{2} } (\theta - 1) 
                     \varepsilon^{\frac{\theta}{2} }  
                            \| \alpha \|^\theta  \nonumber\\   
         & + 2^{2\theta -2} 2^{\frac{\theta}{2} } (\theta - 1) 
                  \Big( \sum_{i \in \mathbb{Z} }
                                \sum_{k \in \mathbb{N} }                                     
                                      L_{\sigma, k,i}^2
                  \Big)^{\frac{\theta}{2} } 
                 \varepsilon^{\frac{\theta}{2} } 
                 \mathbb{E}
                       \bigg[ \int_{\tau-\rho}^{\tau} 
                                          \Big( a \|u^{\varepsilon}(s)\|^{\theta -2} \|u^{\varepsilon}(s)\|^p 
                                                   + \frac{b}{\theta} \|u^{\varepsilon}(s)\|^\theta
                                          \Big)        
                                      ds
                      \bigg]  \nonumber\\  
         & + 2^{2\theta -2} 2^{\frac{\theta}{2} } (\theta - 1) \|L_\sigma\|^\theta 
                  \varepsilon^{\frac{\theta}{2} }  
                     \int_{\tau-\rho}^{\tau} 
                              \mathbb{E} \left[ \|u^{\varepsilon}(s) \|^\theta \right]
                     ds.                  
 \end{align}
 
 For the fifth term on the right-hand side of \eqref{scr-12},
 by the BDG inequality, we obtain
 \begin{align}\label{scr-16} 
       & \theta \sqrt{\varepsilon} \mathbb{E}
                         \bigg[ \sup_{r \in [\tau -\rho, \tau] } 
                                     \int_{\tau-\rho}^{r} 
                                     \sum_{k \in \mathbb{N} }
                                            \| u^{\varepsilon}(s) - u^0(s) \|^{\theta - 2}
                                             \langle h_k(s) + \sigma_k(u^{\varepsilon}(s), \mathcal{L}_{u^{\varepsilon}(s) } ),  
                                              u^{\varepsilon}(s) - u^0(s) 
                                             \rangle dW_k(s)
                         \bigg]  \nonumber\\
 \le & c_1 \theta \sqrt{\varepsilon} \mathbb{E}
                         \bigg[ 
                                   \Big( \int_{\tau-\rho}^{\tau} 
                                                \sum_{k \in \mathbb{N} }
                                                      \| u^{\varepsilon}(s) - u^0(s) \|^{2\theta - 2}
                                                      \| h_k(s) + \sigma_k(u^{\varepsilon}(s), \mathcal{L}_{u^{\varepsilon}(s) } )
                                                      \|^2 ds
                                   \Big)^{\frac{1}{2} }                                        
                         \bigg]  \nonumber\\   
  \le & \frac{1}{4} 
               \mathbb{E}
                      \left[ \sup_{r \in [\tau-\rho, \tau] } 
                                    \| u^{\varepsilon}(r) - u^0(r) \|^{\theta}
                      \right]  \nonumber\\
        & + c_1^2 \theta^2 \varepsilon
                                    \mathbb{E}
                                         \left[ \int_{\tau-\rho}^{\tau} 
                                                 \sum_{k \in \mathbb{N} }
                                                       \| u^{\varepsilon}(s) - u^0(s) \|^{\theta - 2}
                                                       \| h_k(s) + \sigma_k(u^{\varepsilon}(s), \mathcal{L}_{u^{\varepsilon}(s) } )
                                                       \|^2 ds
                                        \right],                                       
 \end{align}
 where the second term on the right-hand side can be estimated by \eqref{scr-15}.

 For the sixth term on the right-hand side of \eqref{scr-12},
 similar to the argument of \eqref{scr-8}, we obtain 
 \begin{align}\label{scr-17}  
       & \mathbb{E}
                    \bigg[ \int_{\tau-\rho}^{\tau} 
                                 \int_{y \in \mathbb{Y} } 
                                     \Bigl( \| u^{\varepsilon}(s) - u^0(s)
	                                              + \sqrt{\varepsilon}\widetilde{h}(s, y) 
	                                              + \sqrt{\varepsilon}
	                                                        \widetilde{\sigma}
	                                                              (u^{\varepsilon}(s), \mathcal{L}_{u^{\varepsilon}(s)}, y)
	                                           \|^{\theta}
      	                                       - \| u^{\varepsilon}(s) - u^0(s) \|^{\theta}  \nonumber\\
	  & \qquad\qquad\quad
	                                           - \sqrt{ \varepsilon} \theta 
	                                                                             \| u^{\varepsilon}(s) - u^0(s) \|^{\theta - 2}
	                                           \langle u^{\varepsilon}(s) - u^0(s),
	                                                        \widetilde{h}(s,y) 
	                                                        + \widetilde{\sigma}
	                                                               (u^{\varepsilon}(s), \mathcal{L}_{u^{\varepsilon}(s)}, y)
	                                           \rangle
	                                 \Bigr) \nu(dy) ds   
	                  \bigg] \nonumber\\ 
 \le & \frac{\theta-2}{\theta} C_{\theta}
 	         \mathbb{E}
 	              \left[ \int_{\tau-\rho}^{\tau}
                               \| u^{\varepsilon}(s) - u^0(s) \|^{\theta} ds
                  \right]  \nonumber\\ 
      & + \widetilde{C}_\theta \varepsilon^{\frac{\theta}{2} }
          + \widetilde{C}_\theta \varepsilon^{\frac{\theta}{2} }  
                    \int_{\tau-\rho}^{\tau}
           	                \big( \| \widetilde{h}(s) \|_{L^2(\mathbb{Y}, \nu; \ell^2) }^\theta
           	                         + \| \widetilde{h}(s) \|^{\theta}_{L^{\theta}(\mathbb{Y}, \nu; \ell^2) }
           	                \big)  
           	        ds  \nonumber\\
      &  + \widetilde{C}_\theta \varepsilon^{\frac{\theta}{2} } 
                   \Big( \sum_{i \in \mathbb{Z} } 
                                           \| L_{ \widetilde{\sigma}, i} \|_{L^2(\mathbb{Y}, \nu; \mathbb{R} ) }^{\frac{2\theta}{\theta -2} } 
                   \Big)^{\frac{\theta -2}{2} } 
                   \mathbb{E}
                         	              \left[  \int_{\tau-\rho}^{\tau}   
                                                      \big( \|u^{\varepsilon}(s)\|^{\theta-2} \|u^{\varepsilon}(s)\|_p^{p} 
                                                             	+ \|u^{\varepsilon}(s)\|^{\theta}  
                                                      \big)
                                                  ds
                                          \right]  \nonumber\\ 	
     & + \widetilde{C}_\theta \varepsilon^{\frac{\theta}{2} } 
                 \Big( \sum_{i \in \mathbb{Z} } 
                                          \| L_{ \widetilde{\sigma}, i} \|_{L^2(\mathbb{Y}, \nu; \mathbb{R} ) }^2 
                \Big)^{\frac{\theta}{2} }
                \int_{\tau-\rho}^{\tau}
                        \mathbb{E} \left[ \| u^{\varepsilon}(s) \|^\theta \right]                                                   
                ds  \nonumber\\
 	 & + \widetilde{C}_\theta \varepsilon^{\frac{\theta}{2} }  
 	                   	            \int_{y \in \mathbb{Y} }
 	                 	                         \Big( \sum_{i \in \mathbb{Z} }
 	                 	                                      \big| L_{\tilde{\sigma},i}(y) \big|^{\frac{\theta}{\theta-1} }
 	                 	                         \Big)^{\theta-1}
 	                 	                     \nu(dy)
                                     \mathbb{E}
 	                   	                  \left[ \int_{\tau-\rho}^{\tau}
 	                 	                               \big( \|u^{\varepsilon}(s)\|^{\theta-2} \|u^{\varepsilon}(s)\|_p^{p} 
 	                 	                                        + \|u^{\varepsilon}(s)\|^{\theta}  
 	                 	                               \big) ds
 	                                    	  \right]  \nonumber\\
  	  & + \widetilde{C}_\theta \varepsilon^{\frac{\theta}{2} }
	                   \int_{y \in \mathbb{Y} }
	                                \Big( \sum_{i \in \mathbb{Z}}
	                                               L_{\tilde{\sigma}, i}(y)
	                                \Big)^{\theta}
	                   \nu(dy)
	                   \mathbb{E}
	                          \left[ \int_{\tau-\rho}^{\tau} 
	                                             \|u^{\varepsilon}(s)\|^{\theta} 
	                                    ds
	                          \right].   	                                       
\end{align} 
 
 For the seventh term on the right-hand side of \eqref{scr-12},
 by the BDG inequality, we obtain
 \begin{align}\label{scr-18}  
      & \theta \sqrt{\varepsilon} 
               \mathbb{E}
                        \bigg[ \sup_{r \in [\tau -\rho, \tau] }
                                     \int_{\tau-\rho}^{r} 
                                        \int_{y \in \mathbb{Y} } 
                                           \| u^{\varepsilon}(s-) - u^0(s-) \|^{\theta - 2}  \nonumber\\
      &\qquad\qquad\qquad\qquad\qquad\qquad                              
                                            \langle \widetilde{h}(s,y) 
                                            	         + \widetilde{\sigma}(u^{\varepsilon}(s-), \mathcal{L}_{u^{\varepsilon}(s)}, y),  
                                                         u^{\varepsilon}(s-) - u^0(s-) 
                                            \rangle  \widetilde{N}(ds, dy)
                        \bigg]  \nonumber\\
\le & c_2 \theta \sqrt{\varepsilon} 
               \mathbb{E}
                        \bigg[ 
                                  \Big( \int_{\tau-\rho}^{\tau} 
                                              \int_{y \in \mathbb{Y} } 
                                                  \| u^{\varepsilon}(s-) - u^0(s-) \|^{2\theta - 2}
                                                  \| \widetilde{h}(s,y) 
                                            	         + \widetilde{\sigma}(u^{\varepsilon}(s-), \mathcal{L}_{u^{\varepsilon}(s)}, y)
                                            	  \|^2
                                               N(ds, dy)
                                  \Big)^{\frac{1}{2} }
                        \bigg]  \nonumber\\
\le &  \frac{1}{4} 
               \mathbb{E}
                      \left[ \sup_{r \in [\tau-\rho, \tau] } 
                                    \| u^{\varepsilon}(r) - u^0(r) \|^{\theta}
                      \right] 
         + c_2^2 \theta (\theta - 2) 
                 \mathbb{E}
                        \left[ \int_{\tau-\rho}^{\tau} \| u^{\varepsilon}(s) - u^0(s) \|^{\theta}
                        \right]  \nonumber\\
     & + 2 c_2^2 \theta \varepsilon^{\frac{\theta}{2} } 
                       \mathbb{E}
                                \bigg[ \int_{\tau-\rho}^{\tau}
                                             \left( \int_{y \in \mathbb{Y} } 
                                                          \| \widetilde{h}(s,y) 
                                            	              + \widetilde{\sigma}(u^{\varepsilon}(s), \mathcal{L}_{u^{\varepsilon}(s)}, y)
                                            	          \|^2
                                                       \nu(dy) 
                                            \right)^{\frac{\theta}{2} } ds 
                        \bigg],                                          
 \end{align}    
 where the third term on the right-hand side can be estimated along with \eqref{scr-17}.

 Then from \eqref{scr-11}-\eqref{scr-18}
 it follows that there exists a constant $M_\theta > 0$ independent of $\rho \in [0,1]$, $\varphi$ and $\psi$, such that for all $r \in [\tau-\rho, \tau]$, $t \ge 3\rho$ and $\varepsilon \in (0, \hat{\varepsilon} )$,
 \begin{align}\label{scr-19}
       & \mathbb{E} 
                  \left[ \sup_{r \in [\tau -\rho, \tau] } 
                               \| u^{\varepsilon} (r) - u^0(r) \|^\theta \| \right]  \nonumber\\
 \le & 2 \mathbb{E} 
                      \left[ \| u^{\varepsilon} (\tau-\rho) - u^0(\tau-\rho)\|^\theta \right]  
          + M_\theta
                    \int_{\tau-\rho}^{\tau}
           	                        \mathbb{E}
           	                           	   \left[ \| u^{\varepsilon}(s) - u^0(s) \|^{\theta}
           	                               \right]
           	                     ds  \nonumber\\    
      & + M_\theta  
                   \int_{\tau-2\rho}^{\tau}
           	                        \mathbb{E}
           	                           	   \left[ \| u^{\varepsilon}(s) - u^0(s) \|^{\theta}
           	                               \right]
           	                     ds  
          + M_\theta \varepsilon^{\frac{\theta}{2} }  \nonumber\\
      & + M_\theta \varepsilon^{\frac{\theta}{2} }  
                           \int_{\tau-\rho}^{\tau}
                  	                \big( \|h(s)\|^\theta
                  	                         + \| \widetilde{h}(s) \|_{L^2(\mathbb{Y}, \nu; \ell^2) }^\theta
                  	                         + \| \widetilde{h}(s) \|^{\theta}_{L^{\theta}(\mathbb{Y}, \nu; \ell^2) }
                  	                \big)  
                  	        ds  \nonumber\\
      & + M_\theta \varepsilon^{\frac{\theta}{2} } 
                          \mathbb{E}
                                	              \left[  \int_{\tau-\rho}^{\tau}   
                                                             \big( \|u^{\varepsilon}(s)\|^{\theta-2} \|u^{\varepsilon}(s)\|_p^{p} 
                                                                    	+ \|u^{\varepsilon}(s)\|^{\theta}  
                                                             \big)
                                                         ds
                                                 \right]  \nonumber\\
 \le & \left( 2 + 3 M_\theta \rho \right) 
              \sup_{r \in [\tau-2\rho, \tau] } 
                   \mathbb{E} 
                         \left[ \| u^{\varepsilon} (r) - u^0(r)\|^\theta \right]  
          + M_\theta \varepsilon^{\frac{\theta}{2} }  \nonumber\\
      & + M_\theta \varepsilon^{\frac{\theta}{2} } e^{\beta \rho}
                           \int_{\tau-\rho}^{\tau} e^{\beta(s-\tau) }
                  	                \big( \|h(s)\|^\theta
                  	                         + \| \widetilde{h}(s) \|_{L^2(\mathbb{Y}, \nu; \ell^2) }^\theta
                  	                         + \| \widetilde{h}(s) \|^{\theta}_{L^{\theta}(\mathbb{Y}, \nu; \ell^2) }
                  	                \big)  
                  	        ds  \nonumber\\
      & + M_\theta \varepsilon^{\frac{\theta}{2} } e^{\beta \rho}
                          \mathbb{E}
                                 \left[  \int_{\tau-\rho}^{\tau} e^{\beta(s-\tau) }  
                                               \big( \|u^{\varepsilon}(s)\|^{\theta-2} \|u^{\varepsilon}(s)\|_p^{p} 
                                                                    	+ \|u^{\varepsilon}(s)\|^{\theta}  
                                               \big)
                                           ds
                                \right]  \nonumber\\     
 \le & \widetilde{M}_{\theta} e^{-\beta t}
                                          \mathbb{E}
                                                 \left[ \| \varphi - \psi \|_\infty^\theta
                                                 \right]  
         + \widetilde{M}_{\theta} \varepsilon^{\frac{\theta}{2} }   
                 e^{-\beta t}
                      \mathbb{E} \left[ \|\varphi\|_{\infty}^{\theta} \right] 
         +  \widetilde{M}_\theta \varepsilon^{\frac{\theta}{2} } \nonumber\\
      & + \widetilde{M}_\theta \varepsilon^{\frac{\theta}{2} } 
             \int_{-\infty}^{\tau} e^{\beta (s-\tau)}
                	             \left( \|g(s)\|^{\theta} 
                	                      + \|h(s)\|^{\theta} 
                	                      + \| \widetilde{h}(s)\|_{L^{\theta}(\mathbb{Y},\nu; \ell^2)}^\theta
                	                        + \| \widetilde{h}(s)\|_{L^{2}(\mathbb{Y},\nu; \ell^2)}^\theta
                	             \right) ds,   
 \end{align}    
 where $\widetilde{M}_\theta > 0$ is a constant depending on $\theta$.
 This completes the proof.
\end{proof}

\begin{remark}
   In proof of Lemma \ref{scr}, 
   the small constant $\delta>0$  introduced in \eqref{dissip1-3}
   is used  to deal with the third and fourth terms on the right-hand side of \eqref{scr-1}
   such that the factor in the front of the  last
   term 
   in \eqref{jul12a1}  is $\varepsilon^{\frac{\theta}{2} }$, which will
   be 
    useful for showing the optimal convergence rate of singleton measure attractor.   
\end{remark}

 In what follows,
 we will present the optimal rate of convergence in terms of the Wasserstein metric for the singleton measure attractors of \eqref{lmvlds-2} as $\varepsilon \to 0$.

   Recall that
  $\mathcal{A}^\varepsilon = \{ \mu_\tau^\varepsilon \}_{\tau \in \mathbb{R} }$
 and
    $\mathcal{A}^0 = \{ \mu_\tau^0 \}_{\tau \in \mathbb{R} }$
are the
 singleton measure attractors in $\mathcal{P}_{\theta}(D_\rho)$ for \eqref{lmvlds-2}  
 and \eqref{lds}, respectively.
 Then we have:

\begin{theorem}\label{conv-rate}
   Suppose that {\bf(H1)}-{\bf(H4)}, \eqref{thetaassum1}, \eqref{thetaassum2}, \eqref{dissip1}, \eqref{dissip2} 
   and \eqref{dissip1-2} hold. 
   Then for every $\tau \in \mathbb{R}$ and $\theta \in \left(2, \tfrac{2(p-2)}{q-2} \right)$,
   there exists a positive constants $C_{\tau, \theta} = C(\tau, \theta) > 0$ and $\hat{\varepsilon} >0$ 
   such that for all $\varepsilon \in (0, \hat{\varepsilon} )$ and $\rho \in [0,1]$, 
   \begin{align} \label{cr-0}
         \mathbb{W}_\theta \left( \mu_\tau^\varepsilon, \mu_{\tau}^0 \right)
         \le C_{\tau, \theta} \, \varepsilon^{\frac{1}{2}},
   \end{align}
   and the convergence order $\frac{1}{2}$ is optimal.
\end{theorem}

\begin{proof}
We first prove \eqref{cr-0} is valid,
and then prove the convergence order $\frac{1}{2}$ is optimal.

 {\bf Step 1}: Prove \eqref{cr-0}.
 
 Given $\tau \in \mathbb{R}$, $t > 0$ and $\rho \in [0,1]$, 
 let $\varphi$,
  $\psi \in L^{\theta}(\Omega, \mathcal{F}_{\tau-t}; D_\rho)$ such that   
         $\mathcal{L}_{\varphi} = \mu_{\tau-t}^\varepsilon$ 
 and $\mathcal{L}_{\psi} = \mu_{\tau-t}^0$.
 Then by the invariance  of $\{ \mu_\tau^\varepsilon \}_{\tau \in \mathbb{R} }$ and $\{ \mu_\tau^0 \}_{\tau \in \mathbb{R} }$, we obtain
    \begin{align} \label{cr-1}
        \mathbb{W}_\theta \left( \mu_{\tau}^\varepsilon, \mu_{\tau}^0 \right) 
    = & \mathbb{W}_\theta 
                       \left( P_{\tau-t, \tau}^{\ast,\varepsilon} \mu_{\tau-t}^\varepsilon,
                                P_{\tau-t, \tau}^{\ast, 0} \mu_{\tau-t}^0
                       \right)  \nonumber\\
 \le & \left( \mathbb{E}
                              \left[ \| u^\varepsilon_\tau(\cdot; \tau-t, \xi) - u^0_\tau(\cdot; \tau-t, \zeta)  
                                       \|_\infty^\theta
                              \right]
              \right)^{\frac{1}{\theta} }.
 \end{align}             
 From \eqref{cr-1} and Lemma \ref{scr}, 
 it follows that there exist the constants $C_{\tau, \theta} > 0$ and $\hat{\varepsilon} = \hat{\varepsilon}(\theta) \in (0, \varepsilon^*)$ such that for all $t \ge 3\rho$ and $\varepsilon \in (0, \hat{\varepsilon} )$,
   \begin{align} \label{cr-2}
       \mathbb{W}_\theta \left( \mu_{\tau}^\varepsilon, \mu_{\tau}^0 \right) 
\le & C_{\tau, \theta} e^{\frac{\beta}{\theta} (\tau-t) }
            \left( \mathbb{E}
                             \left[ \| \varphi - \psi \|_\infty^\theta
                             \right]
            \right)^{\frac{1}{\theta} }  
         + C_{\tau, \theta} e^{\frac{\beta}{\theta} (\tau-t) }
                      \left( \mathbb{E} \left[ \|\varphi\|_{\infty}^{\theta} \right] \right)^{\frac{1}{\theta} }
         + C_{\tau, \theta} \varepsilon^{\frac{1}{2} } \nonumber\\
\le & 2^{1-\frac{1}{\theta} } C_{\tau, \theta} e^{\frac{\beta}{\theta} (\tau-t) }
             \left( \int_{D_\rho} \|y\|_\infty^\theta \mu_{\tau-t}^\varepsilon(dy)
                      + \int_{D_\rho}  \|z\|_\infty^\theta \mu_{\tau-t}^0(dz)
             \right)^{\frac{1}{\theta} }  \nonumber\\
      & + C_{\tau, \theta} e^{\frac{\beta}{\theta} (\tau-t) }
                                   \left( \int_{D_\rho} \|y\|_\infty^\theta \mu_{\tau-t}^\varepsilon(dy) \right)^{\frac{1}{\theta} }
         + C_{\tau, \theta} \varepsilon^{\frac{1}{2} }   \nonumber\\
\le & 2^{1-\frac{1}{\theta} }
C_{\tau, \theta} 
 e^{\frac{\beta}{\theta} (\tau-t) }
               \left( \| \mathcal{A}^\varepsilon(\tau-t) \|_{\mathcal{P}_{\theta}(D_\rho)}
                        + \| \mathcal{A}^0(\tau-t) \|_{\mathcal{P}_{\theta}(D_\rho)}
               \right)  \nonumber\\
     & + C_{\tau, \theta} e^{\frac{\beta}{\theta} (\tau-t) } 
                 \| \mathcal{A}^\varepsilon(\tau-t) \|_{\mathcal{P}_{\theta}(D_\rho)}
         + C_{\tau, \theta} \varepsilon^{\frac{1}{2} } .
 \end{align}
 Letting $t \to \infty$ in \eqref{cr-1}, 
 we obtain that for all $\varepsilon \in (0, \hat{\varepsilon} )$ and $\rho \in [0,1]$,
 \begin{align} \label{cr-3}
       \mathbb{W}_\theta \left( \mu_{\tau}^\varepsilon, \mu_{\tau}^0 \right)
       \le C_{\tau, \theta} \, \varepsilon^{\frac{1}{2} },
 \end{align}
 which implies \eqref{cr-0}.

 {\bf Step 2}:
  Prove  the convergence order $\frac{1}{2}$ in \eqref{cr-3} is optimal.
 
 {Since \eqref{cr-3} holds for all $\rho \in [0,1]$},
 we consider a linear  stochastic lattice system 
 in $\ell^2$ driven by additive L\'{e}vy noise 
 to demonstrate that the
 convergence rate  $\frac{1}{2}$ in \eqref{cr-3}
 cannot be improved in general.
 Consider the stochastic equation:
 \begin{align} \label{special lds}
    du(t) + A u(t) dt + \lambda u(t) dt
    = \sqrt{\varepsilon} 
            \int_{y \in \mathbb{Y} } \widetilde{h}(t, y) \widetilde{N}(dt,dy) \quad \text{in } \ \ell^2,
 \end{align}
 where $u=(u_i)_{i \in \mathbb{Z} }$,
      $(Au)_i = - u_{i+1}(t) + 2u_{i} (t) - u_{i-1}(t)$,
 $\widetilde{h}(s,y) = (\widetilde{h}_i(s,y))_{i \in \mathbb{Z}}$ satisfies \eqref{dissip2} with $\theta = 2$ and $0< \beta < 
  \lambda $. We further assume that
 for every $\tau\in \mathbb{R}$,  
  \be\label{mar01a1}
  \int_{-\infty}^\tau 
  \|\widetilde{h}(s)\|^2 _{L^2(\mathbb{Y}, \nu; \ell^2)}
  ds \neq 0.
  \ee
  
  Note that
  for every initial time $\tau$ and initial value $\xi \in L^2(\Omega, \mathcal{F}_\tau; \ell^2)$,
  system \eqref{special lds}
  has a unique solution
   in $L^2(\Omega, D([\tau, +\infty), \ell^2) )$
   (see e.g. \cite{W2011CSA}).
  Let $\mathbb{A} = A + \lambda I$, where $I$ is the identity operator.
  Then we have
 $$
   \| e^{-\mathbb{A}t} \| \le e^{- \lambda t},  \quad \forall \ t \ge 0,
 $$
and the  solution
of \eqref{special lds} is given by 
 $$ u(t; \tau, \xi)
   = e^{- \mathbb{A}(t - \tau)} \xi
     + \sqrt{\varepsilon} 
              \int_\tau^t 
                  \int_{y \in \mathbb{Y} } 
                           e^{- \mathbb{A}(t-s)} 
                           \widetilde{h}(s, y) \widetilde{N}(ds,dy).
 $$

By the argument of 
  Theorem \ref{sin-attractor},
 we find that \eqref{special lds} has a unique singleton measure attractor in $\mathcal{P}_2(\ell^2)$
 which is  given by an evolution family of 
 probability measures $\{\mu_\tau^\varepsilon\}_{\tau \in \mathbb{R}}$,
 where $\mu_\tau^\varepsilon$
 is the   law  of the random variable:
 $$
 z^\varepsilon(\tau) 
     = \sqrt{\varepsilon} 
             \int_{-\infty}^\tau 
                 \int_{y \in \mathbb{Y}} 
                       e^{- \mathbb{A}(\tau - s)}  \widetilde{h}(s, y) 
             \widetilde{N}(ds,dy) 
     \in L^\theta(\Omega, \ell^2).
     $$
 Indeed, in this case   we have
 \begin{align*}
      u(t; \tau, z^\varepsilon(\tau) )
   = & e^{-\mathbb{A}(t - \tau)} z^\varepsilon(\tau)
         + \sqrt{\varepsilon} 
                 \int_\tau^t 
                     \int_{y \in \mathbb{Y} } 
                              e^{- \mathbb{A}(t - s)} 
                                    \widetilde{h}(s, y) 
                                         \widetilde{N}(ds,dy)  \nonumber\\
   = & \sqrt{\varepsilon} 
                 \int_{-\infty}^\tau 
                      \int_{y \in \mathbb{Y} } 
                                e^{- \mathbb{A}(t - s)} 
                                      \widetilde{h}(s, y) 
                                           \widetilde{N}(ds,dy)
        + \sqrt{\varepsilon} 
                  \int_\tau^t 
                       \int_{y \in \mathbb{Y} } 
                                e^{- \mathbb{A}(t - s)} 
                                      \widetilde{h}(s, y) 
                                           \widetilde{N}(ds,dy)  \nonumber\\
  = & \sqrt{\varepsilon}
               \int_{-\infty}^t 
                    \int_{y \in \mathbb{Y} } 
                         e^{- \mathbb{A}(t - s)}
                               \widetilde{h}(s, y) 
                                     \widetilde{N}(ds,dy)  
  =  z^\varepsilon(t),
 \end{align*}
 which shows that
 the family
  $\{\mu_\tau^\varepsilon\}_{\tau \in \mathbb{R}}$
  is invariant under  $\{P^{*,\varepsilon}_{\tau, t}\}_{t\ge \tau}$;
  that is,
  $P^{*,\varepsilon}_{\tau, t} \mu_\tau^\varepsilon = \mu_t^\varepsilon$
  for all $\tau \in\mathbb{R}$
  and $t\ge \tau$.
   
 If $\varepsilon = 0$, we obtain
 the limiting system
 of \eqref{special lds}: 
 \begin{align} \label{linear ls}
   \frac{d u}{d t} +   A u(t) + \lambda u(t) = 0, \quad t > \tau.
 \end{align}
 Note that
 system \eqref{linear ls}
 has a unique singleton measure attractor
 given by 
 the Dirac measure $\delta_0$
 at $0 \in \ell^2$.   

 Define $\psi_0 (u) = \| u \| $ for every $u\in \ell^2$.
 Then
 $\psi_0 $ is
 a  Lipschitz continuous  function with $\|\psi_0\|_{Lip} \le 1$, and
 thus  we have
 \begin{align} \label{lowbound}
    \mathbb{W}_\theta (\mu_\tau^\varepsilon, \delta_0)
       & \ge \mathbb{W}_1 (\mu_\tau^\varepsilon, \delta_0)
            = \sup_{\| \psi \|_{Lip} \le 1}
                                     \left\{ \int_{\ell^2} \psi(x) \mu_\tau^\varepsilon(dx) 
                                               - \int_{\ell^2} \psi(x) \delta_0(dx)
                                     \right\}  \nonumber\\
       & \ge \int_{\ell^2} \psi_0 (x) \  \mu_\tau^\varepsilon(dx)
                  - \int_{\ell^2} \psi_0 (x) \  \delta_0(dx)
           = \mathbb{E} \left[ \psi_0 (z^\varepsilon (\tau) ) \right] - \psi_0 (0) 
           = \mathbb{E} \left[ \| z^\varepsilon (\tau) \| \right]  \nonumber \\
       & = \sqrt{\varepsilon}
                  \mathbb{E}
                                   \left[
                                          \| \int_{-\infty}^\tau
                                                \int_{y \in \mathbb{Y} } e^{-\mathbb{A}(\tau-s)} \widetilde{h}(s, y) \widetilde{N}(ds,dy)
                                          \| 
                                   \right]
      =  {C}_1 \sqrt{\varepsilon},
 \end{align}
 where
 $C_1 = \mathbb{E}
                              \left[ \| \int_{-\infty}^\tau
                                           \int_{y \in \mathbb{Y} } e^{-\mathbb{A}(\tau-s)} \widetilde{h}(s, y) \widetilde{N}(ds,dy)
                                     \|
                              \right]
                     $.
 By \eqref{dissip2}  we have
 \begin{align*}
    C_1^2 
   \le & \mathbb{E}
                              \left
                              (   \int_{-\infty}^\tau
                                           \int_{y \in \mathbb{Y} } 
                                           \| e^{-\mathbb{A}(\tau-s)} \widetilde{h}(s, y)\|^2  \nu(dy) ds 
                              \right )  \nonumber\\
    \le & \mathbb{E}
                              \left
                              (   \int_{-\infty}^\tau
                                           \int_{y \in \mathbb{Y} } 
                                           \| e^{-\lambda
                                           (\tau-s)} \widetilde{h}(s, y)\|^2  \nu(dy) ds 
                              \right )  \nonumber\\
    \le & \mathbb{E}
                              \left
                              (   \int_{-\infty}^\tau
                                           \int_{y \in \mathbb{Y} } 
                                           \| e^{-\beta 
                                           (\tau-s)} \widetilde{h}(s, y)\|^2  \nu(dy) ds 
                              \right ) 
                              <\infty.
 \end{align*}
 It remains to show $C_1>0$.
 If $C_1=0$, then we have, $\mathbb{P}$-almost surely,
 $$ \| \int_{-\infty}^\tau
              \int_{y \in \mathbb{Y} } e^{-\mathbb{A}(\tau-s)} \widetilde{h}(s, y)  
              \widetilde{N}(ds,dy) \| = 0,
 $$
 and hence
 $$ \mathbb{E} \left[ \| \int_{-\infty}^\tau
                                           \int_{y \in \mathbb{Y} } e^{-\mathbb{A}(\tau-s)} \widetilde{h}(s, y) \widetilde{N}(ds,dy) \|^2 \right] = 0,
 $$
 which implies that
 $$ \int_{-\infty}^\tau
           \int_{y \in \mathbb{Y} }
                  \|e^{-\mathbb{A}(\tau-s)} \widetilde{h}(s, y) \|^2 \nu(dy) ds =0 .
 $$
 Consequently, we get 
   $$ e^{-\mathbb{A}(\tau-s)} \widetilde{h}(s, y) = 0 \  \text{in}\ \ell^2,\  \text{for a.e.}\ (s,y) \in (-\infty, \tau] \times \mathbb{Y},
   $$
  and thus
   $\widetilde{h}(s, y) = 0$
    in $\ell^2$ for a.e. $(s,y) \in (-\infty, \tau] \times \mathbb{Y}$,
 and 
 $ \int_{-\infty}^\tau
                                           \int_{y \in \mathbb{Y} } 
                 \| \widetilde{h}(s, y)   \|^2 \nu(dy) ds   = 0,
 $
 which is in contradiction  \eqref{mar01a1}.
 Therefore, we must have $0<C_1<\infty$,
 which together with \eqref{lowbound} completes the proof.
\end{proof}

\section*{Acknowledgements}

This work is partially supported by the NNSF of China (11971260, 12471167).

%

\end{document}